\documentclass{article}%
\usepackage{cite}
\usepackage{amsmath}
\usepackage{amsfonts}
\usepackage{amssymb}
\usepackage{graphicx}
\usepackage{longtable}
\usepackage[final,notcite,notref]{showkeys}
\usepackage{amscd,enumerate}%
\providecommand{\U}[1]{\protect\rule{.1in}{.1in}}
\newtheorem{theorem}{Theorem}[section]

\newtheorem{assumption}[theorem]{Assumption}

\newtheorem{corollary}[theorem]{Corollary}

\newtheorem{definition}[theorem]{Definition}

\newtheorem{lemma}[theorem]{Lemma}

\newtheorem{proposition}[theorem]{Proposition}
\newtheorem{remark}[theorem]{Remark}

\makeatletter\@addtoreset{equation}{section}\makeatother
\newenvironment{proof}[1][Proof]{\noindent\textbf{#1.} }{\ \rule{0.5em}{0.5em}}
\evensidemargin=\oddsidemargin
\newdimen\dummy
\dummy=\oddsidemargin
\newcommand{\ned}{\boldsymbol{\mathcal N}^I}
\newcommand{\poly}{{\mathcal P}}
\newcommand{\RT}{\mathbf{RT}}
\newcommand{\boldpoly}{\boldsymbol{\mathcal{P}}}

\newcommand{\Picurlcom}{\Pi^{curl,c}_p}
\newcommand{\hatPicurlcom}{\widehat \Pi^{curl,c}_p}
\newcommand{\Picurls}{\Pi^{curl,s}_p}
\newcommand{\hatPicurls}{\widehat\Pi^{curl,s}_p}
\newcommand{\Pigradcom}{\Pi^{grad,c}_{p+1}}

\newcommand{\Pidivcom}{\Pi^{div,c}_{p}}

\allowdisplaybreaks
\begin{document}

\title{Wavenumber-explicit $hp$-FEM analysis for Maxwell's equations with transparent
boundary conditions}
\author{J.M. Melenk\thanks{(melenk@tuwien.ac.at), Institut f\"{u}r Analysis und
Scientific Computing, Technische Universit\"{a}t Wien, Wiedner Hauptstrasse
8-10, A-1040 Wien, Austria}
\and S.A. Sauter\thanks{(stas@math.uzh.ch), Institut f\"{u}r Mathematik,
Universit\"{a}t Z\"{u}rich, Winterthurerstr.~{190}, CH-8057 Z\"{u}rich,
Switzerland}}
\maketitle

\begin{abstract}
The time-harmonic Maxwell equations at high wavenumber $k$ are discretized by
edge elements of degree $p$ on a mesh of width $h$. For the case of a ball as
the computational domain and exact, transparent boundary conditions, we show
quasi-optimality of the Galerkin method under the $k$-explicit scale
resolution condition that a) $kh/p$ is sufficient small and b) $p/\log k$ is
bounded from below. 

\end{abstract}
\tableofcontents

\newpage%


\section*{Glossary and Notation}
\addcontentsline{toc}{section}{\protect\numberline{}Glossary and Notation}%
\newcommand{\fs}{3mm}
\begin{longtable}[c]{lp{0.8\textwidth}}
{\bf general} & \\ \hline \\
$k \ge 1 > 0$
& wavenumber  \\
$\operatorname*{i} $
& imaginary unit $\sqrt{-1}$ \\
$A\lesssim B $
& there exists $C$ independent of $k$, $h$, $p$, and
independent of the functions which,\\
& possibly, appear in $A$ and $B$ so that
$A\leq CB$ holds, see Rem. \ref{Remsmalltilde}\\
{\bf geometry} & \\ \hline \\
$B_1(0)$
& unit ball in ${\mathbb R}^3$ \\
$B_r^+$
& half-balls in ${\mathbb R}^3$ \\
$\Omega$
& domain in ${\mathbb R}^3$ or unit ball $B_1(0)$ in ${\mathbb R}^3$ \\
$\Omega^+$
& ${\mathbb R}^3 \setminus \overline{\Omega}$ \\
$\Gamma = \partial\Omega$
& boundary of $\Omega$ \\
$\mathbf{n}$
& unit normal vector on $\Gamma$ pointing into $\Omega^+$ \\
$\mathbf{n}^{\ast}$
&constant extension of ${\mathbf n}$ into tubular neighborhood of $\Gamma$\\[\fs]
{\bf spaces} & \\ \hline \\
${\mathbf X}:= {\mathbf H}(\Omega,\operatorname{curl})$
& (\ref{defXHcurl}) \\
${\mathbf X}_{0}:= {\mathbf H}_0(\Omega,\operatorname{curl})$
& (\ref{eq:defX0}) \\
${\mathbf H}(\Omega,\operatorname{curl})$, ${\mathbf H}(\Omega,\operatorname{div})$
& (\ref{defXHcurl}), (\ref{DefHOmegadiv}) \\
${\mathbf L}^2(\Omega)$
& space of vector-valued $L^2$-functions \\ 
$H^s(\Omega)$, $H^s(\Gamma)$ 
& scalar-valued Sobolev spaces on $\Omega$, $\Gamma$, Sec.~\ref{SecSobolevOmega}, (\ref{DefHsGammaNorm}) \\
${\mathbf H}^s(\Omega)$
& vector-valued Sobolev spaces on $\Omega$ \\
${\mathbf L}^2_T(\Gamma)$, ${\mathbf H}^s_T(\Gamma)$
& Sobolev space of tangential fields on $\Gamma$, (\ref{DefL2t}), (\ref{DefHsGammaTNorm})\\
${\mathbf H}_{\operatorname{div}}^{-1/2}(\Gamma)$ 
& (\ref{m1/2curldiv}) \\
${\mathbf H}_{\operatorname{curl}}^{-1/2}(\Gamma)$ 
& (\ref{m1/2curldiv}) \\
${\mathbf V}_0 $, ${\mathbf V}_0^\ast$
& spaces of divergence-free functions, see (\ref{defVoa}), (\ref{defVob}) \\
${\mathcal A}(C k^\alpha,\gamma,D)$,
${\mathcal A}^\infty(C k^\alpha,\gamma,D)$, & \\
 ${\mathcal A}(C k^\alpha,\gamma,\Gamma)$
& classes of analytic fcts., Def.~\ref{DefClAnFct}; $C$, $\gamma$, $\alpha$ are independent of $k$ \\[\fs]
{\bf functions} & \\ \hline \\
${\mathbf E}$, 
${\mathbf H}$, 
${\mathbf E}^+$, 
${\mathbf H}^+$ 
& electric and magnetic fields in $\Omega$ and in $\Omega^+$ \\
$Y^m_\ell$, $\lambda_\ell$ 
& eigenfunctions of Laplace-Beltrami, (\ref{eigvalabsLapBelt}); \\
$\tilde Y^m_\ell$
& analytic extension of $Y^m_\ell$ into tubular neighborhood ${\mathcal U}$ of $\Gamma$, (\ref{estgradYtilde}) \\
& for $\Gamma = \partial B_1(0)$, the spherical harmonics \\
${\mathbf T}^m_\ell:=\overrightarrow{\operatorname*{curl}_\Gamma} Y^m_\ell$
& ${\mathbf T}^m_\ell$, $\nabla_\Gamma Y^{m'}_{\ell'}$ are $L^2_T(\Gamma)$-orthgonal basis, 
  cf. \cite[Thm.~{2.4.8}]{Nedelec01} \\
$\iota_\ell$ & index set of indices for eigenvalue 
$\lambda_\ell$, (\ref{eigvalabsLapBelt}); \\
& for the unit sphere, $\iota_\ell = \{-\ell,-\ell+1,\ldots,\ell-1,\ell\}$ \\
$g_{k}$
& Helmholtz fundamental solution, (\ref{fundsol}) \\
${\mathbf G}_{k}$
& Maxwell fundamental solution, (\ref{fundsol}) \\[\fs]
{\bf sesquilinear forms, norms} & \\ \hline \\
$(\cdot,\cdot)$, $(\cdot,\cdot)_{\Gamma}$
& $L^2(\Omega)$-inner prod. and $L^2(\Gamma)$-inner prod. (or duality pairing) \\
$a_k$, $A_k$, $b_k$
&  sesquilinear forms associated with \\
 &Maxwell's equations, (\ref{GalDis}), (\ref{Akcom2scprod}) \\
$b_k^{\operatorname{low}}$, 
$b_k^{\operatorname{high}}$, 
&  low- and high-frequency parts of bilinear form $b_k$, (\ref{defblowbhigh}) \\
$ \left(\kern-.1em \left( \kern-.1em \cdot,\cdot \kern-.1em \right) \kern-.1em \right)
$ &
$ \left(\kern-.1em \left( \kern-.1em \cdot,\cdot \kern-.1em \right) \kern-.1em \right) =
k^2 (\cdot,\cdot)_{L^2(\Omega)} + \operatorname*{i} k b_k((\cdot)^\nabla,(\cdot)^\nabla)$ \\
&$ 
=
k^2 (\cdot,\cdot)_{L^2(\Omega)} + \operatorname*{i} k (T_k (\cdot)_T,(\cdot)_T)_\Gamma$;
see (\ref{2scprod}) \\
$(\cdot,\cdot)_{\operatorname{curl},\Omega,k}$
& $(\operatorname{curl} \cdot, \operatorname{curl} \cdot) + k^2 (\cdot,\cdot)$;  see (\ref{indexedcurl}), \\
$\|\cdot\|_{-1/2,\operatorname{curl}_\Gamma}$, 
$\|\cdot\|_{-1/2,\operatorname{div}_\Gamma}$ 
&  norms on ${\mathbf H}^{-1/2}_{\operatorname{curl}}(\Gamma)$, 
on ${\mathbf H}^{-1/2}_{\operatorname{div}}(\Gamma)$, (\ref{m1/2curldiv}) \\
$\|\cdot\|_{\operatorname{curl},\Omega,k,\lambda}$
& see (\ref{eq:normcurllambda}) \\
$\|\cdot\|_{{\mathcal H},\omega}$ & 
$\|\cdot\|^2_{{\mathcal H},\omega} = \|\nabla \cdot \|^2_{L^2(\omega)} + k^2 \|\cdot\|^2_{L^2(\omega)}$ \\
$\left\langle \cdot,\cdot\right\rangle $, $\left\vert \cdot\right\vert $
& Euclidean scalar product in $\mathbb{C}^{3}$ (conjugation on
second argument), Eucl. norm\\
$N^\prime_{R,p}$, 
$N^\prime_{R,p,q}$, & \\
$M^\prime_{R,p}$, 
$M^\prime_{R,p,q}$, 
$H_{R,p}$
& seminorms to control high order derivatives, (\ref{eq:Nprime}), (\ref{eq:Mprime}), (\ref{eq:Mprimepq})
\\[\fs]
{\bf discrete spaces, meshes} & \\ \hline \\
$\widehat K$ 
& reference tetrahedron \\
${\mathcal T}_h$, $F_K$, $F_K$, $A_K$
& triangulation, element maps, Sec.~\ref{SecCurlConfFEM}, Ass.~\ref{def:element-maps} \\
$S_h $
&  (discrete) subspace of $H^1(\Omega)$; \\
& we require $\nabla S_h \subset {\mathbf X}_h$
and exact seq. property (\ref{eq:intro-exact-sequence}), (\ref{esp}) \\
${\mathbf X}_h $
& (discrete) subspace of ${\mathbf H}(\Omega,\operatorname{curl})$ \\
$h$, $h_{K}$, $p$
& global and local meshwidth (Thm.~\ref{thm:hpFEM-quasioptimality}, (\ref{defhkloc})),
polyn. deg. $p$ \\
$\poly_p$, $\boldpoly_p$
& space of ${\mathbb R}$-valued and ${\mathbb R}^3$-valued polynomials of degree $p$, (\ref{eq:Pp}) \\
$\ned_p(\widehat K)$
& N\'ed\'elec type I space on reference tetrahedron $\widehat K$, (\ref{eq:Np})  \\
$\RT_p(\widehat K)$
& Raviart-Thomas elements on reference tetrahedron $\widehat K$, (\ref{eq:RTp}) \\
$ S_{p+1}({\mathcal T}_h)$, 
$\ned_p({\mathcal T}_h)$,  
&  \\
$\RT_p({\mathcal T}_h)$, 
$Z_p({\mathcal T}_h)$ 
& polyn. spaces on ${\mathcal T}_h$: $H^1(\Omega)$-, ${\mathbf H}(\operatorname{curl},\Omega)$-,  
${\mathbf H}(\operatorname{div},\Omega)$-, and $L^2(\Omega)$-conforming \\[\fs]
{\bf operators} & \\ \hline \\
$\operatorname*{curl}$, $\operatorname*{div}$
& 3D curl and divergence operators \\
$\operatorname*{curl}_\Gamma$, $\operatorname*{div}_\Gamma$
& 2D scalar curl and divergence operators on the surface $\Gamma$, (\ref{sccounttangcurl}) \\
$\overrightarrow{\operatorname*{curl}_\Gamma}$,
$\nabla_\Gamma$,
& 2D vectorial curl and surface gradient operators on $\Gamma$, (\ref{curlvec}) \\
$\Delta_\Gamma$
& surface Laplace-Beltrami operator, (\ref{defLaplBelt}) \\
$T_k $
&  (Maxwell) capacity operator \\
$T_k^{\operatorname{low}}$, 
$T_k^{\operatorname{high}}$ 
&  low- and high-frequency part of capacity operator, (\ref{defblowbhigh}) \\
${\mathcal E}_{\operatorname{curl}}$, ${\mathcal E}_{\operatorname{div}}$,
&  lifting operators (see Thm.~\ref{traceTHM1}) \\
$\Pi_T$, $\Pi_T^+$, $\gamma_T$, $\gamma_T^+$
& trace operators for $\Omega$ and $\Omega^+$, (\ref{eq:trace-operators}), Thm.~\ref{traceTHM1} \\
$(\cdot)_T$
& subscript $T$ indicates tangential trace: ${\mathbf u}_T = \Pi_\tau {\mathbf u}$ \\
$(\cdot)^{\operatorname{high}}$, 
$(\cdot)^{\operatorname{low}}$ 
& ${\mathbf v}^{\operatorname{high}} = H_\Omega {\mathbf v}$, 
 ${\mathbf v}^{\operatorname{low}} = L_\Omega {\mathbf v}$,  \\
$(\cdot)^\nabla$
& gradient part of functions on $\Gamma$, (\ref{utbcHdecomp})\\
$(\cdot)^{\operatorname{curl}}$
& curl part of functions on $\Gamma$, (\ref{utbcHdecomp})\\
$\left[ \cdot\right]_{0,\Gamma}$, $\left[ \cdot\right]_{1,\Gamma}$
& jump operators across $\Gamma$, (\ref{defjumps})\\
$L_\Omega$, $H_\Omega = \operatorname{I} - L_\Omega$, & \\
$L_\Gamma$, $H_\Gamma=\operatorname{I} -L_\Gamma$
& high and low frequency operators with \\
& cut-off parameter $\lambda > 1$, Def.~\ref{DefFreqSplit}, (\ref{eq:defLOmega-strong}) \\
& for the case $\Omega = B_1(0)$, one has $\|L_\Omega\|_{\operatorname{curl},\Omega,k} \leq 1$, 
and $\|H_\Omega\|_{\operatorname{curl},\Omega,k} \leq 2$, (\ref{lowhighuest}) \\
$\mathcal{S}_{-k}^{\operatorname*{Hh}}$
& Helmholtz single layer operator, (\ref{eq:helmholtz-single-layer}) \\
$\mathcal{N}_{-k}^{\operatorname*{Hh}}$
& Helmholtz Newton potential, (\ref{eq:helmholtz-newton-potential}) \\
$\mathcal{S}_{-k}^{\operatorname*{Mw}}$
& Maxwell single layer operator, (\ref{eq:maxwell-single-layer}) \\
$T_\Delta$
& Laplace Dirichlet-to-Neumann operator for $B_1(0)$; Sec.~\ref{SolFormType3})\\
$\Pi^E_h$
& abstract form of $\Picurlcom$ \\
$\Pi^F_h$
& abstract form of $\Pidivcom$ \\
$\Picurlcom$
& ${\mathbf H}(\operatorname{curl})$-conf. commuting diagram projector (matches $\Pigradcom$) \\
$\hatPicurlcom$
& the operator $\Picurlcom$ on the reference tetrahedron \\
$\Picurls$
& ${\mathbf H}(\operatorname{curl})$-conforming approx. operator, \\
& optimal $p$-rates
\emph{simultaneously} in $L^2$ and ${\mathbf H}(\operatorname{curl})$ \\
$\hatPicurls$
& the operator $\Picurls$ on the reference tetrahedron \\
$\Pigradcom$
& $H^1$-conf. commuting diagram projector (matches $\Picurlcom$) \\
$\Pi^{\nabla}$, $\Pi^{\nabla,*}$, $\Pi^{\nabla}_V$, $\Pi^{\nabla,*}_V$
& projection onto $\nabla H^1(\Omega)$
or $V$ w.r.t. $((\cdot,\cdot))$ (Lemma~\ref{Lem2prodwp}) \\
$\Pi^{\nabla}_h$, $\Pi^{\nabla,*}_h$
& projection onto $\nabla S_h$ w.r.t. $((\cdot,\cdot))$ (Lemma~\ref{Lem2prodwp}) \\
$\Pi^{\operatorname*{curl}}$, $\Pi^{\operatorname*{curl},\ast}$
& $\operatorname{I} - \Pi^{\nabla}$ and $\operatorname{I} - \Pi^{\nabla,\ast}$,
see Def.~\ref{DefHelmDecomp}\\
$\Pi^{\operatorname*{curl}}_h$, $\Pi^{\operatorname*{curl},\ast}_h$
& $\operatorname{I} - \Pi^{\nabla}_h$ and $\operatorname{I} - \Pi^{\nabla,\ast}_h$,
see Def.~\ref{DefHelmDecomp}\\
$\Pi^{\operatorname*{comp},\ast}:=L_{\Omega}+\Pi^{\operatorname*{curl},\ast}H_{\Omega}$
& see Def.~\ref{DefHelmDecomp}\\
$\Pi_{h}^{\operatorname*{comp},\ast}:=L_{\Omega}+\Pi_{h}^{\operatorname*{curl},\ast}H_{\Omega}$
& see Def.~\ref{DefHelmDecomp}\\[\fs]
{\bf constants} & \\ \hline \\
$C_{\operatorname*{affine}}$, $C_{\operatorname{metric}}$
&constants measuring the quality of the mesh (Assumption \ref{def:element-maps}) \\
$C_\Gamma$
& continuity of tangential trace operator (\ref{eq:norm-of-Pi_tau}) \\
& (bounded uniformly in $k$) \\
$C_{k}^{L,\Omega}$,\ $C_{k}^{H,\Omega}$
& $\|L_\Omega\|_{\operatorname*{curl},\Omega,k}$, $\|H_\Omega\|_{\operatorname*{curl},\Omega,k}$,
see (\ref{defCkLHOmega}); \\
& for general domains, $C_k^{L,\Omega}$, $C_k^{H,\Omega} = O(k)$; \\
&  for $\Omega = B_1$,  $C_{k}^{L,\Omega}$,$C_{k}^{H,\Omega} = O(1)$ (cf. Cor.~\ref{CorConstantsSphere}) \\
$C_{b,k}^{\nabla,high}$ & continuity of constant of $b_k^{\operatorname{high}}$, see (\ref{defCbkhigh});  \\
& for $\Omega = B_1$: $C_{b,k}^{\nabla,high} = O(1)$ by Cor.~\ref{CorConstantsSphere}\\
$C_{b,k}^{\operatorname{curl},high}$
& continuity const. of $b_k^{\operatorname{curl},\operatorname{high}}$, see (\ref{defCbkhigh}); \\
& for $\Omega = B_1$: $C_{b,k}^{\operatorname{curl},high} = O(1)$ by Cor.~\ref{CorConstantsSphere}\\
$C_{\operatorname{DtN},k}$
& norm of capacity operator $T_k$, (\ref{DefCDtNk}); \\
& for $\Omega = B_1$: $C_{\operatorname{DtN},k} = O(k^2)$ by Cor.~\ref{CorConstantsSphere} \\
$C_{\operatorname{cont},k}$
& cont. const. of $A_k$ and of $((\cdot,\cdot))$, see (\ref{contAkbk}); \\
$C^{\operatorname{high}}_{\operatorname{cont},k}$
& cont. const. of $A_k(H_\Omega \cdot,\cdot)$ , see (\ref{contAkhigh}); \\
& for $\Omega = B_1$: $C_{\operatorname{cont},k} = O(k^3)$ by Cor.~\ref{CorConstantsSphere}\\
$C^{\operatorname{high}}_{b,k}$
& cont. const. of $((\cdot,H_\Omega\cdot))$ and $((H_\Omega\cdot,\cdot))$ , see (\ref{DefCconthighk}); \\
&  for $\Omega = B_1$: $C^{\operatorname{high}}_{b,k} = O(1)$ by Cor.~\ref{CorConstantsSphere}\\
$C_{\Omega,k}$ 
& the embedding constant ${\mathbf V}_0 \subset {\mathbf H}^1(\Omega)$, see (\ref{DefCOmegak});  \\
& for $\Omega = B_1(0)$, $C_{\Omega,k} = 1$ by Lemma~\ref{Lemembedspec} \\
$C^{I}_{k}$ 
& constant in fundamental approximation result, (\ref{bserest}), (\ref{defrhokneu}) \\
$C_{r,k}$ 
& see (\ref{defCrk}), \\
$C_{\#,k}$ 
& see (\ref{defCrk}), \\
$C_{\#\#,k}$ 
& see (\ref{est2ndtermsplit}), \\
$\alpha_j$, $C_{{\mathcal A},j}$, $\gamma_{{\mathcal A},j}$, 
& constants characterizing $k$-dependence in analyticity classes, (\ref{eq:analyticity-classes-d})  \\
$C_{b}$, $C_{b}^\prime$, 
& $O(1)$ constants related to the bilinear form $b_k$; see Props.~\ref{PropFrequbest}, \ref{Propkbk}  \\
$\widetilde C_{b}$ & $O(1)$ cont. const. of $(\!(\cdot,\cdot)\!)$ for $B_1$, if one argument is a 
high frequency, cf.~Prop.~\ref{PropStabCSU}, \\
$C_{\operatorname{rough}}$ & an $O(1)$ constant associated with adjoint solution operator ${\mathcal N}_2$, Prop.~\ref{PropN2Arough}
\\[\fs]
{\bf dual problems and} & \\ 
{\bf approximation } & \\ 
{\bf properties} & \\ \hline \\
$\widehat {\mathcal N}$ 
& (\ref{defNmother}), (\ref{solformdualprob}) \\
${\mathcal N}_1^{\mathcal A}$ 
& solution of an adjoint problem with analytic data, (\ref{adjproblm0}) \\
$\tilde{\eta}_1^{\operatorname{exp}}$
& approximation property related to ${\mathcal N}_1^{\mathcal A}$, (\ref{defetatildeexp}) \\
${\mathcal N}_2$ & adjoint sol. operator, right-hand sides finite regularity (\ref{adjoint3b})\\
${\mathcal N}_3^{\mathcal A}$,  ${\mathcal N}_4^{\mathcal A}$,  & adjoint sol. operator, analytic data, (\ref{smoothdualproblemd}), (\ref{graddoppelKlammer})\\
$\tilde{\eta}_i^{\operatorname{alg}}$, 
$\tilde{\eta}_i^{\operatorname{exp}}$, 
${\eta}_i^{\operatorname{alg}}$, ${\eta}_i^{\operatorname{exp}}$  
& see (\ref{defetatildealg})---(\ref{Defeta7}), (\ref{defeta1exp})---(\ref{etashort}),  \\
 & a tilde indicates that an adjoint sol. operator ${\mathcal N}$ is involved; \\
& $\eta$ indicates a pure approximation property, \\
&  superscript ``$\operatorname{exp}$'' indicates that exponential convergence of $hp$-FEM is expected;  \\
& superscript ``$\operatorname{alg}$'' indicates that algebraic convergence of $hp$-FEM is expected 
\end{longtable}


\section{Introduction\label{sec:intro}}

High-frequency electromagnetic scattering problems are often modelled by the
time-harmonic Maxwell equations (\ref{Maxwellfullspace}), and
the high-frequency case is characterized by a large wavenumber $k>0$. The
solution is then highly oscillatory, and its numerical resolution requires
fine meshes. Besides this natural condition on the discretization, a second,
more subtle issue arises in the high-frequency regime, namely, the difficulty
of Galerkin discretizations to control dispersion errors. That is, in fixed
order methods the discrepancy between the best approximation from the discrete
space and the Galerkin error widens as the wavenumber $k$ increases. It is the
purpose of the present paper to show for a model problem that high order
methods are able to control these dispersion errors and can lead to
quasi-optimality for a fixed (but sufficiently large) number of degrees of
freedom per wavelength.

For the related, simpler case of high-frequency acoustic scattering, which is
modelled by the Helmholtz equation, substantial progress in the understanding
of the dispersive properties of low order and high order methods has been made
in the last decades. We mention the dispersion analyses on regular grids for
fixed order Galerkin methods
\cite{Ihlenburg,Frankp,Ihlenburgbook,BabuskaSauter}, the works
\cite{Ainsworth2004,ainsworth-wajid09,ainsworth-wajid10}, for high order
methods and \cite{ainsworth-monk-muniz06} for a non-conforming discretization
and refer the reader to \cite{MelenkHelmStab2010,mm_stas_helm2} for a more
detailed discussion. These analyses on regular grids give strong arguments for
the numerical observation that high order discretizations are much better
suited to control dispersion errors than low-order methods. For general
meshes, a rigorous argument in favor of high order (conforming and
non-conforming) methods is put forward in the works
\cite{MelenkDiss,MelenkSauterMathComp,mm_stas_helm2,MelenkHelmStab2010,MPS13},
where stability and convergence analyses that are explicit in the mesh size
$h$, the approximation order $p$, and the wavenumber $k$ are provided for
several classes of Helmholtz problems. The underlying principles in these
works are not restricted to FEM discretizations; indeed,
\cite{MelenkLoehndorf} applies these techniques in a Helmholtz BEM context.

The numerical analysis focussing on the dispersive properties of high order
methods for the time-harmonic Maxwell equations is to date significantly less
developed. An analysis on regular grids that is explicit in the polynomial
degree $p$ is available in \cite{ainsworth04b}. A convergence analysis for a
Maxwell problem on general grids that is explicit in the mesh size $h$, the
polynomial degree $p$, and the wavenumber $k$ is the purpose of the present
work. To fix ideas we consider as a model problem the time-harmonic Maxwell
equations (\ref{Maxwellfullspace}) in full space ${\mathbb{R}}^{3}$. Since a
(high order) finite element method (FEM) is our goal, we consider the
equivalent reformulation of the full space problem as a problem in the unit
ball $\Omega=B_{1}(0)$ complemented with transparent boundary conditions on
$\Gamma=\partial\Omega$ (cf. (\ref{electricMWEqOmega})). As we study
conforming Galerkin discretizations, the starting point for the discretization
is the variational formulation (\ref{GalDis}). For this model problem, our
main result is Theorem~\ref{thm:hpFEM-quasioptimality}, which establishes
quasi-optimality of the Galerkin method based on N\'{e}d\'{e}lec type I
elements of degree $p$ under the scale resolution conditions
\begin{equation}
\mbox{ ${kh}/{p} \leq c_1$ }\qquad\mbox{ and }\qquad p\geq c_2\log k
\label{eq:intro-scale-resolution};
\end{equation}
here, $c_2 > 0$ may be chosen arbitrarily and $c_1 > 0$ is 
sufficiently small but independent of $h$, $k$, and $p$. 

We focus here on a conforming Galerkin discretization, which will require the
scale resolution condition (\ref{eq:intro-scale-resolution}) to ensure
existence of the discrete solution. It is worth pointing out that alternatives
to conforming Galerkin methods have been proposed in the literature. Without
attempting completeness and restricting ourselves to approaches based on
higher order polynomials, we mention stabilized methods for Helmholtz
\cite{feng-wu09,feng-wu11,feng-xing13,zhu-wu13} and Maxwell
\cite{feng-wu14,lu-chen-qiu17} problems; hybridizable methods
\cite{chen-lu-xu13}; least-squares type methods \cite{chen-qiu17} and
Discontinous Petrov Galerkin methods,
\cite{petrides-demkowicz17,demkowicz-gopalakrishnan-muga-zitelli12}. In convex
domains or domains with a smooth boundary, $H^{1}$-conforming discretizations
for Maxwell problems can be employed instead of ${\mathbf{H}}%
(\operatorname{curl})$-conforming ones; see \cite{nicaise-tomezyk17},
{\cite{nicaise-tomezyk19}} for a $k$-explicit theory.

We close this introduction by emphasizing that, as in the case of the
Helmholtz equation, the techniques employed in the present work are not
restricted to the model problem under consideration here; in the forthcoming
\cite{MelenkSauterMaxwell_II}, we apply the techniques developed here to
Maxwell's equations equipped with impedance boundary conditions. Finally, a
general note on notation is warranted: as we aim at a $k$-explicit theory, we
indicate constants that (possibly) depend on the wavenumber $k$ by a subscript
$k$.


\subsection{Road Map: Setting}

\label{sec:setting}
Our $k$-explicit convergence analysis of high order FEM for Maxwell's
equations requires a variety of tools including compactness arguments,
$k$-explicit regularity based on decomposing the solution into parts with
finite regularity and analytic parts as developed for the Helmholtz equation,
and commuting diagram operators that are explicit in the polynomial degree
$p$. It may therefore be useful to provide here an outline of the key steps.

The reformulation of the original full space problem (\ref{Maxwellfullspace})
as the problem (\ref{electricMWEqOmega}) in a bounded domain $\Omega
\subset{\mathbb{R}}^{3}$ uses transparent boundary conditions, which are
expressed in terms of the capacity operator $T_{k}$ (see
Section~\ref{SecReformMax} and (\ref{DefTk}) for its explicit series
representation in the case of the unit ball $\Omega= B_{1}(0)$). The pertinent
sesquilinear form that we consider in this work is then
\[
A_{k}({\mathbf{u}},{\mathbf{v}})=(\operatorname{curl}{\mathbf{u}%
},\operatorname{curl}{\mathbf{v}})-k^{2}({\mathbf{u}},{\mathbf{v}%
})-\operatorname*{i}k(T_{k}{\mathbf{u}}_{T},{\mathbf{v}}_{T})_{\Gamma}.
\]
Here, $\left(  \cdot,\cdot\right)  $ is the $L^{2}(\Omega)$ inner product and
$(\cdot,\cdot)_{\Gamma}$ the $L^{2}\left(  \Gamma\right)  $ inner product with
$\Gamma=\partial\Omega$. The subscript $T$ indicates that the tangential
component of the trace is considered. For $\Omega=B_{1}(0)$, our analysis will
be explicit in the wavenumber $k$ and we therefore focus on this case in this
introduction.

\subsection{Road Map: the Maxwell Aspect}

\label{sec:maxwell-aspect}
Let us first discuss the key issues that are specific to discretizations of
Maxwell's equations; in the following Section~\ref{sec:k-explicit-aspect}, we
will focus on the additional difficulties arising from making the error
analysis explicit in $k$. The arguments that we highlight in the current
Section~\ref{sec:maxwell-aspect} are essentially those of
\cite{Monk03,hiptmair-acta,Buffa2005,caorsi-fernandes-raffetto00} and
\cite[Sec.~{7.2}]{Monkbook}.

To understand the Galerkin error for Maxwell's equations, it is imperative to
decompose the various fields in gradient fields and solenoidal fields, both in
$\Omega$ and on the surface $\Gamma$. The tangential field ${\mathbf{u}}_{T}$
is decomposed as a gradient part ${\mathbf{u}}^{\nabla}$ and a (surface)
divergence-free part ${\mathbf{u}}^{\operatorname{curl}}$. The decomposition
${\mathbf{u}}_{T}={\mathbf{u}}^{\nabla}+{\mathbf{u}}^{\operatorname{curl}}$
leads to the decomposition of the sesquilinear form $A_{k}$ as (cf.
(\ref{Akcom2scprod}))%
\[
A_{k}({\mathbf{u}},{\mathbf{v}})=\left(  \operatorname{curl}{\mathbf{u}%
},\operatorname{curl}{\mathbf{v}}\right)  -\operatorname*{i}k\left(
T_{k}{\mathbf{u}}^{\operatorname{curl}},{\mathbf{v}}^{\operatorname{curl}%
}\right)  _{\Gamma}-\underbrace{\left(  k^{2}\left(  {\mathbf{u}},{\mathbf{v}%
}\right)  +\operatorname*{i}k\left(  T_{k}{\mathbf{u}}^{\nabla},{\mathbf{v}%
}^{\nabla}\right)  _{\Gamma}\right)  }_{=:\left(
\kern-.1em%
\left(
\kern-.1em%
\mathbf{u},\mathbf{v}%
\kern-.1em%
\right)
\kern-.1em%
\right)  }%
\]
By \cite[Thm.~{5.3.6}]{Nedelec01}, we have for $\Omega=B_{1}(0)$ sign
properties of the expressions $\operatorname*{i}k(T_{k}{\mathbf{u}%
}^{\operatorname{curl}},{\mathbf{u}}^{\operatorname{curl}})_{\Gamma}$ and
$\left(
\kern-.1em%
\left(
\kern-.1em%
\mathbf{u},\mathbf{u}%
\kern-.1em%
\right)
\kern-.1em%
\right)  $. Furthermore, the curl-part ${\mathbf{u}}^{\operatorname{curl}}$ of
the tangential trace ${\mathbf{u}}_{T}$ vanishes for gradient fields
${\mathbf{u}}=\nabla\varphi$, $\varphi\in H^{1}(\Omega)$. Collecting these
observations, we have:

\begin{enumerate}
[(I)]

\item \label{item:intro-I} \raggedright
$\operatorname*{Re}\left(  (\operatorname*{curl}{\mathbf{u}}%
,\operatorname*{curl}{\mathbf{u}})-\operatorname*{i}k(T_{k}{\mathbf{u}%
}^{\operatorname{curl}},{\mathbf{u}}^{\operatorname{curl}})_{\Gamma}\right)
\geq\Vert\operatorname*{curl}{\mathbf{u}}\Vert^{2}$\quad$\forall\mathbf{u}%
\in\mathbf{X}:={\mathbf{H}}(\Omega,\operatorname{curl})$, \linebreak%
(cf.~\cite[Thm.~{5.3.6}]{Nedelec01}, Lemma~\ref{Lembsplit});

\item \label{item:intro-II} $\operatorname*{Re}\left(
\kern-.1em%
\left(
\kern-.1em%
\nabla\varphi,\nabla\varphi%
\kern-.1em%
\right)
\kern-.1em%
\right)  \geq\left(  k\Vert\nabla\varphi\Vert\right)  ^{2}$\quad
$\forall\varphi\in H^{1}\left(  \Omega\right)  $,$\quad$
(cf.~(\ref{k+normdblescest}));

\item \label{item:intro-III} $A_{k}({\mathbf{u}},\nabla\varphi)=-\left(
\kern-.1em%
\left(
\kern-.1em%
\mathbf{u},\nabla\varphi%
\kern-.1em%
\right)
\kern-.1em%
\right)  \quad\forall\varphi\in H^{1}(\Omega),\ {\mathbf{u}}\in{\mathbf{H}%
}(\Omega,\operatorname*{curl})$,\quad(cf. (\ref{Akcom2scprod}) in conjunction
with Rem.~\ref{rem:decomposition-of-nabla-phi}).
\end{enumerate}

Let ${\mathbf{u}}\in{\mathbf{X}}={\mathbf{H}}(\Omega,\operatorname*{curl})$
and ${\mathbf{u}}_{h}\in{\mathbf{X}}_{h}\subset{\mathbf{X}}$ be its Galerkin
approximation. Then, for arbitrary ${\mathbf{w}}_{h}\in{\mathbf{X}}_{h}$ we
get for the Galerkin error ${\mathbf{e}}_{h}:={\mathbf{u}}-{\mathbf{u}}_{h}$
\begin{align}
\Vert{\mathbf{e}}_{h}\Vert_{\operatorname{curl},\Omega,k}^{2}  &
:=\Vert\operatorname{curl}{\mathbf{e}}_{h}\Vert^{2}+k^{2}\Vert{\mathbf{e}}%
_{h}\Vert^{2}\leq\operatorname*{Re}A_{k}({\mathbf{e}}_{h},{\mathbf{e}}%
_{h})+2\operatorname{Re}\left(
\kern-.1em%
\left(
\kern-.1em%
\mathbf{e}_{h},\mathbf{e}_{h}%
\kern-.1em%
\right)
\kern-.1em%
\right) \label{eq:def-Hcurl-k-intro}\\
&  =\operatorname*{Re}A_{k}({\mathbf{e}}_{h},{\mathbf{u}}-{\mathbf{w}}%
_{h})+2\operatorname{Re}\left(
\kern-.1em%
\left(
\kern-.1em%
{\mathbf{e}}_{h},{\mathbf{u}}-{\mathbf{w}}_{h}%
\kern-.1em%
\right)
\kern-.1em%
\right)  +2\operatorname{Re}\left(
\kern-.1em%
\left(
\kern-.1em%
{\mathbf{e}}_{h},{\mathbf{w}}_{h}-{\mathbf{u}}_{h}%
\kern-.1em%
\right)
\kern-.1em%
\right) \nonumber\\
&  \leq\underbrace{\operatorname*{Re}\left(  A_{k}\left(  {\mathbf{e}}%
_{h},{\mathbf{u}}-{\mathbf{w}}_{h}\right)  +2\left(
\kern-.1em%
\left(
\kern-.1em%
{\mathbf{e}}_{h},{\mathbf{u}}-{\mathbf{w}}_{h}%
\kern-.1em%
\right)
\kern-.1em%
\right)  \right)  }_{=:T_{1}}+2\sup_{{\mathbf{v}}_{h}\in{\mathbf{X}}%
_{h}\backslash\left\{  0\right\}  }\frac{\operatorname{Re}\left(
\kern-.1em%
\left(
\kern-.1em%
{\mathbf{e}}_{h},{\mathbf{v}}_{h}%
\kern-.1em%
\right)
\kern-.1em%
\right)  }{\Vert{\mathbf{v}}_{h}\Vert_{\operatorname{curl},\Omega,k}%
}\underbrace{\Vert{\mathbf{u}}_{h}-{\mathbf{w}}_{h}\Vert_{\operatorname{curl}%
,\Omega,k}}_{\leq\Vert{\mathbf{e}}_{h}\Vert_{\operatorname{curl},\Omega
,k}+\Vert{\mathbf{u}}-{\mathbf{w}}_{h}\Vert_{\operatorname{curl},\Omega,k}}.
\label{eq:intro-20}%
\end{align}
Assuming continuity of $A_{k}$ and $\left(
\kern-.1em%
\left(
\kern-.1em%
\cdot,\cdot%
\kern-.1em%
\right)
\kern-.1em%
\right)  $ with respect to the norm $\Vert\cdot\Vert_{\operatorname{curl}%
,\Omega,k}$ (defined in (\ref{eq:def-Hcurl-k-intro})) this analysis shows that
quasi-optimality of the Galerkin method can be achieved \emph{provided} one
can ensure
\begin{equation}
2\sup_{{\mathbf{v}}_{h}\in{\mathbf{X}}_{h}\backslash\left\{  0\right\}  }%
\frac{\operatorname{Re}\left(
\kern-.1em%
\left(
\kern-.1em%
{\mathbf{e}}_{h},{\mathbf{v}}_{h}%
\kern-.1em%
\right)
\kern-.1em%
\right)  }{\Vert{\mathbf{v}}_{h}\Vert_{\operatorname{curl},\Omega,k}%
\Vert{\mathbf{e}}_{h}\Vert_{\operatorname{curl},\Omega,k}}<1.
\label{eq:duality-1}%
\end{equation}
It is tempting to treat this term by a duality argument. However, the duality
argument cannot be applied directly since the map ${\mathbf{X}}\ni{\mathbf{v}%
}\mapsto\left(
\kern-.1em%
\left(
\kern-.1em%
\cdot,{\mathbf{v}}%
\kern-.1em%
\right)
\kern-.1em%
\right)  \in{\mathbf{X}}^{\prime}$ is not necessarily compact. In the
numerical analysis of Maxwell's equations, this lack of compactness is
addressed by suitable \textquotedblleft continuous\textquotedblright\ and
\textquotedblleft discrete\textquotedblright\ Helmholtz decompositions,
thereby exploiting that ${\mathbf{v}}_{h}$ is from the discrete space
${\mathbf{X}}_{h}$. Specifically, we decompose ${\mathbf{v}}_{h}\in
{\mathbf{X}}_{h}$ in two ways (\textquotedblleft continuous Helmholtz
decomposition\textquotedblright\ and \textquotedblleft discrete Helmholtz
decomposition\textquotedblright) into a divergence-free part and a gradient
part:
\begin{align}
{\mathbf{v}}_{h}  &  =\Pi^{\operatorname{curl},\ast}{\mathbf{v}}_{h}%
+\Pi^{\nabla,\ast}{\mathbf{v}}_{h}\qquad
\mbox{ (with ``continuous'' $\Pi^{\operatorname{curl},\ast} {\mathbf v}_h \in {\mathbf X}$, \
$\Pi^{\nabla,\ast} {\mathbf v}_h\in{\mathbf{X}}\cap \nabla H^1(\Omega)$;
see (\ref{item:intro-IV}))},\label{eq:intro-continuous-helmholtz}\\
{\mathbf{v}}_{h}  &  =\Pi_{h}^{\operatorname{curl},\ast}{\mathbf{v}}_{h}%
+\Pi_{h}^{\nabla,\ast}{\mathbf{v}}_{h}\qquad
\mbox{ (with ``discrete'' $\Pi^{\operatorname{curl},\ast}_h {\mathbf v}_h \in {\mathbf X}_h$, \
$\Pi^{\nabla,\ast}_h {\mathbf v}_h\in{\mathbf{X}_h} \cap \nabla H^1(\Omega)$;
see (\ref{item:intro-V}))}. \label{eq:intro-discrete-helmholtz}%
\end{align}
Since, by construction, $\Pi_{h}^{\nabla,\ast}{\mathbf{v}}_{h}\in{\mathbf{X}%
}_{h}$ is a gradient, the Galerkin orthogonality and the observation
(\ref{item:intro-III}) imply $\left(
\kern-.1em%
\left(
\kern-.1em%
{\mathbf{e}}_{h},\Pi_{h}^{\nabla,\ast}{\mathbf{v}}_{h}%
\kern-.1em%
\right)
\kern-.1em%
\right)  =0.$ Hence, we can write using both decompositions
(\ref{eq:intro-continuous-helmholtz}), (\ref{eq:intro-discrete-helmholtz})
\begin{equation}
\left(
\kern-.1em%
\left(
\kern-.1em%
{\mathbf{e}}_{h},{\mathbf{v}}_{h}%
\kern-.1em%
\right)
\kern-.1em%
\right)  =\left(
\kern-.1em%
\left(
\kern-.1em%
{\mathbf{e}}_{h},\Pi^{\operatorname{curl},\ast}{\mathbf{v}}_{h}%
\kern-.1em%
\right)
\kern-.1em%
\right)  +\left(
\kern-.1em%
\left(
\kern-.1em%
{\mathbf{e}}_{h},\Pi_{h}^{\operatorname{curl},\ast}{\mathbf{v}}_{h}%
-\Pi^{\operatorname{curl},\ast}{\mathbf{v}}_{h}%
\kern-.1em%
\right)
\kern-.1em%
\right)  =:T_{2}+T_{3}. \label{eq:intro-decomposition}%
\end{equation}
The convergence analysis based on this decomposition then relies on a) the
fact that the term $T_{2}=\left(
\kern-.1em%
\left(
\kern-.1em%
{\mathbf{e}}_{h},\Pi^{\operatorname{curl},\ast}{\mathbf{v}}_{h}%
\kern-.1em%
\right)
\kern-.1em%
\right)  $ can be estimated with a duality argument and b) that ${\Pi
}^{\operatorname{curl},\ast}{\mathbf{v}}_{h}-\Pi_{h}^{\operatorname{curl}%
,\ast}{\mathbf{v}}_{h}$ is shown to be small.

The continuous and discrete Helmholtz decompositions
(\ref{eq:intro-continuous-helmholtz}), (\ref{eq:intro-discrete-helmholtz}) are
defined as follows:

\begin{enumerate}
[(I)] \setcounter{enumi}{3}

\item (decomposition in gradient part and divergence-free part)
\label{item:intro-IV} The gradient part $\Pi^{\nabla,\ast}{\mathbf{v}}%
\in\nabla H^{1}\left(  \Omega\right)  $ is defined by the \textquotedblleft
orthogonality\textquotedblright\ condition
\[
\left(
\kern-.1em%
\left(
\kern-.1em%
\nabla\psi,\Pi^{\nabla,\ast}{\mathbf{v}}%
\kern-.1em%
\right)
\kern-.1em%
\right)  =\left(
\kern-.1em%
\left(
\kern-.1em%
\nabla\psi,{\mathbf{v}}%
\kern-.1em%
\right)
\kern-.1em%
\right)  \qquad\forall\psi\in H^{1}(\Omega),
\]
which is well posed by (\ref{item:intro-II}). We set $\Pi
^{\operatorname*{curl},\ast}:=I-\Pi^{\nabla,\ast}$ and denote its range by
${\mathbf{V}}_{0}^{\ast}$. We note that the operators $\Pi^{\nabla,\ast}$ and
$\Pi^{\operatorname*{curl}}$ effect a stable decomposition of the direct sum
${\mathbf{X}}={\mathbf{V}}_{0}^{\ast}\oplus\nabla H^{1}(\Omega)$. The above
mentioned duality argument for $T_{2}$ relies on the compactness of
${\mathbf{X}}\ni{\mathbf{v}}\mapsto\left(
\kern-.1em%
\left(
\kern-.1em%
\cdot,\Pi^{\operatorname{curl},\ast} {\mathbf{v}}
\kern-.1em%
\right)
\kern-.1em%
\right)  \in{\mathbf{X}}^{\prime}$, which is shown in Lemma~\ref{Lemembed} and
ultimately relies on the embedding ${\mathbf{V}}_{0}^{\ast} \subset
{\mathbf{H}}^{1}(\Omega)$.

\item (decomposition of discrete functions in gradient part and discrete
divergence-free part) \label{item:intro-V} Let $S_{h}\subset H^{1}(\Omega)$ be
defined by the requirement that the following (discrete) \emph{exact sequence}
property holds:
\begin{equation}
\begin{CD} S_h @> \nabla >> {\mathbf X}_h @> \operatorname*{curl} >> \operatorname*{curl} {\mathbf X}_h \end{CD} \label{eq:intro-exact-sequence}%
\end{equation}
(cf.~(\ref{exdiscseq_rm}) for the specific example of $hp$-FEM). We define the
discrete version $\Pi_{h}^{\nabla,\ast}:\mathbf{X}\rightarrow\nabla S_{h}$ of
$\Pi^{\nabla,\ast}$ by the \textquotedblleft orthogonality\textquotedblright%
\ condition%
\[
\left(
\kern-.1em%
\left(
\kern-.1em%
\nabla\psi,\Pi_{h}^{\nabla,\ast}{\mathbf{v}}%
\kern-.1em%
\right)
\kern-.1em%
\right)  =\left(
\kern-.1em%
\left(
\kern-.1em%
\nabla\psi,{\mathbf{v}}%
\kern-.1em%
\right)
\kern-.1em%
\right)  \qquad\forall\psi\in S_{h}%
\]
and set $\Pi_{h}^{\operatorname*{curl},\ast}:=I-\Pi_{h}^{\nabla,\ast}$.
\end{enumerate}

While the term $T_{2}$ in (\ref{eq:intro-decomposition}) is treated by a
duality argument, control of the term $T_{3}$ in (\ref{eq:intro-decomposition}%
) relies on the existence of an interpolating projector $\Pi_{h}^{E}$ (and a
companion operator $\Pi_{h}^{F}$) with a commuting diagram property:

\begin{enumerate}
[(I)] \setcounter{enumi}{5}

\item (commuting diagram projector) \label{item:intro-VI} Define ${\mathbf{V}%
}_{0,h}^{\ast}:=\{{\mathbf{v}}\in{\mathbf{V}}_{0}^{\ast}%
\,|\,\operatorname*{curl}{\mathbf{v}}\in\operatorname*{curl}{\mathbf{X}}%
_{h}\}$. We require the existence of an operator $\Pi_{h}^{E}:\mathbf{V}%
_{0,h}^{\ast}+\mathbf{X}_{h}\rightarrow{\mathbf{X}}_{h}$ with the following properties:

\begin{enumerate}
\item \label{item:intro-VI-a} $\Pi_{h}^{E}$ is a projector.

\item \label{item:intro-VI-b} There is a companion operator $\Pi_{h}^{F}$
defined on $\operatorname*{curl}{\mathbf{X}}_{h}$ with the commuting diagram
property $\operatorname*{curl}\Pi_{h}^{E}=\Pi_{h}^{F}\operatorname*{curl}$.

\item \label{item:intro-VI-c} $\Pi_{h}^{E}$ has some approximation properties
in $L^{2}(\Omega)$:
\begin{equation}
k \Vert{\mathbf{v}}-\Pi_{h}^{E}{\mathbf{v}}\Vert\leq\eta_{6}%
^{\operatorname{alg}}\Vert{\mathbf{v}}\Vert_{\operatorname*{curl},\Omega
,k}\qquad\forall{\mathbf{v}}\in{\mathbf{V}}_{0,h}^{\ast}, \label{eq:intro-100}%
\end{equation}
where the parameter $\eta_{6}^{\operatorname{alg}}$ quantifies certain the
approximation properties of ${\mathbf{X}}_{h}$ (e.g., in terms of the mesh
size $h$ and polynomial degree $p$).
\end{enumerate}
\end{enumerate}

\begin{remark}
In the case of $hp$-FEM, the operators $\Pi_{h}^{E}$ and $\Pi_{h}^{F}$ will be
constructed in an element-by-element fashion
(cf.~Def.~\ref{def:element-by-element}) from the operators $\widehat{\Pi}%
_{p}^{\operatorname*{curl},c}$ and $\widehat{\Pi}_{p}^{\operatorname*{div},c}$
(cf.~Theorem~\ref{thm:projection-based-interpolation}) that are defined on the
reference tetrahedron $\widehat{K}$. In the $hp$-FEM setting, the quantity
$\eta_{6}^{\operatorname*{alg}}$ in (\ref{eq:intro-100}) is estimated via
Lemma~\ref{lemma:Picurlcom-approximation},
(\ref{item:lemma:Picurlcom-approximation-iii}) by\footnote{$A\lesssim B$ is
shorthand for $A \leq C B$ for some $C>0$ that is independent of the
wavenumber $k$, the mesh size $h$, the polynomial degree $p$, as well as
functions appearing in $A$ and $B$.} $\eta_{6}^{\operatorname*{alg}}\lesssim
kh/p$; see (\ref{eta6algest}). \hbox{}\hfill\rule{0.8ex}{0.8ex}
\end{remark}

\begin{remark}
\label{Remsmalltilde}Various approximation properties $\eta_{\ell}$ will
appear in our analysis, which depend on the subspace $\mathbf{X}_{h}$. In the
context of $hp$-finite elements, these quantities $\eta_{\ell}$ will depend on
the mesh width $h$, the polynomial order $p$ of approximation, and the
regularity of the functions involved. Given that we focus on high order FEM
with the potential of exponential convergence, we employ the following
notational convention: If some $\eta_{\ell}$ is (generically) algebraically
small in $p$, we employ the superscript \textquotedblleft\thinspace
$\operatorname{alg}$\textquotedblright\ while we use the superscript
\textquotedblleft\thinspace$\operatorname{exp}$\textquotedblright\ if the
quantity is exponentially small.\hbox{}\hfill\rule{0.8ex}{0.8ex}
\end{remark}

The use of the properties of $\Pi_{h}^{E}$ required in (\ref{item:intro-VI})
become apparent if we observe the following arguments for estimating $T_{3}$:

\begin{enumerate}
[(i)]

\item \label{item:intro-step-2} The definition of $\Pi^{\operatorname*{curl}%
,\ast}$ and $\Pi_{h}^{\operatorname*{curl},\ast}$ implies the
\textquotedblleft orthogonality\textquotedblright\
\begin{equation}
\left(
\kern-.1em%
\left(
\kern-.1em%
\nabla\widetilde{\psi}_{h},\left(  \Pi^{\operatorname*{curl},\ast}-\Pi
_{h}^{\operatorname*{curl},\ast}\right)  \mathbf{v}_{h}%
\kern-.1em%
\right)
\kern-.1em%
\right)  =0\qquad\forall\widetilde{\psi}_{h}\in S_{h}. \label{eq:intro-200}%
\end{equation}

\item \label{item:intro-step-1} {}From $\operatorname*{curl}\Pi
^{\operatorname*{curl},\ast}=\operatorname*{curl}\Pi_{h}^{\operatorname*{curl}%
,\ast}=\operatorname*{curl}$ on $\mathbf{X}_{h}$ by
(\ref{eq:intro-continuous-helmholtz}), (\ref{eq:intro-discrete-helmholtz}) we
get for any $\mathbf{v}_{h}\in\mathbf{X}_{h}$%
\begin{align}
\operatorname*{curl}\left(  \Pi_{h}^{\operatorname*{curl},\ast}{\mathbf{v}%
}_{h}-\Pi_{h}^{E}\Pi^{\operatorname*{curl},\ast}{\mathbf{v}}_{h}\right)   &
\overset{\text{(\ref{item:intro-VI-b})}}{=}\operatorname*{curl}\left(  \Pi
_{h}^{\operatorname*{curl},\ast}{\mathbf{v}}_{h}\right)  -\Pi_{h}%
^{F}\operatorname*{curl}\left(  \Pi^{\operatorname*{curl},\ast}{\mathbf{v}%
}_{h}\right)  =\operatorname*{curl}{\mathbf{v}}_{h}-\Pi_{h}^{F}%
\operatorname*{curl}{\mathbf{v}}_{h}\nonumber\\
&  \overset{\text{(\ref{item:intro-VI-b})}}{=}\operatorname*{curl}{\mathbf{v}%
}_{h}-\operatorname*{curl}\Pi_{h}^{E}{\mathbf{v}}_{h}%
\overset{\text{(\ref{item:intro-VI-a})}}{=}\operatorname*{curl}\left(
{\mathbf{v}}_{h}-{\mathbf{v}}_{h}\right)  =0. \label{eq:intro-2000}%
\end{align}

\item \label{item:intro-step-3} By the exact sequence property, the
observation (\ref{eq:intro-2000}) implies that $\Pi_{h}^{\operatorname*{curl}%
,\ast}\mathbf{v}_{h}-\Pi_{h}^{E}\Pi^{\operatorname*{curl},\ast}\mathbf{v}_{h}$
is the gradient of an element of $S_{h}$, i.e., $\Pi_{h}^{\operatorname*{curl}%
,\ast}\mathbf{v}_{h}-\Pi_{h}^{E}\Pi^{\operatorname*{curl},\ast}\mathbf{v}%
_{h}=\nabla\psi_{h}$ for some $\psi_{h}\in S_{h}$.

\item \label{item:intro-step-4} Combining (\ref{item:intro-II}),
(\ref{item:intro-step-3}), (\ref{eq:intro-200}) yields
\begin{align}
k^{2}\left\Vert \left(  \Pi^{\operatorname*{curl},\ast}-\Pi_{h}%
^{\operatorname*{curl},\ast}\right)  \mathbf{v}_{h}\right\Vert ^{2}  &
\overset{\text{(\ref{item:intro-II})}}{\leq}\operatorname*{Re}\left(
\kern-.1em%
\left(
\kern-.1em%
\, \left(  \Pi^{\operatorname*{curl},\ast}-\Pi_{h}^{\operatorname*{curl},\ast
}\right)  \mathbf{v}_{h},\left(  \Pi^{\operatorname*{curl},\ast}-\Pi
_{h}^{\operatorname*{curl},\ast}\right)  \mathbf{v}_{h}%
\kern-.1em%
\right)
\kern-.1em%
\right) \nonumber\\
&  \overset{\text{(\ref{eq:intro-200}), (\ref{item:intro-step-3})}%
}{=}\operatorname*{Re}\left(
\kern-.1em%
\left(
\kern-.1em%
\left(  I-\Pi_{h}^{E}\right)  \Pi^{\operatorname*{curl},\ast}\mathbf{v}%
_{h},\left(  \Pi^{\operatorname*{curl},\ast}-\Pi_{h}^{\operatorname*{curl}%
,\ast}\right)  \mathbf{v}_{h}%
\kern-.1em%
\right)
\kern-.1em%
\right)  . \label{eq:intro-2010}%
\end{align}

\item \label{item:intro-step-5} The continuity of $\left(
\kern-.1em%
\left(
\kern-.1em%
\cdot,\cdot%
\kern-.1em%
\right)
\kern-.1em%
\right)  $ (cf.~(\ref{Cbk}), Prop.~\ref{PropStabCSU}) and using
$\operatorname*{curl}\left(  \left(  I-\Pi_{h}^{E}\right)  \Pi
^{\operatorname*{curl},\ast}\mathbf{v}_{h}\right)  =0=\operatorname*{curl}%
\left(  \Pi^{\operatorname*{curl},\ast}-\Pi_{h}^{\operatorname*{curl},\ast
}\right)  \mathbf{v}_{h}$ (as a consequence of the above calculation), gives
$\left\Vert \left(  I-\Pi_{h}^{E}\right)  \Pi^{\operatorname*{curl},\ast
}\mathbf{v}_{h}\right\Vert _{\operatorname*{curl},\Omega,k}=k\left\Vert
\left(  I-\Pi_{h}^{E}\right)  \Pi^{\operatorname*{curl},\ast}\mathbf{v}%
_{h}\right\Vert $ so that we may continue the estimate (\ref{eq:intro-2010}):
\begin{align*}
k^{2}\left\Vert \left(  \Pi^{\operatorname*{curl},\ast}-\Pi_{h}%
^{\operatorname*{curl},\ast}\right)  \mathbf{v}_{h}\right\Vert ^{2}  &  \leq
C_{\operatorname*{cont},k}\left\Vert \left(  \Pi^{\operatorname*{curl},\ast
}-\Pi_{h}^{\operatorname*{curl},\ast}\right)  \mathbf{v}_{h}\right\Vert
_{\operatorname*{curl},\Omega,k}\left\Vert \left(  I-\Pi_{h}^{E}\right)
\Pi^{\operatorname*{curl},\ast}\mathbf{v}_{h}\right\Vert
_{\operatorname*{curl},\Omega,k}\\
&  =C_{\operatorname*{cont},k}\left(  k\left\Vert \left(  \Pi
^{\operatorname*{curl},\ast}-\Pi_{h}^{\operatorname*{curl},\ast}\right)
\mathbf{v}_{h}\right\Vert \right)  \left(  k\left\Vert \left(  I-\Pi_{h}%
^{E}\right)  \Pi^{\operatorname*{curl},\ast}\mathbf{v}_{h}\right\Vert \right)
.
\end{align*}
Here, the constant $C_{\operatorname*{cont},k}$ could depend on $k$.

\item \label{item:intro-step-6} The final step in treating $T_{3}$ uses the
continuity of $\left(
\kern-.1em%
\left(
\kern-.1em%
\cdot,\cdot%
\kern-.1em%
\right)
\kern-.1em%
\right)  $, the above steps, and the stability of the map ${\mathbf{v}}%
_{h}\mapsto\Pi^{\operatorname{curl},\ast}{\mathbf{v}}_{h}$:
\begin{align*}
|T_{3}|  &  =\left\vert \left(
\kern-.1em%
\left(
\kern-.1em%
{\mathbf{e}}_{h},\left(  \Pi^{\operatorname*{curl},\ast}-\Pi_{h}%
^{\operatorname*{curl},\ast}\right)  \mathbf{v}_{h}%
\kern-.1em%
\right)
\kern-.1em%
\right)  \right\vert \leq C_{\operatorname*{cont},k}\Vert{\mathbf{e}}_{h}%
\Vert_{\operatorname*{curl},\Omega,k}\left\Vert \left(  \Pi
^{\operatorname*{curl},\ast}-\Pi_{h}^{\operatorname*{curl},\ast}\right)
\mathbf{v}_{h}\right\Vert _{\operatorname*{curl},\Omega,k}\\
&  \leq C_{k}\Vert{\mathbf{e}}_{h}\Vert_{\operatorname*{curl},\Omega,k}\left(
k\left\Vert \left(  I-\Pi_{h}^{E}\right)  \Pi^{\operatorname*{curl},\ast
}\mathbf{v}_{h}\right\Vert \right)  \leq C_{k}\eta_{6}^{\operatorname{alg}%
}\Vert{\mathbf{e}}_{h}\Vert_{\operatorname*{curl},\Omega,k}\left\Vert
\Pi^{\operatorname*{curl},\ast}\mathbf{v}_{h}\right\Vert
_{\operatorname*{curl},\Omega,k}\\
&  \leq C_{k}\eta_{6}^{\operatorname{alg}}\Vert{\mathbf{e}}_{h}\Vert
_{\operatorname*{curl},\Omega,k}\Vert{\mathbf{v}}_{h}\Vert
_{\operatorname*{curl},\Omega,k}.
\end{align*}
Here, the constant $C_{k}$ may depend on $k$ (and is, of course, different in
each occurrence). Recalling our starting point (\ref{eq:duality-1}), we
discover that the approximation space ${\mathbf{X}}_{h}$ and the operator
$\Pi_{h}^{E}$ should be such that $\eta_{6}^{\operatorname{alg}}$ can be made
sufficiently small (see (\ref{eta6algest})).
\end{enumerate}

A few more comments concerning the above procedure are in order:

\begin{remark}
\begin{enumerate}
[(a)]

\item The basic estimate (\ref{eq:intro-20}) is formulated in such a way that
one is led to study $\left(
\kern-.1em%
\left(
\kern-.1em%
{\mathbf{e}}_{h},{\mathbf{v}}_{h}%
\kern-.1em%
\right)
\kern-.1em%
\right)  $ with ${\mathbf{v}}_{h}\in{\mathbf{X}}_{h}$ in the \emph{discrete}
space ${\mathbf{X}}_{h}$. This seemingly innocuous choice has far reaching
ramifications. First, one has $\operatorname*{curl}\Pi^{\operatorname*{curl}%
,\ast}\mathbf{v}_{h}=\operatorname*{curl}{\mathbf{v}}_{h}=\operatorname*{curl}%
\Pi_{h}^{\operatorname*{curl},\ast}\mathbf{v}_{h}$, which allows one to
replace the stronger $\Vert\cdot\Vert_{\operatorname*{curl},\Omega,k}$ norm by
the weaker $L^{2}$-norm in the estimates of Step~(\ref{item:intro-step-5}):
$\left\Vert \left(  \Pi^{\operatorname*{curl},\ast}-\Pi_{h}%
^{\operatorname*{curl},\ast}\right)  \mathbf{v}_{h}\right\Vert
_{\operatorname*{curl},\Omega,k}=k\left\Vert \left(  \Pi^{\operatorname*{curl}%
,\ast}-\Pi_{h}^{\operatorname*{curl},\ast}\right)  \mathbf{v}_{h}\right\Vert $
and $\left\Vert \left(  I-\Pi_{h}^{E}\right)  \Pi^{\operatorname*{curl},\ast
}\mathbf{v}_{h}\right\Vert _{\operatorname*{curl},\Omega,k}=k\left\Vert
\left(  I-\Pi_{h}^{E}\right)  \Pi^{\operatorname*{curl},\ast}\mathbf{v}%
_{h}\right\Vert $. Second, the commuting diagram property of $\Pi_{h}^{E}$ and
the (discrete) exact sequence property (\ref{eq:intro-exact-sequence}) are
responsible for the \textquotedblleft orthogonality\textquotedblright%
\ (\ref{eq:intro-200}) (cf. Steps~(\ref{item:intro-step-2}%
)---(\ref{item:intro-step-3})).

\item The $L^{2}$-approximation properties of $\Pi_{h}^{E}$ stipulated in
(\ref{item:intro-VI-c}) can be met because of the special structure of the
space ${\mathbf{V}}_{0,h}^{\ast}$: first, as we discovered in
(\ref{item:intro-IV}), functions from ${\mathbf{V}}_{0}^{\ast}$ are in fact in
${\mathbf{H}}^{1}(\Omega)$. Second, for functions ${\mathbf{v}}\in{\mathbf{V}%
}_{0,h}^{\ast}$ one has that $\operatorname*{curl}{\mathbf{v}}\in
\operatorname*{curl}{\mathbf{X}}_{h}$ is a discrete object. For the specific
case of N\'{e}d\'{e}lec Type I elements of degree $p$, an operator $\Pi
_{h}^{E}$ is constructed on the reference tetrahedron in
Theorem~\ref{thm:projection-based-interpolation} (called $\widehat{\Pi}%
_{p}^{\operatorname*{curl},c}$ there) that exploits these properties and leads
to the quantitative estimate $\eta_{6}^{\operatorname{alg}} = O(hk/p)$. We
flag at this point that, while the space ${\mathbf{V}}_{0}^{\ast}$ is a space
of divergence-free functions, the operator $\widehat{\Pi}_{p}%
^{\operatorname*{curl},c}$ is additionally defined for (elementwise) smooth
(actually, elementwise ${\mathbf{H}}^{1}(\operatorname*{curl})$) functions.
This property will be needed in Section~\ref{sec:k-explicit-aspect} below to
argue the benefits of high order methods. \hbox{}\hfill\rule{0.8ex}{0.8ex}

\end{enumerate}
\end{remark}

\subsection{Road Map: $k$-explicit Estimates\label{sec:k-explicit-aspect}}

The argument outlined above does not take into account how the wavenumber $k$
enters the estimates, which occurs in various places, for example, in the
continuity of $A_{k}$ and $\left(
\kern-.1em%
\left(
\kern-.1em%
\cdot,\cdot%
\kern-.1em%
\right)
\kern-.1em%
\right)  $, the stability of the map $\Pi^{\operatorname*{curl},\ast}$, and
the regularity properties of the solution $\mathbf{z}$ of the dual problem
$A_{k}(\cdot,\mathbf{z})=\left(
\kern-.1em%
\left(
\kern-.1em%
\cdot,\Pi^{\operatorname*{curl},\ast}\mathbf{v}_{h}%
\kern-.1em%
\right)
\kern-.1em%
\right)  $. Indeed, care is required as we only have the $k$-dependent
continuity bounds (cf.~Cor.~\ref{CorConstantsSphere})
\begin{equation}
\left\vert \left(
\kern-.1em%
\left(
\kern-.1em%
{\mathbf{v}},{\mathbf{w}}%
\kern-.1em%
\right)
\kern-.1em%
\right)  \right\vert +\left\vert A_{k}\left(  {\mathbf{v}},{\mathbf{w}%
}\right)  \right\vert \leq Ck^{3}\Vert{\mathbf{v}}\Vert_{\operatorname*{curl}%
,\Omega,k}\Vert{\mathbf{w}}\Vert_{\operatorname*{curl},\Omega,k}.
\label{eq:intro-continuity}%
\end{equation}


\subsubsection{Continuity of $A_{k}$, $\left(
\kern-.1em%
\left(
\kern-.1em%
\cdot,\cdot%
\kern-.1em%
\right)
\kern-.1em%
\right)  $ and Treatment of $T_{1}$\label{sec:intro-T1T2}}

The fundamental ingredient for $k$-explicit bounds that are useful for the
analysis of high-order FEM is the ability to decompose functions ${\mathbf{u}%
}\in{\mathbf{X}}$ into \textquotedblleft high-frequency\textquotedblright%
\ parts $H_{\Omega}{\mathbf{u}}$ and \textquotedblleft
low-frequency\textquotedblright\ parts $L_{\Omega}{\mathbf{u}}$. An
overarching theme of the present work is that the high-frequency component
$H_{\Omega}{\mathbf{u}}$ leads to estimates \emph{uniform} in $k$ in the
expected Sobolev norms; the low-frequency component $L_{\Omega}{\mathbf{u}}$
involves $k$-dependencies, but is smooth (even analytic), which can be
exploited by high order approximation spaces. We note that such decompositions
${\mathbf{u}}=H_{\Omega}{\mathbf{u}}+L_{\Omega}{\mathbf{u}}$ of functions
entail corresponding decompositions of sesquilinear forms such as $A_{k}$ and
$\left(
\kern-.1em%
\left(
\kern-.1em%
\cdot,\cdot%
\kern-.1em%
\right)
\kern-.1em%
\right)  $.

The frequency splitting operators $L_{\Omega}$ and $H_{\Omega}$ are motivated
by an analysis of the $k$-dependence of the continuity constants of $A_{k}$
and $\left(
\kern-.1em%
\left(
\kern-.1em%
\cdot,\cdot%
\kern-.1em%
\right)
\kern-.1em%
\right)  $, e.g., in the bound $\left\vert A_{k}\left(  {\mathbf{u}%
},{\mathbf{v}}\right)  \right\vert \leq C_{\operatorname*{cont},k}%
\Vert{\mathbf{u}}\Vert_{\operatorname{curl},\Omega,k}\Vert{\mathbf{v}}%
\Vert_{\operatorname{curl},\Omega,k}$. One discovers that it is the capacity
operator $T_{k}$ that introduces a $k$-dependence in $C_{\operatorname*{cont}%
,k}$. Inspection of the series expansion of $T_{k}$ in (\ref{DefTk}) (see in
particular Lemma~\ref{Lemzlest}, which gives sharp bounds for the symbol of
the operator $T_{k}$) shows that the $k$-dependence is due to the
low-frequency parts of ${\mathbf{u}}_{T}$. Having identified these components
as the culprits for unfavorable $k$-dependencies, we introduce in
Definition~\ref{DefFreqSplit} the low-frequency operator $L_{\Omega
}:{\mathbf{X}}\rightarrow{\mathbf{X}}$ and the high-frequency operator
$H_{\Omega}=I-L_{\Omega}$ that have, for the case $\Omega=B_{1}(0)$ considered
here, the following properties:

\begin{enumerate}
[(I)] \setcounter{enumi}{6}

\item (stability) \label{item:intro-VII} $\|L_{\Omega}{\mathbf{u}%
}\|_{\operatorname{curl},\Omega,k} \leq\|{\mathbf{u}}\|_{\operatorname{curl}%
,\Omega,k}$ and $\|H_{\Omega}{\mathbf{u}}\|_{\operatorname{curl},\Omega,k}
\leq2 \|{\mathbf{u}}\|_{\operatorname{curl},\Omega,k}$ (cf.~(\ref{lowhighuest})).

\item (smoothness) \label{item:intro-VIII} $L_{\Omega}{\mathbf{u}}$ is
analytic. Specifically, there are $C$, $\alpha$, $\gamma> 0$ independent of
$k$ and ${\mathbf{u}}$ such that $L_{\Omega}{\mathbf{u}} \in{\mathcal{A}}(C
k^{\alpha}\|{\mathbf{u}}\|_{\operatorname{curl},\Omega,k},\gamma,\Omega)$ with
the analyticity class ${\mathcal{A}}$ given by Def.~\ref{DefClAnFct}
(cf.~Theorem~\ref{TheoLau}).

\item ($k$-uniform continuity at the expense of a compact perturbation)
\label{item:intro-IX} For some $C > 0$ independent of $k$
(cf.~Prop.~\ref{PropStabCSU} and Lemma~\ref{LemCDtN} in conjunction with
Cor.~\ref{CorConstantsSphere}):
\begin{align}
\label{eq:intro-refined-continuity-(())}\left\vert \left(
\kern-.1em%
\left(
\kern-.1em%
H_{\Omega}{\mathbf{u}},{\mathbf{v}}%
\kern-.1em%
\right)
\kern-.1em%
\right)  \right\vert +\left\vert \left(
\kern-.1em%
\left(
\kern-.1em%
{\mathbf{v}},H_{\Omega}{\mathbf{u}}%
\kern-.1em%
\right)
\kern-.1em%
\right)  \right\vert  &  \leq C\Vert{\mathbf{u}}\Vert_{\operatorname*{curl}%
,\Omega,k}\Vert{\mathbf{v}}\Vert_{\operatorname*{curl},\Omega,k},\\
\left\vert A_{k}\left(  {H_{\Omega}{\mathbf{u}},{\mathbf{v}}}\right)
\right\vert +\left\vert A_{k}\left(  {\mathbf{v}},H_{\Omega}{\mathbf{u}%
}\right)  \right\vert  &  \leq C\Vert{\mathbf{u}}\Vert_{\operatorname*{curl}%
,\Omega,k}\Vert{\mathbf{v}}\Vert_{\operatorname*{curl},\Omega,k} .
\end{align}

\end{enumerate}

The refined continuity properties of $A_{k}$ and $\left(
\kern-.1em%
\left(
\kern-.1em%
\cdot,\cdot%
\kern-.1em%
\right)
\kern-.1em%
\right)  $ given in (\ref{item:intro-IX}) allow us to estimate the terms
$T_{1}$ in the basic error estimate (\ref{eq:intro-20}) explicitly in $k$.
Abbreviating $\mathbf{v}:=\mathbf{u}-\mathbf{w}_{h}$ and decomposing
$\mathbf{v}^{\operatorname*{low}}:=L_{\Omega}\mathbf{v}$ and $\mathbf{v}%
^{\operatorname*{high}}:=H_{\Omega}\mathbf{v}$ we write%
\begin{align*}
T_{1}  &  =\operatorname{Re}\left(  A_{k}({\mathbf{e}}_{h},{\mathbf{v}%
})+2\left(
\kern-.1em%
\left(
\kern-.1em%
{\mathbf{e}}_{h},{\mathbf{v}}%
\kern-.1em%
\right)
\kern-.1em%
\right)  \right)  =\operatorname{Re}\left(  A_{k}({\mathbf{e}}_{h},H_{\Omega
}{\mathbf{v}})+2\left(
\kern-.1em%
\left(
\kern-.1em%
{\mathbf{e}}_{h},H_{\Omega}{\mathbf{v}}%
\kern-.1em%
\right)
\kern-.1em%
\right)  \right)  +\operatorname{Re}\left(  A_{k}({\mathbf{e}}_{h},L_{\Omega
}{\mathbf{v}})+2\left(
\kern-.1em%
\left(
\kern-.1em%
{\mathbf{e}}_{h},L_{\Omega}{\mathbf{v}}%
\kern-.1em%
\right)
\kern-.1em%
\right)  \right) \\
&  =\underbrace{\operatorname{Re}\left(  A_{k}({\mathbf{e}}_{h},H_{\Omega
}{\mathbf{v}})+2\left(
\kern-.1em%
\left(
\kern-.1em%
{\mathbf{e}}_{h},H_{\Omega}{\mathbf{v}}%
\kern-.1em%
\right)
\kern-.1em%
\right)  \,\right)  }_{=:T_{1,1}}+\underbrace{\operatorname{Re}%
(\operatorname{curl}{\mathbf{e}}_{h},\operatorname{curl}L_{\Omega}{\mathbf{v}%
})}_{=:T_{1,2}}+\operatorname{Re}\Bigl(\underset{=:T_{1,3}\left(  {\mathbf{e}%
}_{h},L_{\Omega}\mathbf{v}\right)  }{\underbrace{-\operatorname*{i}k\left(
T_{k}{\mathbf{e}}_{h}^{\operatorname{curl}},\left(  L_{\Omega}{\mathbf{v}%
}\right)  ^{\operatorname{curl}}\right)  _{\Gamma}+\left(
\kern-.1em%
\left(
\kern-.1em%
{\mathbf{e}}_{h},L_{\Omega}{\mathbf{v}}%
\kern-.1em%
\right)
\kern-.1em%
\right)  }}\Bigr).
\end{align*}
The sesquilinear forms in $T_{1,1}$ and $T_{1,2}$ have good continuity
properties (cf.~(\ref{item:intro-IX}) and (\ref{item:intro-VII}) respectively)
and can be estimated with $k$-independent constants. The term $T_{1,3}$ is
amenable to a treatment by a duality argument: Let $%
\mbox{\boldmath$ \psi$}%
\in{\mathbf{X}}$ solve $A_{k}(\cdot,%
\mbox{\boldmath$ \psi$}%
)=T_{1,3}\left(  \cdot,L_{\Omega}\mathbf{v}\right)  $. By Galerkin
orthogonality satisfied by ${\mathbf{e}}_{h}$ and the stability estimate
(\ref{eq:intro-continuity}), one arrives at
\begin{equation}
\left\vert T_{1,3}\left(  L_{\Omega} {\mathbf{e}}_{h} ,{\mathbf{u}%
}-{\mathbf{w}}_{h}\right)  \right\vert =|A_{k}({\mathbf{e}}_{h},%
\mbox{\boldmath$ \psi$}%
)|\leq Ck^{3}\Vert{\mathbf{e}}_{h}\Vert_{\operatorname*{curl},\Omega,k}\inf_{%
\mbox{\boldmath$ \psi$}%
_{h}\in{\mathbf{X}}_{h}}\Vert%
\mbox{\boldmath$ \psi$}%
-%
\mbox{\boldmath$ \psi$}%
_{h}\Vert_{\operatorname*{curl},\Omega,k}. \label{eq:intro-300}%
\end{equation}
Since $L_{\Omega}({\mathbf{u}}-{\mathbf{w}}_{h})$ is an analytic function by
(\ref{item:intro-VIII}) and the geometry $\Gamma=\partial B_{1}(0)$ is
analytic so is the dual solution $%
\mbox{\boldmath$ \psi$}%
$. As discussed in Proposition \ref{PropRegN1A}, one has the following
analytic regularity assertion:

\begin{enumerate}
[(I)] \setcounter{enumi}{9}

\item \label{item:intro-X} Given ${\mathbf{r}}\in{\mathbf{X}}$, the solutions
$%
\mbox{\boldmath$ \psi$}%
_{1}$, $%
\mbox{\boldmath$ \psi$}%
_{2}\in{\mathbf{X}}$ of $A_{k}(\cdot,%
\mbox{\boldmath$ \psi$}%
_{1})=T_{1,3}\left(  \cdot,L_{\Omega}{\mathbf{r}}\right)  $ and $A_{k}(\cdot,%
\mbox{\boldmath$ \psi$}%
_{2})=\left(
\kern-.1em%
\left(
\kern-.1em%
\cdot,L_{\Omega}{\mathbf{r}}%
\kern-.1em%
\right)
\kern-.1em%
\right)  $ are analytic in $\overline{\Omega}$ and satisfy $%
\mbox{\boldmath$ \psi$}%
_{1}$, $%
\mbox{\boldmath$ \psi$}%
_{2}\in{\mathcal{A}}(Ck^{\alpha}\Vert{\mathbf{r}}\Vert_{\operatorname*{curl}%
,\Omega,k},\gamma)$ for some $C$, $\gamma$, $\alpha\geq0$ independent of $k$
and ${\mathbf{r}}$. The analyticity classes ${\mathcal{A}}$ are introduced in
Def.~\ref{DefClAnFct}.
\end{enumerate}

Since, by (\ref{item:intro-X}), the solution $%
\mbox{\boldmath$ \psi$}%
$ in (\ref{eq:intro-300}) is analytic, exponential approximation properties of
$hp$-FEM spaces will be able to offset the algebraic factor $k^{3}$ in
(\ref{eq:intro-300}). Indeed, we will show in
Lemma~\ref{lemma:Picurls-approximation},
(\ref{item:lemma:Picurls-approximation-ii}) for N\'{e}d\'{e}lec elements of
degree $p$ that the infimum in (\ref{eq:intro-300}) decays exponentially in
$p$ (provided that $kh/p$ is sufficiently small).

\subsubsection{Treatment of $T_{2}$: the $k$-explicit Duality Argument for
$\Pi^{\operatorname*{curl},\ast}\mathbf{v}_{h}$\label{sec:intro-T4}}

The analysis of the terms $T_{2}=\left(
\kern-.1em%
\left(
\kern-.1em%
{\mathbf{e}}_{h},\Pi^{\operatorname{curl},\ast}{\mathbf{v}}_{h}%
\kern-.1em%
\right)
\kern-.1em%
\right)  $ and $T_{3}=\left(
\kern-.1em%
\left(
\kern-.1em%
{\mathbf{e}}_{h},\Pi^{\operatorname{curl},\ast}{\mathbf{v}}_{h}-\Pi
_{h}^{\operatorname{curl},\ast}{\mathbf{v}}_{h}%
\kern-.1em%
\right)
\kern-.1em%
\right)  $ and arising in (\ref{eq:intro-decomposition}) requires us to make
the decompositions ${\mathbf{v}}_{h}=\Pi^{\operatorname{curl},\ast}%
{\mathbf{v}}_{h}+\Pi^{\nabla,\ast}{\mathbf{v}}_{h}=\Pi_{h}%
^{\operatorname{curl},\ast}{\mathbf{v}}_{h}+\Pi_{h}^{\nabla,\ast}{\mathbf{v}%
}_{h}$ in a more careful, $k$-dependent way. The stability property
(\ref{item:intro-IX}) implies $\Vert\Pi^{\nabla,\ast}H_{\Omega}{\mathbf{v}%
}\Vert_{\operatorname*{curl},\Omega,k}\leq C\Vert H_{\Omega}{\mathbf{v}}%
\Vert_{\operatorname*{curl},\Omega,k}\leq C\Vert{\mathbf{v}}\Vert
_{\operatorname*{curl},\Omega,k}$ with $C>0$ independent of $k$ so that
\begin{equation}
\Vert\Pi^{\operatorname*{curl},\ast}H_{\Omega}{\mathbf{v}}\Vert
_{\operatorname*{curl},\Omega,k}\leq C\Vert{\mathbf{v}}\Vert
_{\operatorname*{curl},\Omega,k}, \label{eq:intro-stability-of-Picurl}%
\end{equation}
again with $C>0$ independent of $k$ (cf.\ also (\ref{stabnablacurlcurl})).
These favorable estimates for $H_{\Omega}{\mathbf{v}}$ instead of
${\mathbf{v}}$ directly suggest that we should study, for ${\mathbf{v}}_{h}%
\in{\mathbf{X}}_{h}$, the following decompositions instead of
(\ref{eq:intro-continuous-helmholtz}) (\ref{eq:intro-discrete-helmholtz}):
\begin{align}
{\mathbf{v}}_{h}  &  =\Pi^{\operatorname*{comp},\ast}\mathbf{v}_{h}%
+\Pi^{\nabla,\ast}H_{\Omega}{\mathbf{v}}_{h}\quad\text{with\quad}%
\Pi^{\operatorname*{comp},\ast}:=L_{\Omega}+\Pi^{\operatorname*{curl},\ast
}H_{\Omega},\label{eq:intro-continuous-helmholtz-split}\\
{\mathbf{v}}_{h}  &  =\Pi_{h}^{\operatorname*{comp},\ast}\mathbf{v}_{h}%
+\Pi_{h}^{\nabla,\ast}H_{\Omega}{\mathbf{v}}_{h}\quad\text{with\quad}\Pi
_{h}^{\operatorname*{comp},\ast}:=L_{\Omega}+\Pi_{h}^{\operatorname*{curl}%
,\ast}H_{\Omega}. \label{eq:intro-discrete-helmholtz-split}%
\end{align}
and consequently replacing $T_{2}$ and $T_{3}$ by $\widetilde{T}_{2}=\left(
\kern-.1em%
\left(
\kern-.1em%
{\mathbf{e}}_{h},\Pi^{\operatorname{comp},\ast}{\mathbf{v}}_{h}%
\kern-.1em%
\right)
\kern-.1em%
\right)  $ and $\widetilde{T}_{3}=\left(
\kern-.1em%
\left(
\kern-.1em%
{\mathbf{e}}_{h},\Pi^{\operatorname{comp},\ast}{\mathbf{v}}_{h}-\Pi
_{h}^{\operatorname{comp},\ast}{\mathbf{v}}_{h}%
\kern-.1em%
\right)
\kern-.1em%
\right)  .$
The duality argument for ${\widetilde{T}}_{2}=\left(
\kern-.1em%
\left(
\kern-.1em%
{\mathbf{e}}_{h},\Pi^{\operatorname*{comp},\ast}\mathbf{v}_{h}%
\kern-.1em%
\right)
\kern-.1em%
\right)  =\left(
\kern-.1em%
\left(
\kern-.1em%
{\mathbf{e}}_{h},L_{\Omega}{\mathbf{v}}_{h}%
\kern-.1em%
\right)
\kern-.1em%
\right)  +\left(
\kern-.1em%
\left(
\kern-.1em%
{\mathbf{e}}_{h},\Pi^{\operatorname*{curl},\ast}H_{\Omega}{\mathbf{v}}_{h}%
\kern-.1em%
\right)
\kern-.1em%
\right)  $ is split into two duality arguments. For the first term, one
observes again that $L_{\Omega}{\mathbf{v}}_{h}$ is analytic and so will be
the appropriate dual solution by (\ref{item:intro-X}), which in turn means
that exponential approximability of $hp$-FEM space can be brought to bear. For
the second term, the duality argument requires much more care since
$\Pi^{\operatorname*{curl},\ast}H_{\Omega}{\mathbf{v}}_{h}$ has only finite
regularity. We have (cf.\ Prop.~\ref{PropN2Arough}):

\begin{enumerate}
[(I)] \setcounter{enumi}{10}

\item \label{item:intro-XI} The solution $%
\mbox{\boldmath$ \psi$}%
$ of $A_{k}(\cdot,%
\mbox{\boldmath$ \psi$}%
)=\left(
\kern-.1em%
\left(
\kern-.1em%
\cdot,\Pi^{\operatorname*{curl},\ast}H_{\Omega}{\mathbf{v}}_{h}%
\kern-.1em%
\right)
\kern-.1em%
\right)  $ can be decomposed as $%
\mbox{\boldmath$ \psi$}%
=%
\mbox{\boldmath$ \psi$}%
_{H^{2}}+%
\mbox{\boldmath$ \psi$}%
_{\mathcal{A}}$ with $k^{2}\Vert%
\mbox{\boldmath$ \psi$}%
_{H^{2}}\Vert+\Vert%
\mbox{\boldmath$ \psi$}%
_{H^{2}}\Vert_{{\mathbf{H}}^{2}(\Omega)}\leq C\Vert H_{\Omega}{\mathbf{v}}%
_{h}\Vert_{\operatorname{curl},\Omega,k}\leq C\Vert{\mathbf{v}}_{h}%
\Vert_{\operatorname{curl},\Omega,k}$ and $%
\mbox{\boldmath$ \psi$}%
_{\mathcal{A}}\in{\mathcal{A}}(Ck^{\alpha}\Vert H_{\Omega}{\mathbf{v}}%
_{h}\Vert_{\operatorname{curl},\Omega,k},\gamma,\Omega)$ for some $C$,
$\gamma$, $\alpha\geq0$ independent of $k$ (cf. Def.~\ref{DefClAnFct} for the
definition of the analyticity class ${\mathcal{A}}$).
\end{enumerate}

The decomposition of (\ref{item:intro-XI}) into a part $%
\mbox{\boldmath$ \psi$}%
_{H^{2}}$ with finite regularity in conjunction with $k$-uniform control of
the second derivatives and an analytic part $%
\mbox{\boldmath$ \psi$}%
_{\mathcal{A}}$ is shown in Section~\ref{SecRegDualSol}; it relies on a
solution formula based on Green's function for the Helmholtz equation and the
decomposition is then inferred from the one developed in
\cite{MelenkSauterMathComp}.

\subsubsection{Treatment of $\protect\widetilde{T}_{3}$: Estimating $\left(
\Pi^{\operatorname*{comp},\ast}-\Pi_{h}^{\operatorname*{comp},\ast}\right)
\mathbf{v}_{h}$}

\label{sec:intro-T3}
For the final term, $\widetilde{T}_{3}=\left(
\kern-.1em%
\left(
\kern-.1em%
{\mathbf{e}}_{h},\left(  \Pi^{\operatorname*{comp},\ast}-\Pi_{h}%
^{\operatorname*{comp},\ast}\right)  \mathbf{v}_{h}%
\kern-.1em%
\right)
\kern-.1em%
\right)  $, a new type of duality argument appears. We start by writing%
\begin{align*}
\widetilde{T}_{3}  &  =\left(
\kern-.1em%
\left(
\kern-.1em%
{\mathbf{e}}_{h},\left(  \Pi^{\operatorname*{comp},\ast}-\Pi_{h}%
^{\operatorname*{comp},\ast}\right)  \mathbf{v}_{h}%
\kern-.1em%
\right)
\kern-.1em%
\right)  =\left(
\kern-.1em%
\left(
\kern-.1em%
{\mathbf{e}}_{h},L_{\Omega}\left(  \Pi^{\operatorname*{comp},\ast}-\Pi
_{h}^{\operatorname*{comp},\ast}\right)  \mathbf{v}_{h}%
\kern-.1em%
\right)
\kern-.1em%
\right)  +\left(
\kern-.1em%
\left(
\kern-.1em%
{\mathbf{e}}_{h},H_{\Omega}\left(  \Pi^{\operatorname*{comp},\ast}-\Pi
_{h}^{\operatorname*{comp},\ast}\right)  \mathbf{v}_{h}%
\kern-.1em%
\right)
\kern-.1em%
\right) \\
&  =:\widetilde{T}_{3,1}+\widetilde{T}_{3,2}.
\end{align*}
Exploiting the analyticity of $L_{\Omega}\left(  \Pi^{\operatorname*{comp}%
,\ast}-\Pi_{h}^{\operatorname*{comp},\ast}\right)  \mathbf{v}_{h}$, the first
term, $\widetilde{T}_{3,1}$ can be treated by a duality argument as in
Section~\ref{sec:intro-T1T2}. For the second term, $\widetilde{T}_{3,2}$, we
use (\ref{item:intro-IX}) to estimate
\begin{align*}
|\widetilde{T}_{3,2}|  &  =\left\vert \left(
\kern-.1em%
\left(
\kern-.1em%
{\mathbf{e}}_{h},H_{\Omega}\left(  \Pi^{\operatorname*{comp},\ast}-\Pi
_{h}^{\operatorname*{comp},\ast}\right)  \mathbf{v}_{h}%
\kern-.1em%
\right)
\kern-.1em%
\right)  \right\vert \overset{(\ref{eq:intro-refined-continuity-(())})}{\leq
}C\Vert{\mathbf{e}}_{h}\Vert_{\operatorname*{curl},\Omega,k}\Vert\left(
\Pi^{\operatorname*{comp},\ast}-\Pi_{h}^{\operatorname*{comp},\ast}\right)
\mathbf{v}_{h}\Vert_{\operatorname*{curl},\Omega,k}\\
&  =C\Vert{\mathbf{e}}_{h}\Vert_{\operatorname*{curl},\Omega,k}\left(
k\left\Vert \left(  \Pi^{\operatorname*{comp},\ast}-\Pi_{h}%
^{\operatorname*{comp},\ast}\right)  \mathbf{v}_{h}\right\Vert \right)  ,
\end{align*}
where, in the last step, we used $\operatorname*{curl}\left(  \Pi
^{\operatorname*{comp},\ast}-\Pi_{h}^{\operatorname*{comp},\ast}\right)
\mathbf{v}_{h}=0$. The term $k\left\Vert \left(  \Pi^{\operatorname*{comp}%
,\ast}-\Pi_{h}^{\operatorname*{comp},\ast}\right)  \mathbf{v}_{h}\right\Vert $
is estimated by%
\begin{align*}
\left(  k\left\Vert \left(  \Pi^{\operatorname*{comp},\ast}-\Pi_{h}%
^{\operatorname*{comp},\ast}\right)  \mathbf{v}_{h}\right\Vert \right)
^{2}\overset{\text{(\ref{item:intro-II})}}{\leq}  &  \operatorname*{Re}\left(
%
\kern-.1em%
\left(
\kern-.1em%
\left(  \Pi^{\operatorname*{comp},\ast}-\Pi_{h}^{\operatorname*{comp},\ast
}\right)  \mathbf{v}_{h},\left(  \Pi^{\operatorname*{comp},\ast}-\Pi
_{h}^{\operatorname*{comp},\ast}\right)  \mathbf{v}_{h}%
\kern-.1em%
\right)
\kern-.1em%
\right) \\
\overset{\text{(\ref{eq:intro-200})}}{=}  &  \operatorname*{Re}\left(
\kern-.1em%
\left(
\kern-.1em%
\left(  I-\Pi_{h}^{E}\right)  \Pi^{\operatorname*{comp},\ast}\mathbf{v}%
_{h},\left(  \Pi^{\operatorname*{comp},\ast}-\Pi_{h}^{\operatorname*{comp}%
,\ast}\right)  \mathbf{v}_{h}%
\kern-.1em%
\right)
\kern-.1em%
\right) \\
=  &  \operatorname*{Re}\left(
\kern-.1em%
\left(
\kern-.1em%
H_{\Omega}\left(  I-\Pi_{h}^{E}\right)  \Pi^{\operatorname*{comp},\ast
}\mathbf{v}_{h},\left(  \Pi^{\operatorname*{comp},\ast}-\Pi_{h}%
^{\operatorname*{comp},\ast}\right)  \mathbf{v}_{h}%
\kern-.1em%
\right)
\kern-.1em%
\right) \\
&  +\operatorname*{Re}\left(
\kern-.1em%
\left(
\kern-.1em%
L_{\Omega}\left(  I-\Pi_{h}^{E}\right)  \Pi^{\operatorname*{comp},\ast
}\mathbf{v}_{h},\left(  \Pi^{\operatorname*{comp},\ast}-\Pi_{h}%
^{\operatorname*{comp},\ast}\right)  \mathbf{v}_{h}%
\kern-.1em%
\right)
\kern-.1em%
\right)  =:\widetilde{T}_{4,1}+\widetilde{T}_{4,2}.
\end{align*}
{}From (\ref{eq:intro-refined-continuity-(())}) in (\ref{item:intro-IX}), we
get $\left\vert \widetilde{T}_{4,1}\right\vert \leq C\left(  k\Vert\left(
\Pi^{\operatorname*{comp},\ast}-\Pi_{h}^{\operatorname*{comp},\ast}\right)
\mathbf{v}_{h}\Vert\right)  \left(  k\Vert\left(  I-\Pi_{h}^{E}\right)
\Pi^{\operatorname*{comp},\ast}\mathbf{v}_{h}\Vert\right)  $. We remark that
the above argument glossed over a minor point: In view of the modified
definition of the decomposition (\ref{eq:intro-continuous-helmholtz-split}),
(\ref{eq:intro-discrete-helmholtz-split}), we have to require that the
operator $\Pi_{h}^{E}$ be additionally defined on the space of smooth
functions (in (\ref{item:intro-VI}), the operator $\Pi_{h}^{E}$ is only
defined on ${\mathbf{V}}_{0,h}^{\ast}+{\mathbf{X}}_{h}$) and satisfy some
appropriate stability properties.
The term $\left\Vert \left(  I-\Pi_{h}^{E}\right)  \Pi^{\operatorname*{comp}%
,\ast}\mathbf{v}_{h}\right\Vert $ can be estimated as follows in view of the
definition (\ref{eq:intro-continuous-helmholtz-split}):
\[
\left\Vert \left(  I-\Pi_{h}^{E}\right)  \Pi^{\operatorname*{comp},\ast
}\mathbf{v}_{h}\right\Vert \leq\left\Vert \left(  I-\Pi_{h}^{E}\right)
L_{\Omega}{\mathbf{v}}_{h}\right\Vert +\left\Vert \left(  I-\Pi_{h}%
^{E}\right)  \Pi^{\operatorname*{curl},\ast}H_{\Omega}{\mathbf{v}}%
_{h}\right\Vert =:\widetilde{T}_{5,1}+\widetilde{T}_{5,2}.
\]
For N\'{e}d\'{e}lec elements of degree $p$,
Theorem~\ref{thm:projection-based-interpolation} provides an operator $\Pi
_{h}^{E}$ (its restriction to the reference element $\widehat{K}$ is denoted
there $\widehat{\Pi}_{p}^{\operatorname*{curl},c}$) that is also defined on
(elementwise) smooth functions and has good polynomial approximation
properties.
In particular, by the analyticity of $L_{\Omega}{\mathbf{v}}_{h}$, the term
$\widetilde{T}_{5,1}$ is exponentially small in the polynomial degree $p$ for
N\'{e}d\'{e}lec elements. The term $\widetilde{T}_{5,2}$ can be controlled by
the assumption (\ref{item:intro-VI-c}) and the stability bound
(\ref{eq:intro-stability-of-Picurl}) as
\[
\widetilde{T}_{5,2}\overset{\text{(\ref{item:intro-VI-c})}}{\leq}\eta
_{6}^{\operatorname{alg}}\Vert\Pi^{\operatorname*{curl},\ast}H_{\Omega
}{\mathbf{v}}_{h}\Vert_{\operatorname{curl},\Omega,k}%
\overset{(\ref{eq:intro-stability-of-Picurl})}{\leq}\eta_{6}%
^{\operatorname{alg}}C\Vert{\mathbf{v}}_{h}\Vert_{\operatorname*{curl}%
,\Omega,k}.
\]
The term $\widetilde{T}_{4,2}$ requires a duality argument that exploits the
orthogonality property (\ref{eq:intro-200}). Specifically, the dual problem is
to find $\psi\in H^{1}(\Omega)$ such that
\begin{equation}
\left(
\kern-.1em%
\left(
\kern-.1em%
\nabla\psi,\nabla\widetilde{\psi}%
\kern-.1em%
\right)
\kern-.1em%
\right)  =\left(
\kern-.1em%
\left(
\kern-.1em%
L_{\Omega}\left(  I-\Pi_{h}^{E}\right)  \Pi^{\operatorname*{comp},\ast
}\mathbf{v}_{h},\nabla\widetilde{\psi}%
\kern-.1em%
\right)
\kern-.1em%
\right)  \qquad\forall\widetilde{\psi}\in H^{1}(\Omega).
\label{eq:intro-last-dual-problem}%
\end{equation}
Solvability is ensured by (\ref{item:intro-II}). The analyticity of
$L_{\Omega}\left(  I-\Pi_{h}^{E}\right)  \Pi^{\operatorname*{comp},\ast
}\mathbf{v}_{h}$ and $\partial\Omega$ give that $\psi$ is analytic; we have by
Proposition~\ref{PropN4A} (problem (\ref{eq:intro-last-dual-problem}) is of
Type 2 discussed in Sec.~\ref{SecSolForm}):

\begin{enumerate}
[(I)] \setcounter{enumi}{11}

\item The solution $\psi$ of the problem (\ref{eq:intro-last-dual-problem})
belongs to an analyticity class $\psi\in{\mathcal{A}}(C k^{\alpha
}\|{\mathbf{v}}_{h}\|_{\operatorname{curl},\Omega,k}, \gamma,\Omega)$ for some
$C$, $\alpha$, $\gamma\ge0$ independent of $k$
\end{enumerate}

We obtain, noting that $(\Pi^{\operatorname{comp},\ast} - \Pi
^{\operatorname{comp},\ast}_{h}){\mathbf{v}}_{h}$ satisfies the same
orthogonality condition (\ref{eq:intro-200}) as the difference $(\Pi
^{\operatorname{curl},\ast} - \Pi^{\operatorname{curl},\ast}_{h}){\mathbf{v}%
}_{h}$,
\begin{align*}
T_{4,2}  &  = \left(
\kern-.1em%
\left(
\kern-.1em%
\nabla\psi,(\Pi^{\operatorname{comp},\ast} - \Pi^{\operatorname{comp},\ast
}_{h}){\mathbf{v}}_{h}
\kern-.1em%
\right)
\kern-.1em%
\right)  \overset{(\ref{eq:intro-200})}{=} \inf_{\psi_{h} \in S_{h}} \left(
\kern-.1em%
\left(
\kern-.1em%
\nabla(\psi- \psi_{h}),(\Pi^{\operatorname{comp},\ast} - \Pi
^{\operatorname{comp},\ast}_{h}){\mathbf{v}}_{h}
\kern-.1em%
\right)
\kern-.1em%
\right) \\
&  \overset{(\ref{eq:intro-continuity})}{\leq} C k^{3} \inf_{\psi_{h}
\in{S_{h}} } \|\nabla(\psi- \psi_{h})\|_{\operatorname{curl},\Omega,k}
\|(\Pi^{\operatorname{comp},\ast} - \Pi^{\operatorname{comp},\ast}%
_{h}){\mathbf{v}}_{h}\|_{\operatorname{curl},\Omega,k}\\
&  = C k^{3} \inf_{\psi_{h} \in S_{h}} (k \|\nabla(\psi- \psi_{h})\|) (k
\|(\Pi^{\operatorname{comp},\ast} - \Pi^{\operatorname{comp},\ast}%
_{h}){\mathbf{v}}_{h}\|);
\end{align*}
a more detailed argument can be found in the proof of
Prop.~\ref{PropHOmegaSplit}.

\bigskip The main result of the present work is quasi-optimality of the
${\mathbf{H}}(\Omega,\operatorname{curl})$-conforming discretization: In
Theorem~\ref{thm:quasi-optimal}, we present a fairly abstract convergence
result (which is not fully explicit in $k$). In
Theorem~\ref{thm:hpFEM-quasioptimality} we consider high order N\'ed\'elec
elements and the specific situation of the unit ball $B_{1}(0)$ and show
quasi-optimality of the Galerkin discretization under the scale resolution
condition (\ref{eq:intro-scale-resolution}).


\section{Maxwell's Equations\label{SecME}}


In Sections~\ref{SecMwEqFS} and \ref{SecReformMax} we introduce the strong
form of the Maxwell problem first in the full space ${\mathbb{R}}^{3}$ and
then in an equivalent way on a bounded domain. At this stage we are vague
concerning the precise function spaces and mapping properties of trace
operators. The variational formulation of the problem in a bounded domain is
given in Section \ref{SecMaxwVarForm}, where also the appropriate function
spaces are introduced.

\subsection{Maxwell's Equations in the Full Space $\mathbb{R}^{3}%
$\label{SecMwEqFS}}

We consider the solution of the Maxwell equations in the full space
${\mathbb{R}}^{3}$ with Silver-M\"{u}ller radiation conditions at infinity.
The angular frequency is denoted by $\omega$, the electric permittivity by
$\varepsilon$, and the magnetic permeability by $\mu$. We formulate the
problem in terms of the wavenumber $k=\omega\sqrt{\varepsilon\mu}$, the scaled
magnetic field $\mathbf{\tilde{H}}=\sqrt{\frac{\mu}{\varepsilon}}\mathbf{H}$,
and the scaled electric charge density $\mathbf{\tilde{\boldsymbol{\jmath}}%
=}\sqrt{\mu/\varepsilon}\mathbf{j}$: Find the electric field $\mathbf{E}$ and
the scaled magnetic field $\mathbf{\tilde{H}}$ such that%
\begin{equation}%
\begin{array}
[c]{cccl}%
\operatorname{curl}\mathbf{E}-\operatorname*{i}k\mathbf{\tilde{H}}=\mathbf{0}
& \text{and} & \operatorname{curl}\mathbf{\tilde{H}}+\operatorname*{i}%
k\mathbf{E}=\mathbf{\tilde{\boldsymbol{\jmath}}} & \text{in }\mathbb{R}^{3},\\
&  &  & \\
\left\vert \mathbf{E}-\mathbf{\tilde{H}}\times\dfrac{\mathbf{x}}{r}\right\vert
\leq\dfrac{c}{r^{2}} & \text{and} & \left\vert \mathbf{E}\times\dfrac
{\mathbf{x}}{r}+\mathbf{\tilde{H}}\right\vert \leq\dfrac{c}{r^{2}} & \text{for
}r=\left\Vert \mathbf{x}\right\Vert \rightarrow\infty
\end{array}
\label{Maxwellfullspace}%
\end{equation}
is satisfied in a weak sense. Throughout the paper we assume that the data
$\mathbf{\tilde{\boldsymbol{\jmath}}}$ satisfies the following
Assumption~\ref{AssumptionData}a which is sufficient to prove quasi-optimality
of the Galerkin discretization (cf.~Theorems~\ref{thm:quasi-optimal},
\ref{thm:hpFEM-quasioptimality}) while further assumptions on $\mathbf{\tilde
{\boldsymbol{\jmath}}}$ are needed to obtain convergence \textit{rates}
(cf.~Corollary~\ref{CorConvRates}).

\begin{assumption}
[$\Omega$, $\Gamma$, right-hand side $\mathbf{\tilde{\boldsymbol{\jmath}}}$%
]\label{AssumptionData}a) The scaled electric charge density $\mathbf{\tilde
{\boldsymbol{\jmath}}}$ is a compactly supported distribution (functional on
the space $\mathbf{H}_{\operatorname*{loc}}\left(  \operatorname*{curl}%
,\mathbb{R}^{3}\right)  $ defined in Section~\ref{sec:spaces}) \textbf{ }in
the sense that there exists a bounded, smooth Lipschitz domain $\Omega
\subset\mathbb{R}^{3}$ with simply connected boundary $\Gamma:=\partial\Omega$
that satisfies $\operatorname*{supp}\mathbf{\tilde{\boldsymbol{\jmath}}%
}\subset\Omega$. We denote by $\mathbf{n}$ the unit normal vector on the
boundary $\Gamma$ oriented such that it points into the unbounded exterior
$\Omega^{+}:=\mathbb{R}^{3}\backslash\overline{\Omega}$.

b) The wavenumber $k$ is considered as a real parameter in the
range\footnote{The condition $k\geq1$ can be replaced by $k\geq k_{0}>0$. Our
estimates remain valid for all choices of $k_{0}>0$. The constants in the
estimates are uniform for all $k\geq k_{0}$ while they depend continuously on
$k_{0}$ and, possibly, become large as $k_{0}\rightarrow0$. We use
(\ref{loweromega}) simply to reduce technicalities.}%
\begin{equation}
k\geq1. \label{loweromega}%
\end{equation}

\end{assumption}

\subsection{Reformulation on a Bounded Domain\label{SecReformMax}}

Assumption \ref{AssumptionData} allows us to formulate problem
(\ref{Maxwellfullspace}) in an equivalent way as a transmission problem. For
this we have to introduce in (\ref{eq:trace-operators}) the trace operators
$\Pi_{T}$ and $\gamma_{T}$, which map sufficiently smooth functions
$\mathbf{u}$ in $\overline{\Omega}$ to tangential fields on the surface
$\Gamma$ while the trace operators $\Pi_{T}^{+}$ and $\gamma_{T}^{+}$ denote
the corresponding traces for function $\mathbf{u}^{+}$ in the exterior domain
$\overline{\Omega^{+}}$:%
\begin{equation}%
\begin{array}
[c]{ll}%
\Pi_{T}:\mathbf{u}\mapsto\mathbf{n}\times\left(  \mathbf{u}|_{\Gamma}%
\times\mathbf{n}\right)  , & \gamma_{T}:\mathbf{u}\mapsto\mathbf{u}|_{\Gamma
}\times\mathbf{n},\\
\Pi_{T}^{+}:\mathbf{u}^{+}\mapsto\mathbf{n}\times\left(  \mathbf{u}%
^{+}|_{\Gamma}\times\mathbf{n}\right)  , & \gamma_{T}^{+}:\mathbf{u}%
^{+}\mapsto\mathbf{u}^{+}|_{\Gamma}\times\mathbf{n}.
\end{array}
\label{eq:trace-operators}%
\end{equation}
This allows us to define the jumps for sufficiently smooth functions
$\mathbf{w}$ in the interior and $\mathbf{w}^{+}$ in the exterior domain:%
\begin{equation}
\left[  \left(  \mathbf{w},\mathbf{w}^{+}\right)  \right]  _{0,\Gamma}%
:=\gamma_{T}\mathbf{w-}\gamma_{T}^{+}\mathbf{w}^{+}\mathbf{,\quad}\left[
\left(  \mathbf{w},\mathbf{w}^{+}\right)  \right]  _{1,\Gamma}:=\gamma
_{T}\operatorname{curl}\mathbf{w-}\gamma_{T}^{+}\operatorname{curl}%
\mathbf{w}^{+}. \label{defjumps}%
\end{equation}
With this notation, the problem (\ref{Maxwellfullspace}) takes the form: Find
$\mathbf{E}$, $\mathbf{E}^{+}$, $\mathbf{\tilde{H}}$, $\mathbf{\tilde{H}}^{+}$
such that
%
\begin{subequations}
\label{eq:fullspace}%
\begin{align}
\operatorname{curl}\mathbf{E}-\operatorname*{i}k\mathbf{\tilde{H}}  &
=\mathbf{0}, & \operatorname{curl}\mathbf{\tilde{H}}+\operatorname*{i}%
k\mathbf{E}  &  =\mathbf{\tilde{\boldsymbol{\jmath}}}\qquad\text{in }%
\Omega\text{,}\label{eq:fullspace-a}\\
\operatorname{curl}\mathbf{E}^{+}-\operatorname*{i}k\mathbf{\tilde{H}}^{+}  &
=\mathbf{0}, & \operatorname{curl}\mathbf{\tilde{H}}^{+}+\operatorname*{i}%
k\mathbf{E}^{+}  &  =\mathbf{0}\qquad\text{in }\Omega^{+}%
,\label{eq:fullspace-b}\\
\left[  \left(  \mathbf{E},\mathbf{E}^{+}\right)  \right]  _{0,\Gamma}  &  =0,
& \left[  \left(  \mathbf{E},\mathbf{E}^{+}\right)  \right]  _{1,\Gamma}  &
=\mathbf{0},\label{eq:fullspace-c}\\
&  &  & \\
\left\vert \mathbf{E}^{+}-\mathbf{\tilde{H}}^{+}\times\dfrac{\mathbf{x}}%
{r}\right\vert  &  \leq\dfrac{c}{r^{2}}, & \left\vert \mathbf{E}^{+}%
\times\dfrac{\mathbf{x}}{r}+\mathbf{\tilde{H}}^{+}\right\vert  &  \leq
\dfrac{c}{r^{2}}\qquad\text{for }r=\left\Vert \mathbf{x}\right\Vert
\rightarrow\infty. \label{eq:fullspace-d}%
\end{align}
The key role for formulating this problem as an equation on the bounded domain
$\Omega$ is played by the capacity operator $T_{k}$. This operator associates
to $\mathbf{g}_{T}\in\mathbf{H}_{\operatorname*{curl}}^{-1/2}\left(
\Gamma\right)  $ the value of $\gamma_{T}^{+}\mathbf{\tilde{H}}^{+}$ on
$\Gamma$ where $\left(  \mathbf{E}^{+},\mathbf{\tilde{H}}^{+}\right)  $ solves
the homogeneous Maxwell problem in the exterior domain $\Omega^{+}$ with
Silver-M\"{u}ller radiation conditions at $\infty$ (i.e.,
(\ref{eq:fullspace-b}), (\ref{eq:fullspace-d})) together with Dirichlet
boundary conditions $\gamma_{T}^{+}\mathbf{E}^{+}=\mathbf{g}_{T}%
\times\mathbf{n}$. That is, $T_{k}\mathbf{g}_{T}:=\gamma_{T}^{+}%
\mathbf{\tilde{H}}^{+}$.
\end{subequations}
\begin{remark}
\label{ExCapOp}From \cite[Lemma~{5.4.3}, Thm.~{5.4.6}]{Nedelec01}\footnote{The
function spaces appearing in these statements will be introduced in
Section~\ref{sec:spaces}.} we conclude that the exterior homogeneous Maxwell
equations with given Dirichlet data $\mathbf{g}\in\mathbf{H}%
_{\operatorname*{div}}^{-1/2}\left(  \Gamma\right)  $, i.e., $\gamma_{T}%
^{+}\mathbf{E}^{+}=\mathbf{g}$ on $\Gamma$, for the electric field has a weak
solution $\mathbf{E}^{+}\in\mathbf{H}_{\operatorname*{loc}}\left(
\operatorname*{curl},\Omega^{+}\right)  $, which is unique and satisfies%
\[
\left\Vert \mathbf{E}^{+}\right\Vert _{\operatorname*{curl},B_{R}\left(
0\right)  \cap\Omega^{+},1}\leq C_{R,\Omega}\left\Vert \mathbf{g}\right\Vert
_{\mathbf{H}_{\operatorname*{div}}^{-1/2}\left(  \Gamma\right)  },
\]
where $B_{R}\left(  0\right)  $ is a ball with radius $R$ centered at $0$ such
that $\overline{\Omega}\subset B_{R}\left(  0\right)  $ and $C_{R,\Omega}$ is
a constant which only depends on $\Omega$ and $R$.

This implies that the capacity operator $T_{k}:\mathbf{H}%
_{\operatorname*{curl}}^{-1/2}\left(  \Gamma\right)  \rightarrow
\mathbf{H}_{\operatorname*{div}}^{-1/2}\left(  \Gamma\right)  $ is continuous.
\hbox{}\hfill\rule{0.8ex}{0.8ex}
\end{remark}

The Maxwell equations on the bounded domain are given by%
\[%
\begin{array}
[c]{lll}%
\operatorname{curl}\mathbf{E}-\operatorname*{i}k\mathbf{\tilde{H}}%
=\mathbf{0}, & \operatorname{curl}\mathbf{\tilde{H}}+\operatorname*{i}%
k\mathbf{E}=\mathbf{\tilde{\boldsymbol{\jmath}}} & \text{in }\Omega\text{,}\\
\gamma_{T}\operatorname{curl}\mathbf{E}-\operatorname*{i}kT_{k}\Pi
_{T}\mathbf{E}=\mathbf{0} &  & \text{on }\Gamma.
\end{array}
\]
Eliminating $\mathbf{\tilde{H}}$ from these equations we arrive at the Maxwell
equations for the electric field on a bounded domain $\Omega$%
\begin{equation}%
\begin{array}
[c]{ll}%
\operatorname*{curl}\operatorname{curl}\mathbf{E}-k^{2}\mathbf{E}%
=\operatorname*{i}k\mathbf{\tilde{\boldsymbol{\jmath}}} & \text{in }\Omega
\text{,}\\
\gamma_{T}\operatorname{curl}\mathbf{E}-\operatorname*{i}kT_{k}\Pi
_{T}\mathbf{E}=\mathbf{0} & \text{on }\Gamma.
\end{array}
\label{electricMWEqOmega}%
\end{equation}


\subsection{Sobolev Spaces in $\Omega$ and on $\Gamma$\label{sec:spaces}}

We introduce the pertinent function spaces.

\subsubsection{Sobolev Spaces in $\Omega$\label{SecSobolevOmega}}

By $H^{s}\left(  \Omega\right)  $ we denote the usual Sobolev spaces of index
$s\geq0$ with norm $\left\Vert \cdot\right\Vert _{H^{s}\left(  \Omega\right)
}$. The closure of $C_{0}^{\infty}\left(  \Omega\right)  $ functions with
respect to $\left\Vert \cdot\right\Vert _{H^{s}\left(  \Omega\right)  }$ is
denoted by $H_{0}^{s}\left(  \Omega\right)  $. For $s\geq0$, the dual space of
$H_{0}^{s}\left(  \Omega\right)  $ is denoted by $H^{-s}\left(  \Omega\right)
$. If the functions are vector-valued we indicate this by writing
$\mathbf{H}^{s}\left(  \Omega\right)  $, $\mathbf{H}_{0}^{s}\left(
\Omega\right)  $. For details we refer to \cite{Adams_new}.

The \textit{energy space} for the electric field is given by%
\begin{equation}
\mathbf{X}:=\mathbf{H}\left(  \Omega,\operatorname{curl}\right)  :=\left\{
\mathbf{u}\in\mathbf{L}^{2}\left(  \Omega\right)  \mid\operatorname{curl}%
\mathbf{u}\in\mathbf{L}^{2}\left(  \Omega\right)  \right\}  \label{defXHcurl}%
\end{equation}
equipped with the indexed scalar product and norm%
\begin{equation}
\left(  \mathbf{f},\mathbf{g}\right)  _{\operatorname{curl},\Omega,k}:=\left(
\operatorname{curl}\mathbf{f},\operatorname{curl}\mathbf{g}\right)
+k^{2}\left(  \mathbf{f},\mathbf{g}\right)  \quad\text{and\quad}\left\Vert
\mathbf{f}\right\Vert _{\operatorname{curl},\Omega,k}:=\left(  \mathbf{f}%
,\mathbf{f}\right)  _{\operatorname{curl},\Omega,k}^{1/2}, \label{indexedcurl}%
\end{equation}
where $\left(  \cdot,\cdot\right)  $ denotes the $\mathbf{L}^{2}\left(
\Omega\right)  $-scalar product%
\begin{equation}
\left(  \mathbf{f},\mathbf{g}\right)  :=\int_{\Omega}\left\langle
\mathbf{f},\mathbf{g}\right\rangle . \label{L2Euklid}%
\end{equation}
Here, $\left\langle \cdot,\cdot\right\rangle $ is the Euclidean scalar product
in $\mathbb{C}^{3}$ (with complex conjugation in the second argument). 
The dual space of ${\mathbf X}$ (i.e., the space of continuous 
linear functionals on ${\mathbf X}$) is denoted by ${\mathbf X}^\prime$
and the anti-dual space ${\mathbf X}$ (i.e., the space continuous
anti-linear functionals on ${\mathbf X}$) is denoted by ${\mathbf X}^\times$.
We also
introduce the space%
\begin{equation}
\mathbf{H}\left(  \Omega,\operatorname{div}\right)  :=\left\{  \mathbf{u}%
\in\mathbf{L}^{2}\left(  \Omega\right)  \mid\operatorname{div}\mathbf{u}\in
L^{2}\left(  \Omega\right)  \right\}  . \label{DefHOmegadiv}%
\end{equation}
For \emph{unbounded} domains $D\subset\mathbb{R}^{3}$ we denote
$\mathbf{H}_{\operatorname*{loc}}\left(  D,\operatorname*{curl}\right)  $ the
space of all distributions ${\mathbf{f}}$ with the property that
$\varphi\mathbf{f}\in\mathbf{H}\left(  D,\operatorname*{curl}\right)  $ for
all smooth, compactly supported functions $\varphi\in{C}_{0}^{\infty
}({\mathbb{R}}^{3})$.

\subsubsection{Sobolev Spaces on $\Gamma$}

The Sobolev spaces on the boundary $\Gamma$ are denoted by $H^{s}\left(
\Gamma\right)  $ for scalar-valued functions and by $\mathbf{H}^{s}\left(
\Gamma\right)  $ for vector-valued functions. The range of differentiability
$s$ depends on the smoothness of $\Gamma$. To avoid such technicalities, we
assume throughout the paper that the boundary $\Gamma$ is sufficiently smooth
so that the Sobolev spaces $H^{s}\left(  \Gamma\right)  $, $H^{s}\left(
\Gamma\right)  $ are well defined. A formal definition may be found in
\cite{Mclean00}; however, below and throughout this work, we will use the
characterization in terms of expansions via eigenfunctions of the
Laplace--Beltrami operator. We will need the space $\mathbf{L}_{T}^{2}%
(\Gamma)$ of tangential vector fields given by
\begin{equation}
\mathbf{L}_{T}^{2}\left(  \Gamma\right)  :=\{\mathbf{v}\in\mathbf{L}%
^{2}(\Gamma)|\left\langle \mathbf{n},\mathbf{v}\right\rangle =0\text{ on
}\Gamma\}. \label{DefL2t}%
\end{equation}

For a sufficiently smooth scalar-valued function $u$ and vector-valued
function $\mathbf{v}$ on $\Gamma$, the constant (along the normal direction)
extensions into a sufficiently small three-dimensional neighborhood
${\mathcal{U}}$ of $\Gamma$ is denoted by $u^{\ast}$ and $\mathbf{v}^{\ast}$.
The \textit{surface gradient} $\nabla_{\Gamma}$, the \textit{tangential curl}
$\overrightarrow{\operatorname*{curl}\nolimits_{\Gamma}}$, and the
\textit{surface divergence} $\operatorname*{div}_{\Gamma}$ are defined by
(cf., e.g., \cite{Nedelec01}, \cite{BuffaCostabelSheen})%
\begin{equation}
\nabla_{\Gamma}u:=\left.  \left(  \nabla u^{\star}\right)  \right\vert
_{\Gamma}\text{,\quad}\overrightarrow{\operatorname*{curl}\nolimits_{\Gamma}%
}u:=\nabla_{\Gamma}u\times\mathbf{n}\text{,\quad and\quad}\operatorname*{div}%
\nolimits_{\Gamma}\mathbf{v}=\left.  \left(  \operatorname*{div}%
\mathbf{v}^{\ast}\right)  \right\vert _{\Gamma}\qquad\text{on }\Gamma\text{.}
\label{curlvec}%
\end{equation}
The scalar counterpart of the tangential curl is the \textit{surface curl}%
\begin{equation}
\operatorname*{curl}\nolimits_{\Gamma}\mathbf{v}:=\langle\left.  \left(
\operatorname*{curl}\mathbf{v}^{\ast}\right)  \right\vert _{\Gamma}%
,\mathbf{n}\rangle\qquad\text{on }\Gamma. \label{sccounttangcurl}%
\end{equation}
The composition of the surface divergence and surface gradient leads to the
\textit{scalar Laplace-Beltrami operator}%
\begin{equation}
\Delta_{\Gamma}u=\operatorname*{div}\nolimits_{\Gamma}\nabla_{\Gamma}u.
\label{defLaplBelt}%
\end{equation}
{}From \cite[(2.5.197)]{Nedelec01} we have the relation%
\begin{equation}
\operatorname*{div}\nolimits_{\Gamma}\left(  \mathbf{v}\times\mathbf{n}%
\right)  =\operatorname*{curl}\nolimits_{\Gamma}\mathbf{v}. \label{divcurlrel}%
\end{equation}

The operator $\Delta_{\Gamma}$ is self-adjoint with respect to the
$L^{2}\left(  \Gamma\right)  $ scalar product $\left(  \cdot,\cdot\right)
_{\Gamma}$ and positive semidefinite. It admits a countable sequence of
eigenfunctions in $L^{2}\left(  \Gamma\right)  $ denoted by $Y_{\ell}^{m}$
such that%
\begin{equation}
-\Delta_{\Gamma}Y_{\ell}^{m}=\lambda_{\ell}Y_{\ell}^{m}\qquad\text{\ for }%
\ell=0,1,\ldots.\text{and }m\in\iota_{\ell}. \label{eigvalabsLapBelt}%
\end{equation}
We choose the normalization such that $\left(  Y_{\ell}^{m},Y_{\ell^{\prime}%
}^{m^{\prime}}\right)  _{\Gamma}=\delta_{m,m^{\prime}}\delta_{\ell
,\ell^{\prime}}$ holds. Here, $\iota_{\ell}$ is a finite index set whose
cardinality equals the multiplicity of the eigenvalue $\lambda_{\ell}$, and we
always assume that the eigenvalues $\lambda_{\ell}$ are distinct and ordered
increasingly. We have $\lambda_{0}=0$ and for $\ell\geq1$, they are real and
positive and accumulate at infinity. By Assumption~\ref{AssumptionData} the
surface $\Gamma$ is simply connected so that $\lambda_{0}=0$ is a simple
eigenvalue. {}From \cite[Sec. 5.4]{Nedelec01} we know that any distribution
$w$, defined on the surface $\Gamma$, can formally be expanded with respect to
the basis $Y_{\ell}^{m}$ as%
\[
w=\sum_{\ell=0}^{\infty}\sum_{m\in\iota_{\ell}}w_{\ell}^{m}Y_{\ell}^{m}.
\]
The space $H^{s}\left(  \Gamma\right)  $ can be characterized by%
\begin{equation}
H^{s}\left(  \Gamma\right)  =\left\{  w\in\left(  C^{\infty}\left(
\Gamma\right)  \right)  ^{\prime}\mid\left\Vert w\right\Vert _{H^{s}\left(
\Gamma\right)  }^{2}:=\sum_{\ell=0}^{\infty}\left(  \delta_{\ell,0}%
+\lambda_{\ell}\right)  ^{s}\sum_{m\in\iota_{\ell}}\left\vert w_{\ell}%
^{m}\right\vert ^{2}<\infty\right\}  \label{DefHsGammaNorm}%
\end{equation}
with Kronecker's $\delta_{m,\ell}$. A norm on $H^{s}\left(  \Gamma\right)  $
is given by $\left\Vert \cdot\right\Vert _{H^{s}\left(  \Gamma\right)  }$.

Next, we define spaces of vector-valued functions. By \cite[Sec.~{5.4.1}%
]{Nedelec01}, every function $\mathbf{v}_{T}\in\mathbf{L}_{T}^{2}\left(
\Gamma\right)  $ can be written in the form
\begin{equation}
\mathbf{v}_{T}=\sum_{\ell=1}^{\infty}\sum_{m\in\iota_{\ell}}\left(  v_{\ell
}^{m}\overrightarrow{\operatorname*{curl}\nolimits_{\Gamma}}Y_{\ell}%
^{m}+V_{\ell}^{m}\nabla_{\Gamma}Y_{\ell}^{m}\right)  , \label{vFExp}%
\end{equation}
where the coefficients satisfy $\sum_{\ell=1}^{\infty}\lambda_{\ell}\sum
_{m\in\iota_{\ell}}\left(  \left\vert v_{\ell}^{m}\right\vert ^{2}+\left\vert
V_{\ell}^{m}\right\vert ^{2}\right)  <\infty$. We set%
\begin{equation}
\left\Vert \mathbf{v}_{T}\right\Vert _{\mathbf{H}_{T}^{s}\left(
\Gamma\right)  }^{2}:=\sum_{\ell=1}^{\infty}\lambda_{\ell}^{s+1}\sum
_{m\in\iota_{\ell}}\left(  \left\vert v_{\ell}^{m}\right\vert ^{2}+\left\vert
V_{\ell}^{m}\right\vert ^{2}\right)  . \label{DefHsGammaTNorm}%
\end{equation}
A tangential vector field ${\mathbf{v}}_{T}$ can be decomposed into a surface
gradient and a surface curl part as $\mathbf{v}_{T}=\mathbf{v}%
^{\operatorname*{curl}} +\mathbf{v}^{\nabla}$, where
(cf.~\ref{vFExp})

\begin{equation}%
\begin{array}
[c]{clll}
& \mathbf{v}^{\nabla}:=\sum_{\ell=1}^{\infty}\sum_{m\in\iota_{\ell}}V_{\ell
}^{m}\nabla_{\Gamma}Y_{\ell}^{m} & \text{and} & \mathbf{v}%
^{\operatorname*{curl}}:=\sum_{\ell=1}^{\infty}\sum_{m\in\iota_{\ell}}v_{\ell
}^{m}\overrightarrow{\operatorname*{curl}\nolimits_{\Gamma}}Y_{\ell}^{m}\\
\quad &  &  &
\end{array}
\label{utbcHdecomp}%
\end{equation}

\begin{remark}
\label{rem:decomposition-of-nabla-phi}{For gradient fields $\nabla\varphi$ we
have $(\Pi_{T}\nabla\varphi)^{\operatorname*{curl}}=0$ and $(\Pi_{T}%
\nabla\varphi)^{\nabla}=\nabla_{\Gamma}\varphi$. } \hbox{}\hfill
\rule{0.8ex}{0.8ex}
\end{remark}

The decomposition (\ref{utbcHdecomp}) allows us to express the operators
$\operatorname*{curl}_{\Gamma}$ and $\operatorname*{div}_{\Gamma}$ and the
corresponding norms in terms of the Fourier coefficients: for a tangential
field $\mathbf{v}_{T}$ of the form (\ref{vFExp}), the surface divergence and
surface gradient are defined (formally) as in \cite[(5.4.18)-(5.4.21)]%
{Nedelec01}%
\begin{equation}
\operatorname*{div}\nolimits_{\Gamma}\mathbf{v}_{T}=\sum_{\ell=1}^{\infty
}\lambda_{\ell}\sum_{m\in\iota_{\ell}}V_{\ell}^{m}Y_{\ell}^{m}\quad
\text{and\quad}\operatorname*{curl}\nolimits_{\Gamma}\mathbf{v}_{T}=\sum
_{\ell=1}^{\infty}\lambda_{\ell}\sum_{m\in\iota_{\ell}}v_{\ell}^{m}Y_{\ell
}^{m}. \label{divsucurl=0}%
\end{equation}
The $H^{s}\left(  \Gamma\right)  $ norm (cf.\ (\ref{DefHsGammaNorm})) of
$\operatorname*{curl}_{\Gamma}\left(  \cdot\right)  $ and $\operatorname*{div}%
_{\Gamma}\left(  \cdot\right)  $ can accordingly be expressed in terms of the
Fourier expansions:
\begin{equation}
\left\Vert \operatorname*{curl}\nolimits_{\Gamma}\mathbf{v}_{T}\right\Vert
_{H^{s}\left(  \Gamma\right)  }^{2}=\sum_{\ell=1}^{\infty}\lambda_{\ell}%
^{s+2}\sum_{m\in\iota_{\ell}}\left\vert v_{\ell}^{m}\right\vert ^{2}%
\quad\text{and\quad}\left\Vert \operatorname*{div}\nolimits_{\Gamma}%
\mathbf{v}_{T}\right\Vert _{H^{s}\left(  \Gamma\right)  }^{2}=\sum_{\ell
=1}^{\infty}\lambda_{\ell}^{s+2}\sum_{m\in\iota_{\ell}}\left\vert V_{\ell}%
^{m}\right\vert ^{2}. \label{divscalednorm}%
\end{equation}
{We define}
\begin{subequations}
\label{m1/2curldiv}
\end{subequations}%
\begin{align}
\left\Vert \mathbf{v}_{T}\right\Vert _{-1/2,\operatorname*{curl}_{\Gamma}%
}^{2}  &  =\sum_{\ell=1}^{\infty}\lambda_{\ell}^{1/2}\sum_{m\in\iota_{\ell}%
}\left(  \left(  1+\lambda_{\ell}\right)  \left\vert v_{\ell}^{m}\right\vert
^{2}+\left\vert V_{\ell}^{m}\right\vert ^{2}\right)  ,\tag{%
\ref{m1/2curldiv}%
a}\label{m1/2curldiva}\\
\left\Vert \mathbf{v}_{T}\right\Vert _{-1/2,\operatorname*{div}_{\Gamma}}^{2}
&  =\sum_{\ell=1}^{\infty}\lambda_{\ell}^{1/2}\sum_{m\in\iota_{\ell}}\left(
\left\vert v_{\ell}^{m}\right\vert ^{2}+\left(  1+\lambda_{\ell}\right)
\left\vert V_{\ell}^{m}\right\vert ^{2}\right)  . \tag{%
\ref{m1/2curldiv}%
b}\label{m1/2curldivb}%
\end{align}

The spaces $\mathbf{H}_{\operatorname*{curl}}^{-1/2}\left(  \Gamma\right)  $
and $\mathbf{H}_{\operatorname*{div}}^{-1/2}\left(  \Gamma\right)  $ allow for
orthogonal decompositions on the surface $\Gamma$. From \cite[(5.4.20),
(5.4.21)]{Nedelec01} we conclude that%
\[%
\begin{array}
[c]{ccc}%
\mathbf{v}_{T}\in\mathbf{H}_{\operatorname*{div}}^{-1/2}\left(  \Gamma\right)
& \iff & \mathbf{v}_{T}\text{ is of the form (\ref{vFExp}) and }\left\Vert
\mathbf{v}_{T}\right\Vert _{-1/2,\operatorname*{div}_{\Gamma}}<\infty,\\
\mathbf{v}_{T}\in\mathbf{H}_{\operatorname*{curl}}^{-1/2}\left(  \Gamma\right)
& \iff & \mathbf{v}_{T}\text{ is of the form (\ref{vFExp}) and }\left\Vert
\mathbf{v}_{T}\right\Vert _{-1/2,\operatorname*{curl}_{\Gamma}}<\infty
\end{array}
\]
holds. The system $\left\{  \nabla_{\Gamma}Y_{\ell}^{m}%
,\overrightarrow{\operatorname*{curl}\nolimits_{\Gamma}}Y_{\ell}^{m}\right\}
$ forms an orthogonal basis in $\mathbf{L}_{T}^{2}\left(  \Gamma\right)  $
(cf. \cite[\S \ after (5.4.12)]{Nedelec01}) so that
\begin{equation}
\left(  \mathbf{v}^{\nabla},\mathbf{v}^{\operatorname*{curl}}\right)
_{\mathbf{L}^{2}_{T}\left(  \Gamma\right)  }=0\qquad\forall\mathbf{v}%
\in\mathbf{L}_{T}^{2}\left(  \Gamma\right)  . \label{orthogonality}%
\end{equation}

The following theorem shows that $\mathbf{H}^{-1/2}(\operatorname*{div}%
_{\Gamma},\Gamma)$ and $\mathbf{H}^{-1/2}(\operatorname{curl}_{\Gamma}%
,\Gamma)$ are the correct spaces to define continuous trace operators.

\begin{theorem}
\label{traceTHM1}The trace mappings%
\[
\Pi_{T}:\mathbf{X}\rightarrow\mathbf{H}_{\operatorname*{curl}}^{-1/2}\left(
\Gamma\right)  ,\gamma_{T}:\mathbf{X}\rightarrow\mathbf{H}%
_{\operatorname*{div}}^{-1/2}\left(  \Gamma\right)
\]
are continuous and surjective. Moreover, there exist continuous liftings
$\mathcal{E}_{\operatorname*{curl}}$,$:\mathbf{H}_{\operatorname*{curl}%
}^{-1/2}\left(  \Gamma\right)  \rightarrow\mathbf{X}$ and $\mathcal{E}%
_{\operatorname*{div}}$,$:\mathbf{H}_{\operatorname*{div}}^{-1/2}\left(
\Gamma\right)  \rightarrow\mathbf{X}$ for these trace spaces which are divergence-free.
\end{theorem}

For a proof we refer to \cite{Cessenat_book}, \cite[Thm.~{5.4.2}]{Nedelec01}.
For a vector field $\mathbf{u}\in\mathbf{X}$, we will employ frequently the
notation $\mathbf{u}_{T}:=\Pi_{T}\mathbf{u}$. The continuity constant of
$\Pi_{T}$ is%
\begin{equation}
C_{\Gamma}:=\sup_{\mathbf{v}\in\mathbf{X}\backslash\left\{  0\right\}  }%
\frac{\left\Vert \Pi_{T}\mathbf{v}\right\Vert _{-1/2,\operatorname*{curl}%
_{\Gamma}}}{\left\Vert \mathbf{v}\right\Vert _{\operatorname*{curl},\Omega,1}%
}. \label{eq:norm-of-Pi_tau}%
\end{equation}
The spaces $H_{\operatorname*{curl}}^{-1/2}(\Gamma)$ and
$H_{\operatorname*{div}}^{-1/2}(\Gamma)$ are in duality with respect to
$L_{T}^{2}(\Gamma)$, i.e., $(\cdot,\cdot)_{\Gamma}$ extends continuously to
the duality pairing $H_{\operatorname*{curl}}^{-1/2}(\Gamma)\times
H_{\operatorname*{div}}^{-1/2}(\Gamma)$ and, cf.~\cite[Lemmas~{5.3.1},
{5.4.1}]{Nedelec01},
\begin{equation}
|({\mathbf{u}}_{T},{\mathbf{v}}_{T})_{\Gamma}|\leq\Vert{\mathbf{u}}_{T}%
\Vert_{-1/2,\operatorname{curl}_{\Gamma}}\Vert{\mathbf{v}}_{T}\Vert
_{-1/2,\operatorname{div}_{\Gamma}}\qquad\forall{\mathbf{u}}_{T}\in
{\mathbf{H}}_{\operatorname*{curl}}^{-1/2}(\Gamma),\quad{\mathbf{v}}_{T}%
\in{\mathbf{H}}_{\operatorname*{curl}}^{-1/2}(\Gamma). \label{eq:duality}%
\end{equation}


\subsubsection{The Analyticity Classes ${\mathcal{A}}$}

\label{sec:analyticity-classes}
We introduce classes of analytic functions whose growth of the derivatives (as
the order of differentiation grows) is controlled explicitly in terms of the
wavenumber $k$. For smooth tensor-valued functions ${\mathbf{u}} =
({\mathbf{u}}_{{\mathbf{i}}})_{{\mathbf{i}} \in{\mathbf{I}}}$ on a open
$\omega\subset{\mathbb{R}}^{d}$, where ${\mathbf{I}}$ is a suitable finite
index set and using the usual multi-index conventions
$\mbox{\boldmath$ \alpha$}=\left(  \alpha_{s}\right)  _{s=1}^{d}$, we set
$\left\vert \mbox{\boldmath$ \alpha$}\right\vert =\alpha_{1}+\ldots+\alpha
_{d},$ and abbreviate
\begin{equation}
\label{defLaplhochn}\left\vert \nabla^{n}{\mathbf{u}}\left(  x\right)
\right\vert ^{2}=\sum_{\substack{\mbox{\boldmath$ \alpha$}\in\mathbb{N}%
_{0}^{d}\\\left\vert \mbox{\boldmath$ \alpha$}\right\vert =n}} \sum
_{{\mathbf{i}} \in{\mathbf{I}}} \dbinom{n}{\mbox{\boldmath$ \alpha$} }
\left\vert D^{\mbox{\boldmath$\alpha$}} {\mathbf{u}}_{{\mathbf{i}}} \left(
x\right)  \right\vert ^{2}, \qquad\qquad\dbinom{n}{\mbox{\boldmath$ \alpha$}}%
=\frac{n!}{\alpha_{1}!\cdots\alpha_{d}!},\quad D^{\mbox{\boldmath$ \alpha$}}%
=\partial_{1}^{\alpha_{1}}\partial_{2}^{\alpha_{2}}\ldots\partial_{d}%
^{\alpha_{d}}.
\end{equation}
We then define:

\begin{definition}
\label{DefClAnFct}For an open set $\omega\subset\mathbb{R}^{d}$ and constants
$C_{1}$, $\gamma_{1}>0$, and wavenumber $k\geq1$ (cf.~(\ref{loweromega})), we
set%
\begin{align*}
\mathcal{A}\left(  C_{1},\gamma_{1},\omega\right)   &  :=\left\{  {\mathbf{u}}
\in L^{2}\left(  \omega\right)  \mid\left\Vert \nabla^{n}{\mathbf{u}%
}\right\Vert _{L^{2}\left(  \omega\right)  }\leq C_{1}\gamma_{1}^{n}%
\max\left\{  n+1,k\right\}  ^{n}\;\forall n\in\mathbb{N}_{0}\right\}  ,\\
\mathcal{A}^{\infty}\left(  C_{1},\gamma_{1},\omega\right)   &  :=\left\{
{\mathbf{u}} \in L^{\infty}\left(  \omega\right)  \mid\left\Vert \nabla
^{n}{\mathbf{u}}\right\Vert _{L^{\infty}\left(  \omega\right)  }\leq
C_{1}\gamma_{1}^{n}n!\quad\forall n\in\mathbb{N}_{0}\right\}  .
\end{align*}

For the unit sphere $\Gamma$ in $\mathbb{R}^{3}$ and constants $C_{1}$,
$\gamma_{1}$, and the wavenumber $k\geq1$, we set%
\[
\mathcal{A}\left(  C_{1},\gamma_{1},\Gamma\right)  :=\left\{  \mathbf{f}%
\in\mathbf{L}_{T}^{2}\left(  \Gamma\right)  \mid\left\Vert \nabla_{\Gamma}%
^{n}\mathbf{f}\right\Vert _{L^{2}\left(  \Gamma\right)  }\leq C_{1}\gamma
_{1}^{n}\max\left\{  n+1,k\right\}  ^{n}\;\forall n\in\mathbb{N}_{0}\right\}
,
\]
where $\nabla_{\Gamma}$ denotes the surface gradient as in (\ref{curlvec}) and
the application of $\nabla_{\Gamma}^{n}$ to $\mathbf{f}$ is defined componentwise.
\end{definition}

Membership in the analyticity class ${\mathcal{A}}$ is invariant under
analytic changes of variables and multiplication by analytic functions:

\begin{lemma}
\label{lemma:A-invariant} Let $d\in{\mathbb{N}}$ and $\omega_{1}$, $\omega
_{2}\subset{\mathbb{R}}^{d}$ be bounded, open sets. Let $g:\omega
_{1}\rightarrow\omega_{2}$ be a bijection and analytic on the closure
$\overline{\omega_{1}}$, i.e., there are constants $C_{g}$, $C_{g,\operatorname{inv}%
}$, $\gamma_{g}$ such that%
\[
g\in{\mathcal{A}}^{\infty}(C_{g},\gamma_{g},\omega_{1})\quad\mbox{ and }\quad
\Vert(g^{\prime})^{-1}\Vert_{L^{\infty}(\omega_{1})}\leq
C_{g,\operatorname{inv}}.
\]
Let $f$ be analytic on the closure $\overline{\omega_{2}}$, i.e.,
$f\in{\mathcal{A}}^{\infty}(C_{f},\gamma_{f},\omega_{2})$ for some $C_{f}$,
$\gamma_{f}$. Let ${\mathbf{u}}\in{\mathcal{A}}(C_{\mathbf{u}},\gamma
_{\mathbf{u}},\omega_{2})$ for some $C_{\mathbf{u}}$, $\gamma_{\mathbf{u}}$.
Then there are constants $C^{\prime}$, $\gamma^{\prime}>0$ depending solely on
$C_{g}$, $\gamma_{g}$, $C_{g,\operatorname{inv}}$, $\gamma_{\mathbf{u}}$,
$\gamma_{f}$, and $d$, such that $\widetilde{\mathbf{u}}:=f\cdot({\mathbf{u}%
}\circ g)$ satisfies $\widetilde{\mathbf{u}}\in{\mathcal{A}}(C^{\prime}%
C_{f}C_{\mathbf{u}},\gamma^{\prime},\omega_{1})$ .
\end{lemma}

\begin{proof}
The case $d = 2$ is proved in \cite[Lemma~{4.3.1}]{MelenkHabil}. Inspection of
the proof shows, as was already observed in \cite[Lemma~{C.1}]%
{MelenkSauterMathComp}, that it generalizes to arbitrary $d \in{\mathbb{N}}$.
\end{proof}


\subsection{Variational Formulation of the Electric Maxwell
Equations\label{SecMaxwVarForm}}

We formulate (\ref{electricMWEqOmega}) as a variational problem. We introduce
the sesquilinear forms $a_{k}$, $b_{k}$, $A_{k}:\mathbf{X}\times
\mathbf{X}\rightarrow\mathbb{C}$ by
\begin{subequations}
\label{GalDis}
\end{subequations}%
\begin{equation}%
\begin{array}
[c]{ll}%
a_{k}\left(  \mathbf{u},\mathbf{v}\right)  :=\left(  \operatorname{curl}%
\mathbf{u},\operatorname{curl}\mathbf{v}\right)  -k^{2}\left(  \mathbf{u}%
,\mathbf{v}\right)  , & b_{k}\left(  \mathbf{u}_{T},\mathbf{v}_{T}\right)
:=\left(  T_{k}\mathbf{u}_{T},\mathbf{v}_{T}\right)  _{\Gamma},\\
A_{k}\left(  \cdot,\cdot\right)  :=a_{k}\left(  \cdot,\cdot\right)
-\operatorname*{i}kb_{k}\left(  \Pi_{T}\cdot,\Pi_{T}\cdot\right)  . &
\end{array}
\tag{%
\ref{GalDis}%
a}\label{GalDisa}%
\end{equation}
Then, the weak form of the electric Maxwell equations on a bounded domain
$\Omega\subset\mathbb{R}^{3}$ with transparent boundary conditions reads:
\begin{equation}
\text{given }{F}\in\mathbf{X}^{\times}\quad\text{find }\mathbf{E}%
\in\mathbf{X}\quad\text{such that\quad}A_{k}\left(  \mathbf{E},\mathbf{v}%
\right)  =F\left(  \mathbf{v}\right)  \qquad\forall\mathbf{v}\in\mathbf{X}.
\tag{%
\ref{GalDis}%
b}\label{GalDisb}%
\end{equation}
Note that the strong formulation (\ref{electricMWEqOmega}) corresponds to the
choice $F(  \mathbf{v})  =\left(  \operatorname*{i}k \mathbf{\tilde
{\boldsymbol{\jmath}}},\mathbf{v}\right)  $ in (\ref{GalDisb}).

\begin{theorem}
\label{TheoExUniqCont}Let Assumption \ref{AssumptionData} be satisfied. Let
$A_{k}$ have the form (\ref{GalDisa}). Then, for every $F$ $\in\mathbf{X}%
^{\times}$, problem (\ref{GalDisb}) has a unique solution.
\end{theorem}

%

\proof
Let $B_{R}(  0)  $ denote a ball centered at the origin with
sufficiently large radius $R$ such that $\overline{\Omega}\subset B_{R}\left(
0\right)  $. We consider the electric Maxwell equation in $B_{R}\left(
0\right)  $ of the form: Find $\mathbf{E}_{R}\in\mathbf{H}\left(  B_{R}\left(
0\right)  ,\operatorname*{curl}\right)  $ such that for all ${\mathbf{v}}%
\in{\mathbf{H}}(B_{R}(0),\operatorname{curl})$
\begin{equation}
\left(  \operatorname{curl}\mathbf{E}_{R},\operatorname{curl}\mathbf{v}%
\right)  _{L^{2}\left(  B_{R}\left(  0\right)  \right)  }-k^{2}\left(
\mathbf{E}_{R},\mathbf{v}\right)  _{L^{2}\left(  B_{R}\left(  0\right)
\right)  }-\operatorname*{i}k\left(  T_{k,R}\mathbf{E}_{R,T},\mathbf{v}%
_{T}\right)  _{L^{2}\left(  \partial B_{R}\left(  0\right)  \right)  }%
=F_{R}\left(  \mathbf{v}\right)  , \label{GalDisBR}%
\end{equation}
where $T_{k,R}$ is the capacity operator for the exterior domain
$\mathbb{R}^{3}\backslash\overline{B_{R}\left(  0\right)  }$ and $F_{R}$ is
the extension of $F$ by zero, i.e., $F_{R}\left(  \mathbf{v}\right)
:=F\left(  \left.  \mathbf{v}\right\vert _{\Omega}\right)  $. In
\cite[Lem.~{5.4.4} (with $\Gamma=\emptyset$ therein)]{Nedelec01} an ansatz
$\mathbf{E}_{R}\mathbf{=u}_{R}+\nabla p_{R}$ is employed, where $\mathbf{u}%
_{R}$ and $p_{R}$ are the solutions of a variational saddle point problem. In
\cite[Thm.~{5.4.6}]{Nedelec01}, an inf-sup condition is proved for this saddle
point problem which implies the well-posedness of (\ref{GalDisBR}). The
construction implies that $\mathbf{E}:=\left.  \mathbf{E}_{R}\right\vert
_{\Omega}$ then satisfies (\ref{GalDisb}). On the other hand, every solution
$\mathbf{E}$ of (\ref{GalDisb}) can be extended to a solution of
(\ref{GalDisBR}) by employing the well-posedness of the exterior Dirichlet
problem, \cite[Thm.~{5.4.6}]{Nedelec01}. Since (\ref{GalDisBR}) has a unique
solution also the solution of (\ref{GalDisb}) is unique.%
\endproof

\section{Discretization\label{SecNumDisc}}


\subsection{Abstract Galerkin Discretization\label{AbsGalDisc}}


Let $\mathbf{X}_{h}\subset\mathbf{X}$ denote a finite dimensional subspace.
The Galerkin discretization of (\ref{GalDis}) reads: Find $\mathbf{E}_{h}%
\in\mathbf{X}_{h}$ such that%
\begin{equation}
A_{k}\left(  \mathbf{E}_{h},\mathbf{v}_{h}\right)  =F\left(  \mathbf{v}%
_{h}\right)  \qquad\forall\mathbf{v}_{h}\in\mathbf{X}_{h}. \label{GalDiscMaxw}%
\end{equation}
For the error analysis we will impose an assumption (Assumption~\ref{AdiscSp})
on the space ${\mathbf{X}}_{h}$ by requiring the existence of a suitable
projection onto the space $\mathbf{X}_{h}$. Also for the error analysis we
will make use of a space $S_{h}$ such that the following \emph{exact sequence
property} holds:
\begin{equation}
\label{esp}\begin{CD} S_h @> \nabla >> {\mathbf X}_h @> \operatorname*{curl} >> \operatorname*{curl} {\mathbf X}_h \end{CD}
\end{equation}
In the next section we will introduce the N\'ed\'elec space
$\boldsymbol{\mathcal{N}}^{I}_{p}({\mathcal{T}}_{h})$; for the choice
${\mathbf{X}}_{h} = \boldsymbol{\mathcal{N}}^{I}_{p}({\mathcal{T}}_{h})$, we
will perform the error analysis explicitly in the wavenumber $k$, the mesh
width $h$ and the polynomial degree $p$.


\subsection{Curl-Conforming $hp$-Finite Element Spaces\label{SecCurlConfFEM}}

The classical example of curl-conforming FE spaces are the N\'{e}d\'{e}lec
elements, \cite{nedelec80}. We restrict our attention here to so-called
\textquotedblleft type I\textquotedblright\ elements (sometimes also referred
to as the N\'{e}d\'{e}lec-Raviart-Thomas element) on tetrahedra. These spaces
are based on a regular (no hanging), shape-regular triangulation ${\mathcal{T}}_{h}$ of
$\Omega\subset\mathbb{R}^{3}$. That is, ${\mathcal{T}}_{h}$ satisfies:

\begin{enumerate}
[(i)]

\item The (open) elements $K\in{\mathcal{T}}_{h}$ cover $\Omega$, i.e.,
$\overline{\Omega}=\cup_{K\in{\mathcal{T}}_{h}}\overline{K}$.

\item Associated with each element $K$ is the \emph{element map}, a $C^{1}%
$-diffeomorphism $F_{K}:\overline{\widehat{K}} \rightarrow\overline{K}$. The
set $\widehat{K}$ is the \emph{reference tetrahedron}.

\item Denoting $h_{K}=\operatorname*{diam}K$, there holds, with some
\emph{shape-regularity constant $\gamma$},
\begin{equation}
h_{K}^{-1}\Vert F_{K}^{\prime}\Vert_{L^{\infty}(\widehat{K})}+h_{K}\Vert
(F_{K}^{\prime})^{-1}\Vert_{L^{\infty}(\widehat{K})}\leq\gamma.
\label{defhkloc}%
\end{equation}

\item The intersection of two elements is only empty, a vertex, an edge, a
face, or they coincide (here, vertices, edges, and faces are the images of the
corresponding entities on the reference tetrahedron $\widehat{K}$). The
parametrization of common edges or faces are compatible. That is, if two
elements $K$, $K^{\prime}$ share an edge (i.e., $F_{K}(e)=F_{K^{\prime}%
}(e^{\prime})$ for edges $e$, $e^{\prime}$ of $\widehat{K}$) or a face (i.e.,
$F_{K}(f)=F_{K^{\prime}}(f^{\prime})$ for faces $f$, $f^{\prime}$ of
$\widehat{K}$), then $F_{K}^{-1}\circ F_{K^{\prime}}:f^{\prime}\rightarrow f$
is an affine isomorphism.
\end{enumerate}

The following assumption assumes that the element map $F_{K}$ can be
decomposed as a composition of an affine scaling with an $h$-independent
mapping. We adopt the setting of \cite[Sec.~{5}]{MelenkSauterMathComp} and
assume that the element maps $F_{K}$ of the regular, $\gamma$-shape regular
triangulation ${\mathcal{T}}_{h}$ satisfy the following additional requirements:

\begin{assumption}
[normalizable regular triangulation]\label{def:element-maps} Each element map
$F_{K}$ can be written as $F_{K}=R_{K}\circ A_{K}$, where $A_{K}$ is an
\emph{affine} map and the maps $R_{K}$ and $A_{K}$ satisfy for constants
$C_{\operatorname*{affine}}$, $C_{\operatorname{metric}}$, $\gamma>0$
independent of $K$:
\begin{align*}
&  \Vert A_{K}^{\prime}\Vert_{L^{\infty}(\widehat{K})}\leq
C_{\operatorname*{affine}}h_{K},\qquad\Vert(A_{K}^{\prime})^{-1}%
\Vert_{L^{\infty}(\widehat{K})}\leq C_{\operatorname*{affine}}h_{K}^{-1}\\
&  \Vert(R_{K}^{\prime})^{-1}\Vert_{L^{\infty}(\widetilde{K})}\leq
C_{\operatorname{metric}},\qquad\Vert\nabla^{n}R_{K}\Vert_{L^{\infty
}(\widetilde{K})}\leq C_{\operatorname{metric}}\gamma^{n}n!\qquad\forall
n\in{\mathbb{N}}_{0}.
\end{align*}
Here, $\widetilde{K}=A_{K}(\widehat{K})$ and $h_{K}>0$ is the element diameter.
\end{assumption}

\begin{remark}
A prime example of meshes that satisfy Assumption~\ref{def:element-maps} are
those patchwise structured meshes as described, for example, in
\cite[Ex.~{5.1}]{MelenkSauterMathComp} or \cite[Sec.~{3.3.2}]{MelenkHabil}.
These meshes are obtained by first fixing a macrotriangulation of $\Omega$;
the actual triangulation is then obtained as images of affine triangulations
of the reference element. \hbox{}\hfill\rule{0.8ex}{0.8ex}
\end{remark}

On the reference tetrahedron $\widehat{K}$ we introduce the classical
N\'{e}d\'{e}lec type I and Raviart-Thomas elements of degree $p\geq0$ (see,
e.g., \cite{Monkbook}):
\begin{align}
{\mathcal{P}}_{p}(\widehat{K})  &  :=\operatorname*{span}\{x^{\alpha
}\,|\,|\alpha|\leq p\},\label{eq:Pp}\\
\boldsymbol{\mathcal{N}}_{p}^{\operatorname*{I}}(\widehat{K})  &
:=\{\mathbf{p}(x)+x\times\mathbf{q}(x)\,|\,\mathbf{p},\mathbf{q}%
\in({\mathcal{P}}_{p}(\widehat{K}))^{3}\},\label{eq:Np}\\
\mathbf{RT}_{p}(\widehat{K})  &  :=\{\mathbf{p}(x)+x{q}(x)\,|\,\mathbf{p}%
\in({\mathcal{P}}_{p}(\widehat{K}))^{3},{q}\in{\mathcal{P}}_{p}(\widehat{K}%
)\}. \label{eq:RTp}%
\end{align}
The spaces $S_{p+1}({\mathcal{T}}_{h})$, $\boldsymbol{\mathcal{N}}%
_{p}^{\operatorname*{I}}({\mathcal{T}}_{h})$, $\mathbf{RT}_{p}({\mathcal{T}%
}_{h})$, and $Z_{p}({\mathcal{T}}_{h})$ are then defined as in \cite[(3.76),
(3.77)]{Monkbook} by transforming covariantly the space
$\boldsymbol{\mathcal{N}}_{p}^{\operatorname*{I}}(\widehat{K})$ and with the
aid of the Piola transform the space $\mathbf{RT}_{p}(\widehat{K})$:
%
\begin{subequations}
\label{eq:curl-conforming-hp-spaces}%
\begin{align}
S_{p+1}({\mathcal{T}}_{h})  &  :=\{u\in H^{1}(\Omega)\,|\,u|_{K}\circ F_{K}%
\in{\mathcal{P}}_{p+1}(\widehat{K})\},\\
\boldsymbol{\mathcal{N}}_{p}^{\operatorname*{I}}({\mathcal{T}}_{h})  &
:=\{\mathbf{u}\in\mathbf{H}(\Omega,\operatorname*{curl})\,|\,(F_{K}^{\prime
})^{\intercal}\mathbf{u}|_{K}\circ F_{K}\in\boldsymbol{\mathcal{N}}%
_{p}^{\operatorname*{I}}(\widehat{K})\},\\
\mathbf{RT}_{p}({\mathcal{T}}_{h})  &  :=\{\mathbf{u}\in\mathbf{H}%
(\Omega,\operatorname*{div})\,|\,({\operatorname*{det}F_{K}^{\prime}}%
)(F_{K}^{\prime})^{-1}\mathbf{u}|_{K}\circ F_{K}\in\mathbf{RT}_{p}%
(\widehat{K})\},\\
Z_{p}({\mathcal{T}}_{h})  &  :=\{u\in L^{2}(\Omega)\,|\,u|_{K}\circ F_{K}%
\in{\mathcal{P}}_{p}(\widehat{K})\}.
\end{align}
A key property of these spaces is that we have the following exact sequence,
\cite{Monkbook}:
\end{subequations}
\begin{equation}
\begin{CD} \mathbb{R} @> >> S_{p+1}({\mathcal T}_h) @>\nabla >> \boldsymbol{\mathcal N}^I_p({\mathcal T}_h) @>\operatorname{curl} >> \mathbf{RT}_p({\mathcal T}_p) @> \operatorname{div}>> Z_p({\mathcal T}_h) \end{CD}.
\label{exdiscseq_rm}%
\end{equation}

\section{Stability and Error Analysis\label{SecStabErrA}}

\subsection{The Basic Error Estimate\label{SecBasisErrEst}}

\subsubsection{Preliminaries}

The basic error estimates for curl-conforming Galerkin discretization involve
some $k$-dependent sesquilinear forms and corresponding $k$-dependent norms
which, in turn, are based on Helmholtz decompositions on the surface $\Gamma$.
With start this section with these preliminaries. For the proof of the basic
error estimate (Thm.~\ref{LemMonk}), we introduce the sesquilinear form
$\left(
\kern-.1em%
\left(
\kern-.1em%
\cdot,\cdot%
\kern-.1em%
\right)
\kern-.1em%
\right)  :\mathbf{X}\times\mathbf{X}\rightarrow\mathbb{C}$ by%
\begin{equation}
\left(
\kern-.1em%
\left(
\kern-.1em%
\mathbf{u},\mathbf{v}%
\kern-.1em%
\right)
\kern-.1em%
\right)  :=k^{2}\left(  \mathbf{u},\mathbf{v}\right)  +\operatorname*{i}%
kb_{k}\left(  \mathbf{u}^{\nabla},\mathbf{v}^{\nabla}\right)  .
\label{2scprod}%
\end{equation}
We need some definiteness assumptions for the sesquilinear form $b_{k}\left(
\cdot,\cdot\right)  $. Throughout the paper, we will assume:

\begin{assumption}
\label{Assumpsdefb}The sesquilinear form $b_{k}:\mathbf{X}\times
\mathbf{X}\rightarrow\mathbb{C}$ of (\ref{GalDisa}) satisfies
\begin{subequations}
\label{asdefb}
\end{subequations}%
\begin{equation}%
\begin{array}
[c]{ll}%
\operatorname{Im}b_{k}\left(  \mathbf{u}^{\nabla},\mathbf{u}^{\nabla}\right)
\leq0\quad\text{and\quad}\operatorname{Im}b_{k}\left(  \mathbf{u}%
^{\operatorname*{curl}},\mathbf{u}^{\operatorname*{curl}}\right)  \geq0 &
\forall\mathbf{u}\in\mathbf{X},\\
\operatorname{Re}b_{k}\left(  \mathbf{v}_{T},\mathbf{v}_{T}\right)  >0 &
\forall\mathbf{v}\in\mathbf{X}\backslash\left\{  \mathbf{0}\right\}
\end{array}
\tag{%
\ref{asdefb}%
a}\label{asdefba}%
\end{equation}
and%
\begin{equation}
b_{k}\left(  \mathbf{u}^{\nabla},\mathbf{v}^{\operatorname*{curl}}\right)
=b_{k}\left(  \mathbf{u}^{\operatorname*{curl}},\mathbf{v}^{\nabla}\right)
=0\qquad\forall\mathbf{u},\mathbf{v}\in\mathbf{X}. \tag{%
\ref{asdefb}%
b}\label{orthobk}%
\end{equation}

\end{assumption}

For $\Omega$ being the unit ball, the statements in
Assumption~\ref{Assumpsdefb} are proved in \cite[Sec. 5.3.2]{Nedelec01}; see
also Rem. \ref{RemA41}. Assumption (\ref{asdefb}) implies in particular:%
\begin{align}
A_{k}\left(  \mathbf{u},\mathbf{v}\right)   &  =\left(  \operatorname{curl}%
\mathbf{u},\operatorname{curl}\mathbf{v}\right)  {-\operatorname*{i}%
kb_{k}\left(  \mathbf{u}^{\operatorname*{curl}},\mathbf{v}%
^{\operatorname*{curl}}\right)  }-\left(
\kern-.1em%
\left(
\kern-.1em%
\mathbf{u},\mathbf{v}%
\kern-.1em%
\right)
\kern-.1em%
\right)  {,}\label{Akcom2scprod}\\
A_{k}({\mathbf{u}},\nabla\varphi)  &  =-\left(
\kern-.1em%
\left(
\kern-.1em%
\mathbf{u},\nabla{\varphi}%
\kern-.1em%
\right)
\kern-.1em%
\right)  \qquad\forall{\mathbf{u}}\in{\mathbf{X}},\quad\varphi\in H^{1}%
(\Omega). \label{eq:Ak(unablaphi)}%
\end{align}

The stability and convergence analysis of the Galerkin discretization
(\ref{GalDiscMaxw}) involve a) some frequency splittings on the surface
$\Gamma$ and in the domain $\Omega$ as well as b), some Helmholtz
decomposition for the space $\mathbf{X}$. These splittings will be defined
next while their analysis (for the case of the unit ball) is postponed to
Section~\ref{sec:FreqSpOp}.

\begin{definition}
[frequency splittings]\label{DefFreqSplit} Let $\lambda>1$ be a parameter. For
a tangential field with an expansion of the form (\ref{vFExp}), the
low-frequency operator $L_{\Gamma}$ and high-frequency operator $H_{\Gamma}$
are given by%
\[
L_{\Gamma}\mathbf{v}_{T}:=\sum_{1\leq\ell\leq\lambda k}\sum_{m\in\iota_{\ell}%
}\left(  v_{\ell}^{m}\overrightarrow{\operatorname*{curl}\nolimits_{\Gamma}%
}Y_{\ell}^{m}+V_{\ell}^{m}\nabla_{\Gamma}Y_{\ell}^{m}\right)  \quad
\text{and\quad}H_{\Gamma}:=I-L_{\Gamma}.
\]
The mapping $L_{\Omega}:\mathbf{X}\rightarrow\mathbf{X}$ is the solution
operator of the minimization problem:%
\begin{equation}
\left\Vert L_{\Omega}\mathbf{u}\right\Vert _{\operatorname*{curl},\Omega
,k}=\min_{\substack{\mathbf{v}\in\mathbf{X}\\\Pi_{T}\mathbf{v}=L_{\Gamma
}\mathbf{u}_{T}}}\left\Vert \mathbf{v}\right\Vert _{\operatorname*{curl}%
,\Omega,k}. \label{charctLaminproblem}%
\end{equation}
Set $H_{\Omega}:=I-L_{\Omega}$. We introduce the notation%
\begin{equation}
C_{k}^{L,\Omega}:=\sup_{\mathbf{u}\in\mathbf{X}\backslash\left\{  0\right\}
}\frac{\left\Vert L_{\Omega}\mathbf{u}\right\Vert _{\operatorname*{curl}%
,\Omega,k}}{\left\Vert \mathbf{u}\right\Vert _{\operatorname*{curl},\Omega,k}%
}\quad\text{and\quad}C_{k}^{H,\Omega}:=\sup_{\mathbf{u}\in\mathbf{X}%
\backslash\left\{  0\right\}  }\frac{\left\Vert H_{\Omega}\mathbf{u}%
\right\Vert _{\operatorname*{curl},\Omega,k}}{\left\Vert \mathbf{u}\right\Vert
_{\operatorname*{curl},\Omega,k}}. \label{defCkLHOmega}%
\end{equation}

\end{definition}

\begin{remark}
\label{rem:LOmega-weak} Since
\begin{equation}
\mathbf{X}_{0}:=\left\{  \mathbf{w}\in\mathbf{X}\mid\Pi_{T}\mathbf{w}%
=0\right\}  \label{eq:defX0}%
\end{equation}
is a Hilbert space with respect to $\Vert\cdot\Vert_{\operatorname{curl}%
,\Omega,k}$, the operator $L_{\Omega}:{\mathbf{X}}\rightarrow{\mathbf{X}}$ is
well-defined and bounded and linear (see also \cite{TerasseDiss} and
\cite[Lemma~{3.3}]{Veit}). The function $L_{\Omega}\mathbf{u}$ can be
characterized equivalently to (\ref{charctLaminproblem}) as the solution of
the following variational problem: Find $L_{\Omega}\mathbf{u}\in\mathbf{X}$
with $\Pi_{T}L_{\Omega}\mathbf{u}=L_{\Gamma}\mathbf{u}_{T}$ such that%
\begin{equation}
\left(  \operatorname*{curl}L_{\Omega}\mathbf{u},\operatorname*{curl}%
\mathbf{w}\right)  +k^{2}\left(  L_{\Omega}\mathbf{u},\mathbf{w}\right)
=0\qquad\forall\mathbf{w}\in\mathbf{X}_{0}. \label{eq:defLOmega-weak}%
\end{equation}
The strong formulation of (\ref{eq:defLOmega-weak}) is thus%
\begin{subequations}
\label{eq:defLOmega-strong}
\end{subequations}%
\begin{align}
\operatorname*{curl}\operatorname*{curl}L_{\Omega}\mathbf{u}+k^{2}L_{\Omega
}\mathbf{u}=0\qquad &  \text{in }\Omega,\tag{%
\ref{eq:defLOmega-strong}%
a}\label{eq:defLOmega-strong-a}\\
\Pi_{T}L_{\Omega}\mathbf{u}=L_{\Gamma}\mathbf{u}_{T}\qquad &  \text{on
}\partial\Omega\text{.} \tag{%
\ref{eq:defLOmega-strong}%
b}\label{eq:defLOmega-strong-c}%
\end{align}
By applying the divergence operator to (\ref{eq:defLOmega-strong-a}) we get%
\begin{equation}
\operatorname{div}L_{\Omega}{\mathbf{u}}=0\qquad\mbox{in $\Omega$,} \tag{%
\ref{eq:defLOmega-strong}%
c}\label{eq:defLOmega-strong-b}%
\end{equation}
\hbox{}\hfill\rule{0.8ex}{0.8ex}
\end{remark}


Clearly the following commuting properties are valid%
\begin{equation}
\Pi_{T}L_{\Omega}=L_{\Gamma}\Pi_{T}\quad\text{and\quad}\Pi_{T}H_{\Omega}%
=\Pi_{T}-L_{\Gamma}\Pi_{T}=\left(  I-L_{\Gamma}\right)  \Pi_{T}=H_{\Gamma}%
\Pi_{T}. \label{commpropfreq}%
\end{equation}

\begin{remark}
For the special case of a ball $\Omega=B_{1}\left(  0\right)  $, we will
derive in Section~\ref{SecAnLOmega} $k$-independent estimates for the
continuity constants $C_{k}^{L,\Omega}$ and $C_{k}^{H,\Omega}$. In the general
case, one can show estimates of the form $C_{k}^{L,\Omega}\leq\tilde{C}k$ and
$C_{k}^{H,\Omega}\leq1+\tilde{C}k$ for some $\tilde{C}>0$ independent of $k$
by the following argument based on the ($k$-independent) lifting operator
${\mathcal{E}}_{\operatorname{curl}}:{\mathbf{H}}_{\operatorname{curl}}%
^{-1/2}(\Gamma)\rightarrow{\mathbf{H}}(\Omega,\operatorname{curl})$ provided
by Theorem~\ref{traceTHM1}: The ansatz $L_{\Omega}{\mathbf{u}}={\mathbf{U}%
}-{\mathbf{U}}_{0}$ with $\mathbf{U=}\mathcal{E}_{\operatorname*{curl}%
}L_{\Gamma}\mathbf{u}_{T}$ leads to the equation
\[
\operatorname*{curl}\operatorname*{curl}\mathbf{U}_{0}+k^{2}\mathbf{U}%
_{0}=\operatorname*{curl}\operatorname*{curl}\mathbf{U}+k^{2}\mathbf{U}%
\quad\text{in }\Omega,\Pi_{T}\mathbf{U}_{0}=0\quad\text{on }\partial\Omega.
\]
Hence,
\[
\left\Vert \mathbf{U}_{0}\right\Vert _{\operatorname*{curl},\Omega,k}%
\leq\left\Vert \mathbf{U}\right\Vert _{\operatorname*{curl},\Omega,k}\leq
Ck\left\Vert L_{\Gamma}\mathbf{u}_{T}\right\Vert _{-1/2,\operatorname{curl}%
_{\Gamma}}\leq Ck\left\Vert \mathbf{u}_{T}\right\Vert
_{-1/2,\operatorname{curl}_{\Gamma}}\leq CC_{\Gamma}k\left\Vert \mathbf{u}%
\right\Vert _{\operatorname*{curl},\Omega,1}\leq CC_{\Gamma}k\left\Vert
\mathbf{u}\right\Vert _{\operatorname*{curl},\Omega,k}%
\]
from which we get $\left\Vert L_{\Omega}\mathbf{u}\right\Vert
_{\operatorname*{curl},\Omega,k}\leq\tilde{C}k\left\Vert \mathbf{u}\right\Vert
_{\operatorname*{curl},\Omega,k}$, i.e., $C_{k}^{L,\Omega}\leq\tilde{C}k$. The
triangle inequality gives $C_{k}^{H,\Omega}\leq1+C_{k}^{L,\Omega}$.
\hbox{}\hfill\rule{0.8ex}{0.8ex}
\end{remark}

The operators $L_{\Gamma}$ and $L_{\Omega}$ map into low frequency modes which
correspond to smooth functions (since the eigenfunctions $Y_{\ell}^{m}$ are
smooth by the smoothness of $\Gamma$) and, hence, can be approximated well by
$hp$ finite elements. We also use the operators $L_{\Gamma}$ and $H_{\Gamma}$
to define the high- and low frequency parts of the sesquilinear form $b_{k}$.

\begin{definition}
\label{Deflowhigh}The low- and high-frequency parts of the capacity operator
and the sesquilinear form $b_{k}$ are given by%
\begin{equation}%
\begin{array}
[c]{ll}%
T_{k}^{\operatorname*{low}}:=T_{k}L_{\Gamma}, & T_{k}^{\operatorname*{high}%
}:=T_{k}H_{\Gamma},\\
b_{k}^{\operatorname*{low}}\left(  \cdot,\cdot\right)  :=b_{k}\left(
\cdot,L_{\Gamma}\cdot\right)  , & b_{k}^{\operatorname*{high}}\left(
\cdot,\cdot\right)  :=b_{k}\left(  \cdot,H_{\Gamma}\cdot\right)  .
\end{array}
\label{defblowbhigh}%
\end{equation}
The continuity constants of the high-frequency parts of $b_{k}$ are given by
\begin{subequations}
\label{Cbkhightot}
\end{subequations}%
\begin{align}
C_{b,k}^{\nabla,\operatorname*{high}}  &  :=k\sup_{\mathbf{u},\mathbf{v}%
\in\mathbf{X}\backslash\left\{  0\right\}  }\frac{\max\left\{  \left\vert
b_{k}\left(  \mathbf{u}^{\nabla},\left(  H_{\Omega}\mathbf{v}\right)
^{\nabla}\right)  \right\vert ,\left\vert b_{k}\left(  \left(  H_{\Omega
}\mathbf{u}\right)  ^{\nabla},\mathbf{v}^{\nabla}\right)  \right\vert
\right\}  }{\left\Vert \mathbf{u}\right\Vert _{\operatorname*{curl},\Omega
,k}\left\Vert \mathbf{v}\right\Vert _{\operatorname*{curl},\Omega,k}},\tag{%
\ref{Cbkhightot}%
a}\label{defCbkhigh}\\
C_{b,k}^{\operatorname*{curl},\operatorname*{high}}  &  :=\sup_{\mathbf{u}%
,\mathbf{v}\in\mathbf{X}\backslash\left\{  0\right\}  }\frac{k\left\vert
b_{k}\left(  \mathbf{u}^{\operatorname*{curl}},\left(  H_{\Omega}%
\mathbf{v}\right)  ^{\operatorname*{curl}}\right)  \right\vert }{\left\Vert
\mathbf{u}\right\Vert _{\operatorname*{curl},\Omega,k}\left\Vert
\mathbf{v}\right\Vert _{\operatorname*{curl},\Omega,k}}. \tag{%
\ref{Cbkhightot}%
b}\label{defCbkhighcurl}%
\end{align}

\end{definition}

\begin{lemma}
\label{LemCDtN}The capacity operator $T_{k}:\mathbf{H}_{\operatorname*{curl}%
}^{-1/2}\left(  \Gamma\right)  \rightarrow\mathbf{H}_{\operatorname*{div}%
}^{-1/2}\left(  \Gamma\right)  $ is continuous with continuity constant%
\begin{equation}
C_{\operatorname*{DtN},k}:=\left\Vert T_{k}\right\Vert _{\mathbf{H}%
_{\operatorname*{div}}^{-1/2}\left(  \Gamma\right)  \leftarrow\mathbf{H}%
_{\operatorname*{curl}}^{-1/2}\left(  \Gamma\right)  }<\infty.
\label{DefCDtNk}%
\end{equation}
The sesquilinear form $A_{k}:\mathbf{X}\times\mathbf{X}\rightarrow\mathbb{C}$
is continuous. For all $\mathbf{u},\mathbf{v}\in\mathbf{X}$ it
holds\label{contAkbk}
\begin{align}
\max\left\{  \left\vert A_{k}\left(  \mathbf{u},\mathbf{v}\right)  \right\vert
\, ,\,\left\vert \left(
\kern-.1em%
\left(
\kern-.1em%
\mathbf{u},\mathbf{v}%
\kern-.1em%
\right)
\kern-.1em%
\right)  \right\vert \right\}   &  \leq C_{\operatorname*{cont},k}\left\Vert
\mathbf{u}\right\Vert _{\operatorname*{curl},\Omega,1}\left\Vert
\mathbf{v}\right\Vert _{\operatorname*{curl},\Omega,1}\quad\text{with
}C_{\operatorname*{cont},k}:=k^{2}+C_{\Gamma}^{2}C_{\operatorname*{DtN}%
,k}k,\label{Cbk}\\
\max\left\{  \left\vert \left(
\kern-.1em%
\left(
\kern-.1em%
\mathbf{u},H_{\Omega}\mathbf{v}%
\kern-.1em%
\right)
\kern-.1em%
\right)  \right\vert \,,\, \left\vert \left(
\kern-.1em%
\left(
\kern-.1em%
H_{\Omega}\mathbf{u},\mathbf{v}%
\kern-.1em%
\right)
\kern-.1em%
\right)  \right\vert \right\}   &  \leq C_{b,k}^{\operatorname*{high}%
}\left\Vert \mathbf{u}\right\Vert _{\operatorname*{curl},\Omega,k}\left\Vert
\mathbf{v}\right\Vert _{\operatorname*{curl},\Omega,k}\quad\;\;\text{with
}C_{b,k}^{\operatorname*{high}}:=C_{k}^{H,\Omega}+C_{b,k}^{\nabla
,\operatorname*{high}},\label{DefCconthighk}\\
\max\{|A_{k}(H_{\Omega}{\mathbf{u}},{\mathbf{v}})|, |A_{k}({\mathbf{u}%
},H_{\Omega}{\mathbf{v}})|\}  &  \leq C^{\operatorname{high}}%
_{\operatorname{cont},k} \|{\mathbf{u}}\|_{\operatorname{curl},\Omega,k}
\|{\mathbf{v}}\|_{\operatorname{curl},\Omega,k} \label{contAkhigh}%
\end{align}
with $C^{\operatorname{high}}_{\operatorname*{cont},k}:= C^{H,\Omega}_{k} +
C^{\operatorname{curl},\operatorname{high}}_{b,k} + C^{\nabla
,\operatorname{high}}_{b,k}$.
\end{lemma}

%

\proof
The continuity of $T_{k}:\mathbf{H}_{\operatorname*{curl}}^{-1/2}\left(
\Gamma\right)  \rightarrow\mathbf{H}_{\operatorname*{div}}^{-1/2}\left(
\Gamma\right)  $ is asserted in Remark~\ref{ExCapOp}. For the sesquilinear
form $A_{k}$ we employ%
\[
\left\vert A_{k}\left(  \mathbf{u},\mathbf{v}\right)  \right\vert
\leq\left\Vert \mathbf{u}\right\Vert _{\operatorname*{curl},\Omega
,k}\left\Vert \mathbf{v}\right\Vert _{\operatorname*{curl},\Omega
,k}+k\left\vert b_{k}\left(  \mathbf{u}_{T},\mathbf{v}_{T}\right)  \right\vert
.
\]
For the last term, we obtain%
\begin{align}
&  k\left\vert b_{k}\left(  \mathbf{u}_{T},\mathbf{v}_{T}\right)  \right\vert
=k\left\vert \left(  T_{k}\mathbf{u}_{T},\mathbf{v}_{T}\right)  _{\Gamma
}\right\vert \overset{(\ref{eq:duality})}{\leq}k\left\Vert T_{k}\mathbf{u}%
_{T}\right\Vert _{-1/2,\operatorname*{div}_{\Gamma}}\left\Vert \mathbf{v}%
_{T}\right\Vert _{-1/2,\operatorname*{curl}_{\Gamma}}\nonumber\\
&  \qquad\leq C_{\operatorname*{DtN},k}k\left\Vert \mathbf{u}_{T}\right\Vert
_{-1/2,\operatorname*{curl}_{\Gamma}}\left\Vert \mathbf{v}_{T}\right\Vert
_{-1/2,\operatorname*{curl}_{\Gamma}}\overset{\text{(\ref{eq:norm-of-Pi_tau}%
)}}{\leq}C_{\Gamma}^{2}C_{\operatorname*{DtN},k}k\left\Vert \mathbf{u}%
\right\Vert _{\operatorname*{curl},\Omega,1}\left\Vert \mathbf{v}\right\Vert
_{\operatorname*{curl},\Omega,1}. \label{kbkestabs}%
\end{align}
For the continuity bound of the sesquilinear form $\left(
\kern-.1em%
\left(
\kern-.1em%
\cdot,\cdot%
\kern-.1em%
\right)
\kern-.1em%
\right)  $ we obtain similarly as before%
\begin{align*}
\left\vert \left(
\kern-.1em%
\left(
\kern-.1em%
\mathbf{u},\mathbf{v}%
\kern-.1em%
\right)
\kern-.1em%
\right)  \right\vert  &  \leq k^{2}\left\Vert \mathbf{u}\right\Vert
^{2}\left\Vert \mathbf{v}\right\Vert ^{2}+C_{\operatorname*{DtN},k}k\left\Vert
\mathbf{u}^{\nabla}\right\Vert _{-1/2,\operatorname*{curl}_{\Gamma}}\left\Vert
\mathbf{v}^{\nabla}\right\Vert _{-1/2,\operatorname*{curl}_{\Gamma}}\\
&  \overset{(\ref{m1/2curldiv})}{\leq}k^{2}\left\Vert \mathbf{u}\right\Vert
^{2}\left\Vert \mathbf{v}\right\Vert ^{2}+C_{\operatorname*{DtN},k}k\left\Vert
\mathbf{u}_{T}\right\Vert _{-1/2,\operatorname*{curl}_{\Gamma}}\left\Vert
\mathbf{v}_{T}\right\Vert _{-1/2,\operatorname*{curl}_{\Gamma}}\\
&  \leq C_{\operatorname*{cont},k}\left\Vert \mathbf{u}\right\Vert
_{\operatorname*{curl},\Omega,1}\left\Vert \mathbf{v}\right\Vert
_{\operatorname*{curl},\Omega,1}.
\end{align*}
For the high frequency estimate of $\left(
\kern-.1em%
\left(
\kern-.1em%
\cdot,\cdot%
\kern-.1em%
\right)
\kern-.1em%
\right)  $ we employ%
\begin{align*}
\left\vert \left(
\kern-.1em%
\left(
\kern-.1em%
\mathbf{u},H_{\Omega}\mathbf{v}%
\kern-.1em%
\right)
\kern-.1em%
\right)  \right\vert  &  \leq\left(  k\left\Vert \mathbf{u}\right\Vert
\right)  \left(  \left(  k\left\Vert H_{\Omega}\mathbf{v}\right\Vert \right)
\right)  +k\left\vert b_{k}\left(  \mathbf{u}^{\nabla},\left(  H_{\Omega
}\mathbf{v}\right)  ^{\nabla}\right)  \right\vert \\
&  \leq\left(  k\left\Vert \mathbf{u}\right\Vert \right)  C_{k}^{H,\Omega
}\left\Vert \mathbf{v}\right\Vert _{\operatorname*{curl},\Omega,k}%
+C_{b,k}^{\nabla,\operatorname*{high}}\left\Vert \mathbf{u}\right\Vert
_{\operatorname*{curl},\Omega,k}\left\Vert \mathbf{v}\right\Vert
_{\operatorname*{curl},\Omega,k}\\
&  \leq\left(  C_{k}^{H,\Omega}+C_{b,k}^{\nabla,\operatorname*{high}}\right)
\left\Vert \mathbf{u}\right\Vert _{\operatorname*{curl},\Omega,k}\left\Vert
\mathbf{v}\right\Vert _{\operatorname*{curl},\Omega,k}.
\end{align*}
The estimates with interchanged arguments follow along the same lines. The
bound (\ref{contAkhigh}) follows similarly.
\endproof

Next, we introduce frequency-dependent Helmholtz decompositions for the space
$\mathbf{X}$. Let $V\subset H^{1}\left(  \Omega\right)  $ be a closed subspace
(the choice $V=H^{1}\left(  \Omega\right)  $ is allowed). Note that this
implies $\nabla V\subset\mathbf{X}$. Consider the problems%
\begin{align}
&  \text{Given }\mathbf{w}\in\mathbf{X}\text{, find }\Pi_{V}^{\nabla
}{\mathbf{w}}\in\nabla V\quad\text{s.t.\quad}\left(
\kern-.1em%
\left(
\kern-.1em%
\Pi_{V}^{\nabla}{\mathbf{w}},%
\mbox{\boldmath$ \xi$}%
\kern-.1em%
\right)
\kern-.1em%
\right)  =\left(
\kern-.1em%
\left(
\kern-.1em%
\mathbf{w},%
\mbox{\boldmath$ \xi$}%
\kern-.1em%
\right)
\kern-.1em%
\right)  \qquad\forall%
\mbox{\boldmath$ \xi$}%
\in\nabla V.\label{vargammaold}\\
&  \text{Given }\mathbf{w}\in\mathbf{X}\text{, find }\Pi_{V}^{\nabla,\ast
}{\mathbf{w}}\in\nabla V\quad\text{s.t.\quad}\left(
\kern-.1em%
\left(
\kern-.1em%
\mbox{\boldmath$ \xi$}%
,\Pi_{V}^{\nabla,\ast}{\mathbf{w}}%
\kern-.1em%
\right)
\kern-.1em%
\right)  =\left(
\kern-.1em%
\left(
\kern-.1em%
\mbox{\boldmath$ \xi$}%
,\mathbf{w}%
\kern-.1em%
\right)
\kern-.1em%
\right)  \qquad\forall%
\mbox{\boldmath$ \xi$}%
\in\nabla V. \label{vargammaoldadj}%
\end{align}

\begin{lemma}
\label{Lem2prodwp}Let assumption (\ref{asdefba}) be satisfied. Let $V\subset
H^{1}\left(  \Omega\right)  $ be a closed subspace. Then, problems
(\ref{vargammaold}) and (\ref{vargammaoldadj}) are both uniquely solvable.
Thus, the operators $\Pi_{V}^{\nabla}$ and $\Pi_{V}^{\nabla,\ast}$ are well defined.
\end{lemma}

%

\proof
The definiteness of $\operatorname*{Im}b_{k}\left(  \left(  \cdot\right)
^{\nabla},\left(  \cdot\right)  ^{\nabla}\right)  $ (cf. (\ref{asdefba}))
leads to%
\begin{equation}
\operatorname{Re}\left(
\kern-.1em%
\left(
\kern-.1em%
\nabla\xi,\nabla\xi%
\kern-.1em%
\right)
\kern-.1em%
\right)  =\left(  k\left\Vert \nabla\xi\right\Vert \right)  ^{2}%
-k\operatorname{Im}b_{k}\left(  \left(  \nabla\xi\right)  ^{\nabla},\left(
\nabla\xi\right)  ^{\nabla}\right)  \geq\left(  k\left\Vert \nabla
\xi\right\Vert \right)  ^{2} \quad\forall\xi\in H^{1}(\Omega).
\label{k+normdblescest}%
\end{equation}
{}From (\ref{Cbk}) we furthermore get $\left|  \left(
\kern-.1em%
\left(
\kern-.1em%
\mathbf{w},\nabla\xi%
\kern-.1em%
\right)
\kern-.1em%
\right)  \right|  \leq C_{\operatorname{cont},k} \|{\mathbf{w}}%
\|_{\operatorname{curl},\Omega,1} \|\nabla\xi\|_{\operatorname{curl},\Omega,1}
= C_{\operatorname{cont},k} \|{\mathbf{w}}\|_{\operatorname{curl},\Omega,1} (k
\|\nabla\xi\|)$, which shows the well-posedness of $\Pi^{\nabla}_{V}$. The
well-posedness of $\Pi^{\nabla,\ast}_{V}$ is shown analogously.
\endproof

For $V=S_{h}$, we abbreviate $\Pi_{S_{h}}^{\nabla}$ by $\Pi_{h}^{\nabla}$ and
$\Pi_{S_{h}}^{\nabla,\ast}$ by $\Pi_{h}^{\nabla,\ast}$ while for
$V=H^{1}\left(  \Omega\right)  $ we use the shorthands $\Pi^{\nabla}$ for
$\Pi_{H^{1}\left(  \Omega\right)  }^{\nabla}$ and $\Pi^{\nabla,\ast}$ for
$\Pi_{H^{1}\left(  \Omega\right)  }^{\nabla,\ast}$.

A central role in the analysis is played by the spaces (cf.~\cite[p.220]%
{Nedelec01})
%
\begin{subequations}
\label{defVo}%
\begin{align}
\mathbf{V}_{0}  &  :=\left\{  \mathbf{u}\in\mathbf{X}\mid\left(
\kern-.1em%
\left(
\kern-.1em%
\mathbf{u},\nabla\xi%
\kern-.1em%
\right)
\kern-.1em%
\right)  =0\qquad\forall\xi\in H^{1}\left(  \Omega\right)  \right\}
\overset{(\ref{eq:Ak(unablaphi)})}{=}\left\{  {\mathbf{u}}\in{\mathbf{X}%
}\,|\,A_{k}({\mathbf{u}},\nabla\xi)=0\quad\forall\xi\in H^{1}(\Omega)\right\}
,\label{defVoa}\\
\mathbf{V}_{0}^{\ast}  &  :=\left\{  \mathbf{u}\in\mathbf{X}\mid\left(
\kern-.1em%
\left(
\kern-.1em%
\nabla\xi,\mathbf{u}%
\kern-.1em%
\right)
\kern-.1em%
\right)  =0\quad\forall\xi\in H^{1}\left(  \Omega\right)  \right\}
\overset{(\ref{eq:Ak(unablaphi)})}{=}\left\{  {\mathbf{u}}\in{\mathbf{X}%
}\,|\,A_{k}(\nabla\xi,{\mathbf{u}})=0\quad\forall\xi\in H^{1}(\Omega)\right\}
. \label{defVob}%
\end{align}
These spaces of divergence-free functions are the ranges of the operators
$\Pi^{\operatorname*{curl}}$ and $\Pi^{\operatorname*{curl},\ast}$ given by
\end{subequations}
\begin{equation}
\Pi^{\operatorname*{curl}}:=I-\Pi^{\nabla},\quad\Pi^{\operatorname*{curl}%
,\ast}:=I-\Pi^{\nabla,\ast}. \label{eq:Pidivnull}%
\end{equation}

\begin{lemma}
\label{LemNablaCurl}Let Assumption~\ref{Assumpsdefb} be satisfied. There holds
for all $\mathbf{v}\in\mathbf{X}$
\begin{subequations}
\label{stabnablacurl}
\end{subequations}%
\begin{align}
\left\Vert \Pi^{\nabla,\ast} H_{\Omega}\mathbf{v}\right\Vert
_{\operatorname*{curl},\Omega,k}  &  \leq C_{b,k}^{\operatorname*{high}%
}\left\Vert \mathbf{v}\right\Vert _{\operatorname*{curl},\Omega,k},\tag{%
\ref{stabnablacurl}%
a}\label{stabnablacurlnabla}\\
\left\Vert \Pi^{\operatorname*{curl},\ast} H_{\Omega}\mathbf{v} \right\Vert
_{\operatorname*{curl},\Omega,k}  &  \leq\left(  C_{k}^{H,\Omega}%
+C_{b,k}^{\nabla,\operatorname*{high}}\right)  \left\Vert \mathbf{v}%
\right\Vert _{\operatorname*{curl},\Omega,k}. \tag{%
\ref{stabnablacurl}%
b}\label{stabnablacurlcurl}%
\end{align}

\end{lemma}

%

\proof
We employ (\ref{k+normdblescest}) and $\operatorname*{curl}\Pi^{\nabla,\ast
}H_{\Omega}\mathbf{v}=0$ to obtain%
\[
\left(  k\left\Vert \Pi^{\nabla,\ast}H_{\Omega}\mathbf{v}\right\Vert \right)
^{2}\leq\operatorname{Re}\left(
\kern-.1em%
\left(
\kern-.1em%
\Pi^{\nabla,\ast}H_{\Omega}\mathbf{v},\Pi^{\nabla,\ast}H_{\Omega}\mathbf{v}%
\kern-.1em%
\right)
\kern-.1em%
\right)  =\operatorname{Re}\left(
\kern-.1em%
\left(
\kern-.1em%
\Pi^{\nabla,\ast}H_{\Omega}\mathbf{v},H_{\Omega}\mathbf{v}%
\kern-.1em%
\right)
\kern-.1em%
\right)  \overset{(\ref{DefCconthighk})}{\leq}C_{b,k}^{\operatorname*{high}%
}k\left\Vert \Pi^{\nabla,\ast}H_{\Omega}\mathbf{v}\right\Vert \left\Vert
\mathbf{v}\right\Vert _{\operatorname*{curl},\Omega,k},
\]
so that (\ref{stabnablacurlnabla}) follows.
Estimate (\ref{stabnablacurlcurl}) is obtained from (\ref{stabnablacurlnabla})
and the triangle inequality using $\Pi^{\operatorname*{curl},\ast}%
=I-\Pi^{\nabla,\ast}$:%
\[
\left\Vert \Pi^{\operatorname*{curl},\ast}H_{\Omega}\mathbf{v}\right\Vert
_{\operatorname*{curl},\Omega,k}\leq\left\Vert H_{\Omega}\mathbf{v}\right\Vert
_{\operatorname*{curl},\Omega,k}+\left\Vert \Pi^{\nabla,\ast}H_{\Omega
}{\mathbf{v}}\right\Vert _{\operatorname*{curl},\Omega,k}%
\overset{\text{(\ref{defCkLHOmega}), (\ref{stabnablacurlnabla})}}{=}\left(
C_{k}^{H,\Omega}+C_{b,k}^{\nabla,\operatorname*{high}}\right)  \left\Vert
\mathbf{v}\right\Vert _{\operatorname*{curl},\Omega,k}.
\]%
\endproof

It is finally convenient to introduce the discrete counterparts of these
operators:
\begin{equation}
\Pi_{h}^{\operatorname*{curl}}:=I-\Pi_{h}^{\nabla},\quad\Pi_{h}%
^{\operatorname*{curl},\ast}:=I-\Pi_{h}^{\nabla,\ast}. \label{eq:Pidivnullh}%
\end{equation}
The operators $\Pi^{\nabla}$ and $\Pi^{\operatorname*{curl}}$ (analogously
$\Pi^{\nabla,\ast}$ and $\Pi^{\operatorname*{curl},\ast}$) can be used to
define a Helmholtz decomposition of ${\mathbf{u}}\in{\mathbf{X}}$ into a
gradient part and a divergence-free part. Since favorable stability properties
of $\Pi^{\nabla}$ (and thus also of $\Pi^{\operatorname*{curl}}$) will only be
available for high-frequency functions, the decomposition
(\ref{HelmDecompCont}) below involves additionally the frequency-splitting
operators $H_{\Omega}$ and $L_{\Omega}$.

\begin{definition}
[Helmholtz decompositions]\label{DefHelmDecomp}
For $\mathbf{u}$, $\mathbf{v}\in\mathbf{X}$ we set
\begin{subequations}
\label{HelmDecompCont}
\end{subequations}%
\begin{equation}%
\begin{array}
[c]{ll}%
\mathbf{u}=\Pi^{\operatorname*{comp}}\mathbf{u}+\Pi^{\nabla}H_{\Omega
}{\mathbf{u}} & \text{with }\Pi^{\operatorname*{comp}}:=L_{\Omega}%
+\Pi^{\operatorname*{curl}}H_{\Omega},
\end{array}
\tag{%
\ref{HelmDecompCont}%
a}\label{HelmDecompConta}%
\end{equation}
The adjoint splitting is based on the operator $\Pi^{\nabla,\ast}$ and given
by%
\begin{equation}%
\begin{array}
[c]{ll}%
\mathbf{v}=\Pi^{\operatorname*{comp},\ast}\mathbf{v}+\Pi^{\nabla,\ast
}H_{\Omega}{\mathbf{v}} & \text{with }\Pi^{\operatorname*{comp},\ast
}:=L_{\Omega}+\Pi^{\operatorname*{curl},\ast}H_{\Omega}.
\end{array}
\tag{%
\ref{HelmDecompCont}%
b}\label{HelmDecompContb}%
\end{equation}
The discrete counterparts of these splittings are%
\begin{equation}%
\begin{array}
[c]{ll}%
\mathbf{u}=\Pi_{h}^{\operatorname*{comp}}\mathbf{u}+\Pi_{h}^{\nabla}H_{\Omega
}{\mathbf{u}} & \text{with }\Pi_{h}^{\operatorname*{comp}}:=L_{\Omega}+\Pi
_{h}^{\operatorname*{curl}}H_{\Omega},\\
\mathbf{v}=\Pi_{h}^{\operatorname*{comp},\ast}\mathbf{v}+\Pi_{h}^{\nabla,\ast
}H_{\Omega}{\mathbf{v}} & \text{with }\Pi_{h}^{\operatorname*{comp},\ast
}:=L_{\Omega}+\Pi_{h}^{\operatorname*{curl},\ast}H_{\Omega}.
\end{array}
\label{HelmDecompDiscrete}%
\end{equation}

\end{definition}

The next lemma characterizes the spaces ${\mathbf{V}}_{0}$ and ${\mathbf{V}%
}^{\star}_{0}$ in terms of the capacity operator $T_{k}$:

\begin{lemma}
\label{LemEq2scprodstrong}Let Assumption~\ref{Assumpsdefb} be satisfied. Then:
$\mathbf{u}\in\mathbf{V}_{0}$ if and only if \label{ImpBed}%
\begin{equation}
\mathbf{u}\in\mathbf{X}\text{ satisfies\quad}\operatorname{div}\mathbf{u}%
=0\quad\text{in }L^{2}\left(  \Omega\right)  \quad\wedge\quad\operatorname*{i}%
k\left\langle \mathbf{u},\mathbf{n}\right\rangle +\operatorname{div}_{\Gamma
}T_{k}\Pi_{T}\mathbf{u}=0\quad\text{in }H^{-1/2}\left(  \Gamma\right)  .
\label{ImpBeda}%
\end{equation}
Furthermore, $\mathbf{v}\in\mathbf{V}_{0}^{\ast}$ if and only if%
\begin{equation}
\operatorname{div}\mathbf{v}=0\quad\text{in }L^{2}\left(  \Omega\right)
\quad\wedge\quad\operatorname*{i}k\left\langle \mathbf{v},\mathbf{n}%
\right\rangle -\operatorname*{div}\nolimits_{\Gamma}T_{-k}\Pi_{T}%
\mathbf{v}=0\quad\text{in }H^{-1/2}\left(  \Gamma\right)  . \label{ImpBedb}%
\end{equation}

\end{lemma}

%

\proof
We only show the equivalence (\ref{ImpBedb}), since (\ref{ImpBeda}) follows by
the same reasoning. Integration by parts applied to the condition $\left(
\kern-.1em%
\left(
\kern-.1em%
\nabla\xi,\mathbf{v}%
\kern-.1em%
\right)
\kern-.1em%
\right)  =0$ yields, for all $\xi\in H^{1}\left(  \Omega\right)  $,%
\begin{align*}
0  &  =\left(
\kern-.1em%
\left(
\kern-.1em%
\nabla\xi,\mathbf{v}%
\kern-.1em%
\right)
\kern-.1em%
\right)  =k^{2}\left(  \nabla\xi,\mathbf{v}\right)  +\operatorname*{i}%
kb_{k}\left(  \left(  \nabla\xi\right)  ^{\nabla},\mathbf{v}^{\nabla}\right)
=-k^{2}\left(  \xi,\operatorname*{div}\mathbf{v}\right)  +k^{2}\left(
\xi,\left\langle \mathbf{v},\mathbf{n}\right\rangle \right)  _{\Gamma
}+\operatorname*{i}k\left(  T_{k}\left(  \nabla\xi\right)  _{T},\mathbf{v}%
_{T}\right)  _{\Gamma}\\
&  =-k^{2}\left(  \xi,\operatorname*{div}\mathbf{v}\right)  +k^{2}\left(
\xi,\left\langle \mathbf{v},\mathbf{n}\right\rangle \right)  _{\Gamma
}+\operatorname*{i}k\left(  \left(  \nabla\xi\right)  _{T},T_{k}^{\ast
}\mathbf{v}_{T}\right)  _{\Gamma}\\
&  =-k^{2}\left(  \xi,\operatorname*{div}\mathbf{v}\right)  +\operatorname*{i}%
k\left(  \xi,\operatorname*{i}k\left\langle \mathbf{v},\mathbf{n}\right\rangle
-\operatorname*{div}\nolimits_{\Gamma}T_{k}^{\ast}\Pi_{T}\mathbf{v}\right)
_{\Gamma},
\end{align*}
where $T_{k}^{\ast}$ is the adjoint of $T_{k}$ given by%
\begin{equation}
\left(  T_{k}%
\mbox{\boldmath$ \phi$}%
,%
\mbox{\boldmath$ \psi$}%
\right)  _{\Gamma}=\left(
\mbox{\boldmath$ \phi$}%
,T_{k}^{\ast}%
\mbox{\boldmath$ \psi$}%
\right)  _{\Gamma}\qquad\forall%
\mbox{\boldmath$ \phi$}%
,%
\mbox{\boldmath$ \psi$}%
\in\mathbf{H}_{\operatorname*{curl}}^{-1/2}(\Gamma). \label{defTomegaast}%
\end{equation}
Since\footnote{This follows by representing $T_{k}$ by trace operators and
boundary/volume potentials for the electric Maxwell equation as, e.g.,
explained in \cite{Buffa_Hipt_Bem_Maxwell}, and by applying the rules for
computing the adjoint of composite operators.}$\left(  \operatorname*{i}%
kT_{k}\right)  ^{\ast}=-\operatorname*{i}kT_{k}^{\ast}=-\operatorname*{i}%
kT_{-k}$,this is equivalent to (\ref{ImpBedb}).%
\endproof

\begin{corollary}
Let the right-hand side in (\ref{GalDisb}) be defined by $F\left(
\mathbf{v}\right)  =\left(  \operatorname*{i}k\mathbf{\tilde{\boldsymbol{\jmath}}%
},\mathbf{v}\right)  $ for some $\mathbf{\tilde{\boldsymbol{\jmath}}\in H}\left(
\Omega,\operatorname*{div}\right)  $ with $\operatorname*{div}\mathbf{\tilde
{\boldsymbol{\jmath}}}=0$ and $\mathbf{\tilde{\boldsymbol{\jmath}}} \cdot{\mathbf{n}}
= 0$ on $\Gamma$. Then the solution $\mathbf{E}$ satisfies $\mathbf{E}%
\in\mathbf{V}_{0}$.
\end{corollary}

%

\proof
The conditions $\operatorname{div}\tilde{\boldsymbol{\jmath}}=0$ and
$\tilde{\boldsymbol{\jmath}}\cdot {\mathbf n}=0$ imply $(\tilde{\boldsymbol{\jmath}%
},\nabla p)=0$ for all $p\in H^{1}(\Omega)$. Hence, $A_{k}(\mathbf{E},\nabla
p)=F(\nabla p)=0$ for all $p\in H^{1}(\Omega)$. By (\ref{defVoa}) 
we get $\mathbf{E}\in
V_{0}$.
\endproof

Next, we will prove that the spaces $\mathbf{V}_{0}$ and $\mathbf{V}_{0}%
^{\ast}$ are subspaces of $\mathbf{H}^{1}\left(  \Omega\right)  $. For the
special case of $\Gamma$ being the unit sphere, the constants in the norm
equivalences can be determined explicitly -- these details can be found in
Lemma~\ref{Lemembedspec}.

\begin{lemma}
\label{Lemembed}
Let Assumption~\ref{Assumpsdefb} be valid. 
Let $\mathbf{V}_{0}$, $\mathbf{V}_{0}^{\ast}$ be defined as in
(\ref{defVo}). Then,
\begin{equation}
\mathbf{V}_{0}\cup\mathbf{V}_{0}^{\ast}\subset\mathbf{H}^{1}\left(
\Omega\right)  . \label{uinV0uinV0starH1}%
\end{equation}
There exists a constant $C_{\Omega,k}>0$ such that%
\begin{equation}
\left\Vert \mathbf{u}\right\Vert _{H^{1}\left(  \Omega\right)  }\leq
C_{\Omega,k}\left\Vert \mathbf{u}\right\Vert _{\operatorname*{curl},\Omega
,1}\qquad\forall\mathbf{u}\in\mathbf{V}_{0}\cup\mathbf{V}_{0}^{\ast}.
\label{DefCOmegak}%
\end{equation}
Moreover, 
the mappings
$\mathbf{X}\ni\mathbf{u}\mapsto\left(
\kern-.1em%
\left(
\kern-.1em%
\Pi^{\operatorname*{curl}}\mathbf{u},\cdot%
\kern-.1em%
\right)
\kern-.1em%
\right)  \in\mathbf{X}^{\times}$ and $\mathbf{X}\ni\mathbf{v}\mapsto\left(
\kern-.1em%
\left(
\kern-.1em%
\cdot,\Pi^{\operatorname*{curl},\ast}\mathbf{v}%
\kern-.1em%
\right)
\kern-.1em%
\right) \in {\mathbf X}^\prime  $ are compact.
\end{lemma}

%

\proof
Let $\mathbf{u}\in\mathbf{V}_{0}$. The function $T_{k}\Pi_{T}\mathbf{u}$ is
computed by first solving the exterior problem (cf. Remark \ref{ExCapOp})%
\begin{equation}%
\begin{array}
[c]{ll}%
\operatorname*{curl}\operatorname{curl}\mathbf{u}^{+}-k^{2}\mathbf{u}^{+}=0 &
\text{in }\Omega^{+},\\
\left[  \left(  \mathbf{u},\mathbf{u}^{+}\right)  \right]  _{0,\Gamma}=0 &
\text{on }\Gamma
\end{array}
\label{extV0probl}%
\end{equation}
with Silver-M\"{u}ller radiation conditions and then setting $T_{k}\Pi
_{T}\mathbf{u=}\frac{1}{\operatorname*{i}k}\gamma_{T}^{+}\operatorname*{curl}%
\mathbf{u}^{+}$. Since the tangential components of $\mathbf{u}$ and
$\mathbf{u}^{+}$ coincide on $\Gamma$, the function $\mathbf{U}:\mathbb{R}%
^{3}\rightarrow\mathbb{C}$ defined by $\left.  \mathbf{U}\right\vert _{\Omega
}=\mathbf{u}$ and $\left.  \mathbf{U}\right\vert _{\Omega^{+}}=\mathbf{u}^{+}$
(and $\Gamma$ considered as a set of measure zero) is in $\mathbf{H}%
_{\operatorname*{loc}}\left(  \operatorname*{curl},\mathbb{R}^{3}\right)  $.
Then, for all $\mathbf{v}\in\mathbf{C}_{0}^{\infty}\left(  \mathbb{R}%
^{3}\right)  $ it holds%
\begin{align}
&  \left(  \operatorname*{curl}\mathbf{U},\operatorname*{curl}\mathbf{v}%
\right)  _{\mathbf{L}^{2}\left(  \mathbb{R}^{3}\backslash\Gamma\right)
}-k^{2}\left(  \mathbf{U},\mathbf{v}\right)  _{\mathbf{L}^{2}\left(
\mathbb{R}^{3}\right)  }=a_{k}\left(  \mathbf{u},\mathbf{v}\right)  +\left(
\operatorname*{curl}\mathbf{u}^{+},\operatorname*{curl}\mathbf{v}\right)
_{\mathbf{L}^{2}\left(  \Omega^{+}\right)  }-k^{2}\left(  \mathbf{u}%
^{+},\mathbf{v}\right)  _{\mathbf{L}^{2}\left(  \Omega^{+}\right)
}\nonumber\\
&  \text{\qquad\qquad}=a_{k}\left(  \mathbf{u},\mathbf{v}\right)  +\left(
\operatorname*{curl}\operatorname*{curl}\mathbf{u}^{+}-k^{2}\mathbf{u}%
^{+},\mathbf{v}\right)  _{\mathbf{L}^{2}\left(  \Omega^{+}\right)  }-\left(
\gamma_{T}\operatorname*{curl}\mathbf{u}^{+},\mathbf{v}\right)  _{\Gamma
}\nonumber\\
&  \text{\qquad\qquad}=a_{k}\left(  \mathbf{u},\mathbf{v}\right)  -\left(
\gamma_{T}\operatorname*{curl}\mathbf{u}^{+},\mathbf{v}\right)  _{\Gamma
}\nonumber\\
&  \text{\qquad\qquad}=a_{k}\left(  \mathbf{u},\mathbf{v}\right)
-\operatorname*{i}k\left(  T_{k}\mathbf{u}_{T},\mathbf{v}_{T}\right)
_{\Gamma}=a_{k}\left(  \mathbf{u},\mathbf{v}\right)  -\operatorname*{i}%
kb_{k}\left(  \mathbf{u}_{T},\mathbf{v}_{T}\right)  . \label{fulllocal}%
\end{align}
If we test with gradients $\mathbf{v}=\nabla\varphi$ for $\varphi\in
C_{0}^{\infty}\left(  \mathbb{R}^{3}\right)  $ we obtain%
\begin{align*}
\left(  \operatorname*{curl}\mathbf{U},\operatorname*{curl}\nabla
\varphi\right)  _{\mathbf{L}^{2}\left(  \mathbb{R}^{3}\backslash\Gamma\right)
}-k^{2}\left(  \mathbf{U},\nabla\varphi\right)  _{\mathbf{L}^{2}\left(
\mathbb{R}^{3}\right)  }  &  =-k^{2}\left(  \mathbf{U},\nabla\varphi\right)
_{\mathbf{L}^{2}\left(  \mathbb{R}^{3}\right)  }=k^{2}\left(
\operatorname*{div}\mathbf{U},\varphi\right)  _{L^{2}\left(  \mathbb{R}%
^{3}\right)  },\\
\left(  \operatorname*{curl}\mathbf{U},\operatorname*{curl}\nabla
\varphi\right)  _{\mathbf{L}^{2}\left(  \mathbb{R}^{3}\backslash\Gamma\right)
}-k^{2}\left(  \mathbf{U},\nabla\varphi\right)  _{\mathbf{L}^{2}\left(
\mathbb{R}^{3}\right)  }  &  \overset{\text{(\ref{fulllocal})}}{=}a_{k}\left(
\mathbf{u},\nabla\varphi\right)  -\operatorname*{i}kb_{k}\left(
\mathbf{u}_{T},\left(  \nabla\varphi\right)  _{T}\right)
\overset{\text{Rem.~\ref{rem:decomposition-of-nabla-phi}}}{=}-\left(
\kern-.1em%
\left(
\kern-.1em%
\mathbf{u},\nabla\varphi%
\kern-.1em%
\right)
\kern-.1em%
\right)  .
\end{align*}
Since $\mathbf{u}\in\mathbf{V}_{0}$ implies $\left(
\kern-.1em%
\left(
\kern-.1em%
\mathbf{u},\nabla\varphi%
\kern-.1em%
\right)
\kern-.1em%
\right)  =0$, the combination of the previous two equations leads to
$\operatorname*{div}\mathbf{U}=0$ in $\mathbb{R}^{3}$. Hence%
\begin{equation}
\mathbf{U}\in\mathbf{H}_{\operatorname*{loc}}\left(  \mathbb{R}^{3}%
,\operatorname*{div}\right)  \cap\mathbf{H}_{\operatorname*{loc}}\left(
\mathbb{R}^{3},\operatorname*{curl}\right)  . \label{bigUincurldiv}%
\end{equation}
Let $B_{R}\left(  0\right)  $ denote the ball with radius $0<R<\infty$ and
centered at $0$ such that $\overline{\Omega}\subset B_{R}\left(  0\right)  $.

Next, we show $\mathbf{U}\in\mathbf{H}^{1}\left(  B_{R}\left(  0\right)
\right)  $. From (\ref{bigUincurldiv}) we conclude that $\mathbf{v}%
:=\chi\mathbf{U}\in\mathbf{H}_{\operatorname*{loc}}\left(  \mathbb{R}%
^{3},\operatorname*{div}\right)  \cap\mathbf{H}_{\operatorname*{loc}}\left(
\mathbb{R}^{3},\operatorname*{curl}\right)  $ for any smooth cut-off function
$\chi$; in particular we choose $\chi$ such that $\left.  \chi\right\vert
_{B_{R}\left(  0\right)  }=1$. The Fourier transform $\mathbf{\hat{v}}$ then
satisfies $\left\langle
\mbox{\boldmath$ \xi$}%
,\mathbf{\hat{v}}\right\rangle \in L^{2}\left(  \mathbb{R}^{3}\right)  $ as
well as $%
\mbox{\boldmath$ \xi$}%
\times\mathbf{\hat{v}}\in\mathbf{L}^{2}\left(  \mathbb{R}^{3}\right)  $. From
$\left\vert
\mbox{\boldmath$ \xi$}%
\right\vert ^{2}\left\vert \mathbf{\hat{v}}\left(
\mbox{\boldmath$ \xi$}%
\right)  \right\vert ^{2}=\left\langle
\mbox{\boldmath$ \xi$}%
,\mathbf{\hat{v}}\right\rangle ^{2}+\left|
\mbox{\boldmath$ \xi$}%
\times\mathbf{\hat{v}}\right|  ^{2}$ we infer $\nabla\mathbf{v\in L}%
^{2}\left(  \mathbb{R}^{3}\right)  $ and, in turn, $\nabla\mathbf{U}%
\in\mathbf{L}^{2}\left(  B_{R}\left(  0\right)  \right)  $. Since
(\ref{bigUincurldiv}) directly implies $\mathbf{U}\in\mathbf{L}^{2}\left(
B_{R}\left(  0\right)  \right)  $ we have proved $\mathbf{U}\in\mathbf{H}%
^{1}\left(  B_{R}\left(  0\right)  \right)  $ with 
\[
\left\Vert \mathbf{u}\right\Vert _{H^{1}\left(  \Omega\right)  }\leq\left\Vert
\mathbf{U}\right\Vert _{\mathbf{H}^{1}\left(  B_{R}\left(  0\right)  \right)
}\leq C_{R}\left(  \left\Vert \mathbf{U}\right\Vert _{\operatorname*{curl}%
,B_{R}\left(  0\right)  ,1}+\left\Vert \operatorname*{div}\mathbf{U}%
\right\Vert _{\mathbf{L}^{2}\left(  B_{R}\left(  0\right)  \right)  }\right)
.
\]
We already know that $\operatorname*{div}\mathbf{U}=0$ in $\mathbb{R}^{3}$ so
that%
\begin{equation}
\left\Vert \mathbf{u}\right\Vert _{H^{1}\left(  \Omega\right)  }\leq
C_{R}\left(  \left\Vert \mathbf{u}\right\Vert _{\operatorname*{curl},\Omega
,1}+\left\Vert \mathbf{u}^{+}\right\Vert _{\operatorname*{curl},\Omega^{+}\cap
B_{R}\left(  0\right)  ,1}\right)  . \label{uH1uU}%
\end{equation}

An inspection of the proof of \cite[Thm.~{5.4.6}]{Nedelec01} implies that%
\[
\left\Vert \mathbf{u}_{+}\right\Vert _{\operatorname*{curl},\Omega^{+}\cap
B_{R}\left(  0\right)  ,1}\leq C_{k}\left\Vert \gamma_{\tau}^{+}\mathbf{u}%
_{+}\right\Vert _{H_{\operatorname*{div}}^{-1/2}\left(  \Gamma\right)
}\overset{\text{(\ref{extV0probl})}}{=}C_{k}\left\Vert \gamma_{\tau}%
\mathbf{u}\right\Vert _{H_{\operatorname*{div}}^{-1/2}\left(  \Gamma\right)
}\overset{\text{Thm. \ref{traceTHM1}}}{\leq}C_{k}^{\prime}\left\Vert
\mathbf{u}\right\Vert _{\operatorname*{curl},\Omega,1}.
\]
The combination with (\ref{uH1uU}) leads to (\ref{DefCOmegak}) for
$\mathbf{u}\in\mathbf{V}_{0}$ with a constant $C_{\Omega,k}$, possibly
depending on $\Omega$ and $k$. The inclusion $\mathbf{V}_{0}^{\ast}%
\subset\mathbf{H}^{1}\left(  \Omega\right)  $ in (\ref{uinV0uinV0starH1}) and
(\ref{DefCOmegak}) for $\mathbf{u}\in\mathbf{V}_{0}^{\ast}$ follows by the
same reasoning.

Next we prove that the mapping $\mathbf{X}\ni\mathbf{u}\mapsto\left(
\kern-.1em%
\left(
\kern-.1em%
\Pi^{\operatorname*{curl}}\mathbf{u},\cdot%
\kern-.1em%
\right)
\kern-.1em%
\right)  \in\mathbf{X}^{\times}$ is compact. The $\mathbf{L}^{2}\left(
\Omega\right)  $ part of this mapping is compact since $\Pi
^{\operatorname*{curl}}\mathbf{u}\in\mathbf{V}_{0}\subset\mathbf{H}^{1}\left(
\Omega\right)  \overset{\operatorname*{comp}}{\hookrightarrow}\mathbf{L}%
^{2}\left(  \Omega\right)  $. Hence, it remains to prove the compactness of
\begin{equation}
\mathbf{X}\ni\mathbf{u}\mapsto\left(  T_{k}\left(  \Pi^{\operatorname*{curl}%
}\mathbf{u}\right)  ^{\nabla},\left(  \cdot\right)  ^{\nabla}\right)
_{\Gamma}\in\mathbf{X}^{\times}. \label{compdoppelbili}%
\end{equation}
We set $\mathbf{u}_{0}:=\Pi^{\operatorname*{curl}}\mathbf{u}$ and write
$\Pi_{T}\mathbf{u}_{0}=:\mathbf{u}_{0}^{\operatorname*{curl}}+\mathbf{u}%
_{0}^{\nabla}$ according to (\ref{utbcHdecomp}). For an element ${\mathbf{v}%
}\in{\mathbf{X}}$, we decompose ${\mathbf{v}}_{T}={\mathbf{v}}%
^{\operatorname{curl}}+\nabla_{\Gamma}\varphi$ for $\varphi\in H^{1/2}%
(\Gamma)/\mathbb{R}$; the mapping ${\mathbf{X}}\ni{\mathbf{v}}\mapsto
\varphi\in H^{1/2}(\Gamma)/\mathbb{R}$ is continuous. Then
\[
\left(  T_{k}{\mathbf{u}}_{0}^{\nabla},{\mathbf{v}}^{\nabla}\right)  _{\Gamma
}\overset{(\ref{orthobk})}{=}\left(  T_{k}{\mathbf{u}}_{0},{\mathbf{v}%
}^{\nabla}\right)  _{\Gamma}=\left(  T_{k}{\mathbf{u}}_{0},\nabla_{\Gamma
}\varphi\right)  _{\Gamma}=-\left(  \operatorname{div}_{\Gamma}T_{k}%
{\mathbf{u}}_{0},\varphi\right)  _{\Gamma}\overset{\text{(\ref{ImpBeda})}%
}{=}\operatorname*{i}k\left(  \langle{\mathbf{u}}_{0}\cdot{\mathbf{n}}%
\rangle,\varphi\right)  _{\Gamma}.
\]
Since, ${\mathbf{u}}_{0}\in{\mathbf{V}}_{0}\subset{\mathbf{H}}^{1}(\Omega)$,
we have $\langle{\mathbf{u}}_{0}\cdot{\mathbf{n}}\rangle\in H^{1/2}(\Gamma)$.
Hence, we arrive at
\[
\left\vert \left(  T_{k}\left(  \Pi^{\operatorname*{curl}}\mathbf{u}\right)
^{\nabla},\left(  {\mathbf{v}}\right)  ^{\nabla}\right)  _{\Gamma}\right\vert
=\left\vert \operatorname*{i}k\left(  \langle{\mathbf{u}}_{0}\cdot{\mathbf{n}%
}\rangle,\varphi\right)  _{\Gamma}\right\vert \leq k\Vert\langle{\mathbf{u}%
}_{0}\cdot{\mathbf{n}}\rangle\Vert_{H^{1/2}(\Gamma)}\Vert\varphi
\Vert_{H^{-1/2}(\Gamma)}.
\]
Since $\varphi\in H^{1/2}\left(  \Gamma\right)  \overset{\operatorname*{comp}%
}{\hookrightarrow}H^{-1/2}\left(  \Gamma\right)  $ the compactness of the
mapping (\ref{compdoppelbili}) follows.

The compactness of the mapping $\mathbf{X}\ni\mathbf{v}\mapsto\left(  \left(
\cdot\right)  ^{\nabla},T_{k}^{\ast}\left(  \Pi^{\operatorname*{curl},\ast
}\mathbf{v}\right)  ^{\nabla}\right)  _{\Gamma}\in\mathbf{X}^{\prime}$ follows
analogously.
\endproof

\subsubsection{Abstract Error Estimate}

We have collected all ingredients to prove the quasi-optimal error estimate
for the Galerkin solution in the following Theorem~\ref{LemMonk}. It is the
\textquotedblleft Maxwell generalization\textquotedblright\ of the Galerkin
convergence theory for sesquilinear forms satisfying a G{\aa }rding
inequality, going back to \cite{Monk03}; various generalizations of this
technique can be found in \cite{hiptmair-acta,Buffa2005}. We follow
\cite[Sec.~{7.2}]{Monkbook}. For $\mathbf{w}\in\mathbf{X}\backslash\left\{
0\right\}  $ we introduce the quantity%
\begin{equation}
\delta_{k}\left(  {\mathbf{w}}\right)  :=\sup_{{\mathbf{v}}\in{\mathbf{X}}%
_{h}\backslash\left\{  0\right\}  }\left(  2\frac{\operatorname{Re}\left(
\kern-.1em%
\left(
\kern-.1em%
\mathbf{w},\mathbf{v}_{h}%
\kern-.1em%
\right)
\kern-.1em%
\right)  }{\left\Vert \mathbf{w}\right\Vert _{\operatorname{curl},\Omega
,k}\Vert{\mathbf{v}}_{h}\Vert_{\operatorname{curl},\Omega,k}}\right)  .
\label{eq:key-term}%
\end{equation}
We need an adjoint approximation property $\tilde{\eta}_{1}^{\exp}$ defined
via the following dual problem: For given $\mathbf{w}$, $\mathbf{h}%
\in\mathbf{X}$, find $\widehat{\mathcal{N}}\left(  \mathbf{w},\mathbf{h}%
\right)  \in\mathbf{X}$ such that%
\begin{equation}
A_{k}\left(  \mathbf{v},\widehat{\mathcal{N}}(\mathbf{w},\mathbf{h})\right)
=\left(
\kern-.1em%
\left(
\kern-.1em%
\mathbf{v},\mathbf{w}%
\kern-.1em%
\right)
\kern-.1em%
\right)  -\operatorname*{i}kb_{k}\left(  \mathbf{v}^{\operatorname*{curl}%
},\mathbf{h}^{\operatorname*{curl}}\right)  \qquad\forall\mathbf{v}%
\in\mathbf{X}. \label{defNmother}%
\end{equation}
In (\ref{solformdualprob}) we will present an explicit solution formula for
this problem, which directly implies existence of a solution. The operator
$\mathcal{N}_{1}^{\mathcal{A}}:\mathbf{X}\rightarrow\mathbf{X}$ then is given
by $\mathcal{N}_{1}^{\mathcal{A}}\left(  \mathbf{w}\right)
:=\widehat{\mathcal{N}}\left(  L_{\Omega}\mathbf{w},L_{\Omega}\mathbf{w}%
\right)  $, i.e.,
\begin{equation}
A_{k}\left(  \mathbf{v},\mathcal{N}_{1}^{\mathcal A}\mathbf{w}\right)  =\left(
\kern-.1em%
\left(
\kern-.1em%
\mathbf{v},L_{\Omega}\mathbf{w}%
\kern-.1em%
\right)
\kern-.1em%
\right)  -\operatorname*{i}kb_{k}\left(  \mathbf{v}^{\operatorname*{curl}%
},\left(  L_{\Omega}\mathbf{w}\right)  ^{\operatorname*{curl}}\right)
\qquad\forall\mathbf{v}\in\mathbf{X}. \label{adjproblm0}%
\end{equation}
The adjoint approximation property $\tilde{\eta}_{1}^{\exp}$ is defined by%
\begin{equation}
\tilde{\eta}_{1}^{\exp}:=\tilde{\eta}_{1}^{\exp}\left(  \mathbf{X}_{h}\right)
:=\sup_{\mathbf{w}\in\mathbf{X}\backslash\left\{  0\right\}  }\inf
_{\mathbf{z}_{h}\in\mathbf{X}_{h}}\frac{\left\Vert \mathcal{N}_{1}%
^{\mathcal{A}}\mathbf{w}-\mathbf{z}_{h}\right\Vert _{\operatorname*{curl}%
,\Omega,k}}{\left\Vert \mathbf{w}\right\Vert _{\operatorname*{curl},\Omega,k}%
}. \label{defetatildeexp}%
\end{equation}

\begin{theorem}
\label{LemMonk}{Let (\ref{asdefb}) be satisfied. Let $\mathbf{E}\in\mathbf{X}$
and $\mathbf{E}_{h}\in\mathbf{X}_{h}$ satisfy
\begin{equation}
A_{k}\left(  \mathbf{E}-\mathbf{E}_{h},\mathbf{v}_{h}\right)  =0\qquad
\forall\mathbf{v}_{h}\in\mathbf{X}_{h}. \label{7.12}%
\end{equation}
Assume that }$\delta_{k}\left(  \mathbf{e}_{h}\right)  <1$ for {$\mathbf{e}%
_{h}:=\mathbf{E}-\mathbf{E}_{h}$}. {Then, }$\mathbf{e}_{h}${ satisfies, for
all $\mathbf{w}_{h}\in\mathbf{X}_{h}$,} the quasi-optimal error estimate
\begin{equation}
\left\Vert \mathbf{e}_{h}\right\Vert _{\operatorname*{curl},\Omega,k}\leq
\frac{C_{k}^{\operatorname*{I}}+\delta_{k}\left(  {\mathbf{e}}_{h}\right)
}{1-\delta_{k}\left(  \mathbf{e}_{h}\right)  }\left\Vert \mathbf{E}%
-\mathbf{w}_{h}\right\Vert _{\operatorname*{curl},\Omega,k} \label{bserest}%
\end{equation}
with%
\begin{equation}
C_{k}^{\operatorname*{I}}:=1+C_{b,k}^{\operatorname*{high}}+C_{b,k}%
^{\operatorname*{curl},\operatorname*{high}}+C_{\operatorname*{cont},k}%
\tilde{\eta}_{1}^{\exp} \label{defrhokneu}%
\end{equation}

\end{theorem}

%

\proof
The assumed sign conditions of $T_{k}$ (cf. (\ref{asdefba})) imply%
\begin{align*}
\left\Vert \mathbf{e}_{h}\right\Vert _{\operatorname*{curl},\Omega,k}^{2}  &
\leq\left(  \operatorname{curl}\mathbf{e}_{h},\operatorname{curl}%
\mathbf{e}_{h}\right)  +k^{2}\left(  \mathbf{e}_{h},\mathbf{e}_{h}\right)
-k\operatorname{Im}b_{k}\left(  \mathbf{e}_{h}^{\nabla},\mathbf{e}_{h}%
^{\nabla}\right)  {+k\operatorname{Im}b_{k}\left(  \mathbf{e}_{h}%
^{\operatorname*{curl}},\mathbf{e}_{h}^{\operatorname*{curl}}\right)  }\\
&  =\operatorname{Re}A_{k}\left(  \mathbf{e}_{h},\mathbf{e}_{h}\right)
+2\operatorname{Re}\left(
\kern-.1em%
\left(
\kern-.1em%
\mathbf{e}_{h},\mathbf{e}_{h}%
\kern-.1em%
\right)
\kern-.1em%
\right)  .
\end{align*}
We employ Galerkin orthogonality for the first term to obtain for any
$\mathbf{w}_{h}\in\mathbf{X}_{h}$%
\begin{align*}
\left\Vert \mathbf{e}_{h}\right\Vert _{\operatorname*{curl},\Omega,k}^{2}  &
\leq\operatorname{Re}A_{k}\left(  \mathbf{e}_{h},\mathbf{E}-\mathbf{w}%
_{h}\right)  +2\operatorname{Re}\left(
\kern-.1em%
\left(
\kern-.1em%
\mathbf{e}_{h},\mathbf{E}-\mathbf{w}_{h}%
\kern-.1em%
\right)
\kern-.1em%
\right)  +\delta_{k}\left(  {\mathbf{e}}_{h}\right)  \left\Vert \mathbf{e}%
_{h}\right\Vert _{\operatorname*{curl},\Omega,k}\underset{\leq\left\Vert
\mathbf{e}_{h}\right\Vert _{\operatorname*{curl},\Omega,k}+\left\Vert
\mathbf{E}-\mathbf{w}_{h}\right\Vert _{\operatorname*{curl},\Omega
,k}}{\underbrace{\left\Vert \mathbf{E}_{h}-\mathbf{w}_{h}\right\Vert
_{\operatorname*{curl},\Omega,k}}}.
\end{align*}
We write $A_{k}$ in the form (\ref{Akcom2scprod}) so that%
\begin{align}
\left(  1-\delta_{k}\left(  \mathbf{e}_{h}\right)  \right)  \left\Vert
\mathbf{e}_{h}\right\Vert _{\operatorname*{curl},\Omega,k}^{2}  &
\leq\left\vert \left(  \operatorname{curl}\mathbf{e}_{h},\operatorname{curl}%
\left(  \mathbf{E}-\mathbf{w}_{h}\right)  \right)  \right\vert
+\operatorname{Re}\left\{  \left(
\kern-.1em%
\left(
\kern-.1em%
\mathbf{e}_{h},\mathbf{E}-\mathbf{w}_{h}%
\kern-.1em%
\right)
\kern-.1em%
\right)  {-}\operatorname*{i}{kb_{k}\left(  \mathbf{e}_{h}%
^{\operatorname*{curl}},\left(  \mathbf{E}-\mathbf{w}_{h}\right)
^{\operatorname*{curl}}\right)  }\right\} \label{errest1mdelta}\\
&  +\delta_{k}\left(  {\mathbf{e}}_{h}\right)  \left\Vert \mathbf{e}%
_{h}\right\Vert _{\operatorname*{curl},\Omega,k}\left\Vert \mathbf{E}%
-\mathbf{w}_{h}\right\Vert _{\operatorname*{curl},\Omega,k}.\nonumber
\end{align}

The sesquilinear forms in $\left\{  \ldots\right\}  \ $ will be seen to allow
for good continuity constants when applied to high frequency functions while
these constants grow with $k$ when being applied to low frequency functions.
For a function $\mathbf{v}\in\mathbf{X}$ we therefore introduce the splitting
into a high-frequency and low-frequency part $\mathbf{v}=\mathbf{v}%
^{\operatorname*{high}}+\mathbf{v}^{\operatorname*{low}}:=H_{\Omega}%
\mathbf{v}+L_{\Omega}\mathbf{v}$ and get by using (\ref{adjproblm0})%
\begin{align}
\left(
\kern-.1em%
\left(
\kern-.1em%
\mathbf{e}_{h},\mathbf{v}%
\kern-.1em%
\right)
\kern-.1em%
\right)  {-}\operatorname*{i}{kb_{k}\left(  \mathbf{e}_{h}%
^{\operatorname*{curl}},\mathbf{v}^{\operatorname*{curl}}\right)  }  &
{=}\left(
\kern-.1em%
\left(
\kern-.1em%
\mathbf{e}_{h},\mathbf{v}^{\operatorname*{high}}%
\kern-.1em%
\right)
\kern-.1em%
\right)  {-}\operatorname*{i}{kb_{k}\left(  \mathbf{e}_{h}%
^{\operatorname*{curl}},\left(  \mathbf{v}^{\operatorname*{high}}\right)
^{\operatorname*{curl}}\right)  } +\left(
\kern-.1em%
\left(
\kern-.1em%
\mathbf{e}_{h},\mathbf{v}^{\operatorname*{low}}%
\kern-.1em%
\right)
\kern-.1em%
\right)  {-}\operatorname*{i}{kb_{k}\left(  \mathbf{e}_{h}%
^{\operatorname*{curl}},\left(  \mathbf{v}^{\operatorname*{low}}\right)
^{\operatorname*{curl}}\right)  }\nonumber\\
&  =\left(
\kern-.1em%
\left(
\kern-.1em%
\mathbf{e}_{h},\mathbf{v}^{\operatorname*{high}}%
\kern-.1em%
\right)
\kern-.1em%
\right)  {-}\operatorname*{i}{kb_{k}\left(  \mathbf{e}_{h}%
^{\operatorname*{curl}},\left(  \mathbf{v}^{\operatorname*{high}}\right)
^{\operatorname*{curl}}\right)  }+A_{k}\left(  \mathbf{e}_{h},\mathcal{N}%
_{1}^{\mathcal{A}}\mathbf{v}\right)  . \label{sesquicontsplit}%
\end{align}
We employ the continuity estimate of (\ref{defCbkhigh}) to get%
\begin{align*}
\left\vert \left(
\kern-.1em%
\left(
\kern-.1em%
\mathbf{e}_{h},\mathbf{v}^{\operatorname*{high}}%
\kern-.1em%
\right)
\kern-.1em%
\right)  \right\vert  &  \leq\left(  k\left\Vert \mathbf{e}_{h}\right\Vert
\right)  \left(  k\left\Vert \mathbf{v}^{\operatorname*{high}}\right\Vert
\right)  +k\left\vert b_{k}\left(  \mathbf{e}_{h}^{\nabla},\left(
\mathbf{v}^{\operatorname*{high}}\right)  ^{\nabla}\right)  \right\vert \\
&  \leq\left\Vert \mathbf{e}_{h}\right\Vert _{\operatorname*{curl},\Omega
,k}\left\Vert H_{\Omega}\mathbf{v}\right\Vert _{\operatorname*{curl},\Omega
,k}+C_{b,k}^{\nabla,\operatorname*{high}}\left\Vert \mathbf{e}_{h}\right\Vert
_{\operatorname*{curl},\Omega,k}\left\Vert \mathbf{v}\right\Vert
_{\operatorname*{curl},\Omega,k} \leq C_{b,k}^{\operatorname*{high}}\left\Vert
\mathbf{e}_{h}\right\Vert _{\operatorname*{curl},\Omega,k}\left\Vert
\mathbf{v}\right\Vert _{\operatorname*{curl},\Omega,k}.
\end{align*}
For the second term in (\ref{sesquicontsplit}) we use (\ref{defCbkhighcurl})
and obtain in a similar fashion%
\[
\left\vert {kb_{k}\left(  \mathbf{e}_{h}^{\operatorname*{curl}},\left(
\mathbf{v}^{\operatorname*{high}}\right)  ^{\operatorname*{curl}}\right)
}\right\vert \leq C_{b,k}^{\operatorname*{curl},\operatorname*{high}%
}\left\Vert \mathbf{e}_{h}\right\Vert _{\operatorname*{curl},\Omega
,k}\left\Vert \mathbf{v}\right\Vert _{\operatorname*{curl},\Omega,k}.
\]
For the last term in (\ref{sesquicontsplit}), the combination of Galerkin
orthogonality, the continuity estimate (\ref{Cbk}) and the definition of
$\tilde{\eta}_{1}^{\operatorname{exp}}$ in (\ref{defetatildeexp}) gives
\[
\left\vert A_{k}\left(  \mathbf{e}_{h},\mathcal{N}_{1}^{\mathcal{A}}%
\mathbf{v}\right)  \right\vert =\inf_{\mathbf{w}_{h}\in\mathbf{X}_{h}%
}\left\vert A_{k}\left(  \mathbf{e}_{h},\mathcal{N}_{1}^{\mathcal{A}%
}\mathbf{v}-\mathbf{w}_{h}\right)  \right\vert \leq\tilde{\eta}_{1}^{\exp
}C_{\operatorname*{cont},k}\left\Vert \mathbf{e}_{h}\right\Vert
_{\operatorname*{curl},\Omega,k}\left\Vert \mathbf{v}\right\Vert
_{\operatorname*{curl},\Omega,k}.
\]
Thus%
\[
\left\vert \left(
\kern-.1em%
\left(
\kern-.1em%
\mathbf{e}_{h},\mathbf{v}%
\kern-.1em%
\right)
\kern-.1em%
\right)  {-}\operatorname*{i}{kb_{k}\left(  \mathbf{e}_{h}%
^{\operatorname*{curl}},\mathbf{v}^{\operatorname*{curl}}\right)  }\right\vert
\leq\left(  C_{b,k}^{\operatorname*{high}}+C_{b,k}^{\operatorname*{curl}%
,\operatorname*{high}}+\tilde{\eta}_{1}^{\exp}C_{\operatorname*{cont}%
,k}\right)  \left\Vert \mathbf{e}_{h}\right\Vert _{\operatorname*{curl}%
,\Omega,k}\left\Vert \mathbf{v}\right\Vert _{\operatorname*{curl},\Omega,k}.
\]
This allows us to continue the error estimation in (\ref{errest1mdelta})
resulting in%
\[
\left(  1-\delta_{k}\left(  \mathbf{e}_{h}\right)  \right)  \left\Vert
\mathbf{e}_{h}\right\Vert _{\operatorname*{curl},\Omega,k}\leq\left(
1+C_{b,k}^{\operatorname*{high}}+C_{b,k}^{\operatorname*{curl}%
,\operatorname*{high}}+\delta_{k}\left(  {\mathbf{e}}_{h}\right)  +\tilde
{\eta}_{1}^{\exp}C_{\operatorname*{cont},k}\right)  \left\Vert \mathbf{E}%
-\mathbf{w}_{h}\right\Vert _{\operatorname*{curl},\Omega,k}.
\]
%

\endproof

This theorem implies that quasi-optimality of the Galerkin method is ensured
if $\delta_{k}({\mathbf{e}}_{h})<1$. As will be shown in
Theorem~\ref{thm:quasi-optimal} below, this condition also implies existence
and uniqueness of the Galerkin approximation ${\mathbf{E}}_{h}$. In the
following, we will focus on estimating $\delta({\mathbf{e}}_{h})$, heavily
exploiting the Galerkin orthogonality (\ref{7.12}). For the case $\Omega
=B_{1}(0)$ we will derive $k$-explicit estimates for the constants in
(\ref{bserest}) in Corollary~\ref{CorConstantsSphere}. In this case, the
constants $C_{b,k}^{\operatorname{high}}$, $C_{b,k}^{\operatorname{curl}%
,\operatorname{high}}$ are independent of $k$; $C_{\operatorname*{cont}%
,k}=O(k^{3})$ grows algebraically in $k$, which can be offset by controlling
$\tilde{\eta}_{1}^{\exp}$.

\subsection{Splittings of $\mathbf{v}_{h}$ for Estimating of $\delta
({\mathbf{e}}_{h})$}

It remains to estimate $\delta({\mathbf{e}}_{h})$ in (\ref{eq:key-term}). In
this section, we will introduce some frequency-dependent Helmholtz
decompositions for a splitting of the term $\left(
\kern-.1em%
\left(
\kern-.1em%
\mathbf{e}_{h},\mathbf{v}_{h}%
\kern-.1em%
\right)
\kern-.1em%
\right)  $.

For $\mathbf{v}\in\mathbf{X}$ we introduce two decompositions according to
Definition~\ref{DefHelmDecomp}. Let $\mathbf{v}^{\operatorname*{low}%
}:=L_{\Omega}\mathbf{v}$ and $\mathbf{v}^{\operatorname*{high}}:=H_{\Omega
}\mathbf{v}$. Then,%
\begin{equation}%
\begin{array}
[c]{ll}%
\mathbf{v}=\Pi_{h}^{\operatorname*{comp},\ast}\mathbf{v}+\Pi_{h}^{\nabla
,\star}{\mathbf{v}}^{\operatorname*{high}} & \text{with }\Pi_{h}%
^{\operatorname*{comp},\ast}\text{ as in (\ref{HelmDecompDiscrete}),}\\
\mathbf{v}=\Pi^{\operatorname*{comp},\ast}\mathbf{v}+\Pi^{\nabla,\star
}\mathbf{v}^{\operatorname*{high}} & \text{with }\Pi^{\operatorname*{comp}%
,\ast}\text{ as in (\ref{HelmDecompContb}).}%
\end{array}
\label{decoecheck1}%
\end{equation}
An important point to note is that for $\mathbf{v}_{h}\in\mathbf{X}_{h}$ we
have $\Pi_{h}^{\operatorname*{comp},\ast}\mathbf{v}_{h}\in{\mathbf{X}}_{h}$
and, for any $\mathbf{v}\in\mathbf{X}$, we have $\Pi_{h}^{\nabla,\star
}{\mathbf{v}}^{\operatorname*{high}}\in\nabla S_{h}\subset{\mathbf{X}}_{h}$.
However, $\Pi^{\operatorname*{comp}}\mathbf{v}_{h}$ and $\Pi^{\nabla
}{\mathbf{v}}_{h}^{\operatorname*{high}}$ are only in ${\mathbf{X}}$ and
$\nabla H^{1}(\Omega)$. From $\operatorname*{curl}\left(  \Pi_{h}%
^{\nabla,\star}{\mathbf{v}}_{h}^{\operatorname*{high}}\right)  =0$ and
Galerkin orthogonality we conclude that%
\begin{equation}
\left(
\kern-.1em%
\left(
\kern-.1em%
\mathbf{e}_{h},\Pi_{h}^{\nabla,\star}{\mathbf{v}}_{h}^{\operatorname*{high}}%
\kern-.1em%
\right)
\kern-.1em%
\right)  \overset{\text{(\ref{Akcom2scprod}),
Rem.~\ref{rem:decomposition-of-nabla-phi}}}{=}-A_{k}\left(  \mathbf{e}_{h}%
,\Pi_{h}^{\nabla,\star}{\mathbf{v}}_{h}^{\operatorname*{high}}\right)  =0
\label{GalOrth2ndBrack}%
\end{equation}
since $\Pi_{h}^{\nabla,\star}{\mathbf{v}}_{h}^{\operatorname*{high}}\in\nabla
S_{h}\subset{\mathbf{X}}_{h}$.
We employ the splitting%
\[
\mathbf{v}_{h}=\Pi^{\operatorname*{comp},\ast}\mathbf{v}_{h}+\left(  \Pi
_{h}^{\operatorname*{comp},\ast}-\Pi^{\operatorname*{comp},\ast}\right)
\mathbf{v}_{h}+\Pi_{h}^{\nabla,\ast}\mathbf{v}_{h}^{\operatorname*{high}}%
\]
and arrive via (\ref{GalOrth2ndBrack}) at our main splitting
\begin{subequations}
\label{gammamainsplittot}
\end{subequations}%
\begin{align}
\left(
\kern-.1em%
\left(
\kern-.1em%
\mathbf{e}_{h},\mathbf{v}_{h}%
\kern-.1em%
\right)
\kern-.1em%
\right)  =  &  \left(
\kern-.1em%
\left(
\kern-.1em%
\mathbf{e}_{h},\left(  \Pi_{h}^{\operatorname*{comp},\ast}-\Pi
^{\operatorname*{comp},\ast}\right)  \mathbf{v}_{h}%
\kern-.1em%
\right)
\kern-.1em%
\right)  +\left(
\kern-.1em%
\left(
\kern-.1em%
\mathbf{e}_{h},\Pi^{\operatorname*{comp},\ast}\mathbf{v}_{h}%
\kern-.1em%
\right)
\kern-.1em%
\right) \tag{%
\ref{gammamainsplittot}%
a}\label{gammamainsplit}\\
=  &  \left(
\kern-.1em%
\left(
\kern-.1em%
\mathbf{e}_{h},\left(  \Pi_{h}^{\operatorname*{comp},\ast}\mathbf{v}_{h}%
-\Pi^{\operatorname*{comp},\ast}\mathbf{v}_{h}\right)  ^{\operatorname*{high}}%
\kern-.1em%
\right)
\kern-.1em%
\right)  +\left(
\kern-.1em%
\left(
\kern-.1em%
\mathbf{e}_{h},\left(  \Pi_{h}^{\operatorname*{comp},\ast}\mathbf{v}_{h}%
-\Pi^{\operatorname*{comp},\ast}\mathbf{v}_{h}\right)  ^{\operatorname*{low}}%
\kern-.1em%
\right)
\kern-.1em%
\right) \tag{%
\ref{gammamainsplittot}%
b}\label{gammamainsplit3}\\
&  +\left(
\kern-.1em%
\left(
\kern-.1em%
\mathbf{e}_{h},\mathbf{v}_{h}^{\operatorname*{low}}%
\kern-.1em%
\right)
\kern-.1em%
\right)  +\left(
\kern-.1em%
\left(
\kern-.1em%
\mathbf{e}_{h},\Pi^{\operatorname*{curl},\ast}\mathbf{v}_{h}%
^{\operatorname*{high}}%
\kern-.1em%
\right)
\kern-.1em%
\right)  . \tag{%
\ref{gammamainsplittot}%
c}\label{gammamainsplit2}%
\end{align}

\subsection{Adjoint Approximation Properties\label{SecAdjProbl}}

The error analysis involve solution operators for some adjoint problems and we
introduce here corresponding approximation properties that measure how well
these adjoint solutions can be approximated by functions in the Galerkin space
$\mathbf{X}_{h}$ and its companion space $S_{h}$. One of these approximation
properties involve the existence of an interpolating projector that will also
be introduced in this section.

Recall the definition of $\mathbf{V}_{0}^{\ast}$ of (\ref{defVob}). We set%
\begin{equation}
\label{eq:V0h}\mathbf{V}_{0,h}^{\ast}:=\left\{  \mathbf{v}\in\mathbf{V}%
_{0}^{\ast}\mid\operatorname*{curl}\mathbf{v}\in\operatorname*{curl}%
\mathbf{X}_{h}\right\}  .
\end{equation}
The following assumption stipulates the existence of a projector $\Pi_{h}%
^{E}:\mathbf{V}_{0,h}^{\ast}+\operatorname{Range}(L_{\Omega})
+\mathbf{X}_{h}\rightarrow\mathbf{X}_{h}$.

\begin{assumption}
\label{AdiscSp}
There exists a linear operator $\Pi_{h}^{E}:\mathbf{V}_{0,h}^{\ast} +
\operatorname{Range}(L_{\Omega})
+\mathbf{X}_{h}\rightarrow{\mathbf{X}}_{h}$ with the following properties:

\begin{enumerate}
\item[a.] $\Pi_{h}^{E}$ is a projection, i.e., the restriction $\left.
\Pi_{h}^{E}\right\vert _{\mathbf{X}_{h}}$ is the identity on $\mathbf{X}_{h}$.

\item[b.] There exists a companion operator $\Pi_{h}^{F}:\operatorname*{curl}%
\mathbf{X}_{h}\rightarrow\operatorname*{curl}\mathbf{X}_{h}$ with the property
$\operatorname*{curl}\Pi_{h}^{E}=\Pi_{h}^{F}\operatorname*{curl}$.
\end{enumerate}
\end{assumption}

Now we formulate the arising adjoint problems along their solution operators:
We introduce the solution operators $\mathcal{N}_{2}$, $\mathcal{N}%
_{3}^{\mathcal{A}}$ for the following adjoint problems%
\begin{subequations}
\label{smoothdualproblem}
\end{subequations}
%
\begin{align}
A_{k}\left(  \mathbf{w},\mathcal{N}_{2}\mathbf{r}\right)   &  =\left(
\kern-.1em%
\left(
\kern-.1em%
\mathbf{w},\mathbf{r}%
\kern-.1em%
\right)
\kern-.1em%
\right)  \qquad\forall\mathbf{w}\in\mathbf{X},\quad\forall\mathbf{r}%
\in\mathbf{V}_{0}^{\ast},\tag{%
\ref{smoothdualproblem}%
a}\label{adjoint3b}\\
A_{k}\left(  \mathbf{w},\mathcal{N}_{3}^{\mathcal{A}}\mathbf{r}\right)   &
=\left(
\kern-.1em%
\left(
\kern-.1em%
\mathbf{w},L_{\Omega}\mathbf{r}%
\kern-.1em%
\right)
\kern-.1em%
\right)  \qquad\forall\mathbf{w}\in\mathbf{X},\quad\forall\mathbf{r}%
\in\mathbf{X,} \tag{%
\ref{smoothdualproblem}%
b}\label{smoothdualproblemd}%
\end{align}
i.e.,
\[
\mathcal{N}_{2}\mathbf{r}:=\widehat{\mathcal{N}}\left(  \mathbf{r}%
,\mathbf{0}\right)  \quad\text{and\quad}\mathcal{N}_{3}^{\mathcal{A}%
}\mathbf{r}:=\mathcal{N}_{2}\left(  L_{\Omega}\mathbf{r}\right)
\mathbf{=}\widehat{\mathcal{N}}\left(  L_{\Omega}\mathbf{r},\mathbf{0}\right)
.
\]
The solution operator $\mathcal{N}_{4}^{\mathcal{A}}:\mathbf{X}\rightarrow
H^{1}\left(  \Omega\right)  /\mathbb{R}$ is related to some Poisson problem
and given by%
\begin{equation}
-A_{k}(\nabla {\mathcal{N}}_{4}^{\mathcal{A}}{\mathbf{r}},\nabla\xi
)\overset{(\ref{eq:Ak(unablaphi)})}{=}\left(
\kern-.1em%
\left(
\kern-.1em%
\nabla\mathcal{N}_{4}^{\mathcal{A}}\mathbf{r},\nabla\xi%
\kern-.1em%
\right)
\kern-.1em%
\right)  =\left(
\kern-.1em%
\left(
\kern-.1em%
L_{\Omega}\mathbf{r},\nabla\xi%
\kern-.1em%
\right)
\kern-.1em%
\right)  \qquad\forall\xi\in H^{1}\left(  \Omega\right)  . \tag{%
\ref{smoothdualproblem}%
c}\label{graddoppelKlammer}%
\end{equation}

We introduce the adjoint approximation properties\footnote{We write
$\tilde{\eta}_{\ell}$ for an approximation property which involves a
\textit{solution operator and }$\eta_{\ell}$ for a \textquotedblleft
pure\textquotedblright\ approximation property for a given space/set of
functions.}%
\begin{align}
\tilde{\eta}_{2}^{\operatorname{alg}}  &  :=\tilde{\eta}_{2}%
^{\operatorname{alg}}\left(  \mathbf{X}_{h}\right)  :=\sup_{\mathbf{v}_{0}%
\in\mathbf{V}_{0}^{\ast}\backslash\left\{  0\right\}  }\inf_{\mathbf{w}_{h}%
\in\mathbf{X}_{h}}\frac{\left\Vert \mathcal{N}_{2}\mathbf{v}_{0}%
-\mathbf{w}_{h}\right\Vert _{\operatorname*{curl},\Omega,k}}{\left\Vert
\mathbf{v}_{0}\right\Vert _{\operatorname*{curl},\Omega,k}}%
,\label{defetatildealg}\\
\tilde{\eta}_{3}^{\exp}  &  :=\tilde{\eta}_{3}^{\exp}\left(  \mathbf{X}%
_{h}\right)  :=\sup_{\mathbf{r}\in\mathbf{X}\backslash\left\{  0\right\}
}\inf_{\mathbf{w}_{h}\in\mathbf{X}_{h}}\frac{\left\Vert \mathcal{N}%
_{3}^{\mathcal{A}}\mathbf{r}-\mathbf{w}_{h}\right\Vert _{\operatorname*{curl}%
,\Omega,k}}{\left\Vert \mathbf{r}\right\Vert _{\operatorname*{curl},\Omega,k}%
},\label{defetatilde3exp}\\
\tilde{\eta}_{4}^{\exp}  &  :=\tilde{\eta}_{4}^{\exp}\left(  S_{h}\right)
:=\sup_{\mathbf{r}\in\mathbf{X}\backslash\left\{  0\right\}  }\inf_{v_{h}\in
S_{h}}\frac{\left\Vert \nabla\left(  \mathcal{N}_{4}^{\mathcal{A}}%
\mathbf{r}-v_{h}\right)  \right\Vert }{\left\Vert \mathbf{r}\right\Vert
_{\operatorname*{curl},\Omega,1}},\label{DefEta5New}\\
\tilde{\eta}_{5}^{\exp}  &  :=\tilde{\eta}_{5}^{\exp}\left(  \mathbf{X}%
_{h}\right)  :=\sup_{\mathbf{r}\in\mathbf{X}\backslash\left\{  0\right\}
}\inf_{\mathbf{w}_{h}\in\mathbf{X}_{h}}\frac{\left\Vert L_{\Omega}%
\mathbf{r}-\mathbf{w}_{h}\right\Vert _{\operatorname*{curl},\Omega,k}%
}{\left\Vert \mathbf{r}\right\Vert _{\operatorname*{curl},\Omega,k}%
},\label{Defeta6}\\
\eta_{6}^{\operatorname{alg}}  &  :=\eta_{6}^{\operatorname{alg}}\left(
\mathbf{X}_{h},\Pi_{h}^{E}\right)  :=\sup_{\substack{\mathbf{w}\in
\mathbf{V}_{0}^{\ast}\backslash\left\{  0\right\}  \colon
\\\operatorname*{curl}{\mathbf{w}}\in\operatorname*{curl}{\mathbf{X}}_{h}%
}}\frac{k\left\Vert \mathbf{w}-\Pi_{h}^{E}\mathbf{w}\right\Vert }{\left\Vert
\mathbf{w}\right\Vert _{\mathbf{H}^{1}\left(  \Omega\right)  }},\label{PiEhc}%
\\
\tilde{\eta}_{7}^{\exp}  &  :=\tilde{\eta}_{7}^{\exp}\left(  \mathbf{X}%
_{h},\Pi_{h}^{E}\right)  :=\sup_{{\mathbf{r}}\in{\mathbf{X}}\setminus
\{0\}}\frac{k\left\Vert L_{\Omega}\mathbf{r}-\Pi_{h}^{E}L_{\Omega}{\mathbf{r}%
}\right\Vert }{\left\Vert \mathbf{r}\right\Vert _{\operatorname*{curl}%
,\Omega,k}}. \label{Defeta7}%
\end{align}

In Section~\ref{SecSplitting}~
we will derive the following estimates for the terms in (\ref{gammamainsplit2}%
). Let $\mathbf{r}:=\Pi_{h}^{\operatorname*{comp},\ast}\mathbf{v}_{h}%
-\Pi^{\operatorname*{comp},\ast}\mathbf{v}_{h}$. Then%
\begin{align*}
\left\vert \left(
\kern-.1em%
\left(
\kern-.1em%
\mathbf{e}_{h},\mathbf{r}^{\operatorname*{high}}%
\kern-.1em%
\right)
\kern-.1em%
\right)  \right\vert  &  \overset{\text{Prop. \ref{PropHOmegaSplit}}}{\leq
}C_{b,k}^{\operatorname*{high}}C_{r,k}\left\Vert \mathbf{e}_{h}\right\Vert
_{\operatorname*{curl},\Omega,k}\left\Vert \mathbf{v}_{h}\right\Vert
_{\operatorname*{curl},\Omega,k},\\
\left\vert \left(
\kern-.1em%
\left(
\kern-.1em%
\mathbf{e}_{h},\Pi^{\operatorname*{curl},\ast}\mathbf{v}_{h}%
^{\operatorname*{high}}%
\kern-.1em%
\right)
\kern-.1em%
\right)  \right\vert  &  \overset{\text{Prop. \ref{PropehPicurlvhhigh}}}{\leq
}C_{\#\#,k}\left(  C_{\#\#,k}+C_{b,k}^{\operatorname*{curl}%
,\operatorname*{high}}+C_{\operatorname*{cont},k}\tilde{\eta}_{5}^{\exp
}\right)  \tilde{\eta}_{2}^{\operatorname{alg}}\left\Vert \mathbf{e}%
_{h}\right\Vert _{\operatorname*{curl},\Omega,k}\left\Vert \mathbf{v}%
_{h}\right\Vert _{\operatorname*{curl},\Omega,k},\\
\left\vert \left(
\kern-.1em%
\left(
\kern-.1em%
\mathbf{e}_{h},L_{\Omega}\mathbf{r}%
\kern-.1em%
\right)
\kern-.1em%
\right)  \right\vert +\left\vert \left(
\kern-.1em%
\left(
\kern-.1em%
\mathbf{e}_{h},L_{\Omega}\mathbf{v}_{h}%
\kern-.1em%
\right)
\kern-.1em%
\right)  \right\vert  &  \overset{\text{Prop. \ref{PropLowerOrderTerm}}}{\leq
}C_{\operatorname*{cont},k}\tilde{\eta}_{3}^{\exp}\left(  1+C_{r,k}\right)
\left\Vert \mathbf{e}_{h}\right\Vert _{\operatorname*{curl},\Omega
,k}\left\Vert \mathbf{v}_{h}\right\Vert _{\operatorname*{curl},\Omega,k}.
\end{align*}
We combine this together with (\ref{gammamainsplittot}) and (\ref{eq:key-term}%
) to obtain%
\begin{equation}
\delta_{k}\left(  {\mathbf{e}}_{h}\right)  \leq\delta_{k}^{\operatorname{I}%
}:=2\left(  C_{b,k}^{\operatorname*{high}}C_{r,k}+C_{\#\#,k}\left(
C_{\#\#,k}+C_{b,k}^{\operatorname*{curl},\operatorname*{high}}%
+C_{\operatorname*{cont},k}\tilde{\eta}_{5}^{\exp}\right)  \tilde{\eta}%
_{2}^{\operatorname{alg}}+C_{\operatorname*{cont},k}\tilde{\eta}_{3}^{\exp
}\left(  1+C_{r,k}\right)  \right)  . \label{defdeltakromI}%
\end{equation}

\begin{theorem}
\label{thm:quasi-optimal}Let Assumption~\ref{AssumptionData} be satisfied and
let $\mathbf{E}$ be the solution of Maxwell's equations (\ref{GalDisb}).
Assume that $\delta_{k}^{\operatorname{I}}$ in (\ref{defdeltakromI}) is
smaller than $1$. Then the discrete problem (\ref{GalDiscMaxw}) has a unique
solution $\mathbf{E}_{h}$, which satisfies the quasi-optimal error estimate%
\begin{equation}
\left\Vert \mathbf{e}_{h}\right\Vert _{\operatorname{curl},\Omega,k}\leq
\frac{C_{k}^{\operatorname*{I}}+\delta_{k}^{\operatorname{I}}}{1-\delta
_{k}^{\operatorname{I}}}\inf_{\mathbf{w}_{h}\in\mathbf{X}_{h}}\left\Vert
\mathbf{E}-\mathbf{w}_{h}\right\Vert _{\operatorname*{curl},\Omega,k}.
\label{quasiopt}%
\end{equation}

\end{theorem}

%

\proof
The proof uses the same arguments as the proof of \cite[Thm.~{3.9}%
]{MelenkLoehndorf}. Under the assumption that a solution exists, the
quasi-optimal error estimate (\ref{quasiopt}) follows from (\ref{bserest}) and
the assumption $\delta_{k}^{\operatorname{I}}<1$. Next, we will prove
uniqueness of problem (\ref{GalDiscMaxw}). We show that if $\mathbf{E}_{h}$
solves%
\[
A_{k}\left(  \mathbf{E}_{h},\mathbf{v}_{h}\right)  =0\quad\forall
\mathbf{v}_{h}\in\mathbf{X}_{h},
\]
then $\mathbf{E}_{h}=0$. This is the Galerkin discretization of the continuous
problem: Find $\mathbf{E}\in\mathbf{X}$ such that $A_{k}\left(  \mathbf{E}%
,\mathbf{v}\right)  =0$ for all $\mathbf{v}\in\mathbf{X}$. Theorem
\ref{TheoExUniqCont} implies that $\mathbf{E}=\mathbf{0}$ is the unique
solution. Hence $\mathbf{e}_{h}=\mathbf{E}-\mathbf{E}_{h}=-\mathbf{E}_{h}$
satisfies the error estimate%
\[
\left\Vert \mathbf{E}_{h}\right\Vert _{\operatorname{curl},\Omega
,k}=\left\Vert \mathbf{e}_{h}\right\Vert _{\operatorname{curl},\Omega,k}%
\leq\frac{C_{k}^{\operatorname{I}}+\delta_{k}^{\operatorname{I}}}{1-\delta
_{k}^{\operatorname{I}}}\inf_{\mathbf{w}_{h}\in\mathbf{X}_{h}}\left\Vert
\mathbf{E}-\mathbf{w}_{h}\right\Vert _{\operatorname*{curl},\Omega,k}=0
\]
since $\mathbf{E}=\mathbf{0}$. Hence $\mathbf{E}_{h}=0$. Since
(\ref{GalDiscMaxw}) is finite dimensional, uniqueness implies existence.%
\endproof

\subsection{$k$-explicit $hp$-FEM}

In this section, we show that the choice $\left(  {\mathbf{X}}_{h}%
,S_{h}\right)  :=\left(  \boldsymbol{\mathcal{N}}_{p}^{\operatorname*{I}%
}\left(  {\mathcal{T}}_{h}\right)  ,S_{p+1}\left(  \mathcal{T}\right)
\right)  $ for properly chosen mesh size $h$ and $k$-dependent polynomial
degree $p\geq1$ leads to a $k$-\textit{independent} quasi-optimality constant
in (\ref{bserest}). We adopt the setting described in
Section~\ref{SecCurlConfFEM}. That is, we let ${\mathcal{T}}_{h}$ be a mesh
satisfying the assumptions of Section~\ref{SecCurlConfFEM} and
Assumption~\ref{def:element-maps}. 
The 
operators $\Pi_{h}^{E}$ and $\Pi_{h}^{F}$, whose existence is required in
Section~\ref{AbsGalDisc} to Section~\ref{SecApproxOps} may be chosen to be
$\Pi_{p}^{\operatorname*{curl},c}$ and $\Pi_{p}^{\operatorname*{div},c}$ of 
Theorem~\ref{thm:projection-based-interpolation}.
%
%

\subsubsection{Applications to the Case $\Omega= B_{1}(0)$}

The adjoint approximation properties $\tilde{\eta}_{\ell}^{\operatorname{alg}%
}$, $\tilde{\eta}_{\ell}^{\operatorname{exp}}$ involve solution operators
whose regularity are investigated in Sections~\ref{SecAnLOmega} and
\ref{sec:dual-problems} for the unit ball $\Omega= B_{1}(0)$.
In particular, we show in Proposition~\ref{PropN2Arough} that the solution
operator $\mathcal{N}_{2}$ allows for a stable additive splitting
$\mathcal{N}_{2}=\mathcal{N}_{2}^{\operatorname*{rough}}+\mathcal{N}%
_{2}^{\mathcal{A}}$, where $\mathcal{N}_{2}^{\mathcal{A}}$ maps into some
analyticity class and $\mathcal{N}_{2}^{\operatorname*{rough}}:\mathbf{V}%
_{0}^{\ast}\rightarrow\mathbf{H}^{2}\left(  \Omega\right)  $ satisfies the
estimate $\left\Vert \mathcal{N}_{2}^{\operatorname*{rough}}\mathbf{v}%
_{0}\right\Vert _{\mathbf{H}^{2}\left(  \Omega\right)  }\leq
C_{\operatorname*{rough}}k\left\Vert \mathbf{v}_{0}\right\Vert
_{\operatorname*{curl},\Omega,k}$. In Theorem~\ref{TheoLau} and
Propositions~\ref{PropN2Arough}, \ref{PropN3}, \ref{PropN4A}, \ref{PropRegN1A}
we show that all other solution operators map into some analyticity class,
more precisely, for all $\mathbf{r}\in\mathbf{X}$ and $\mathbf{v}_{0}%
\in\mathbf{V}_{0}^{\ast}$, it holds\footnote{For the last relation, we have
estimated $\left\Vert \cdot\right\Vert _{\operatorname*{curl},\Omega,1}%
\leq\left\Vert \cdot\right\Vert _{\operatorname*{curl},\Omega,k}$ in
(\ref{defCuprime}) (using (\ref{loweromega})) to simplify technicalities.}
with $\alpha_{1}=3$, $\alpha_{2}=3$, $\alpha_{3}=3$, $\alpha_{4}=5/2$,
$\alpha_{5}=3/2$
%
\begin{subequations}
\label{eq:analyticity-classes-d}%
\begin{align}
\mathcal{N}_{j}^{\mathcal{A}}\mathbf{r}  &  \in\mathcal{A}\left(
C_{\mathcal{A},j}k^{\alpha_{j}}\left\Vert \mathbf{r}\right\Vert
_{\operatorname*{curl},\Omega,k},\gamma_{\mathcal{A},j},\Omega\right)  ,\quad
j=1,3,\\
\mathcal{N}_{2}^{\mathcal{A}}\mathbf{v}_{0}  &  \in\mathcal{A}\left(
C_{\mathcal{A},2}k^{\alpha_{2}}\left\Vert \mathbf{v}_{0}\right\Vert
_{\operatorname*{curl},\Omega,k},\gamma_{\mathcal{A},2},\Omega\right)  ,\\
\nabla\mathcal{N}_{4}^{\mathcal{A}}\mathbf{r}  &  \in\mathcal{A}\left(
C_{\mathcal{A},4}k^{\alpha_{4}}\left\Vert \mathbf{r}\right\Vert
_{\operatorname*{curl},\Omega,1},\gamma_{\mathcal{A},4},\Omega\right)  ,\\
L_{\Omega}\mathbf{r}  &  \in\mathcal{A}\left(  C_{\mathcal{A},5}k^{\alpha_{5}%
}\left\Vert \mathbf{r}\right\Vert _{\operatorname*{curl},\Omega,1}%
,\gamma_{\mathcal{A},5},\Omega\right)  \text{,}%
\end{align}
This allows us to estimate those adjoint approximations that involve solution
operators by simpler approximation properties, which we will introduce next.
We set%
\end{subequations}
\begin{align}
\eta_{1}^{\exp}\left(  \gamma,\mathbf{X}_{h}\right)   &  :=\sup_{\mathbf{z}%
\in\mathcal{A}\left(  1,\gamma,\Omega\right)  }\inf_{\mathbf{w}_{h}%
\in\mathbf{X}_{h}}\left\Vert \mathbf{z}-\mathbf{w}_{h}\right\Vert
_{\operatorname*{curl},\Omega,k},\label{defeta1exp}\\
\eta_{2}^{\operatorname{alg}}\left(  \mathbf{X}_{h}\right)   &  :=\sup
_{\substack{{\mathbf{z}}\in{\mathbf{H}}^{2}(\Omega)\\\Vert{\mathbf{z}}%
\Vert_{{\mathbf{H}}^{2}(\Omega)}\leq k}} \inf_{{\mathbf{w}}_{h} \in
{\mathbf{X}}_{h}} \left\Vert {\mathbf{z}}-{\mathbf{w}}_{h} \right\Vert
_{\operatorname{curl},\Omega,k},\label{defeta2alg}\\
\eta_{4}^{\exp}\left(  \gamma,S_{h}\right)   &  :=\sup_{\nabla z\in
\mathcal{A}\left(  1,\gamma,\Omega\right)  }\inf_{v_{h}\in S_{h}}\left\Vert
\nabla\left(  z-v_{h}\right)  \right\Vert ,\label{defeta4exp}\\
\eta_{7}^{\exp}\left(  \gamma,{\mathbf{X}}_{h}\right)   &  :=k\sup
_{\mathbf{z}\in\mathcal{A}\left(  1,\gamma,\Omega\right)  }\left\Vert
\mathbf{z}-\Pi_{h}^{E}{\mathbf{z}}\right\Vert , \label{defeta7exp}%
\end{align}
and obtain
\begin{equation}%
\begin{array}
[c]{ll}%
\tilde{\eta}_{1}^{\exp}\leq C_{\mathcal{A},1}k^{\alpha_{1}}\eta_{1}^{\exp
}\left(  \gamma_{\mathcal{A},1},\mathbf{X}_{h}\right)  , & \tilde{\eta}%
_{2}^{\operatorname{alg}}\leq C_{\operatorname*{rough}}\eta_{2}%
^{\operatorname{alg}}\left(  \mathbf{X}_{h}\right)  +C_{\mathcal{A}%
,2}k^{\alpha_{2}}\eta_{1}^{\exp}\left(  \gamma_{\mathcal{A},2},\mathbf{X}%
_{h}\right)  ,\\
\tilde{\eta}_{3}^{\exp}\leq C_{\mathcal{A},3}k^{\alpha_{3}}\eta_{1}^{\exp
}\left(  \gamma_{\mathcal{A},3},\mathbf{X}_{h}\right)  , & \tilde{\eta}%
_{4}^{\exp}\leq C_{\mathcal{A},4}k^{\alpha_{4}}\eta_{4}^{\exp}\left(
\gamma_{\mathcal{A},4},S_{h}\right)  ,\\
\tilde{\eta}_{5}^{\exp}\leq C_{\mathcal{A},5}k^{\alpha_{5}}\eta_{1}^{\exp
}\left(  \gamma_{\mathcal{A},5},\mathbf{X}_{h}\right)  , & \tilde{\eta}%
_{7}^{\exp}\leq C_{\mathcal{A},5}k^{\alpha_{5}}\eta_{7}^{\exp}\left(
\gamma_{\mathcal{A},5},\mathbf{X}_{h}\right)  .
\end{array}
\label{eta06est}%
\end{equation}

\begin{corollary}
\label{cor:cond-for-quasioptimality} Let $\Omega=B_{1}(0)$ and recall the
definition of $\alpha_{\ell}$ before (\ref{eq:analyticity-classes-d}). Define%
\begin{equation}%
\begin{array}
[c]{ll}%
\tilde{\eta}_{1,\ast}^{\exp}:=\max_{j\in\left\{  1,2,3,5\right\}
}C_{\mathcal{A},j}\eta_{1}^{\exp}\left(  \gamma_{\mathcal{A},j},\mathbf{X}%
_{h}\right)  , & \tilde{\eta}_{2,\ast}^{\operatorname{alg}}%
:=C_{\operatorname*{rough}}\eta_{2}^{\operatorname{alg}}\left(  \mathbf{X}%
_{h}\right)  ,\\
\tilde{\eta}_{4,\ast}^{\exp}:=C_{\mathcal{A},4}\eta_{4}^{\exp}\left(
\gamma_{\mathcal{A},4},S_{h}\right)  , & \tilde{\eta}_{7,\ast}^{\exp
}:=C_{\mathcal{A},5}\eta_{7}^{\exp}\left(  \gamma_{\mathcal{A},5}%
,\mathbf{X}_{h}\right)  .
\end{array}
\label{etashort}%
\end{equation}
For $0<\tau\leq1$ sufficiently small but independent of $k$, and any
$0<\varepsilon_{\ell}\leq\tau$, $\ell\in\left\{  1,2,4,6,7\right\}  $, select
the mesh size $h$ and the polynomial degree $p$ for the $hp$-finite element
space $\mathbf{X}_{h}$ such that $\mathbf{X}_{h}$ and its companion space
$S_{h}$ (cf.~(\ref{esp})) satisfy Assumption~\ref{AdiscSp} and the
approximation properties:%
\begin{equation}
k^{\alpha_{3}+3}\tilde{\eta}_{1,\ast}^{\exp}\leq\varepsilon_{1},\quad
\tilde{\eta}_{2,\ast}^{\operatorname{alg}}\leq\varepsilon_{2},\quad
k^{\alpha_{4}+1}\tilde{\eta}_{4,\ast}^{\exp}\leq\varepsilon_{4},\quad\eta
_{6}^{\operatorname{alg}}\leq\varepsilon_{6},\quad k^{\alpha_{5}}\tilde{\eta
}_{7,\ast}^{\exp}\leq\varepsilon_{7}. \label{smallnessconditions}%
\end{equation}
Then, the quantity $\delta_{k}^{\operatorname{I}}$ in (\ref{defdeltakromI}),
(\ref{quasiopt}) can be estimated by $\delta_{k}^{\operatorname{I}}<1/2$, and
the discrete problem (\ref{GalDiscMaxw}) has a unique solution $\mathbf{E}%
_{h}$, which satisfies the quasi-optimal error estimate%
\begin{equation}
\left\Vert \mathbf{e}_{h}\right\Vert _{\operatorname{curl},\Omega,k}\leq
C\inf_{\mathbf{w}_{h}\in\mathbf{X}_{h}}\left\Vert \mathbf{E}-\mathbf{w}%
_{h}\right\Vert _{\operatorname*{curl},\Omega,k} \label{quasiopt2}%
\end{equation}
for a constant $C$ independent of $k$.
\end{corollary}

%

\proof
We estimate $\delta_{k}^{\operatorname{I}}$ of (\ref{defdeltakromI}) termwise
by using (\ref{eta06est}), (\ref{etashort}), and the values of $\alpha_{j}$.
From Corollary \ref{CorConstantsSphere}, we deduce that the constants
$C_{\#,k}$, $C_{\#\#,k}$ in (\ref{defCrk}) and (\ref{est2ndtermsplit}) are in
fact bounded uniformly in $k$. Hence%
\[
C_{r,k}\leq C\left(  \varepsilon_{6}+\varepsilon_{7}\right)
\]
for a constant $C$ independent of $k$. Again from Corollary
\ref{CorConstantsSphere} and (\ref{eta06est}), it follows that%
\[
\delta_{k}^{\operatorname{I}}\leq C\left(  \varepsilon_{6}+\varepsilon
_{7}+\left(  1+k^{\alpha_{5}+3}\tilde{\eta}_{1,\ast}^{\exp}\right)  \left(
\tilde{\eta}_{2,\ast}^{\operatorname{alg}}+k^{\alpha_{2}}\tilde{\eta}_{1,\ast
}^{\exp}\right)  +k^{\alpha_{3}+3}\tilde{\eta}_{1,\ast}^{\exp}\left(
1+\varepsilon_{6}+\varepsilon_{7}\right)  \right)
\]
for a constant $C$ independent of $k$. We use $\alpha_{3}+3\geq\max\left\{
\alpha_{1}+3,\alpha_{5}+3,\alpha_{2}\right\}  $ and the conditions in
(\ref{smallnessconditions}) along with $\varepsilon_{\ell}\leq\tau\leq1$ to
obtain%
\[
\delta_{k}^{\operatorname{I}}\leq C\left(  \varepsilon_{1}+\varepsilon
_{2}+\varepsilon_{6}+\varepsilon_{7}\right)  \leq\tilde{C}\tau
\]
for a constant $\tilde{C}$ independent of $k$. Hence, the condition
$0<\tau<\left(  2\tilde{C}\right)  ^{-1}$ implies $\delta_{k}%
^{\operatorname{I}}<1/2$ and existence and uniqueness of the discrete solution
follow from Theorem \ref{thm:quasi-optimal}.

To prove that the quasi-optimality constant $C$ in (\ref{quasiopt2}) is
independent of $k$ we use (\ref{quasiopt}) so that it remains to prove that
$C_{k}^{\operatorname{I}}$ in (\ref{quasiopt}) (cf.~(\ref{defrhokneu})) is
bounded independently of $k$. This, in turn, is a direct consequence of
Corollary~\ref{CorConstantsSphere} and
\[
C_{\operatorname*{cont},k}\tilde{\eta}_{1}^{\exp}\overset{\text{Cor.
\ref{CorConstantsSphere}}}{\leq}Ck^{3}\tilde{\eta}_{1}^{\exp}%
\overset{\text{(\ref{eta06est})}}{\leq}Ck^{\alpha_{1}+3}\eta_{1}^{\exp}\leq
Ck^{\alpha_{3}+3}\eta_{1}^{\exp}\overset{\text{(\ref{smallnessconditions}%
)}}{\leq}C\varepsilon_{1}\leq C\tau\leq C
\]
independent of $k$.%
\endproof

\subsubsection{$hp$-FEM: Results\label{Seceta1-4}}


\begin{theorem}
\label{thm:hpFEM-quasioptimality}Let $\Omega=B_{1}\left(  0\right)  $ be the
unit ball and let ${\mathbf{E}}$ denote the exact solution of (\ref{GalDisb}).
Let the mesh ${\mathcal{T}}_{h}$ satisfy Assumption~\ref{def:element-maps} and
set $h:=\max_{K\in{\mathcal{T}}}h_{K}$. Let $S_{h}=S_{p+1}({\mathcal{T}}_{h})$
and ${\mathbf{X}}_{h}=\boldsymbol{\mathcal{N}}_{p}^{\operatorname*{I}%
}({\mathcal{T}}_{h})$. Fix $c_{2} > 0$. Then there exist constants $c_{1}$,
$C>0$ depending solely on $R$, $c_2$, 
and the constants $C_{\operatorname{affine}}$,
$C_{\operatorname{metric}}$, $\gamma$ of Assumption~\ref{def:element-maps}
such the following holds: If $k\geq1$, $p\geq1$, $h > 0$ satisfy
\begin{equation}
\frac{kh}{p}\leq c_{1}\quad\mbox{ and }\quad p\geq c_{2}\log k,
\label{eq:hpFEM-scale-res}%
\end{equation}
then the Galerkin approximation ${\mathbf{E}}_{h}\in{\mathbf{X}}_{h}$
(cf.~(\ref{GalDiscMaxw})) exists and satisfies
\begin{equation}
\left\Vert {\mathbf{E}}-{\mathbf{E}}_{h}\right\Vert _{\Omega
,\operatorname{curl},k}\leq C\inf_{{\mathbf{w}}_{h}\in{\mathbf{X}}_{h}%
}\left\Vert {\mathbf{E}}-{\mathbf{w}}_{h}\right\Vert _{\Omega
,\operatorname{curl},k}. \label{eq:thm:hpFEM-quasioptimality-10}%
\end{equation}

\end{theorem}

%

\proof
The proof consists in checking the conditions of
Corollary~\ref{cor:cond-for-quasioptimality}. The infima in $\eta
_{2}^{\operatorname{alg}}$, $\eta_{j}^{\operatorname{exp}}$, $j\in\{1,4\}$ are
estimated with the aid the specific approximation operator $\Pi
^{\operatorname{curl},s}$ analyzed in Lemma~\ref{lemma:Picurls-approximation}%
:
\begin{equation}
\eta_{2}^{\operatorname{alg}}\leq\sup_{\substack{{\mathbf{z}}\in{\mathbf{H}%
}^{2}(\Omega)\\\Vert{\mathbf{z}}\Vert_{{\mathbf{H}}^{2}(\Omega)}\leq
k}}\left\Vert {\mathbf{z}}-\Pi_{p}^{\operatorname*{curl},s}{\mathbf{z}%
}\right\Vert _{\operatorname{curl},\Omega,k}%
\overset{\text{Lemma~\ref{lemma:Picurls-approximation},
(\ref{item:lemma:Picurls-approximation-i})}}{\lesssim}\left(  \frac{h}%
{p}+\frac{h^{2}}{p^{2}}k\right)  k=\frac{kh}{p}+\left(  \frac{kh}{p}\right)
^{2}. \label{esteta2alfhp}%
\end{equation}
The terms $\eta_{j}^{\operatorname{exp}}$, $j\in\left\{  1,4\right\}  $
involve the approximation of analytic functions: The term $\eta_{1}%
^{\operatorname{exp}}$ is an approximation from ${\mathbf{X}}_{h}%
=\boldsymbol{\mathcal{N}}_{p}^{I}({\mathcal{T}}_{h})$ and estimated with
Lemma~\ref{lemma:Picurls-approximation},
(\ref{item:lemma:Picurls-approximation-ii}); the term $\eta_{4}%
^{\operatorname{exp}}$ contains an approximation from $S_{h}=S_{p+1}%
({\mathcal{T}}_{h})$ and is taken from the proof of \cite[Thm.~{5.5}%
]{MelenkSauterMathComp}:
\begin{equation}
\sum_{j\in\left\{  1,4\right\}  }\eta_{j}^{\exp}\lesssim\left(  \frac
{h}{h+\sigma}\right)  ^{p}+k\left(  \frac{kh}{\sigma p}\right)  ^{p}+k\left\{
\left(  \frac{h}{h+\sigma}\right)  ^{p+1}+\left(  \frac{kh}{\sigma p}\right)
^{p+1}\right\}  . \label{edtatilde14est}%
\end{equation}
The terms $\eta_{6}^{\operatorname{alg}}$, $\eta_{7}^{\exp}$ involve the
operator $\Pi_{p}^{\operatorname*{curl},c}$. These are estimated in
Lemma~\ref{lemma:Picurlcom-approximation}. Specifically, $\eta_{6}%
^{\operatorname{alg}}$ is controlled with
Lemma~\ref{lemma:Picurlcom-approximation},
(\ref{item:lemma:Picurlcom-approximation-iii}) and $\eta_{7}^{\exp}$ is
controlled with Lemma~\ref{lemma:Picurlcom-approximation},
(\ref{item:lemma:Picurlcom-approximation-ii}) after the observation
(\ref{eq:analyticity-classes-d}) that $L_{\Omega}{\mathbf{v}}$ is in an
analyticity class:%
\begin{align}
\eta_{6}^{\operatorname{alg}}  &  \lesssim\frac{hk}{p},\label{eta6algest}\\
\eta_{7}^{\exp}  &  \lesssim k\left(  \left(  \frac{h}{h+\sigma}\right)
^{p+1}+\left(  \frac{kh}{\sigma p}\right)  ^{p+1}\right)  . \label{eta7expest}%
\end{align}
Selecting $c_{1}$ sufficiently small and using
Lemma~\ref{lemma:resolution-condition-detail} allows us to conclude the
proof.
\endproof

\begin{corollary}
\label{CorConvRates}Let Assumption~\ref{AssumptionData} be satisfied, and let
the right-hand side in (\ref{GalDisb}) be defined by $F\left(  \mathbf{v}%
\right)  =\left(  \operatorname*{i}k\mathbf{\tilde{\boldsymbol{\jmath}}%
},\mathbf{v}\right)  $ for some $\mathbf{\tilde{\boldsymbol{\jmath}}}\in\left\{
\mathbf{u}\in\mathbf{L}^{2}\left(  \Omega\right)  \mid\operatorname*{div}%
\mathbf{u}=0\wedge\left\langle \left.  \mathbf{u}\right\vert _{\Gamma
},\mathbf{n}\right\rangle =0\right\}  $. Let the assumptions of
Theorem~\ref{thm:hpFEM-quasioptimality} be satisfied. Then under the scale
resolution condition (\ref{eq:hpFEM-scale-res}), the Galerkin approximation
${\mathbf{E}}_{h}\in{\mathbf{X}}_{h}$ (cf.~(\ref{GalDiscMaxw})) exists and
satisfies
\begin{equation}
\left\Vert {\mathbf{E}}-{\mathbf{E}}_{h}\right\Vert _{\Omega
,\operatorname{curl},k}\leq C\frac{kh}{p}\left\Vert \mathbf{\tilde
{\boldsymbol{\jmath}}}\right\Vert _{L^{2}(\Omega)}.
\end{equation}

\end{corollary}

%

\proof
Under the assumption of this corollary the solution $\mathbf{E}$ is the
restriction of the electric field of the full space problem
(\ref{Maxwellfullspace}) (with right-hand side defined as the extension of
$\mathbf{\tilde{\boldsymbol{\jmath}}}$ to $\mathbb{R}^{3}$ by zero). In Section
\ref{SecSolForm}, we will derive a solution formula (\ref{solformdualprob})
for an adjoint Maxwell problem which can be easily adapted to the original
Maxwell problem and to our assumption on the data $\mathbf{\tilde
{\boldsymbol{\jmath}}}$. We obtain%
\[
\mathbf{E}\left(  \mathbf{x}\right)  =\operatorname*{i}k\int_{\Omega}%
g_{k}\left(  \left\Vert \mathbf{x}-\mathbf{y}\right\Vert \right)
\mathbf{\tilde{\boldsymbol{\jmath}}}\left(  \mathbf{y}\right)  d\mathbf{y\quad
\forall x}\in\Omega,
\]
where $g_{k}$ is the fundamental solution of the Helmholtz equation
(\ref{fundsol}). From \cite[Lemma 3.5]{MelenkSauterMathComp}, we know that
there exist constants $c$, $C>0$ independent of $k$ and $\mathbf{\tilde
{\boldsymbol{\jmath}}}$ such that, for every $\mu>1$, there exists a $\mu$- and
$k$-dependent splitting $\mathbf{E}=\mathbf{E}_{H^{2}}+\mathbf{E}%
_{{\mathcal{A}}}$ with
\begin{subequations}
\label{formuladecomplemma}
\begin{align}
\Vert\nabla^{m}\mathbf{E}_{H^{2}}\Vert_{L^{2}(\Omega)}  &  \leq C\left(
1+\frac{1}{{\mu^{2}-1}}\right)  \left(  \mu k\right)  ^{m-1}\left\Vert
\mathbf{\tilde{\boldsymbol{\jmath}}}\right\Vert _{L^{2}(\Omega)}\qquad\forall
m\in\{0,1,2\},
\label{formuladecomplemmaa}\\
\Vert\nabla^{n}\mathbf{E}_{{\mathcal{A}}}\Vert_{L^{2}(\Omega)}  &  \leq
C\mu\left(  \gamma\mu k\right)  ^{n}\Vert\mathbf{\tilde{\boldsymbol{\jmath}}}%
\Vert_{L^{2}(\Omega)}\qquad\forall n\in{\mathbb{N}}_{0}. 
\label{formuladecomplemmab}%
\end{align}
\end{subequations}
As in (\ref{esteta2alfhp}) and (\ref{edtatilde14est}) we obtain constants $C$,
$\sigma> 0$ independent of $k$, $h$, $p$, and $\mathbf{\tilde{\boldsymbol{\jmath}%
}}$:
\begin{align*}
\left\Vert \mathbf{E}_{H^{2}}-\Pi_{p}^{\operatorname*{curl},s}\mathbf{E}%
_{H^{2}}\right\Vert _{\operatorname{curl},\Omega,k}  &  \leq C\frac{kh}%
{p}\left\Vert \mathbf{\tilde{\boldsymbol{\jmath}}}\right\Vert _{L^{2}(\Omega)},\\
\left\Vert \mathbf{E}_{\mathcal{A}}-\Pi_{p}^{\operatorname*{curl},s}%
\mathbf{E}_{\mathcal{A}}\right\Vert _{\operatorname{curl},\Omega,k}  &  \leq
C\left(  \left(  \frac{h}{h+\sigma}\right)  ^{p}+k\left(  \frac{kh}{\sigma
p}\right)  ^{p}+k\left(  \frac{h}{h+\sigma}\right)  ^{p+1}+\left(  \frac
{kh}{\sigma p}\right)  ^{p+1}\right)  \left\Vert \mathbf{\tilde{\boldsymbol{\jmath
}}}\right\Vert _{L^{2}(\Omega)}.
\end{align*}
Suitably choosing $c_{1}$, $c_{2}$ in condition (\ref{eq:hpFEM-scale-res})
implies the result.
\endproof

\begin{remark}
Our convergence theory allows also for a $k$-explicit $h$-version analysis:
the combination of Corollary \ref{cor:cond-for-quasioptimality} with the
definition of $\alpha_{\ell}$ (before (\ref{eq:analyticity-classes-d})) and
estimates (\ref{esteta2alfhp})-(\ref{eta7expest}) leads to a scale resolution
condition for fixed polynomial degree $p$ which reads $k^{p+s}h^{p}\lesssim1$
for $s=7$. However, we do not expect that this value of $s$ is sharp since our
goal is to prove quasi-optimality under the scale-resolution condition
$p\gtrsim\log k$ (cf. (\ref{eq:intro-scale-resolution})), which does not
require optimal bounds for the powers $\alpha_{\ell}$ in the stability
estimates.
\hbox{}\hfill\rule{0.8ex}{0.8ex}
\end{remark}

\section{$k$-explicit Analysis of Operators for $\Omega= B_{1}(0)$}

\label{sec:FreqSpOp}

A key ingredient of wavenumber-explicit estimates for the terms in the
splitting (\ref{gammamainsplittot}b,c) of $\left(
\kern-.1em%
\left(
\kern-.1em%
\mathbf{e}_{h},\mathbf{v}_{h}%
\kern-.1em%
\right)
\kern-.1em%
\right)  $ are $k$-explicit estimates of the capacity operator $T_{k}$ for the
low- and high-frequency parts of the arguments as these, in turn, allow for a
$k$-explicit analysis of the continuity properties of the sesquilinear form
$\left(
\kern-.1em%
\left(
\kern-.1em%
\cdot,\cdot%
\kern-.1em%
\right)
\kern-.1em%
\right)  $, the operators $\Pi^{\nabla,\ast}$, $\Pi^{\operatorname{comp},\ast
}$, and $A_{k}$. Our analysis of the operator $T_{k}$ is based on the explicit
knowledge of the Fourier coefficients and hence we restrict in this section to
the case that $\Omega=B_{1}\left(  0\right)  $ is the ball with radius $1$
centered at the origin. These estimates will be derived in
Section~\ref{SecSplitting} and applied to the different terms of the splitting
of $\left(
\kern-.1em%
\left(
\kern-.1em%
\mathbf{e}_{h},\mathbf{v}_{h}%
\kern-.1em%
\right)
\kern-.1em%
\right)  $ in Sections~\ref{SecHOmega}--\ref{SecLterms}.

We also analyze in the present section the operator $L_{\Omega}$ and show that
it maps into an analyticity class. The fact that we consider $\Omega=
B_{1}(0)$ here implies the \textsl{a priori} bound $\|L_{\Omega}{\mathbf{v}}
\|_{\operatorname{curl},\Omega,k} \leq\|{\mathbf{v}}\|_{\operatorname{curl}%
,\Omega,k}$ which, in turn, leads to the quantitative assertion $L_{\Omega
}{\mathbf{v}} \in{\mathcal{A}}(C k^{3/2} \|{\mathbf{v}}\|_{\operatorname{curl}%
,\Omega,1},\gamma,\Omega)$ in Theorem~\ref{TheoLau}.


\subsection{The Capacity Operator $T_{k}$ on the Sphere}

We restrict to the case that $\Omega^{+}:=\mathbb{R}^{3}\backslash
\overline{\Omega}$, where $\Omega=B_{1}\left(  0\right)  $ is the open unit
ball with boundary $\Gamma$. Let $T_{k}:H_{\operatorname{curl}}^{-1/2}\left(
\Gamma\right)  \rightarrow H_{\operatorname{div}}^{-1/2}\left(  \Gamma\right)
$ be the capacity operator that was introduced in the paragraph before
Remark~\ref{ExCapOp}.

In the case of the sphere the eigenfunctions of the negative Laplace-Beltrami
operator \textquotedblleft$-\Delta_{\Gamma}$\textquotedblright\ are given by
the spherical harmonics $Y_{\ell}^{m}$ (cf.~\cite[Sec.~{2.4.1}]{Nedelec01})
with eigenvalues $\lambda_{\ell}=\ell\left(  \ell+1\right)  $. In this case,
the index set $\iota_{\ell}$ in (\ref{eigvalabsLapBelt}) is given by
\begin{equation}
\label{eq:iota-sphere}\iota_{\ell}=\left\{  -\ell,-\ell+1,\ldots,\ell
-1,\ell\right\}  .
\end{equation}

We introduce the decomposition of $\mathbf{E}_{T}$ according to
(\ref{utbcHdecomp}) (cf.~\cite[(5.3.87)]{Nedelec01})
\begin{subequations}
\label{hodgesphere}
\end{subequations}%
\begin{equation}
\mathbf{E}_{T}=\mathbf{E}^{\operatorname*{curl}}+\mathbf{E}^{\nabla}, \tag{%
\ref{hodgesphere}%
a}\label{hodgespherea}%
\end{equation}
where%
\begin{equation}
\mathbf{E}^{\operatorname*{curl}}:=\sum_{\ell=1}^{\infty}\sum_{m\in\iota
_{\ell}}u_{\ell}^{m}\mathbf{T}_{\ell}^{m}\quad\text{and\quad}\mathbf{E}%
^{\nabla}=\nabla_{\Gamma}p\quad\text{with\quad}p:=\sum_{\ell=1}^{\infty}%
\sum_{m\in\iota_{\ell}}U_{\ell}^{m}Y_{\ell}^{m} \tag{%
\ref{hodgesphere}%
b}\label{hodgesphereb}%
\end{equation}
with the vectorial spherical harmonics $\mathbf{T}_{\ell}^{m}%
:=\overrightarrow{\operatorname*{curl}\nolimits_{\Gamma}}Y_{\ell}^{m}$
(cf.~\cite[(2.4.152), (2.4.173)]{Nedelec01}). This implies $\operatorname{div}%
_{\Gamma}\mathbf{E}_{T}^{\operatorname*{curl}}=0$.

\begin{remark}
\label{RemRedHodge}For the expansion of a tangential field, e.g.,
$\mathbf{E}_{T}$ the summation starts with $\ell=1$ since $\mathbf{T}%
_{0}=\nabla_{\Gamma}Y_{0}^{0}=\mathbf{0}$, i.e.,%
\begin{equation}
\mathbf{E}_{T}=\sum_{\ell=1}^{\infty}\sum_{m\in\iota_{\ell}}\left(  u_{\ell
}^{m}\mathbf{T}_{\ell}^{m}+U_{\ell}^{m}\nabla Y_{\ell}^{m}\right)
.\label{redhodege}%
\end{equation}
\hbox{}\hfill\rule{0.8ex}{0.8ex}
\end{remark}

\begin{lemma}
\label{Lembsplit}Let $\mathbf{E}_{T}\in H_{\operatorname{curl}}^{-1/2}\left(
\Gamma\right)  $ be decomposed as in (\ref{hodgesphere}). Then (cf.
Assumption~\ref{Assumpsdefb})%
\begin{equation}
\operatorname{div}_{\Gamma}\mathbf{E}^{\operatorname*{curl}}=0\quad
\text{and\quad}\left(  \mathbf{E}^{\operatorname*{curl}},\mathbf{E}^{\nabla
}\right)  _{\Gamma}=0=b_{k}\left(  \mathbf{E}^{\operatorname*{curl}%
},\mathbf{E}^{\nabla}\right)  =b_{k}\left(  \mathbf{E}^{\nabla},\mathbf{E}%
^{\operatorname*{curl}}\right)  . \label{orthorel}%
\end{equation}
Furthermore, we have the definiteness relations: For all $\mathbf{E}%
\in\mathbf{X}$ it holds%
\begin{equation}
\operatorname{Im}b_{k}\left(  \mathbf{E}^{\operatorname*{curl}},\mathbf{E}%
^{\operatorname*{curl}}\right)  \geq0\mathbf{\quad}\text{and\quad
}\operatorname{Im}b_{k}\left(  \mathbf{E}^{\nabla},\mathbf{E}^{\nabla}\right)
\leq0. \label{Imb1}%
\end{equation}

\end{lemma}

%

\proof
It follows from \cite[(5.3.87) and (5.3.91)]{Nedelec01} that the first term in
(\ref{orthorel}) is zero. Integration by parts for the second term in
(\ref{orthorel}) and using $\operatorname{div}_{\Gamma}\mathbf{E}%
^{\operatorname*{curl}}=0$ shows that the second term vanishes. The third term
in (\ref{orthorel}) vanishes as a consequence of $\operatorname*{div}_{\Gamma
}\mathbf{E}^{\operatorname*{curl}}=0$ and \cite[(5.3.109)]{Nedelec01}. The
last term is zero since $T_{k}\mathbf{E}^{\nabla}$ is a linear combination of
$\nabla_{\Gamma}Y_{\ell}^{m}$ (cf.~\cite[(5.3.87) and (5.3.88)]{Nedelec01})
and $\left(  \nabla_{\Gamma}Y_{\ell}^{m},\mathbf{E}^{\operatorname*{curl}%
}\right)  _{\Gamma}=- \left(  Y_{\ell}^{m},\operatorname{div}_{\Gamma}%
\mathbf{E}^{\operatorname*{curl}}\right)  _{\Gamma}=0$.

The first inequality in (\ref{Imb1}) follows from \cite[(5.3.107)]{Nedelec01}
and the second one is a consequence of \cite[(5.3.106)]{Nedelec01}.%
\endproof

Any tangential vector field $\mathbf{u}_{T}\in\mathbf{H}_{\operatorname*{curl}%
}^{-1/2}\left(  \Gamma\right)  $ can be expanded in terms of surface gradients
of spherical harmonics $Y_{\ell}^{m}$ and vectorial spherical harmonics
$\mathbf{T}_{\ell}^{m}$ via%
\begin{equation}
\mathbf{u}_{T}=\mathbf{u}^{\operatorname*{curl}}+\mathbf{u}^{\nabla}
\label{uTaddsplit}%
\end{equation}
with\
\[
\mathbf{u}^{\operatorname*{curl}}:=\sum_{\ell=1}^{\infty}\sum_{m\in\iota
_{\ell}}u_{\ell}^{m}\mathbf{T}_{\ell}^{m}\quad\text{and\quad}\mathbf{u}%
^{\nabla}:=\nabla_{\Gamma}p\quad\text{with\quad}p:=\sum_{\ell=1}^{\infty}%
\sum_{m\in\iota_{\ell}}U_{\ell}^{m}Y_{\ell}^{m}.
\]
The application of the capacity operator $T_{k}$ to $\mathbf{u}_{T}$ has the
explicit form (cf.~\cite[(5.3.88)]{Nedelec01})%
\begin{equation}
T_{k}\mathbf{u}_{T}=\sum_{\ell=1}^{\infty}\frac{z_{\ell}\left(  k\right)
+1}{\operatorname*{i}k}\sum_{m\in\iota_{\ell}}u_{\ell}^{m}\mathbf{T}_{\ell
}^{m}+\sum_{\ell=1}^{\infty}\frac{\operatorname*{i}k}{z_{\ell}\left(
k\right)  +1}\sum_{m\in\iota_{\ell}}U_{\ell}^{m}\nabla_{\Gamma}Y_{\ell}^{m},
\label{DefTk}%
\end{equation}
where%
\begin{equation}
z_{\ell}\left(  r\right)  :=r\frac{\left(  h_{\ell}^{\left(  1\right)
}\right)  ^{\prime}\left(  r\right)  }{h_{\ell}^{\left(  1\right)  }\left(
r\right)  }=-\frac{p_{\ell}\left(  r^{-2}\right)  }{q_{\ell}\left(
r^{-2}\right)  }+\operatorname*{i}\frac{r}{q_{\ell}\left(  r^{-2}\right)  },
\label{zlconjcompl}%
\end{equation}
with the spherical Hankel functions $h_{\ell}^{\left(  1\right)  }$, and
$p_{\ell}$, $q_{\ell}$ are polynomials of degree $\ell$ with real coefficients
(cf.~\cite[(2.6.19)-(2.6.22)]{Nedelec01}).

\begin{lemma}
\label{Lemzlest} Let $\lambda_{0}>1$ arbitrary but fixed. Then there exists
$C_{0}$ depending only on $\lambda_{0}$ such that for any $\lambda\geq
\lambda_{0}$:%
\begin{equation}
\frac{k}{\left\vert z_{n}\left(  k\right)  +1\right\vert }\leq\left\{
\begin{array}
[c]{ll}%
{2\sqrt{2}}k & \forall n\in\mathbb{N}_{0},\\
{2\sqrt{2}}\left(  \frac{2}{\lambda_{0}}+1\right)  \dfrac{k}{\left(
n+1\right)  } & n>\lambda k^{2},\\
C_{0}\frac{k}{n+1} & n\geq\lambda k.
\end{array}
\right.  \label{estk/z+1}%
\end{equation}
It holds%
\begin{equation}
\frac{\left\vert z_{n}\left(  k\right)  +1\right\vert }{k}\leq1+\frac{n}{k}.
\label{znestabove}%
\end{equation}

\end{lemma}

Estimate (\ref{znestabove}) follows from \cite[(5.3.95)]{Nedelec01}. The proof
of (\ref{estk/z+1}) is rather technical and postponed to the
Appendix~\ref{SecA}.

\begin{remark}
\label{RemA41}{}From (\ref{Imb1}), (\ref{DefTk}), and (\ref{zlconjcompl}) we
conclude that Assumption~\ref{Assumpsdefb} is satisfied for the sphere.
\hbox{}\hfill\rule{0.8ex}{0.8ex}
\end{remark}


\subsection{Analysis of Frequency Splittings $L_{\Gamma}$, $H_{\Gamma}$ on the
Surface of the Sphere\label{SubSecFrequSplit}}


\subsubsection{Analyticity of $L_{\Gamma}$}


\begin{lemma}
\label{LemLgammaPitu} Let $\Omega= B_{1}(0)$, and let the frequency filter
$L_{\Gamma}$ be given by Definition~\ref{DefFreqSplit} with a cut-off
parameter $\lambda\geq\lambda_{0}>1$. Then:

\begin{enumerate}
[(i)]

\item \label{item:LemLgammaPitu-ii} There exists a fixed tubular neighborhood
${\mathcal{U}}_{\Gamma}$ of $\Gamma$ and constants $C_{2}$, $\gamma_{2}$
independent of $k$ (but dependent on $\Gamma$, $\lambda$) such that for each
${\mathbf{u}} \in{\mathbf{H}}(\operatorname{curl},\Omega)$ there is an
extension ${\mathbf{U}} \in{\mathcal{A}}(C_{2} k^{3/2} \|{\mathbf{u}%
}\|_{\operatorname{curl},\Omega,1},\gamma_{2},{\mathcal{U}}_{\Gamma})$ of
$L_{\Gamma}{\mathbf{u}}_{T}$ to ${\mathcal{U}}_{\Gamma}$.

\item \label{item:LemLgammaPitu-i} The function $L_{\Gamma}{\mathbf{u}}_T$
belongs to the class $\mathcal{A}\left(  C_{1}k^{3/2}\left\Vert \mathbf{u}%
\right\Vert _{\operatorname*{curl},\Omega,1},\gamma_{1},\Gamma\right)  $,
where $C_{1}$, $\gamma_{1}$ are constants which are independent of $k$ and
$\mathbf{u}$. In particular, $\|L_{\Gamma} {\mathbf{u}}_T \|_{H^{1/2}%
(\Gamma)} \leq C_{1}^{\prime}k^{2} \|{\mathbf{u}}\|_{\operatorname{curl}%
,\Omega,1}$.
\end{enumerate}
\end{lemma}

\begin{proof}
Before proving Lemma~\ref{LemLgammaPitu}, we mention that the algebraic growth
rates with respect to $k$ are likely suboptimal. However, sharper estimates
would require more technicalities. We start by noting that the analyticity of
$\Gamma$ provides that the eigenfunctions $Y_{\ell}^{m}$ of the
Laplace-Beltrami operator have analytic extensions $\tilde{Y}_{\ell}^{m}$ to a
tubular neighborhood ${\mathcal{U}}_{\Gamma}$ of $\Gamma$. A quantitative
bound in terms of the eigenvalue $\lambda_{\ell}$ is given in
\cite[Lemma~{C.1}]{Melenk_map_helm_int}
\begin{equation}
\left\Vert \nabla^{n}\tilde{Y}_{\ell}^{m}\right\Vert _{L^{2}\left(
\mathcal{U}_{\Gamma}\right)  }\leq C_{S}\max\left\{  \sqrt{\lambda_{\ell}%
},n\right\}  ^{n}\gamma_{S}^{n}\qquad\forall n\in{\mathbb{N}}_{0}
\label{estgradYtilde}%
\end{equation}
for some $C_{S}$, $\gamma_{S}$ depending solely on $\Gamma$. We recall
specifically that the eigenvalues $\lambda_{\ell}$ of the Laplace-Beltrami
operator on the sphere are $\lambda_{\ell}=\ell(\ell+1)$.

\emph{Proof of (\ref{item:LemLgammaPitu-ii}):} Let $\mathbf{u}_{T}$ denote a
tangential field on the sphere with the representation (cf.~(\ref{redhodege}))%
\[
\mathbf{u}_{T}=\sum_{\ell=1}^{\infty}\sum_{m\in\iota_{\ell}}\left(  u_{\ell
}^{m}\mathbf{T}_{\ell}^{m}+U_{\ell}^{m}\nabla_{\Gamma}Y_{\ell}^{m}\right)
,\qquad{\mathbf{T}}_{\ell}^{m}=\overrightarrow{\operatorname{curl}_{\Gamma}%
}Y_{\ell}^{m}=\nabla_{\Gamma}Y_{\ell}^{m}\times{\mathbf{n}}.
\]
With the extension ${\mathbf{n}}^{\ast}$ of the normal vector ${\mathbf{n}}$
that is constant in normal direction, we may define the extension
${\mathbf{U}}$ of $L_{\Gamma}{\mathbf{u}}_{T}$ as
\[
{\mathbf{U}}=\sum_{1\leq\ell\leq\lambda k}\sum_{m\in\iota_{\ell}}u_{\ell}%
^{m}\nabla\tilde{Y}_{\ell}^{m}\times{\mathbf{n}}^{\ast}+U_{\ell}^{m}%
\nabla\tilde{Y}_{\ell}^{m}.
\]
By the analyticity of ${\mathbf{n}}^{\ast}$ we get from
Lemma~\ref{lemma:A-invariant} and (\ref{estgradYtilde}) that, for some
$C^{\prime}$, $\tilde{\gamma}>0$ depending solely on $\Gamma$,
\begin{equation}
\Vert\nabla^{n}(\nabla\tilde{Y}_{\ell}^{m}\times{\mathbf{n}}^{\ast}%
)\Vert_{L^{2}({\mathcal{U}}_{\Gamma})}+\Vert\nabla^{n}(\nabla\tilde{Y}_{\ell
}^{m})\Vert_{L^{2}({\mathcal{U}}_{\Gamma})}\leq C^{\prime}\sqrt{\lambda_{\ell
}}\,\tilde{\gamma}^{n}\max\{\sqrt{\lambda_{\ell}},n\}^{n}\qquad\forall
n\in{\mathbb{N}}_{0}. \label{eq:analyticity-extension}%
\end{equation}
We take $\ell\leq\lambda k$ into account (and using $\lambda_{\ell}=\ell
(\ell+1)$) which allows us to estimate ${\mathbf{U}}$ by
\begin{align}
\left\Vert \nabla^{n}{\mathbf{U}}\right\Vert _{L^{2}(\mathcal{U}_{\Gamma})}
&  \leq\sum_{1\leq\ell\leq\lambda k}\sum_{m\in\iota_{\ell}}\left(  \left\vert
u_{\ell}^{m}\right\vert \left\Vert \nabla^{n}(\nabla\tilde{Y}_{\ell}^{m}%
\times{\mathbf{n}}^{\ast})\right\Vert _{L^{2}(\mathcal{U}_{\Gamma}%
)}+\left\vert U_{\ell}^{m}\right\vert \left\Vert \nabla^{n+1}\tilde{Y}_{\ell
}^{m}\right\Vert _{L^{2}(\mathcal{U}_{\Gamma})}\right) \nonumber\\
&  \lesssim\tilde{\gamma}^{n}\sum_{1\leq\ell\leq\lambda k}\max\left\{
\sqrt{\lambda_{\ell}},n\right\}  ^{n}\lambda_{\ell}^{1/4}\lambda_{\ell}%
^{1/4}\sum_{m=-\ell}^{\ell}\left(  \left\vert u_{\ell}^{m}\right\vert
+\left\vert U_{\ell}^{m}\right\vert \right) \nonumber\\
&  \lesssim\tilde{\gamma}^{n}\left(  \sum_{1\leq\ell\leq\lambda k}%
\sum_{m=-\ell}^{\ell}\max\left\{  \sqrt{\lambda_{\ell}},n\right\}
^{2n}\lambda_{\ell}^{1/2}\right)  ^{1/2}\left(  \sum_{1\leq\ell\leq\lambda
k}\sum_{m=-\ell}^{\ell}\lambda_{\ell}^{1/2}\left(  |u_{\ell}^{m}|+|U_{\ell
}^{m}|\right)  ^{2}\right)  ^{1/2}\nonumber\\
&  \overset{\text{(\ref{m1/2curldiva})}}{\lesssim}\tilde{\gamma}^{n}\left(
\lambda k+1\right)  ^{3/2}\max\left\{  \lambda k+1,n\right\}  ^{n}\left\Vert
L_{\Gamma}\mathbf{u}_{T}\right\Vert _{-1/2,\operatorname{curl}_{\Gamma}}.
\end{align}
Since 
\begin{equation}
\label{eq:LemLgammaPitu} 
\Vert L_{\Gamma}{\mathbf{u}}_{T}\Vert_{-1/2,\operatorname{curl},\Gamma
}\leq\Vert{\mathbf{u}}_{T}\Vert_{-1/2,\operatorname{curl},\Gamma}\lesssim
\Vert{\mathbf{u}}\Vert_{\operatorname{curl},\Omega,1},
\end{equation}
the proof of
(\ref{item:LemLgammaPitu-ii}) is complete.

\emph{Proof of (\ref{item:LemLgammaPitu-i}):} An application of the
multiplicative trace inequality would allow us to infer from
(\ref{item:LemLgammaPitu-ii}) the assertion $L_{\Gamma}{\mathbf{u}_T}
\in{\mathcal{A}}(C_{1} k^{2} \|{\mathbf{u}}\|_{\operatorname{curl},\Omega
,1},\gamma_{1},\Gamma)$ for suitable $C_{1}$, $\gamma_{1}$. The sharper
statement follows by repeating the arguments of (\ref{item:LemLgammaPitu-ii})
starting with the assertion of \cite[Lemma~{C.1}]{Melenk_map_helm_int} that
\begin{equation}
\|\nabla^{n}_{\Gamma}Y^{m}_{\ell}\|_{L^{2}(\Gamma)} \leq C_{S} \gamma_{S}^{n}
\max\{\sqrt{\lambda_{\ell}},n\}^{n} \qquad\forall n \in{\mathbb{N}}_{0}.
\end{equation}

\end{proof}


\subsubsection{Estimates for High and Low Frequency Parts of the Capacity
Operator}

In this section we derive continuity estimates for the sesquilinear form
$b_{k}$. The $k$-dependence is different for the low- and high-frequency parts
of the tangential fields and for the summands in the splitting $\mathbf{u}%
_{T}=\mathbf{u}^{\operatorname*{curl}}+\mathbf{u}^{\nabla}$. In
Proposition~\ref{PropFrequbest} we derive such estimates for the tangential
fields while these estimates are lifted to the space $\mathbf{X}$ and some
subspaces thereof in Proposition~\ref{Propkbk}.

\begin{remark}
If $\Gamma$ is the surface of the unit ball then there holds for all
$\mathbf{u}_{T},\mathbf{v}_{T}\in\mathbf{H}_{\operatorname*{curl}}%
^{-1/2}\left(  \Gamma\right)  $%
\begin{align*}
b_{k}^{\operatorname*{high}}\left(  \mathbf{u}_{T},\mathbf{v}_{T}\right)   &
\overset{\text{Def. \ref{Deflowhigh}}}{=}\left(  T_{k}\mathbf{u}_{T}%
,H_{\Gamma}\mathbf{v}_{T}\right)  _{\Gamma}\overset{\text{Def.
\ref{DefFreqSplit}}}{=}\left(  H_{\Gamma}T_{k}\mathbf{u}_{T},\mathbf{v}%
_{T}\right)  _{\Gamma}\\
&  \overset{\text{(\ref{DefTk})}}{=}\left(  T_{k}H_{\Gamma}\mathbf{u}%
_{T},\mathbf{v}_{T}\right)  _{\Gamma}\overset{\text{Def. \ref{Deflowhigh}}%
}{=}\left(  T_{k}^{\operatorname*{high}}\mathbf{u}_{T},\mathbf{v}_{T}\right)
_{\Gamma}.
\end{align*}
Analogous relations hold for the low frequency part $b_{k}%
^{\operatorname*{low}}$. \hbox{}\hfill\rule{0.8ex}{0.8ex}
\end{remark}

\begin{proposition}
\label{PropFrequbest}
With the frequency filters $L_{\Gamma}$, $H_{\Gamma}$ of
Definition~\ref{DefFreqSplit} given by a cut-off parameter $\lambda\geq
\lambda_{0}>1$ the sesquilinear form $b_{k}$ can be written as
\begin{equation}
b_{k}\left(  \mathbf{u}_{T},\mathbf{v}_{T}\right)  =b_{k}\left(
\mathbf{u}^{\operatorname*{curl}},\mathbf{v}^{\operatorname*{curl}}\right)
+b_{k}^{\operatorname*{high}}\left(  \mathbf{u}^{\nabla},\mathbf{v}^{\nabla
}\right)  +b_{k}^{\operatorname*{low}}\left(  \mathbf{u}^{\nabla}%
,\mathbf{v}^{\nabla}\right)  \quad\forall\mathbf{u}_{T},\mathbf{v}_{T}%
\in\mathbf{H}_{\operatorname*{curl}}^{-1/2}\left(  \Gamma\right)  ,
\label{bksplitfirst}%
\end{equation}
and there is $C_{b}>0$ depending solely on $\lambda_{0}$ such that the
following holds:
\begin{align}
\left\vert b_{k}\left(  \mathbf{u}^{\operatorname*{curl}},\mathbf{v}%
^{\operatorname*{curl}}\right)  \right\vert  &  \leq C_{b}\left(  \frac{1}%
{k}\left\Vert \operatorname*{curl}\nolimits_{\Gamma}\mathbf{u}_{T}\right\Vert
_{H^{-1/2}\left(  \Gamma\right)  }\left\Vert \operatorname*{curl}%
\nolimits_{\Gamma}\mathbf{v}_{T}\right\Vert _{H^{-1/2}\left(  \Gamma\right)
}\right. \nonumber\\
&  \qquad\left.  +\left(  1+\lambda\right)  \left\Vert \operatorname*{curl}%
\nolimits_{\Gamma}L_{\Gamma}\mathbf{u}_{T}\right\Vert _{H^{-1}\left(
\Gamma\right)  }\left\Vert \operatorname*{curl}\nolimits_{\Gamma}L_{\Gamma
}\mathbf{v}_{T}\right\Vert _{H^{-1}\left(  \Gamma\right)  }\right)
,\nonumber\\
\left\vert b_{k}^{\operatorname*{high}}\left(  \mathbf{u}%
^{\operatorname*{curl}},\mathbf{v}^{\operatorname*{curl}}\right)  \right\vert
&  \leq C_{b}\frac{1}{k}\left\Vert \operatorname*{curl}\nolimits_{\Gamma
}\mathbf{u}_{T}\right\Vert _{H^{-1/2}\left(  \Gamma\right)  }\left\Vert
\operatorname*{curl}\nolimits_{\Gamma}\mathbf{v}_{T}\right\Vert _{H^{-1/2}%
\left(  \Gamma\right)  },\label{blowest}\\
\left\vert b_{k}^{\operatorname*{low}}\left(  \mathbf{u}^{\nabla}%
,\mathbf{v}^{\nabla}\right)  \right\vert  &  \leq C_{b}\lambda^{\rho}%
k^{\rho+1}\left\Vert \operatorname*{div}\nolimits_{\Gamma}L_{\Gamma}%
\mathbf{u}_{T}\right\Vert _{H^{-1-\rho/2}\left(  \Gamma\right)  }\left\Vert
\operatorname*{div}\nolimits_{\Gamma}L_{\Gamma}\mathbf{v}_{T}\right\Vert
_{H^{-1-\rho/2}\left(  \Gamma\right)  }\nonumber
\end{align}
for $0\leq\rho\leq2$. If $\operatorname*{div}_{\Gamma}\mathbf{u}_{T}\in
H^{\rho_{1}}\left(  \Gamma\right)  $ and $\operatorname*{div}\mathbf{v}_{T}\in
H^{\rho_{2}}\left(  \Gamma\right)  $ for some $\rho_{1}+\rho_{2}+3\geq0$ we
have
\begin{equation}
\left\vert b_{k}^{\operatorname*{high}}\left(  \mathbf{u}^{\nabla}%
,\mathbf{v}^{\nabla}\right)  \right\vert \leq C_{b}\frac{k}{\left(  \lambda
k\right)  ^{\rho_{1}+\rho_{2}+3}}\left\Vert \operatorname*{div}%
\nolimits_{\Gamma}\mathbf{u}_{T}\right\Vert _{H^{\rho_{1}}\left(
\Gamma\right)  }\left\Vert \operatorname*{div}\nolimits_{\Gamma}\mathbf{v}%
_{T}\right\Vert _{H^{\rho_{2}}\left(  \Gamma\right)  }. \label{estbhigh}%
\end{equation}

\end{proposition}

%

\proof
The equality (\ref{bksplitfirst}) follows from Lemma \ref{Lembsplit}.

By using the orthogonality relations of $\mathbf{T}_{\ell}^{m}$ and
$\nabla_{\Gamma}Y_{\ell}^{m}$, the representations in \cite[Sec.~{5.3.2}%
]{Nedelec01} give us%
\begin{align*}
\left\vert b_{k}\left(  \mathbf{u}^{\operatorname*{curl}},\mathbf{v}%
^{\operatorname*{curl}}\right)  \right\vert  &  =\left\vert \operatorname*{i}%
\sum_{\ell=1}^{\infty}\left(  -\frac{z_{\ell}\left(  k\right)  +1}{k}\right)
\sum_{m=-\ell}^{\ell}u_{\ell}^{m}\overline{v_{\ell}^{m}}\left(  \mathbf{T}%
_{\ell}^{m},\mathbf{T}_{\ell}^{m}\right)  _{\Gamma}\right\vert \\
&  \overset{\text{\cite[(2.4.155)]{Nedelec01}}}{\leq}\left\vert
\operatorname*{i}\sum_{\ell=1}^{\infty}\ell\left(  \ell+1\right)  \left(
-\frac{z_{\ell}\left(  k\right)  +1}{k}\right)  \sum_{m=-\ell}^{\ell}u_{\ell
}^{m}\overline{v_{\ell}^{m}}\right\vert \\
&  \overset{\text{(\ref{znestabove})}}{\leq}\left\vert \sum_{\ell=1}^{\infty
}\ell\left(  \ell+1\right)  \left(  1+\frac{\ell}{k}\right)  \sum_{m\in
\iota_{\ell}}u_{\ell}^{m}\overline{v_{\ell}^{m}}\right\vert \\
&  \overset{\text{(\ref{divscalednorm})}}{\leq}2\left(  \frac{1}{k}\left\Vert
\operatorname*{curl}\nolimits_{\Gamma}\mathbf{u}_{T}\right\Vert _{H^{-1/2}%
\left(  \Gamma\right)  }\left\Vert \operatorname*{curl}\nolimits_{\Gamma
}\mathbf{v}_{T}\right\Vert _{H^{-1/2}\left(  \Gamma\right)  }\right. \\
&  \qquad\left.  +\left(  1+\lambda\right)  \left\Vert \operatorname*{curl}%
\nolimits_{\Gamma}L_{\Gamma}\mathbf{u}_{T}\right\Vert _{H^{-1}\left(
\Gamma\right)  }\left\Vert \operatorname*{curl}\nolimits_{\Gamma}L_{\Gamma
}\mathbf{v}_{T}\right\Vert _{H^{-1}\left(  \Gamma\right)  }\right)  .
\end{align*}
This leads to the first estimate in (\ref{blowest}). In a similar way we
obtain for the high-frequency part%
\begin{align*}
\left\vert b_{k}^{\operatorname*{high}}\left(  \mathbf{u}%
^{\operatorname*{curl}},\mathbf{v}^{\operatorname*{curl}}\right)  \right\vert
&  \leq\left\vert \sum_{\ell\geq\lambda k}^{\infty}\ell\left(  \ell+1\right)
\left(  1+\frac{\ell}{k}\right)  \sum_{m\in\iota_{\ell}}u_{\ell}^{m}%
\overline{v_{\ell}^{m}}\right\vert \\
&  \leq\frac{2}{k}\sum_{\ell\geq\lambda k}^{\infty}\ell^{2}\left(
\ell+1\right)  \sum_{m\in\iota_{\ell}}\left\vert u_{\ell}^{m}\right\vert
\left\vert v_{\ell}^{m}\right\vert \\
&  \overset{\text{(\ref{divscalednorm})}}{\leq}\frac{2}{k}\left\Vert
\operatorname*{curl}\nolimits_{\Gamma}H_{\Gamma}\mathbf{u}_{T}\right\Vert
_{H^{-1/2}\left(  \Gamma\right)  }\left\Vert \operatorname*{curl}%
\nolimits_{\Gamma}H_{\Gamma}\mathbf{v}_{T}\right\Vert _{H^{-1/2}\left(
\Gamma\right)  }%
\end{align*}
so that the second estimate in (\ref{blowest}) follows. For the third one and
(\ref{estbhigh}) we obtain%
\begin{align*}
b_{k}^{\operatorname*{low}}\left(  \mathbf{u}^{\nabla},\mathbf{v}^{\nabla
}\right)   &  =\operatorname*{i}\sum_{1\leq\ell\leq\lambda k}\ell\left(
\ell+1\right)  \sum_{m\in\iota_{\ell}}\left(  \frac{k}{z_{\ell}\left(
k\right)  +1}U_{\ell}^{m}\overline{V_{\ell}^{m}}\right)  ,\\
b_{k}^{\operatorname*{high}}\left(  \mathbf{u}^{\nabla},\mathbf{v}^{\nabla
}\right)   &  =\operatorname*{i}\sum_{\ell>\lambda k}\ell\left(
\ell+1\right)  \sum_{m\in\iota_{\ell}}\left(  \frac{k}{z_{\ell}\left(
k\right)  +1}U_{\ell}^{m}\overline{V_{\ell}^{m}}\right)  .
\end{align*}
By using (\ref{Lemzlest}) and (\ref{estk/z+1}) we get for any $0\leq\rho\leq2$%
\begin{align*}
\left\vert b_{k}^{\operatorname*{low}}\left(  \mathbf{u}^{\nabla}%
,\mathbf{v}^{\nabla}\right)  \right\vert  &  \leq{2\sqrt{2}}k\sum_{1 \leq
\ell\leq\lambda k}\ell\left(  \ell+1\right)  \sum_{m\in\iota_{\ell}}\left\vert
U_{\ell}^{m}\right\vert \left\vert V_{\ell}^{m}\right\vert \\
&  \leq{4\sqrt{2}}k\left(  \lambda k\right)  ^{\rho}\sum_{1 \leq\ell
\leq\lambda k}\ell^{2-\rho}\sum_{m\in\iota_{\ell}}\left\vert U_{\ell}%
^{m}\right\vert \left\vert V_{\ell}^{m}\right\vert \\
&  \leq{16\sqrt{2}}\lambda^{\rho}k^{\rho+1}\sum_{1 \leq\ell\leq\lambda k}%
\ell^{2}\left(  \ell+1\right)  ^{-\rho}\sum_{m\in\iota_{\ell}}\left\vert
U_{\ell}^{m}\right\vert \left\vert V_{\ell}^{m}\right\vert \\
&  \overset{\text{(\ref{divscalednorm})}}{\leq}{16\sqrt{2}}\lambda^{\rho
}k^{\rho+1}\left\Vert \operatorname*{div}\nolimits_{\Gamma}L_{\Gamma
}\mathbf{u}_{T}\right\Vert _{H^{-1-\rho/2}}\left\Vert \operatorname*{div}%
\nolimits_{\Gamma}L_{\Gamma}\mathbf{u}_{T}\right\Vert _{H^{-1-\rho/2}}.
\end{align*}
For $\rho_{1}+\rho_{2}+3\geq0$ we get from (\ref{estk/z+1})%
\begin{align*}
\left\vert b_{k}^{\operatorname*{high}}\left(  \mathbf{u}^{\nabla}%
,\mathbf{v}^{\nabla}\right)  \right\vert  &  \leq C_{0}k\sum_{\ell>\lambda
k}\ell\sum_{m\in\iota_{\ell}}\left\vert U_{\ell}^{m}\right\vert \left\vert
V_{\ell}^{m}\right\vert \\
&  \leq\frac{C_{0}k}{\left(  \lambda k\right)  ^{\rho_{1}+\rho_{2}+3}}%
\sum_{\ell>\lambda k}\ell^{4+\rho_{1}+\rho_{2}}\sum_{m\in\iota_{\ell}%
}\left\vert U_{\ell}^{m}\right\vert \left\vert V_{\ell}^{m}\right\vert \\
&  \overset{\text{(\ref{divscalednorm})}}{\leq}C\frac{k}{\left(  \lambda
k\right)  ^{\rho_{1}+\rho_{2}+3}}\left\Vert \operatorname*{div}%
\nolimits_{\Gamma}\mathbf{u}_{T}\right\Vert _{H^{\rho_{1}}\left(
\Gamma\right)  }\left\Vert \operatorname*{div}\nolimits_{\Gamma}\mathbf{v}%
_{T}\right\Vert _{H^{\rho_{2}}\left(  \Gamma\right)  }.
\end{align*}%
\endproof

\begin{proposition}
\label{Propkbk} There is a constant $C_{b}^{\prime}> 0$ depending solely on
$\lambda_{0}$ such that the following holds:

Let $\mathbf{u},\mathbf{v}\in\mathbf{X}$. Then:
\begin{equation}
\left\vert b_{k}\left(  \mathbf{u}^{\nabla},\mathbf{v}^{\nabla}\right)
\right\vert \leq C_{b}^{\prime}k^{2}\left\Vert \mathbf{u}\right\Vert
_{\operatorname*{curl},\Omega,1}\left\Vert \mathbf{v}\right\Vert
_{\operatorname*{curl},\Omega,1}. \label{estdaechle}%
\end{equation}
Let $\mathbf{u}_{0}\in\mathbf{V}_{0}$ and $\mathbf{v}\in\mathbf{X}$. Then:
\begin{subequations}
\label{frequency_reg}
\end{subequations}%
\begin{align}
\left\vert kb_{k}^{\operatorname*{high}}\left(  \mathbf{u}_{0}^{\nabla
},\mathbf{v}^{\nabla}\right)  \right\vert  &  \leq C_{b}^{\prime}\frac
{k}{\lambda}\left\Vert \mathbf{u}_{0}\right\Vert _{\operatorname*{curl}%
,\Omega,1}\left\Vert \mathbf{v}\right\Vert _{\operatorname*{curl},\Omega
,1},\tag{%
\ref{frequency_reg}%
a}\label{frequency_rega}\\
\left\vert kb_{k}^{\operatorname*{low}}\left(  \mathbf{u}_{0}^{\nabla
},\mathbf{v}^{\nabla}\right)  \right\vert  &  \leq C_{b}^{\prime}\lambda
k^{3}\left\Vert \operatorname*{div}\nolimits_{\Gamma}L_{\Gamma}\mathbf{u}%
_{0,T}\right\Vert _{H^{-3/2}\left(  \Gamma\right)  }\left\Vert
\operatorname*{div}\nolimits_{\Gamma}L_{\Gamma}\mathbf{v}_{T}\right\Vert
_{H^{-3/2}\left(  \Gamma\right)  }. \tag{%
\ref{frequency_reg}%
b}\label{frequency_regb}%
\end{align}
For $\mathbf{u}\in\mathbf{X}$ and $\mathbf{v}_{0}\in\mathbf{V}_{0}^{\ast}$ it
holds
\begin{subequations}
\label{frequency_reg_adj}
\end{subequations}%
\begin{align}
\left\vert kb_{k}^{\operatorname*{high}}\left(  \mathbf{u}^{\nabla}%
,\mathbf{v}_{0}^{\nabla}\right)  \right\vert  &  \leq C_{b}^{\prime}\frac
{k}{\lambda}\left\Vert \mathbf{u}\right\Vert _{\operatorname*{curl},\Omega
,1}\left\Vert \mathbf{v}_{0}\right\Vert _{\operatorname*{curl},\Omega,1},\tag{%
\ref{frequency_reg_adj}%
a}\label{frequency_reg_adja}\\
\left\vert kb_{k}^{\operatorname*{low}}\left(  \mathbf{u}^{\nabla}%
,\mathbf{v}_{0}^{\nabla}\right)  \right\vert  &  \leq C_{b}^{\prime}\lambda
k^{3}\left\Vert \operatorname*{div}\nolimits_{\Gamma}L_{\Gamma}\mathbf{u}%
_{T}\right\Vert _{H^{-3/2}\left(  \Gamma\right)  }\left\Vert
\operatorname*{div}\nolimits_{\Gamma}L_{\Gamma}\mathbf{v}_{0,T}\right\Vert
_{H^{-3/2}\left(  \Gamma\right)  }. \tag{%
\ref{frequency_reg_adj}%
b}\label{frequency_reg_adjb}%
\end{align}
For $\mathbf{u}_{0}\in\mathbf{V}_{0}$, $\mathbf{v}_{0}\in\mathbf{V}_{0}^{\ast
}$ and $p,q\in H^{1}\left(  \Omega\right)  $ it holds
\begin{subequations}
\label{bhighgrad}
\end{subequations}%
\begin{align}
\left\vert kb_{k}^{\operatorname*{high}}\left(  \mathbf{u}_{0}^{\nabla
},\left(  \nabla p\right)  ^{\nabla}\right)  \right\vert  &  \leq\frac
{C_{b}^{\prime}}{\lambda}\left\Vert \mathbf{u}_{0}\right\Vert
_{\operatorname*{curl},\Omega,1}\left\Vert \nabla p\right\Vert
_{\operatorname*{curl},\Omega,k},\tag{%
\ref{bhighgrad}%
a}\label{bhighgrada}\\
\left\vert kb_{k}^{\operatorname*{high}}\left(  \left(  \nabla p\right)
^{\nabla},\mathbf{v}_{0}^{\nabla}\right)  \right\vert  &  \leq\frac
{C_{b}^{\prime}}{\lambda}\left\Vert \nabla p\right\Vert _{\operatorname*{curl}%
,\Omega,k}\left\Vert \mathbf{v}_{0}\right\Vert _{\operatorname*{curl}%
,\Omega,1},\tag{%
\ref{bhighgrad}%
b}\label{bhighgradb}\\
\left\vert kb_{k}^{\operatorname*{high}}\left(  \left(  \nabla p\right)
^{\nabla},\left(  \nabla q\right)  ^{\nabla}\right)  \right\vert  &  \leq
C_{b}^{\prime}\left\Vert \nabla p\right\Vert _{\operatorname*{curl},\Omega
,k}\left\Vert \nabla q\right\Vert _{\operatorname*{curl},\Omega,k}. \tag{%
\ref{bhighgrad}%
c}\label{bhighgradc}%
\end{align}
For $\mathbf{u}_{0}\in\mathbf{V}_{0}$ and $\mathbf{v}_{0}\in\mathbf{V}%
_{0}^{\ast}$ it holds
\begin{subequations}
\label{frequency_reg_2}
\end{subequations}%
\begin{align}
\left\vert kb_{k}^{\operatorname*{high}}\left(  \mathbf{u}_{0}^{\nabla
},\mathbf{v}_{0}^{\nabla}\right)  \right\vert  &  \leq\frac{C_{b}^{\prime}%
}{\lambda^{2}}\left\Vert \mathbf{u}_{0}\right\Vert _{\operatorname*{curl}%
,\Omega,1}\left\Vert \mathbf{v}_{0}\right\Vert _{\operatorname*{curl}%
,\Omega,1},\tag{%
\ref{frequency_reg_2}%
a}\label{frequency_reg_2a}\\
\left\vert kb_{k}^{\operatorname*{low}}\left(  \mathbf{u}_{0}^{\nabla
},\mathbf{v}_{0}^{\nabla}\right)  \right\vert  &  \leq C_{b}^{\prime}\lambda
k^{3}\left\Vert \mathbf{u}_{0}\right\Vert _{\operatorname*{curl},\Omega
,1}\left\Vert \mathbf{v}_{0}\right\Vert _{\operatorname*{curl},\Omega,1}.
\tag{%
\ref{frequency_reg_2}%
b}\label{frequency_reg_2b}%
\end{align}
For $\mathbf{u},\mathbf{v}\in\mathbf{H}^{1}\left(  \Omega\right)  $ and
$\mathbf{w}\in\mathbf{X}$, it holds
\begin{subequations}
\label{freq_reg_H1H1}
\end{subequations}%
\begin{align}
\left\vert kb_{k}^{\operatorname*{high}}\left(  \mathbf{u}^{\nabla}%
,\mathbf{v}^{\nabla}\right)  \right\vert  &  \leq\frac{C_{b}^{\prime}}%
{\lambda^{2}}\left\Vert \mathbf{u}\right\Vert _{\mathbf{H}^{1}\left(
\Omega\right)  }\left\Vert \mathbf{v}\right\Vert _{\mathbf{H}^{1}\left(
\Omega\right)  },\tag{%
\ref{freq_reg_H1H1}%
a}\label{freq_reg_H1H1a}\\
\left\vert kb_{k}^{\operatorname*{high}}\left(  \mathbf{w}^{\nabla}%
,\mathbf{v}^{\nabla}\right)  \right\vert  &  \leq C_{b}^{\prime}\frac
{k}{\lambda}\left\Vert \mathbf{w}\right\Vert _{\operatorname*{curl},\Omega
,1}\left\Vert \mathbf{v}\right\Vert _{\mathbf{H}^{1}\left(  \Omega\right)  }.
\tag{%
\ref{freq_reg_H1H1}%
b}\label{freq_reg_H1H1b}%
\end{align}

\end{proposition}

%

\proof
\emph{Proof of (\ref{estdaechle}):} We combine the last estimate in
(\ref{blowest}) (for $\rho=1$ and $\lambda=\lambda_{0}$) with (\ref{estbhigh})
(for $\rho_{1}=\rho_{2}=-3/2$ and $\lambda=\lambda_{0}$) and obtain%
\begin{align}
\left\vert b_{k}\left(  \mathbf{u}^{\nabla},\mathbf{v}^{\nabla}\right)
\right\vert  &  \leq\left\vert b_{k}^{\operatorname*{low}}\left(
\mathbf{u}^{\nabla},\mathbf{v}^{\nabla}\right)  \right\vert +\left\vert
b_{k}^{\operatorname*{high}}\left(  \mathbf{u}^{\nabla},\mathbf{v}^{\nabla
}\right)  \right\vert \nonumber\\
&  \leq C_{b}\left(  \lambda_{0}k^{2}\left\Vert \operatorname*{div}%
\nolimits_{\Gamma}L_{\Gamma}\mathbf{u}_{T}\right\Vert _{H^{-3/2}\left(
\Gamma\right)  }\left\Vert \operatorname*{div}\nolimits_{\Gamma}L_{\Gamma
}\mathbf{v}_{T}\right\Vert _{H^{-3/2}\left(  \Gamma\right)  }\right.
\nonumber\\
&  \left.  \quad+k\left\Vert \operatorname*{div}\nolimits_{\Gamma}%
\mathbf{u}_{T}\right\Vert _{H^{-3/2}\left(  \Gamma\right)  }\left\Vert
\operatorname*{div}\nolimits_{\Gamma}\mathbf{v}_{T}\right\Vert _{H^{-3/2}%
\left(  \Gamma\right)  }\right) \nonumber\\
&  \leq C\left(  1+\lambda_{0}k\right)  k\left\Vert \operatorname*{div}%
\nolimits_{\Gamma}\mathbf{u}_{T}\right\Vert _{H^{-3/2}\left(  \Gamma\right)
}\left\Vert \operatorname*{div}\nolimits_{\Gamma}\mathbf{v}_{T}\right\Vert
_{H^{-3/2}\left(  \Gamma\right)  }\nonumber\\
&  \leq C\left(  1+\lambda_{0}k\right)  k\left\Vert \mathbf{u}_{T}\right\Vert
_{H^{-1/2}\left(  \Gamma\right)  }\left\Vert \mathbf{v}_{T}\right\Vert
_{H^{-1/2}\left(  \Gamma\right)  }\label{bkdaechleest}\\
&  \leq C\left(  1+\lambda_{0}k\right)  k\left\Vert \mathbf{u}\right\Vert
_{\operatorname*{curl},\Omega,1}\left\Vert \mathbf{v}\right\Vert
_{\operatorname*{curl},\Omega,1}.\nonumber
\end{align}
\emph{Proof of (\ref{frequency_reg}), (\ref{frequency_reg_adj}),
(\ref{freq_reg_H1H1b}):} Let $\mathbf{u}\in\mathbf{H}^{1}\left(
\Omega\right)  $ and $\mathbf{v}\in\mathbf{X}$. Choose $\rho_{1}=-1/2$ and
$\rho_{2}=-3/2$ in (\ref{estbhigh}) to obtain%
\begin{align}
\left\vert kb_{k}^{\operatorname*{high}}\left(  \mathbf{u}^{\nabla}%
,\mathbf{v}^{\nabla}\right)  \right\vert  &  \leq C_{b}\frac{k}{\lambda
}\left\Vert \operatorname*{div}\nolimits_{\Gamma}\mathbf{u}_{T}\right\Vert
_{H^{-1/2}\left(  \Gamma\right)  }\left\Vert \operatorname*{div}%
\nolimits_{\Gamma}\mathbf{v}_{T}\right\Vert _{H^{-3/2}\left(  \Gamma\right)
}\nonumber\\
&  \leq C\frac{k}{\lambda}\left\Vert \mathbf{u}_{T}\right\Vert _{\mathbf{H}%
^{1/2}\left(  \Gamma\right)  }\left\Vert \mathbf{v}_{T}\right\Vert
_{\mathbf{H}^{-1/2}\left(  \Gamma\right)  }\nonumber\\
&  \leq C\frac{k}{\lambda}\left\Vert \mathbf{u}\right\Vert _{\mathbf{H}%
^{1}\left(  \Omega\right)  }\left\Vert \mathbf{v}\right\Vert
_{\operatorname*{curl},\Omega,1}. \label{omegabhighstart}%
\end{align}
This shows (up to interchanging the roles of ${\mathbf{u}}$ and ${\mathbf{v}}%
$) the estimate (\ref{freq_reg_H1H1b}). Since $\mathbf{V}_{0}\subset
\mathbf{H}^{1}\left(  \Omega\right)  $, we may apply estimate
(\ref{omegabhighstart}) to $\mathbf{u}\in\mathbf{V}_{0}$ and $\mathbf{v}%
\in\mathbf{X}$. Lemma~\ref{Lemembedspec} implies the estimate $\Vert
\mathbf{u}\Vert_{\mathbf{H}^{1}(\Omega)}\leq\Vert\mathbf{u}\Vert
_{\operatorname*{curl},\Omega,1}$ so that (\ref{frequency_rega}) follows. For
the low frequency part we get from (\ref{blowest}) for $\rho=1$ the estimate%
\[
\left\vert kb_{k}^{\operatorname*{low}}\left(  \mathbf{u}_{0}^{\nabla
},\mathbf{v}^{\nabla}\right)  \right\vert \leq C_{b}\lambda k^{3}\left\Vert
\operatorname*{div}\nolimits_{\Gamma}L_{\Gamma}\mathbf{u}_{0,T}\right\Vert
_{H^{-3/2}\left(  \Gamma\right)  }\left\Vert \operatorname*{div}%
\nolimits_{\Gamma}L_{\Gamma}\mathbf{v}_{T}\right\Vert _{H^{-3/2}\left(
\Gamma\right)  },
\]
which is (\ref{frequency_regb}). For $\mathbf{u}\in\mathbf{X}$ and
$\mathbf{v}_{0}\in\mathbf{V}_{0}^{\ast}$, estimates (\ref{frequency_reg_adj})
follow by the same arguments and interchanging the roles of $\mathbf{u}$ and
$\mathbf{v}$.

\emph{Proof of (\ref{bhighgrad}):} For $\mathbf{u}_{0}\in\mathbf{V}_{0}$ and
$p\in H^{1}\left(  \Omega\right)  $ we employ (\ref{frequency_rega}) with
$\mathbf{v}=\nabla p$ and $\operatorname*{curl}\nabla p=0$ so that%
\begin{align*}
\left\vert kb_{k}^{\operatorname*{high}}\left(  \mathbf{u}_{0}^{\nabla
},\left(  \nabla p\right)  ^{\nabla}\right)  \right\vert  &  \leq C\frac
{k}{\lambda}\left\Vert \mathbf{u}_{0}\right\Vert _{\operatorname*{curl}%
,\Omega,1}\left\Vert \nabla p\right\Vert _{\operatorname*{curl},\Omega
,1}=\frac{C}{\lambda}\left\Vert \mathbf{u}_{0}\right\Vert
_{\operatorname*{curl},\Omega,1}\left(  k\left\Vert \nabla p\right\Vert
\right) \\
&  \leq\frac{C}{\lambda}\left\Vert \mathbf{u}_{0}\right\Vert
_{\operatorname*{curl},\Omega,1}\left\Vert \nabla p\right\Vert
_{\operatorname*{curl}\Omega,k},
\end{align*}
which shows (\ref{bhighgrada}). The proof of (\ref{bhighgradb}) is just a
repetition of the previous arguments while the proof of (\ref{bhighgradc})
uses (\ref{estbhigh}) with $\rho_{1}=\rho_{2}=-3/2$:%
\begin{align*}
\left\vert kb_{k}^{\operatorname*{high}}\left(  \left(  \nabla p\right)
^{\nabla},\left(  \nabla q\right)  ^{\nabla}\right)  \right\vert  &  \leq
C_{b}k^{2}\left\Vert \operatorname*{div}\nolimits_{\Gamma}\left(  \nabla
p\right)  _{T}\right\Vert _{H^{-3/2}\left(  \Gamma\right)  }\left\Vert
\operatorname*{div}\nolimits_{\Gamma}\left(  \nabla q\right)  _{T}\right\Vert
_{H^{-3/2}\left(  \Gamma\right)  }\\
&  \leq C k^{2}\left\Vert \nabla p\right\Vert _{\operatorname*{curl},\Omega
,1}\left\Vert \nabla q\right\Vert _{\operatorname*{curl},\Omega,1}\\
&  =C k^{2}\left\Vert \nabla p\right\Vert \left\Vert \nabla q\right\Vert
=C\left\Vert \nabla p\right\Vert _{\operatorname*{curl},\Omega,k}\left\Vert
\nabla q\right\Vert _{\operatorname*{curl},\Omega,k},
\end{align*}
where the second step uses the same arguments as in (\ref{bkdaechleest}).

\emph{Proof of (\ref{frequency_reg_2}), (\ref{freq_reg_H1H1a}):} For any
$\mathbf{u},\mathbf{v}\in\mathbf{H}^{1}\left(  \Omega\right)  $ we may choose
$\rho_{1}=\rho_{2}=-1/2$ in (\ref{estbhigh}) to obtain%
\begin{align}
\left\vert kb_{k}^{\operatorname*{high}}\left(  \mathbf{u}^{\nabla}%
,\mathbf{v}^{\nabla}\right)  \right\vert  &  \leq\frac{C_{b}}{\lambda^{2}%
}\left\Vert \operatorname*{div}\nolimits_{\Gamma}\mathbf{u}_{T}\right\Vert
_{H^{-1/2}\left(  \Gamma\right)  }\left\Vert \operatorname*{div}%
\nolimits_{\Gamma}\mathbf{v}_{T}\right\Vert _{H^{-1/2}\left(  \Gamma\right)
}\label{kbkhighest}\\
&  \leq\frac{C}{\lambda^{2}}\left\Vert \mathbf{u}_{T}\right\Vert
_{\mathbf{H}^{1/2}\left(  \Gamma\right)  }\left\Vert \mathbf{v}_{T}\right\Vert
_{\mathbf{H}^{1/2}\left(  \Gamma\right)  }\leq\frac{C}{\lambda^{2}}\left\Vert
\mathbf{u}\right\Vert _{\mathbf{H}^{1}\left(  \Omega\right)  }\left\Vert
\mathbf{v}\right\Vert _{\mathbf{H}^{1}\left(  \Omega\right)  }.\nonumber
\end{align}
This proves (\ref{freq_reg_H1H1a}). If we assume in addition $\mathbf{u}%
_{0}\in\mathbf{V}_{0}$ and $\mathbf{v}_{0}\in\mathbf{V}_{0}^{\ast}$
we can appeal to Lemma~\ref{Lemembedspec} to get
\[
\left\vert kb_{k}^{\operatorname*{high}}\left(  \mathbf{u}_{0}^{\nabla
},\mathbf{v}_{0}^{\nabla}\right)  \right\vert \leq\frac{C}{\lambda^{2}%
}\left\Vert \mathbf{u}_{0}\right\Vert _{\operatorname*{curl},\Omega
,1}\left\Vert \mathbf{v}_{0}\right\Vert _{\operatorname*{curl},\Omega,1},
\]
i.e., (\ref{frequency_reg_2a}). For (\ref{frequency_reg_2b}) we employ the
last equation in (\ref{blowest}) for $\rho=1$ and proceed as for
(\ref{frequency_reg_2a}).%
\endproof

\subsection{Analysis of Frequency Splittings $L_{\Omega}$, $H_{\Omega}$ for
the Case $\Omega=B_{1}\left(  0\right)  $\label{SecAnLOmega}}

The operator $L_{\Omega}$ is defined in Definition~\ref{DefFreqSplit} as the
minimum norm extension of $L_{\Gamma}$ with respect to the norm $\Vert
\cdot\Vert_{\operatorname{curl},\Omega,k}$. {}From
Lemma~\ref{lemma:minimization-by-vsh} we have the following stability estimate
for the case $\Omega=B_{1}(0)$:
\begin{subequations}
\label{lowhighuest}
\end{subequations}%
\begin{equation}
\left\Vert L_{\Omega}\mathbf{u}\right\Vert _{\operatorname*{curl},\Omega
,k}\leq\left\Vert \mathbf{u}\right\Vert _{\operatorname*{curl},\Omega,k}.
\tag{%
\ref{lowhighuest}%
a}\label{lowhighuesta}%
\end{equation}
By the triangle inequality we infer that also $H_{\Omega}$ is stable%
\begin{equation}
\left\Vert H_{\Omega}\mathbf{u}\right\Vert _{\operatorname*{curl},\Omega
,k}\leq2\left\Vert \mathbf{u}\right\Vert _{\operatorname*{curl},\Omega,k}.
\tag{%
\ref{lowhighuest}%
b}\label{lowhighuestb}%
\end{equation}

\begin{theorem}
\label{TheoLau}Let $\Omega=B_{1}(0)$. Then the low-frequency part $L_{\Omega
}{\mathbf{u}}$ satisfies
\begin{equation}
\Vert L_{\Omega}{\mathbf{u}}\Vert_{\operatorname{curl},\Omega,k}\leq
\Vert{\mathbf{u}}\Vert_{\operatorname{curl},\Omega,k}\quad\mbox{ and }\quad
\operatorname{div}L_{\Omega}{\mathbf{u}}=0. \label{eq:TheoLau-simple}%
\end{equation}
Furthermore, $L_{\Omega}{\mathbf{u}}\in{\mathcal{A}}(C_{\mathcal{A},5}%
C_{u}^{\prime\prime},\gamma_{\mathcal{A},5},\Omega)$ with
\begin{equation}
C_{u}^{\prime\prime}=k^{3/2}\left\Vert \mathbf{u}\right\Vert
_{\operatorname*{curl},\Omega,1}. \label{defCuprime}%
\end{equation}
The constants $C_{\mathcal{A},5}$, $\gamma_{\mathcal{A},5}$ are independent of
$k$ and ${\mathbf{u}}$ but depend on the choice of the cut-off parameter
$\lambda$. Furthermore, there exists a tubular neighborhood ${\mathcal{U}%
}_{\Gamma}$ of $\Gamma$ such that $L_{\Omega}{\mathbf{u}}$ is analytic on
$\Omega\cup{\mathcal{U}}_{\Gamma}$ with $L_{\Omega}{\mathbf{u}}\in
{\mathcal{A}}(C_{\mathcal{A},5}^{\prime}C_{u}^{\prime\prime},\gamma
_{\mathcal{A},5}^{\prime},\Omega\cup{\mathcal{U}}_{\Gamma})$.
\end{theorem}

\begin{proof}
\emph{1.~step (interior regularity):} Using the vector identity
\begin{equation}
\label{eq:rotrot}\operatorname{curl}\operatorname{curl} = -\Delta+
\nabla\operatorname{div}%
\end{equation}
we infer from (\ref{eq:defLOmega-strong}) that $-\Delta L_{\Omega}{\mathbf{u}}
+ k^{2} L_{\Omega}{\mathbf{u}}= 0$ in $\Omega$. Interior regularity in the
form \cite[Prop.~{5.5.1}]{MelenkHabil} then gives $L_{\Omega}{\mathbf{u}}
\in{\mathcal{A}}(C_{R} k^{-1}\|L_{\Omega}{\mathbf{u}}\|_{\operatorname{curl}%
,\Omega,k},\gamma_{R},B_{R})$ for any ball $B_{R} \subset\Omega$, where the
constants $C_{R}$, $\gamma_{R}$ are independent of $k$ and ${\mathbf{u}}$ (but
depend on $R$). Noting (\ref{lowhighuest}) shows the desired analyticity
assertion for the interior of $\Omega$.

\emph{2.~step (smoothness up to the boundary and ${\mathbf{H}}^{1}%
$-estimates):} Let the tubular neighborhood ${\mathcal{U}}_{\Gamma}$ of
$\Gamma$ and the extension ${\mathbf{U}}\in{\mathcal{A}}(Ck^{3/2}%
\Vert{\mathbf{u}}\Vert_{\operatorname{curl},\Omega,1},\gamma_{2},{\mathcal{U}%
}_{\Gamma})$ of $L_{\Gamma}{\mathbf{u}}_{T}$ be given by
Lemma~\ref{LemLgammaPitu} and write $L_{\Omega}{\mathbf{u}}={\mathbf{U}%
}+\widetilde{\mathbf{u}}$. By the triangle inequality we have
%
\begin{subequations}
\label{eq:apriori-utilde}%
\begin{align}
k\Vert\widetilde{\mathbf{u}}\Vert_{{\mathbf{L}}^{2}({\mathcal{U}}_{\Gamma})}
&  \leq\Vert L_{\Omega}{\mathbf{u}}\Vert_{\operatorname{curl},\Omega,k}%
+\Vert{\mathbf{U}}\Vert_{\operatorname{curl},{\mathcal{U}}_{\Gamma},k}\leq
Ck^{5/2}\Vert{\mathbf{u}}\Vert_{\operatorname{curl},\Omega,1}%
,\label{eq:apriori-utilde-a}\\
\Vert\operatorname{curl}\widetilde{\mathbf{u}}\Vert_{{\mathbf{L}}%
^{2}({\mathcal{U}}_{\Gamma})}  &  \leq\Vert L_{\Omega}{\mathbf{u}}%
\Vert_{\operatorname{curl},\Omega,k}+\Vert{\mathbf{U}}\Vert
_{\operatorname{curl},{\mathcal{U}}_{\Gamma},k}\leq Ck^{5/2}\Vert{\mathbf{u}%
}\Vert_{\operatorname{curl},\Omega,1}. \label{eq:apriori-utilde-b}%
\end{align}
In view of (\ref{eq:defLOmega-strong}) the function $\widetilde{\mathbf{u}}$
satisfies
%
\end{subequations}
\begin{subequations}
\label{eq:TheorLau-10}%
\begin{align}
\operatorname{curl}\operatorname{curl}\widetilde{\mathbf{u}}+k^{2}%
\widetilde{\mathbf{u}}  &  ={\mathbf{f}}:=-\operatorname{curl}%
\operatorname{curl}{\mathbf{U}}-k^{2}{\mathbf{U}}\quad
\mbox{ in ${\mathcal U}_\Gamma \cap\Omega$},\label{eq:TheorLau-10-a}\\
\operatorname{div}\widetilde{\mathbf{u}}  &  =G:=-\operatorname{div}%
{\mathbf{U}}\quad
\mbox{ in ${\mathcal U}_\Gamma\cap\Omega$},\label{eq:TheorLau-10-b}\\
\Pi_{T}\widetilde{\mathbf{u}}  &  =0\quad\mbox{on $\Gamma$.}
\label{eq:TheorLau-10-c}%
\end{align}
We have (suitably adjusting the constants $\gamma_{2}$)
\end{subequations}
\begin{equation}
{\mathbf{f}}\in\mathcal{A}(Ck^{7/2}\Vert{\mathbf{u}}\Vert_{\operatorname{curl}%
,\Omega,1},\gamma_{2},\mathcal{U}_{\Gamma}\cap\Omega),\qquad G\in
\mathcal{A}(Ck^{5/2}\Vert{\mathbf{u}}\Vert_{\operatorname{curl},\Omega
,1},\gamma_{2},\mathcal{U}_{\Gamma}\cap\Omega). \label{eq:TheorLau-100}%
\end{equation}
The analyticity of $\mathbf{U}$, Lemma~\ref{lemma:maxwell-H2-regularity}, and
a simple induction argument (to deal with the presence of the lower order term
$k^{2}\widetilde{\mathbf{u}}$) shows that $\widetilde{\mathbf{u}}$ is in
$C^{\infty}(\mathcal{U}_{\Gamma}\cap\Omega)$. Additionally, by suitably
localizing, Lemma~\ref{lemma:maxwell-H2-regularity},
(\ref{item:lemma:maxwell-H2-regularity-i}) gives for a suitable subset
${\mathcal{U}}_{\Gamma}^{\prime}\subset{\mathcal{U}}_{\Gamma}$
\begin{equation}
\Vert\widetilde{\mathbf{u}}\Vert_{{\mathbf{H}}^{1}({\mathcal{U}}_{\Gamma
}^{\prime})}\leq C\left[  \Vert\widetilde{\mathbf{u}}\Vert_{{\mathbf{L}}%
^{2}({\mathcal{U}}_{\Gamma})}+\Vert\operatorname{curl}\widetilde{\mathbf{u}%
}\Vert_{{\mathbf{L}}^{2}(\Omega\cap{\mathcal{U}}_{\Gamma})}+\Vert
\operatorname{div}\widetilde{\mathbf{u}}\Vert_{L^{2}(\Omega\cap{\mathcal{U}%
}_{\Gamma})}\right]  \overset{(\ref{eq:apriori-utilde}%
),(\ref{eq:TheorLau-10-b})}{\leq}Ck^{5/2}\Vert{\mathbf{u}}\Vert
_{\operatorname{curl},\Omega,1}. \label{eq:TheorLau-42}%
\end{equation}
For notational convenience, we will henceforth denote ${\mathcal{U}}_{\Gamma
}^{\prime}$ again by ${\mathcal{U}}_{\Gamma}$.

\emph{3.~step (analytic regularity of $\widetilde{\mathbf{u}}$):} Quantitative
bounds for higher derivatives of $\widetilde{\mathbf{u}}$ are obtained by
locally flattening the boundary. By the analyticity of $\Gamma$ and the
compactness of $\Gamma$ there are $R_{0}$, $C_{\chi}$, $\gamma_{\chi}>0$ such
that for each ${\mathbf{x}}_{0} \in\Gamma$ we can find a parametrization
$\chi_{{\mathbf{x}}_{0}} \in{\mathcal{A}}^{\infty}(C_{\chi},\gamma_{\chi
},B_{R_{0}}(0))$ with the following properties\footnote{The third condition is
not essential but leads to a significant simplification as the ensuing
(\ref{eq:local-decoupling}) effects a decoupling of the elliptic system
(\ref{eq:transformed-system}) into three scalar problems at $0$.}:

\begin{enumerate}
\item $\chi_{{\mathbf{x}}_{0}}(0) = {\mathbf{x}}_{0}$ and, for $B_{R_{0}}%
^{+}:= \{\widehat{\mathbf{x}} \in B_{R_{0}}(0)\,|\, \widehat{\mathbf{x}}_{3} >
0\}$ and $\widehat{\Gamma}_{R_{0}}:= \{\widehat{\mathbf{x}} \in B_{R_{0}%
}(0)\,|\, {\mathbf{x}}_{3} = 0\}$, we have $V_{{\mathbf{x}}_{0}}:=
\chi_{{\mathbf{x}}_{0}}(B_{R_{0}}^{+}) \subset\Omega$ as well as
$\chi_{{\mathbf{x}}_{0}}( \widehat{\Gamma}_{R_{0}}) \subset\Gamma$.

\item For $\widehat{\mathbf{x}} \in\widehat{\Gamma}_{R_{0}}$ the vectors
${\mathbf{t}}_{{\mathbf{x}}_{0}}^{i}:= \partial_{i} \chi_{{\mathbf{x}}_{0}%
}(\widehat{\mathbf{x}})$, $i\in\{1,2\}$, span the tangent plane of $\Gamma$ at
$\chi_{{\mathbf{x}}_{0}}(\widehat{\mathbf{x}})$ and ${\mathbf{n}}({\mathbf{x}%
}):= - \partial_{3} \chi_{{\mathbf{x}}_{0}}(\widehat{\mathbf{x}})$ is the
outward normal vector.

\item The Jacobian $D\chi_{{\mathbf{x}}_{0}}(0)\in{\mathbb{R}}^{3\times3}$ is
orthogonal, i.e., $(D\chi_{{\mathbf{x}}_{0}}(0))^{\intercal}(D\chi
_{{\mathbf{x}}_{0}}(0))=\operatorname{I}$.
\end{enumerate}

The transformation of the system (\ref{eq:TheorLau-10}) on $V_{{\mathbf{x}%
}_{0}}$ to the half-ball $B_{R_{0}}^{+}$ is effected with a covariant
transformation of the dependent variable $\widetilde{\mathbf{u}}$ by setting
$\widetilde{\mathbf{u}}^{\operatorname{cov}}:=(D\chi_{{\mathbf{x}}_{0}%
})^{\intercal}\widetilde{\mathbf{u}}\circ\chi_{{\mathbf{x}}_{0}}$. We recall
the formula (see, e.g., \cite[Cor.~{3.58}]{Monkbook})%
\[
\frac{1}{\det\left(  D\chi_{{\mathbf{x}}_{0}}\right)  }\left(  D\chi
_{\mathbf{x}_{0}}\right)  \operatorname*{curl}\mathbf{w}^{\operatorname*{cov}%
}=\left(  \operatorname*{curl}\mathbf{w}\right)  \circ\chi_{{\mathbf{x}}_{0}}%
\]
and introduce the two pointwise symmetric positive definite matrices
\begin{equation}
{\mathbf{A}}:=\frac{1}{\det(D\chi_{{\mathbf{x}}_{0}})}(D\chi_{{\mathbf{x}}%
_{0}})^{\intercal}(D\chi_{{\mathbf{x}}_{0}}),\qquad{\mathbf{B}}:=\left(
\det(D\chi_{{\mathbf{x}}_{0}})\right)  (D\chi_{{\mathbf{x}}_{0}})^{-1}%
(D\chi_{{\mathbf{x}}_{0}})^{-T}; \label{eq:TheoLau-10}%
\end{equation}
note that ${\mathbf{A}}$, ${\mathbf{B}}\in{\mathcal{A}}^{\infty}(C^{\prime
},\gamma^{\prime},B_{R_{0}}^{+})$ for some constants $C^{\prime}$,
$\gamma^{\prime}$ that depend solely on $\Gamma$ (which is fixed in our case
$\Gamma:=\partial B_{1}\left(  0\right)  $). We also note that, since
$D\chi_{{\mathbf{x}}_{0}}(0)$ is assumed to be orthogonal, we have
\begin{equation}
{\mathbf{A}}(0)={\mathbf{B}}(0)=\operatorname{I}\in{\mathbb{R}}^{3\times3}.
\label{eq:local-decoupling}%
\end{equation}

{}From (\ref{eq:TheorLau-10-a}) we obtain for all ${\mathbf{V}}\in
C_{0}^{\infty}(B_{R_{0}}^{+})$
\[
\int_{B_{R_{0}}^{+}}\left(  \frac{1}{\det D\chi_{{\mathbf{x}}_{0}}%
}\left\langle \left(  D\chi_{{\mathbf{x}}_{0}}\right)  \operatorname*{curl}%
\widetilde{\mathbf{u}}^{\operatorname*{cov}},\left(  D\chi_{{\mathbf{x}}_{0}%
}\right)  \operatorname*{curl}\mathbf{V}\right\rangle +\left(  \det
D\chi_{{\mathbf{x}}_{0}}\right)  k^{2}\left\langle \widetilde{\mathbf{u}%
}^{\operatorname*{cov}},\left(  D\chi_{{\mathbf{x}}_{0}}\right)  ^{-1}\left(
D\chi_{{\mathbf{x}}_{0}}\right)  ^{-\intercal}\mathbf{V}\right\rangle \right)
=\int_{B_{R_{0}}^{+}}\left\langle \widehat{\mathbf{f}},\mathbf{V}%
\right\rangle
\]
with $\widehat{\mathbf{f}}:=\det(D\chi_{{\mathbf{x}}_{0}})(D\chi_{{\mathbf{x}%
}_{0}})^{-1}{\mathbf{f}}\circ\chi_{{\mathbf{x}}_{0}}$. The strong form of this
equation is
%
\begin{subequations}
\label{eq:transformed-system}
\end{subequations}%
\begin{equation}
\operatorname*{curl}\left(  {\mathbf{A}}\operatorname*{curl}%
\widetilde{\mathbf{u}}^{\operatorname{conv}}\right)  +k^{2}{\mathbf{B}%
}\widetilde{\mathbf{u}}^{\operatorname{conv}}=\widehat{\mathbf{f}}%
\quad\mbox{ in $B_{R_0}^+$.} \tag{%
\ref{eq:transformed-system}%
a}\label{eq:transformed-system-a}%
\end{equation}
The transformation of the divergence condition (\ref{eq:TheorLau-10-b}) to
$B_{R_{0}}^{+}$ is:
\begin{equation}
\operatorname{div}\left(  {\mathbf{B}}\widetilde{\mathbf{u}}%
^{\operatorname{cov}}\right)  =\widehat{G}:=\det(D\chi_{{\mathbf{x}}_{0}%
})G\circ\chi_{{\mathbf{x}}_{0}}\quad\mbox{in $B_{R_0}^+$}. \tag{%
\ref{eq:transformed-system}%
b}\label{eq:transformed-system-b}%
\end{equation}
The covariant transformation leaves the homogeneous tangential trace
(\ref{eq:TheorLau-10-c}) invariant:
\begin{equation}
\Pi_{T}\widetilde{\mathbf{u}}^{\operatorname{cov}}=0\quad
\mbox{ on $\widehat \Gamma_{R_0}$.} \tag{%
\ref{eq:transformed-system}%
c}\label{eq:transformed-system-c}%
\end{equation}
We rewrite the equations (\ref{eq:transformed-system}) in the form
(\ref{eq:local-system}). To that end, we note that the solution
$\widetilde{\mathbf{u}}^{\operatorname{cov}}$ is smooth (up to the boundary
$\widehat{\Gamma}_{R_{0}}$) by Step~2 so that the manipulations are
admissible; we also note ${\mathbf{A}}(0)={\mathbf{B}}(0)=\operatorname{I}$ by
(\ref{eq:local-decoupling}). Adding the gradient of equation
(\ref{eq:transformed-system-b}) to equation (\ref{eq:transformed-system-a})
and taking the trace of (\ref{eq:transformed-system-b}) on $\widehat{\Gamma
}_{R_{0}}$ as well as taking note of (\ref{eq:transformed-system-c}) (to
obtain both (\ref{eq:local-transformed-curl-system-b}) and
(\ref{eq:local-transformed-curl-system-c})) gives a system of the following
form:
%
\begin{subequations}
\label{eq:local-transformed-curl-system}
\end{subequations}%
\begin{align}
&  -\sum_{\alpha,\beta,j=1}^{3}\partial_{\alpha}\left(  A_{\alpha\beta}%
^{ij}\partial_{\beta}\widetilde{\mathbf{u}}_{j}^{\operatorname{cov}}\right)
+\sum_{j,\beta=1}^{3}B_{\beta}^{ij}\partial_{\beta}\widetilde{\mathbf{u}}%
_{j}^{cov}+\sum_{j=1}^{3}\left(  C^{ij}+k^{2}{\mathbf{B}}_{ij}\right)
\widetilde{\mathbf{u}}_{j}^{\operatorname{cov}}=\widehat{\mathbf{f}}%
_{i}+\partial_{i}\widehat{G},\quad\mbox{ on $B^+_{R_0}$},\quad i=1,2,3,\tag{%
\ref{eq:local-transformed-curl-system}%
a}\label{eq:local-transformed-curl-system-a}\\
&  \widetilde{\mathbf{u}}_{i}^{\operatorname{cov}}=0\quad
\mbox{ on $\Gamma_{R_0}$},\quad i=1,2,\tag{%
\ref{eq:local-transformed-curl-system}%
b}\label{eq:local-transformed-curl-system-b}\\
&  \partial_{3}\widetilde{\mathbf{u}}_{3}^{\operatorname{cov}}=\widehat{G}%
-(\sum_{i=1}^{3}\partial_{i}{\mathbf{B}}_{i3})\widetilde{\mathbf{u}}%
_{3}^{\operatorname{cov}}-\sum_{i=1}^{2}{\mathbf{B}}_{i3}\partial
_{i}\widetilde{\mathbf{u}}_{3}^{\operatorname{cov}}-\sum_{j=1}^{2}{\mathbf{B}%
}_{3j}\partial_{3}\widetilde{\mathbf{u}}_{j}^{\operatorname{cov}}%
-({\mathbf{B}}_{33}-1)\partial_{3}\widetilde{\mathbf{u}}_{3}%
^{\operatorname{cov}}.\quad\mbox{ on $\widehat \Gamma_{R_0}$.} \tag{%
\ref{eq:local-transformed-curl-system}%
c}\label{eq:local-transformed-curl-system-c}%
\end{align}
The tensors $(A_{\alpha\beta}^{ij})_{i,j,\alpha,\beta}$, $(B_{\beta}%
^{ij})_{i,j,\beta}$, and $(C^{ij})_{i,j}$ are analytic on $B_{R_{0}}^{+}$ and,
with constants $C^{\prime\prime}$, $\gamma^{\prime\prime}$, depending solely
on $\Gamma$ (being fixed in our case by $\Gamma:=\partial B_{1}\left(
0\right)  $), we have $(A_{\alpha\beta}^{ij})_{i,j,\alpha,\beta}$, $(B_{\beta
}^{ij})_{i,j,\beta},(C^{ij})_{i,j}\in{\mathcal{A}}^{\infty}(C^{\prime\prime
},\gamma^{\prime\prime},B_{R_{0}}^{+})$. Additionally, we have the structural
property (cf.~(\ref{eq:rotrot}) and (\ref{eq:local-decoupling}))
\begin{equation}
A_{\alpha\beta}^{ij}(0)=\delta_{ij}\delta_{\alpha\beta},\qquad\forall
i,j,\alpha,\beta,\qquad{\mathbf{B}}_{j3}(0)={\mathbf{B}}_{3j}(0)=0=\mathbf{B}%
_{33}(0)-1\quad\mbox{ for $j \in \{1,2\}$}.\qquad
\label{eq:structural-property}%
\end{equation}
Lemma~\ref{lemma:A-invariant} and (\ref{eq:TheorLau-100}) imply, for suitable
constants $C$, $\gamma_{3}$,
\begin{align}
&  \widehat{f}\in{\mathcal{A}}(Ck^{7/2}\Vert{\mathbf{u}}\Vert
_{\operatorname{curl},\Omega,1},\gamma_{3},B_{R_{0}}^{+}),\qquad\widehat{G}%
\in{\mathcal{A}}(Ck^{5/2}\Vert{\mathbf{u}}\Vert_{\operatorname{curl},\Omega
,1},\gamma_{3},B_{R_{0}}^{+}),\label{eq:TheorLau-112}\\
&  k\Vert\widetilde{\mathbf{u}}^{\operatorname{cov}}\Vert_{{\mathbf{L}}%
^{2}(B_{R_{0}}^{+})}+\Vert\widetilde{\mathbf{u}}^{\operatorname{cov}}%
\Vert_{{\mathbf{H}}^{1}(B_{R_{0}}^{+})}\overset{(\ref{eq:apriori-utilde-a}%
),(\ref{eq:TheorLau-42})}{\leq}Ck^{5/2}\Vert{\mathbf{u}}\Vert
_{\operatorname{curl},\Omega,1}.
\end{align}
Dividing (\ref{eq:local-transformed-curl-system}) by $k^{2}$ makes
Theorem~\ref{thm:local-regularity} applicable with $\varepsilon=k^{-1}$ and
the constants $C_{f}$, $C_{G}$, $C_{C}$ there are of the form $C_{f}%
=O(k^{3/2}\Vert{\mathbf{u}}\Vert_{\operatorname{curl},\Omega,1})$,
$C_{G}=O(k^{3/2}\Vert{\mathbf{u}}\Vert_{\operatorname{curl},\Omega,1})$,
$C_{C}=O(1)$. Theorem~\ref{thm:local-regularity} yields a $R>0$ and constants
$C$, $\gamma$ such that for $B_{R}^{+}:=\{\widehat{\mathbf{x}}\in
B_{R}(0)\,|\,\widehat{\mathbf{x}}_{3}>0\}$ we have $\widetilde{\mathbf{u}%
}^{\operatorname{cov}}\in{\mathcal{A}}(C_{u},\gamma,B_{R}^{+})$, where
\[
C_{u}=k^{3/2}\Vert{\mathbf{u}}\Vert_{\operatorname{curl},\Omega,1}.
\]
Transforming back using again Lemma~\ref{lemma:A-invariant} gives for
$V_{\mathbf{x}_0}=\chi_{{\mathbf{x}}_{0}}(B_{R}^{+})$ the analytic regularity assertion
$\widetilde{\mathbf{u}}\in{\mathcal{A}}(CC_{u},\gamma,V_{\mathbf{x}_0})$ for suitable
constants $C$, $\gamma$. A covering argument completes the estimate of
$\widetilde{\mathbf{u}}$ on ${\mathcal{U}}_{\Gamma}$.\medskip
\end{proof}

The normal trace $\langle L_{\Omega}{\mathbf{v}},{\mathbf{n}}\rangle$ is also
analytic. We have:

\begin{lemma}
\label{lemma:ntimesLomega-analytic} Let $\Omega=B_{1}(0)$. There is a tubular
neighborhood ${\mathcal{U}}_{\Gamma}$ of $\Gamma$ and there are constants
$C_{\mathcal{A},\Gamma}$, $\gamma_{\mathcal{A},\Gamma}$, $C_{\mathcal{A}%
,\Gamma}^{\prime}$, $\gamma_{\mathcal{A},\Gamma}^{\prime}$, $C_{\mathcal{A}%
,\Gamma}^{\prime^{\prime}}$, $\gamma_{\mathcal{A},\Gamma}^{\prime\prime}$,
$b>0$ depending only on the choice of cut-off parameter $\lambda$ such that
for any ${\mathbf{v}}\in{\mathbf{H}}(\operatorname{curl},\Omega)$ the normal
trace $\langle L_{\Omega}{\mathbf{v}},{\mathbf{n}}\rangle$ on $\Gamma$
satisfies the following:

\begin{enumerate}
[(i)]

\item \label{item:lemma:ntimesLomega-analytic-i} $g_{1}:=\langle L_{\Omega
}{\mathbf{v}},{\mathbf{n}}\rangle$ has an analytic extension $g_{1}^{\ast}$ to
${\mathcal{U}}_{\Gamma}$ with $g_{1}^{\ast}\in{\mathcal{A}}(C_{\mathcal{A}%
,\Gamma}k^{3/2}\Vert{\mathbf{v}}\Vert_{\operatorname{curl},\Omega,1}%
,\gamma_{\mathcal{A},\Gamma},{\mathcal{U}}_{\Gamma})$.

\item \label{item:lemma:ntimesLomega-analytic-ii} $\langle L_{\Omega
}{\mathbf{v}},{\mathbf{n}}\rangle\in{\mathcal{A}}(C_{\mathcal{A},\Gamma} k^{2}
\|{\mathbf{v}}\|_{\operatorname{curl},\Omega,1},\gamma_{\mathcal{A},\Gamma
},\Gamma)$.

\item \label{item:lemma:ntimesLomega-analytic-iii} The expansion coefficients
$\kappa_{\ell}^{m}$ of
\begin{equation}
\langle L_{\Omega}{\mathbf{v}},{\mathbf{n}}\rangle=\sum_{\ell=1}^{\infty}%
\sum_{m\in\iota_{\ell}}\kappa_{\ell}^{m}Y_{\ell}^{m}
\label{eq:expansion-nLOmega}%
\end{equation}
satisfy%
\begin{align}
|\kappa_{\ell}^{m}|  &  \leq C_{\mathcal{A},\Gamma}^{\prime}\Vert{\mathbf{v}%
}\Vert_{\operatorname{curl},\Omega,1}%
\begin{cases}
k^{1/2} & \mbox{if $\ell \leq \gamma^\prime_{{\mathcal A},\Gamma} k$}\\
k^{2}e^{-b\ell} & \mbox{if $\ell > \gamma^\prime_{{\mathcal A},\Gamma} k$.}
\end{cases}
\label{eq:lemma:ntimesLomega-analytic-14}\\
\sum_{\ell\leq k\gamma_{\mathcal{A},\Gamma}^{\prime}}\sum_{m\in\iota_{\ell}%
}|\kappa_{\ell}^{m}|  &  \leq C_{\mathcal{A},\Gamma}^{\prime^{\prime}}%
k^{1/2}\Vert{\mathbf{v}}\Vert_{\operatorname{curl},\Omega,1}%
,\label{eq:lemma:ntimesLomega-analytic-10}\\
\sum_{\ell>k\gamma_{\mathcal{A},\Gamma}^{\prime}}\sum_{m\in\iota_{\ell}%
}|\kappa_{\ell}^{m}|(\ell+1)^{\alpha}  &  \leq C_{\mathcal{A},\Gamma}%
^{\prime\prime}k^{2}(\gamma_{\mathcal{A},\Gamma}^{\prime\prime})^{\alpha
+1}(\alpha+1)^{\alpha+1}\Vert{\mathbf{v}}\Vert_{\operatorname{curl},\Omega
,1}\qquad\forall\alpha\geq0. \label{eq:lemma:ntimesLomega-analytic-12}%
\end{align}

\end{enumerate}
\end{lemma}

\proof\emph{Proof of (\ref{item:lemma:ntimesLomega-analytic-i}):} {}From
Theorem~\ref{TheoLau} we infer for suitable $C$, $\gamma$ that $L_{\Omega
}{\mathbf{v}}$ is in fact analytic on $\Omega\cup{\mathcal{U}}_{\Gamma}$ and
satisfies there
\begin{equation}
L_{\Omega}{\mathbf{v}}\in{\mathcal{A}}(Ck^{3/2}\Vert{\mathbf{v}}%
\Vert_{\operatorname{curl},\Omega,1},\gamma,\Omega\cup{\mathcal{U}}_{\Gamma})
\label{eq:analyticity-LOmegav-II}%
\end{equation}
The extension $g_{1}^{\ast}$ of $g_{1}=\langle L_{\Omega}{\mathbf{v}%
},{\mathbf{n}}\rangle$ into ${\mathcal{U}}_{\Gamma}$ is taken as $g_{1}^{\ast
}:=\left\langle L_{\Omega}\mathbf{v},\mathbf{n}^{\ast}\right\rangle $ where
$\mathbf{n}^{\ast}\left(  \mathbf{x}\right)  :=\mathbf{x}/\left\Vert
\mathbf{x}\right\Vert $ is the extension of the normal vector ${\mathbf{n}}$
to ${\mathcal{U}}_{\Gamma}$. By the analyticity of ${\mathbf{n}}^{\ast}$ and
(\ref{eq:analyticity-LOmegav-II}) we may apply Lemma~\ref{lemma:A-invariant}
to get with suitable constants $\tilde{C}$, $\tilde{\gamma}$ independent of
$k$ and $\mathbf{v}$,
\begin{equation}
g_{1}^{\ast}\in\mathcal{A}\left(  \tilde{C}k^{3/2}\left\Vert \mathbf{v}%
\right\Vert _{\operatorname*{curl},\Omega,1},\tilde{\gamma},\mathcal{U}%
_{\Gamma}\right)  . \label{g1starA-vorn}%
\end{equation}
\emph{Proof of (\ref{item:lemma:ntimesLomega-analytic-ii}):} Since for smooth
$w$ we have the pointwise bound $|\nabla_{\Gamma}w|\leq|(\nabla w)|_{\Gamma}%
|$, we get from a multiplicative trace inequality (see, e.g., \cite[Thm.~{A.2}%
]{melenk_multtrace})
\[
\left\Vert \nabla_{\Gamma}^{n}g_{1}\right\Vert _{\Gamma}\leq C\left(
\left\Vert \nabla^{n}g_{1}^{\ast}\right\Vert _{L^{2}\left(  \mathcal{U}%
_{\Gamma}\right)  }\left\Vert \nabla^{n}g_{1}^{\ast}\right\Vert _{H^{1}\left(
\mathcal{U}_{\Gamma}\right)  }\right)  ^{1/2}\qquad\forall n\in\mathbb{N}_{0}%
\]
so that $\displaystyle g_{1}\in\mathcal{A}\left(  C_{1}k^{2}\left\Vert
\mathbf{v}\right\Vert _{\operatorname*{curl},\Omega,1},\gamma_{1}%
,\Gamma\right)  $ for suitable $C_{1}$, $\gamma_{1}$; this is the second statement.

\emph{Proof of (\ref{item:lemma:ntimesLomega-analytic-iii}):} \emph{1.~step:}
By \cite[(2.5.212)]{Nedelec01}, the Laplace-Beltrami operator can be expressed
in terms of differential operators in ambient space: $\Delta u=\Delta_{\Gamma
}u+2H\partial_{n}u+\partial_{n}^{2}u$, where $H\equiv1$ is the mean curvature
of the unit sphere. Applying this to $u=g_{1}^{\ast}$ implies for some $C$,
$\gamma_{2}>0$ independent of $k$ and $j$ again with the trace inequality
\begin{equation}
\Vert\Delta_{\Gamma}^{j}g_{1}\Vert_{L^{2}(\Gamma)}\leq Ck^{2}\Vert{\mathbf{v}%
}\Vert_{\operatorname{curl},\Omega,1}\gamma_{2}^{2j}\max\{k,2j\}^{2j},
\label{eq:Delta^j}%
\end{equation}
\emph{2.~step:} Recall that by (\ref{eq:iota-sphere}) we have
$\operatorname{card}\iota_{\ell}=2\ell+1$ and that, by
(\ref{eq:TheoLau-simple}), $\Vert L_{\Omega}{\mathbf{v}}\Vert_{{\mathbf{H}%
}(\operatorname{div},\Omega)}=\Vert L_{\Omega}{\mathbf{v}}\Vert_{{L}%
^{2}(\Omega)}\leq k^{-1}\Vert L_{\Omega}{\mathbf{v}}\Vert_{\operatorname{curl}%
,\Omega,k}\leq k^{-1}\Vert{\mathbf{v}}\Vert_{\operatorname{curl},\Omega,k}$.
By orthonormality of the $Y_{\ell}^{m}$, the expansion coefficients
$\kappa_{\ell}^{m}$ are given by $\displaystyle\kappa_{\ell}^{m}%
=(g_{1},Y_{\ell}^{m})_{\Gamma}$. We estimate the low-order coefficients
($\ell\leq k\gamma_{\mathcal{A},\Gamma}^{\prime}$) by%
\begin{align}
\left\vert \kappa_{\ell}^{m}\right\vert  &  \leq\sum_{\ell\leq k\gamma
_{\mathcal{A},\Gamma}^{\prime}}\sum_{m\in\iota_{\ell}}|\kappa_{\ell}^{m}%
|\leq\left(  \sum_{\ell\leq k\gamma_{\mathcal{A},\Gamma}^{\prime}}\sum
_{m\in\iota_{\ell}}|\kappa_{\ell}^{m}|^{2}\lambda_{\ell}^{-1/2}\right)
^{1/2}\left(  \sum_{\ell\leq k\gamma_{\mathcal{A},\Gamma}^{\prime}}\sum
_{m\in\iota_{\ell}}\lambda_{\ell}^{1/2}\right)  ^{1/2}\lesssim\Vert\langle
L_{\Omega}{\mathbf{v}},{\mathbf{n}}\rangle\Vert_{H^{-1/2}(\Gamma)}%
k^{3/2}\nonumber\\
&  \lesssim k^{3/2}\Vert L_{\Omega}{\mathbf{v}}\Vert_{{\mathbf{H}%
}(\operatorname{div},\Omega)}\lesssim k^{1/2}\Vert{\mathbf{v}}\Vert
_{\operatorname{curl},\Omega,k} \label{eq:foo-100}%
\end{align}
which implies the first estimate in (\ref{eq:lemma:ntimesLomega-analytic-14})
and (\ref{eq:lemma:ntimesLomega-analytic-10}).

\noindent\emph{3.~step:} The minimum of $x\mapsto x^{x}$ is attained at $1/e$
with value $\operatorname{e}^{-1/\operatorname{e}}<1$. Hence, there are
$q\in(0,1)$ and $\delta>0$ such that
\begin{equation}
x^{x}\leq q<1\qquad\forall x\in\lbrack1/\operatorname{e}-\delta
,1/\operatorname{e}+\delta]. \label{eq:xhochx}%
\end{equation}
\emph{4.~step:} 
For $\gamma_{3}:=\max\{1,2 \gamma_2/\delta\}>0$ the
following implication holds for $k \ge 1$:%
\begin{equation}
\ell\geq\gamma_{3}k\quad\Longrightarrow\quad j:=\left\lfloor \frac{\ell
}{2\gamma_{2}\operatorname{e}}\right\rfloor \quad\mbox{ satisfies }j\geq
k\quad\mbox{ and }\quad\frac{2j\gamma_{2}}{\ell}\in\lbrack1/\operatorname{e}%
-\delta,1/\operatorname{e}+\delta]. \label{eq:choice-of-j}%
\end{equation}
\emph{5.~step:} Given $\ell\geq\gamma_{3}k$ we select $j$ as in
(\ref{eq:choice-of-j}). Using the orthonormality of the $Y_{\ell}^{m}$ with
the eigenvalues $\lambda_{\ell}=\ell(\ell+1)$ of $-\Delta_{\Gamma}$, we
compute
\begin{align*}
|\kappa_{\ell}^{m}|  &  =\lambda_{\ell}^{-j}\left\vert \left(  g_{1}%
,(-\Delta_{\Gamma})^{j}Y_{\ell}^{m}\right)  _{\Gamma}\right\vert
=\lambda_{\ell}^{-j}\left\vert \left(  (-\Delta_{\Gamma})^{j}g_{1},Y_{\ell
}^{m}\right)  _{\Gamma}\right\vert \overset{(\ref{eq:Delta^j})}{\leq}%
Ck^{2}\Vert{\mathbf{v}}\Vert_{\operatorname{curl},\Omega,1}\gamma_{2}%
^{2j}(\ell(\ell+1))^{-j}\max\{k,2j\}^{2j}\\
&  \leq Ck^{2}\Vert{\mathbf{v}}\Vert_{\operatorname{curl},\Omega,1}\gamma
_{2}^{2j}(\ell(\ell+1))^{-j}(2j)^{2j}\leq Ck^{2}\Vert{\mathbf{v}}%
\Vert_{\operatorname{curl},\Omega,1}\left(  (2j\gamma_{2}/\ell)^{2j\gamma
_{2}/\ell}\right)  ^{\ell/\gamma_{2}}\overset{(\ref{eq:choice-of-j}%
),(\ref{eq:xhochx})}{\leq}Ck^{2}\Vert{\mathbf{v}}\Vert_{\operatorname{curl}%
,\Omega,1}q^{\ell/\gamma_{2}}.
\end{align*}
This and (\ref{eq:foo-100}) shows the bound
(\ref{eq:lemma:ntimesLomega-analytic-14}) for $\kappa_{\ell}^{m}$.

\emph{6.~step:} We show (\ref{eq:lemma:ntimesLomega-analytic-12}). We start
with the observation
\begin{equation}
\sup_{x>0}x^{\alpha}\operatorname{e}^{-x}\leq\alpha^{\alpha}\operatorname{e}%
^{-\alpha}\qquad\forall\alpha\geq0. \label{eq:xalpha}%
\end{equation}
Then,
\begin{align*}
\sum_{\ell\geq k\gamma_{\mathcal{A},\Gamma}^{\prime}}\sum_{m\in\iota_{\ell}%
}|\kappa_{\ell}^{m}|(\ell+1)^{\alpha}  &
\overset{\text{(\ref{eq:lemma:ntimesLomega-analytic-14})}}{\leq}%
C_{\mathcal{A},\Gamma}^{\prime}k^{2}\Vert{\mathbf{v}}\Vert
_{\operatorname{curl},\Omega,1}\sum_{\ell>k\gamma_{\mathcal{A},\Gamma}%
^{\prime}}(\ell+1)^{\alpha}(2\ell+1)e^{-b\ell}\\
&  \lesssim C_{\mathcal{A},\Gamma}^{\prime}k^{2}\Vert{\mathbf{v}}%
\Vert_{\operatorname{curl},\Omega,1}\sum_{\ell=1}^{\infty}(\ell+1)^{\alpha
+1}\operatorname{e}^{-b(\ell+1)}.
\end{align*}
Upon writing
\[
(\ell+1)^{\alpha+1}\operatorname{e}^{-b(\ell+1)}=\left(  \frac{(\ell+1)b}%
{2}\right)  ^{\alpha+1}\operatorname{e}^{-b(\ell+1)/2}\left(  \frac{2}%
{b}\right)  ^{\alpha+1}\operatorname{e}^{-b(\ell+1)/2}%
\overset{(\ref{eq:xalpha})}{\leq}(\alpha+1)^{\alpha+1}\left(  \frac{2}%
{b}\right)  ^{\alpha+1}\operatorname{e}^{-b(\ell+1)/2}%
\]
we see that the infinite sum can be controlled in the desired fashion.
\endproof


\subsection{Helmholtz Decomposition}

The stability properties of the operators $\Pi^{\operatorname{comp}}$,
$\Pi^{\nabla}$, $\Pi^{\operatorname{comp}}_{h}$, $\Pi^{\nabla}_{h}$ and the
splittings induced by them in Definition~\ref{DefHelmDecomp} are characterized
in Lemma~\ref{LemNablaCurl} in terms of the constants $C^{\operatorname{high}%
}_{b,k}$, $C^{H,\Omega}_{k}$, $C^{\nabla,\operatorname{high}}_{b,k}$. For the
case of the unit ball $B_{1}(0)$ we show in Lemma~\ref{Lemstablesplit} that
these constants can be bounded uniformly in $k$. We furthermore track the
dependence of these constants on the cut-off parameter $\lambda>1$ that enters
the definition of $L_{\Omega}$ and $H_{\Omega}$
(cf.~Definition~\ref{DefFreqSplit}). We track the $\lambda$-dependence with
the aid of the norm
\begin{equation}
\left\Vert \mathbf{u}\right\Vert _{\operatorname*{curl},\Omega,k,\lambda
}:=\left(  k^{2}\left\Vert \mathbf{u}\right\Vert ^{2}+\frac{1}{\lambda^{2}%
}\left\Vert \operatorname*{curl}\mathbf{u}\right\Vert ^{2}\right)  ^{1/2}.
\label{eq:normcurllambda}%
\end{equation}

\begin{lemma}
[Stability of the splitting]\label{Lemstablesplit} Let $\Omega=B_{1}(0)$ and
$\lambda\geq\lambda_{0}>1$. Then there exists $C>0$ depending solely on
$\lambda_{0}$ such that the decompositions of
$\mathbf{u}$, $\mathbf{v}\in\mathbf{X}$ as
\begin{align*}
\mathbf{u}  &  =\Pi^{\operatorname{comp}}{\mathbf{u}}+\Pi^{\nabla}H_{\Omega
}{\mathbf{u}}=(\mathbf{u}^{\operatorname*{low}}+\Pi^{\operatorname*{curl}%
}\mathbf{u}^{\operatorname*{high}})+\Pi^{\nabla}\mathbf{u}%
^{\operatorname*{high}},\\
\mathbf{v}  &  =\Pi^{\operatorname{comp},\ast}{\mathbf{v}}+\Pi^{\nabla,\ast
}H_{\Omega}{\mathbf{v}}=(\mathbf{v}^{\operatorname*{low}}+\Pi
^{\operatorname*{curl},\ast}\mathbf{v}^{\operatorname*{high}})+\Pi
^{\nabla,\ast}\mathbf{v}^{\operatorname*{high}},
\end{align*}
where $\mathbf{u}^{\operatorname*{low}}:=L_{\Omega}\mathbf{u}$, $\mathbf{u}%
^{\operatorname*{high}}:=H_{\Omega}\mathbf{u}$, $\mathbf{v}%
^{\operatorname*{low}}=L_{\Omega}\mathbf{v}$, and $\mathbf{v}%
^{\operatorname*{high}}:=H_{\Omega}\mathbf{v}$ satisfy: 
%
\begin{subequations}
\label{msplitI}%
\begin{align}
k\left\Vert \Pi^{\operatorname*{curl}}\mathbf{u}^{\operatorname*{high}%
}\right\Vert +k\left\Vert \Pi^{\nabla}\mathbf{u}^{\operatorname*{high}%
}\right\Vert  &  \leq C\left\Vert \mathbf{u}^{\operatorname*{high}}\right\Vert
_{\operatorname*{curl},\Omega,k,\lambda}\leq C\left\Vert \mathbf{u}\right\Vert
_{\operatorname*{curl},\Omega,k}, &  & \tag{%
\ref{msplitI}%
a}\label{msplitI-a}\\
\left\Vert \operatorname*{curl}\left(  \Pi^{\operatorname*{curl}}%
\mathbf{u}^{\operatorname*{high}}\right)  \right\Vert  &  =\left\Vert
\operatorname*{curl}\mathbf{u}^{\operatorname*{high}}\right\Vert
\leq2\left\Vert \mathbf{u}\right\Vert _{\operatorname*{curl},\Omega,k}, &
\operatorname*{curl}\left(  \Pi^{\nabla}\mathbf{u}^{\operatorname*{high}%
}\right)   &  =0,\tag{%
\ref{msplitI}%
b}\label{msplitI-b}\\
\Vert\Pi^{\operatorname*{curl}}\mathbf{u}^{\operatorname*{high}}%
\Vert_{{\mathbf{H}}^{1}(\Omega)}  &  \leq C\Vert{\mathbf{u}}\Vert
_{\operatorname{curl},\Omega,k}. &  &  \tag{%
\ref{msplitI}%
c}\label{msplitI-c}%
\end{align}
Analogous estimates hold for $\Pi^{\operatorname*{curl},\ast}\mathbf{v}%
^{\operatorname*{high}}$ and $\Pi^{\nabla,\ast}\mathbf{v}%
^{\operatorname*{high}}$.
\end{subequations}
\end{lemma}

%

\proof
For $\mathbf{u}\in\mathbf{X}$, choose $p\in H^{1}\left(  \Omega\right)
/\mathbb{R}$ such that $\Pi^{\nabla}\mathbf{u}^{\operatorname*{high}}=\nabla
p$, and $\mathbf{u}_{0}:=\Pi^{\operatorname*{curl}}\mathbf{u}%
^{\operatorname*{high}}\in\mathbf{V}_{0}$. A direct consequence is the second
relation in (\ref{msplitI-b}): $\operatorname*{curl}\left(  \Pi^{\nabla
}\mathbf{u}^{\operatorname*{high}}\right)  =0$. For the remaining estimates,
we first collect some simple facts about this splitting.

1) The definition of the space $\mathbf{V}_{0}^{\ast}$ implies $0=\left(
\kern-.1em%
\left(
\kern-.1em%
\nabla p,\mathbf{v}_{0}%
\kern-.1em%
\right)
\kern-.1em%
\right)  =\left(
\kern-.1em%
\left(
\kern-.1em%
\mathbf{u}^{\operatorname*{high}}-\mathbf{u}_{0},\mathbf{v}_{0}%
\kern-.1em%
\right)
\kern-.1em%
\right)  $ for all $\mathbf{v}_{0}\in\mathbf{V}_{0}^{\ast}$.

2) In Lemma~\ref{Lemembedspec}, we prove for the unit ball%
\begin{equation}
\left\Vert \mathbf{u}_{0}\right\Vert _{\mathbf{H}^{1}\left(  \Omega\right)
}\leq\left\Vert \mathbf{u}_{0}\right\Vert _{\operatorname*{curl},\Omega,1}.
\label{u0H1curl}%
\end{equation}
Together with (\ref{lowhighuest}), we obtain (\ref{msplitI-c}).

3) $\operatorname*{curl}\nabla p=0$ implies%
\begin{equation}
\operatorname*{curl}\mathbf{u}^{\operatorname*{high}}=\operatorname*{curl}%
\mathbf{u}_{0}. \label{curlid}%
\end{equation}
The combination with (\ref{lowhighuest}) leads to the first relation in
(\ref{msplitI-b}) and the subsequent estimate follows from (\ref{lowhighuest}%
). Note that (\ref{lowhighuest}) also implies the second estimate in
(\ref{msplitI-a}).

4) Since $\mathbf{u}_{0}\in\mathbf{V}_{0}$, the definition (\ref{defVo})
implies%
\begin{equation}
\operatorname{Re}\left(
\kern-.1em%
\left(
\kern-.1em%
\mathbf{u}_{0},\mathbf{u}_{0}%
\kern-.1em%
\right)
\kern-.1em%
\right)  =\operatorname{Re}\left(
\kern-.1em%
\left(
\kern-.1em%
\mathbf{u}_{0},\mathbf{u}^{\operatorname*{high}}%
\kern-.1em%
\right)
\kern-.1em%
\right)  -\operatorname{Re}\left(
\kern-.1em%
\left(
\kern-.1em%
\mathbf{u}_{0},\nabla p%
\kern-.1em%
\right)
\kern-.1em%
\right)  \overset{(\ref{defVoa})}{=}\operatorname{Re}\left(
\kern-.1em%
\left(
\kern-.1em%
\mathbf{u}_{0},\mathbf{u}^{\operatorname*{high}}%
\kern-.1em%
\right)
\kern-.1em%
\right)  . \label{uu0sub}%
\end{equation}

5) The weighted $\mathbf{L}^{2}\left(  \Omega\right)  $-norm of $\mathbf{u}%
_{0}$ can be estimated via%
\begin{align}
k^{2}\left\Vert \mathbf{u}_{0}\right\Vert ^{2}  &  \overset{\text{(\ref{Imb1}%
)}}{\leq}k^{2}\left\Vert \mathbf{u}_{0}\right\Vert ^{2}-\operatorname*{Im}%
kb_{k}\left(  \mathbf{u}_{0}^{\nabla},\mathbf{u}_{0}^{\nabla}\right)
=\operatorname{Re}\left(
\kern-.1em%
\left(
\kern-.1em%
\mathbf{u}_{0},\mathbf{u}_{0}%
\kern-.1em%
\right)
\kern-.1em%
\right)  \overset{\text{(\ref{uu0sub})}}{=}\operatorname{Re}\left(
\kern-.1em%
\left(
\kern-.1em%
\mathbf{u}_{0},\mathbf{u}^{\operatorname*{high}}%
\kern-.1em%
\right)
\kern-.1em%
\right) \label{omega2epsmue}\\
&  =\operatorname{Re}\left(  k^{2}\left(  \mathbf{u}_{0},\mathbf{u}%
^{\operatorname*{high}}\right)  +\operatorname*{i}k\left(  b_{k}%
^{\operatorname*{low}}\left(  \mathbf{u}_{0}^{\nabla},\left(  \mathbf{u}%
^{\operatorname*{high}}\right)  ^{\nabla}\right)  +b_{k}^{\operatorname*{high}%
}\left(  \mathbf{u}_{0}^{\nabla},\left(  \mathbf{u}^{\operatorname*{high}%
}\right)  ^{\nabla}\right)  \right)  \right) \nonumber\\
&  \leq\frac{1}{2}\left(  k\left\Vert \mathbf{u}_{0}\right\Vert \right)
^{2}+\frac{1}{2}\left(  k\left\Vert \mathbf{u}^{\operatorname*{high}%
}\right\Vert \right)  ^{2}+k\left\vert b_{k}^{\operatorname*{low}}\left(
\mathbf{u}_{0}^{\nabla},\left(  \mathbf{u}^{\operatorname*{high}}\right)
^{\nabla}\right)  \right\vert +k\left\vert b_{k}^{\operatorname*{high}}\left(
\mathbf{u}_{0}^{\nabla},\left(  \mathbf{u}^{\operatorname*{high}}\right)
^{\nabla}\right)  \right\vert . \label{omega2epsmue3}%
\end{align}
{}From (\ref{commpropfreq}), we conclude that $\left(  \mathbf{u}%
^{\operatorname*{high}}\right)  ^{\nabla}=\sum_{\ell>\lambda k}\sum_{m\in
\iota_{\ell}}U_{\ell}^{m}\nabla_{\Gamma}Y_{\ell}^{m}$ and it follows from the
definition of $b_{k}^{\operatorname*{low}}$ in (\ref{defblowbhigh}) that
$b_{k}^{\operatorname*{low}}\left(  \mathbf{u}_{0}^{\nabla},\left(
\mathbf{u}^{\operatorname*{high}}\right)  ^{\nabla}\right)  =0$.

Next, we estimate the last term in (\ref{omega2epsmue3}). Our decomposition
$\mathbf{u}^{\operatorname*{high}}=\mathbf{u}_{0}+\nabla p$ leads to%
\begin{equation}
k\left\vert b_{k}^{\operatorname*{high}}\left(  \mathbf{u}_{0}^{\nabla
},\left(  \mathbf{u}^{\operatorname*{high}}\right)  ^{\nabla}\right)
\right\vert \leq k\left\vert b_{k}^{\operatorname*{high}}\left(
\mathbf{u}_{0}^{\nabla},\mathbf{u}_{0}^{\nabla}\right)  \right\vert
+k\left\vert b_{k}^{\operatorname*{high}}\left(  \mathbf{u}_{0}^{\nabla
},\nabla_{\Gamma}p\right)  \right\vert . \label{0.8}%
\end{equation}
The first term can be estimated by using (\ref{freq_reg_H1H1a}):%
\begin{equation}
k\left\vert b_{k}^{\operatorname*{high}}\left(  \mathbf{u}_{0}^{\nabla
},\mathbf{u}_{0}^{\nabla}\right)  \right\vert \overset{(\ref{freq_reg_H1H1a}%
)}{\leq}\frac{C_{b}^{\prime}}{\lambda^{2}}\left\Vert \mathbf{u}_{0}\right\Vert
_{\mathbf{H}^{1}\left(  \Omega\right)  }^{2}\overset{\text{(\ref{u0H1curl}%
)}}{\leq}\frac{C_{b}^{\prime}}{\lambda^{2}}\left\Vert \mathbf{u}%
_{0}\right\Vert _{\operatorname*{curl},\Omega,1}^{2}=\frac{C_{b}^{\prime}%
}{\lambda^{2}}\left(  \left\Vert \mathbf{u}_{0}\right\Vert ^{2}+\left\Vert
\operatorname*{curl}\mathbf{u}^{\operatorname*{high}}\right\Vert ^{2}\right)
. \label{0.9}%
\end{equation}
For the second term of the right-hand side of (\ref{0.8}) we assume that $p\in
C^{\infty}\left(  \overline{\Omega}\right)  $ since the result for general
$p\in H^{1}\left(  \Omega\right)  $ follows by a density argument. We obtain%
\begin{align}
k\left\vert b_{k}^{\operatorname*{high}}\left(  \mathbf{u}_{0}^{\nabla}%
,\nabla_{\Gamma}p\right)  \right\vert  &  \overset{\text{(\ref{freq_reg_H1H1b}%
)}}{\leq}\frac{C_{b}^{\prime}}{\lambda}k\left\Vert \nabla p\right\Vert
_{\operatorname*{curl},\Omega,1}\left\Vert \mathbf{u}_{0}\right\Vert
_{\mathbf{H}^{1}\left(  \Omega\right)  }\overset{\substack{
\text{(\ref{u0H1curl})},\\ \operatorname*{curl}\nabla p=0}}{\leq}\frac
{C_{b}^{\prime}}{\lambda}k\left\Vert \nabla p\right\Vert \left(  \left\Vert
\mathbf{u}_{0}\right\Vert ^{2}+\left\Vert \operatorname*{curl}\mathbf{u}%
^{\operatorname*{high}}\right\Vert ^{2}\right)  ^{1/2}\label{0.10}\\
&  \leq\frac{C_{b}^{\prime}}{\lambda}\left(  k\left\Vert \mathbf{u}%
^{\operatorname*{high}}\right\Vert +k\left\Vert \mathbf{u}_{0}\right\Vert
\right)  \left(  \left\Vert \mathbf{u}_{0}\right\Vert +\left\Vert
\operatorname*{curl}\mathbf{u}^{\operatorname*{high}}\right\Vert \right)
.\nonumber
\end{align}
Inserting (\ref{0.9}), (\ref{0.10}) into (\ref{0.8}) and employing
Cauchy-Schwarz inequalities with $\eta>0$ leads to
\[
k\left\vert b_{k}^{\operatorname*{high}}\left(  \mathbf{u}_{0}^{\nabla
},\left(  \mathbf{u}^{\operatorname*{high}}\right)  ^{\nabla}\right)
\right\vert \leq C_{b}^{\prime}\left(  \left(  \frac{3}{2\lambda^{2}k^{2}%
}+\frac{1}{\lambda k}+\frac{\eta}{2}\right)  \left(  k\left\Vert
\mathbf{u}_{0}\right\Vert \right)  ^{2}+\left(  k\left\Vert \mathbf{u}%
^{\operatorname*{high}}\right\Vert \right)  ^{2}+\left(  \frac{3+\eta^{-1}}%
{2}\right)  \left(  \frac{\left\Vert \operatorname*{curl}\mathbf{u}%
^{\operatorname*{high}}\right\Vert }{\lambda}\right)  ^{2}\right)  .
\]
We combine this estimate with (\ref{omega2epsmue}) and absorb the first term
on the right-hand side of (\ref{omega2epsmue3}) into the left-hand side of
(\ref{omega2epsmue}) to obtain (using $\lambda k > 1$)
\[
\frac{k^{2}}{2}\left\Vert \mathbf{u}_{0}\right\Vert ^{2}\leq C_{b}^{\prime
}\left(  \frac{5}{2\lambda k}+\frac{\eta}{2}\right)  k^{2}\left\Vert
\mathbf{u}_{0}\right\Vert ^{2}+\left(  \frac{1}{2}+C_{b}^{\prime}\right)
\left(  k\left\Vert \mathbf{u}^{\operatorname*{high}}\right\Vert \right)
^{2}+C_{b}^{\prime}\left(  \frac{3+\eta^{-1}}{2}\right)  \left(
\frac{\left\Vert \operatorname*{curl}\mathbf{u}^{\operatorname*{high}%
}\right\Vert }{\lambda}\right)  ^{2}.
\]

We first consider the case $k\geq\max\left\{  1,\frac{20C_{b}^{\prime}%
}{\lambda}\right\}  $ and choose $\eta=\frac{1}{4C_{b}^{\prime}}$. This leads
to%
\begin{align}
\frac{k^{2}}{4}\left\Vert \mathbf{u}_{0}\right\Vert ^{2}  &  \leq C_{1}%
k^{2}\left\Vert \mathbf{u}^{\operatorname*{high}}\right\Vert ^{2}+\frac{C_{2}%
}{\lambda^{2}}\left\Vert \operatorname*{curl}\mathbf{u}^{\operatorname*{high}%
}\right\Vert ^{2}\label{estu0stab}\\
\text{with\quad}C_{1}  &  :=\left(  \frac{1}{2}+C_{b}^{\prime}\right)  ,\quad
C_{2}:=C_{b}^{\prime}\left(  \frac{3+4C_{b}^{\prime}}{2}\right)  .\nonumber
\end{align}
This yields the first estimate for the first term in (\ref{msplitI-a}) (for
the considered range of $k$).

For $1\leq k\leq\max\left\{  1,\frac{20C_{b}^{\prime}}{\lambda}\right\}  $, we
estimate the term $k\left\vert b_{k}^{\operatorname*{high}}\left(
\mathbf{u}_{0}^{\nabla},\left(  \mathbf{u}^{\operatorname*{high}}\right)
^{\nabla}\right)  \right\vert $ in (\ref{omega2epsmue3}) by using
(\ref{frequency_rega}) and $\eta>0$%
\begin{align*}
k\left\vert b_{k}^{\operatorname*{high}}\left(  \mathbf{u}_{0}^{\nabla
},\left(  \mathbf{u}^{\operatorname*{high}}\right)  ^{\nabla}\right)
\right\vert  &  \leq C_{b}^{\prime}\frac{k}{\lambda}\left\Vert \mathbf{u}%
_{0}\right\Vert _{\operatorname*{curl},\Omega,1}\left\Vert \mathbf{u}%
^{\operatorname*{high}}\right\Vert _{\operatorname*{curl},\Omega,1}\\
&  \leq C_{b}^{\prime}\frac{k}{2\lambda}\left(  \eta\left\Vert \mathbf{u}%
_{0}\right\Vert _{\operatorname*{curl},\Omega,1}^{2}+\frac{1}{\eta}\left\Vert
\mathbf{u}^{\operatorname*{high}}\right\Vert _{\operatorname*{curl},\Omega
,1}^{2}\right) \\
&  \overset{\text{(\ref{curlid})}}{\leq}C_{b}^{\prime}\frac{k}{2\lambda
}\left(  \eta\left\Vert \mathbf{u}_{0}\right\Vert ^{2}+\left(  \eta+\frac
{1}{\eta}\right)  \left\Vert \mathbf{u}^{\operatorname*{high}}\right\Vert
_{\operatorname*{curl},\Omega,1}^{2}\right)  .
\end{align*}
The combination of this\ estimate with (\ref{omega2epsmue}) (by taking into
account $b_{k}^{\operatorname*{low}}\left(  \mathbf{u}_{0}^{\nabla},\left(
\mathbf{u}^{\operatorname*{high}}\right)  ^{\nabla}\right)  =0$) leads to%
\begin{align*}
k^{2}\left\Vert \mathbf{u}_{0}\right\Vert ^{2}  &  \leq\frac{1}{2}\left(
k\left\Vert \mathbf{u}_{0}\right\Vert \right)  ^{2}+\frac{1}{2}\left(
k\left\Vert \mathbf{u}^{\operatorname*{high}}\right\Vert \right)  ^{2}\\
&  +C_{b}^{\prime}\frac{k}{2\lambda}\left(  \eta\left\Vert \mathbf{u}%
_{0}\right\Vert ^{2}+\left(  \eta+\frac{1}{\eta}\right)  \left\Vert
\mathbf{u}^{\operatorname*{high}}\right\Vert _{\operatorname*{curl},\Omega
,1}^{2}\right) \\
&  =\left(  \frac{1}{2}+C_{b}^{\prime}\frac{\eta}{2\lambda k}\right)  \left(
k\left\Vert \mathbf{u}_{0}\right\Vert \right)  ^{2}+\left(  \frac{1}{2}%
+C_{b}^{\prime}\frac{\eta+\eta^{-1}}{2\lambda k}\right)  k^{2}\left\Vert
\mathbf{u}^{\operatorname*{high}}\right\Vert _{\operatorname*{curl},\Omega
,1}^{2}.
\end{align*}
Recall $\lambda k\geq\lambda>1$. The choice $\eta=\frac{1}{2C_{b}^{\prime}}$
leads to%
\begin{align*}
k^{2}\left\Vert \mathbf{u}_{0}\right\Vert ^{2}  &  \leq\left(  \frac{1}%
{2}+\frac{1}{4\lambda k}\right)  \left(  k\left\Vert \mathbf{u}_{0}\right\Vert
\right)  ^{2}+\left(  \frac{1}{2}+\left(  \frac{1}{4}+\left(  C_{b}^{\prime
}\right)  ^{2}\right)  \frac{1}{\lambda k}\right)  k^{2}\left\Vert
\mathbf{u}^{\operatorname*{high}}\right\Vert _{\operatorname*{curl},\Omega
,1}^{2}\\
&  \leq\frac{3}{4}\left(  k\left\Vert \mathbf{u}_{0}\right\Vert \right)
^{2}+\left(  \frac{1}{2}+\left(  \frac{1}{4}+\left(  C_{b}^{\prime}\right)
^{2}\right)  \right)  k^{2}\left\Vert \mathbf{u}^{\operatorname*{high}%
}\right\Vert _{\operatorname*{curl},\Omega,1}^{2}.
\end{align*}
The first term on the right-hand side can be absorbed into the left-hand side.
Since $k\leq\frac{20C_{b}^{\prime}}{\lambda}$, we get%
\[
k^{2}\left\Vert \mathbf{u}_{0}\right\Vert ^{2}\leq C\left\Vert \mathbf{u}%
^{\operatorname*{high}}\right\Vert _{\operatorname*{curl},\Omega,k,\lambda
}^{2}\leq C\left\Vert \mathbf{u}\right\Vert _{\operatorname*{curl},\Omega
,k}^{2}.
\]

The $L^{2}$ estimate for $\nabla p$ follows by a triangle inequality:%
\[
k\left\Vert \nabla p\right\Vert \leq k\left(  \left\Vert \mathbf{u}%
^{\operatorname*{high}}\right\Vert +\left\Vert \mathbf{u}_{0}\right\Vert
\right)  \overset{\text{(\ref{lowhighuest}),}\ \text{(\ref{estu0stab})}}{\leq
}C_{3}\left(  k\left\Vert \mathbf{u}^{\operatorname*{high}}\right\Vert
+\lambda^{-1}\left\Vert \operatorname*{curl}\mathbf{u}^{\operatorname*{high}%
}\right\Vert \right)  \leq C_{3}^{\prime}\left\Vert \mathbf{u}%
^{\operatorname*{high}}\right\Vert _{\operatorname*{curl},\Omega,k,\lambda
}\leq C_{3}^{\prime\prime}\left\Vert \mathbf{u}\right\Vert
_{\operatorname*{curl},\Omega,k}.
\]

The estimates for $\left\Vert \Pi^{\operatorname*{curl},\ast}\mathbf{v}%
^{\operatorname*{high}}\right\Vert +k\left\Vert \Pi^{\nabla,\ast}\mathbf{v}%
^{\operatorname*{high}}\right\Vert $ are derived by repeating the arguments
above.%
\endproof

By similar techniques we will prove next that if one argument in $\left(
\kern-.1em%
\left(
\kern-.1em%
\cdot,\cdot%
\kern-.1em%
\right)
\kern-.1em%
\right)  $ has only high-frequency components then we get $k$-independent
continuity estimates (cf.~also (\ref{DefCconthighk}) for the general case):

\begin{proposition}
\label{PropStabCSU} Let $\Omega= B_{1}(0)$ and $\lambda\ge\lambda_{0} > 1$.
Then there exists $\widetilde{C}_{b} > 0$ depending solely on $\lambda_{0}$
such that for all $\mathbf{u}$, $\mathbf{v}\in\mathbf{X}$
\begin{align}
\left\vert \left(
\kern-.1em%
\left(
\kern-.1em%
H_{\Omega}\mathbf{u},\mathbf{v}%
\kern-.1em%
\right)
\kern-.1em%
\right)  \right\vert +\left\vert \left(
\kern-.1em%
\left(
\kern-.1em%
\mathbf{u},H_{\Omega}\mathbf{v}
\kern-.1em%
\right)
\kern-.1em%
\right)  \right\vert  &  \leq\widetilde{C}_{b}\left\Vert \mathbf{u}\right\Vert
_{\operatorname*{curl},\Omega,k,\lambda}\left\Vert \mathbf{v}\right\Vert
_{\operatorname*{curl},\Omega,k,\lambda},\\
\left\vert \left(
\kern-.1em%
\left(
\kern-.1em%
\mathbf{u},\mathbf{v}%
\kern-.1em%
\right)
\kern-.1em%
\right)  \right\vert  &  \leq C_{\operatorname*{cont},k}\left\Vert
\mathbf{u}\right\Vert _{\operatorname*{curl},\Omega,1}\left\Vert
\mathbf{v}\right\Vert _{\operatorname*{curl},\Omega,1},
\label{lowfrequencyCSU}%
\end{align}
where $C_{\operatorname*{cont},k}\leq\widetilde{C}_{b}k^{3}$.
\end{proposition}

%

\proof
For $\mathbf{u}$, $\mathbf{v}\in\mathbf{X}$, write $\mathbf{u}%
^{\operatorname*{high}}:=H_{\Omega}\mathbf{u}$, $\mathbf{v}%
^{\operatorname*{high}}:=H_{\Omega}\mathbf{v}$. Choose $p$, $q\in H^{1}\left(
\Omega\right)  $ such that $\Pi^{\nabla}\mathbf{u}^{\operatorname*{high}%
}=\nabla p$, $\Pi^{\nabla,\ast}\mathbf{v}^{\operatorname*{high}}=\nabla q$ and
set $\mathbf{u}_{0}=\mathbf{u}^{\operatorname*{high}}-\nabla p$,
$\mathbf{v}_{0}=\mathbf{v}^{\operatorname*{high}}-\nabla q$. Since $\Pi
_{T}H_{\Omega}=H_{\Gamma}\Pi_{T}$ (cf.~(\ref{commpropfreq})) we have%
\[
\left\vert \left(
\kern-.1em%
\left(
\kern-.1em%
\mathbf{u}^{\operatorname*{high}},\mathbf{v}%
\kern-.1em%
\right)
\kern-.1em%
\right)  \right\vert \leq\left(  k\left\Vert \mathbf{u}^{\operatorname*{high}%
}\right\Vert \right)  \left(  k\left\Vert \mathbf{v}\right\Vert \right)
+\left\vert kb_{k}^{\operatorname*{high}}\left(  \left(  \mathbf{u}%
^{\operatorname*{high}}\right)  ^{\nabla},\mathbf{v}^{\nabla}\right)
\right\vert .
\]
For the boundary term, we get%
\begin{align}
\left\vert kb_{k}^{\operatorname*{high}}\left(  \left(  \mathbf{u}%
^{\operatorname*{high}}\right)  ^{\nabla},\mathbf{v}^{\nabla}\right)
\right\vert  &  \leq\left\vert kb_{k}^{\operatorname*{high}}\left(
\mathbf{u}_{0}^{\nabla},\mathbf{v}_{0}^{\nabla}\right)  \right\vert
+\left\vert kb_{k}^{\operatorname*{high}}\left(  \left(  \nabla p\right)
^{\nabla},\mathbf{v}_{0}^{\nabla}\right)  \right\vert \nonumber\\
&  +\left\vert kb_{k}^{\operatorname*{high}}\left(  \mathbf{u}_{0}^{\nabla
},\left(  \nabla q\right)  ^{\nabla}\right)  \right\vert +\left\vert
kb_{k}^{\operatorname*{high}}\left(  \left(  \nabla p\right)  ^{\nabla
},\left(  \nabla q\right)  ^{\nabla}\right)  \right\vert \nonumber\\
&  \overset{\text{(\ref{bhighgrad}), (\ref{frequency_reg_2a})}}{\leq}%
\frac{C_{b}^{\prime}}{\lambda^{2}}\left\Vert \mathbf{u}_{0}\right\Vert
_{\operatorname*{curl},\Omega,1}\left\Vert \mathbf{v}_{0}\right\Vert
_{\operatorname*{curl},\Omega,1}+\frac{C_{b}^{\prime}}{\lambda}\left(
k\left\Vert \nabla p\right\Vert \right)  \left\Vert \mathbf{v}_{0}\right\Vert
_{\operatorname*{curl},\Omega,1}\nonumber\\
&  \qquad+\frac{C_{b}^{\prime}}{\lambda}\left\Vert \mathbf{u}_{0}\right\Vert
_{\operatorname*{curl},\Omega,1}k\left\Vert \nabla q\right\Vert +C_{b}%
^{\prime}\left(  k\left\Vert \nabla p\right\Vert \right)  \left(  k\left\Vert
\nabla q\right\Vert \right) \nonumber\\
&  \overset{\text{(\ref{msplitI})}}{\leq}\widetilde{C}_{b}\left\Vert
\mathbf{u}\right\Vert _{\operatorname*{curl},\Omega,k,\lambda}\left\Vert
\mathbf{v}\right\Vert _{\operatorname*{curl},\Omega,k,\lambda}.
\label{estCkbk}%
\end{align}
The estimate for $((\mathbf{u},\mathbf{v}^{\operatorname*{high}}))$ follows
from the same arguments.

It remains to prove estimate (\ref{lowfrequencyCSU}). We choose $\lambda
=\lambda_{0}=O\left(  1\right)  $ in all splittings and estimates and start
with%
\[
\left\vert \left(
\kern-.1em%
\left(
\kern-.1em%
L_{\Omega}\mathbf{u},\mathbf{v}%
\kern-.1em%
\right)
\kern-.1em%
\right)  \right\vert \leq k^{2}\left\vert \left(  L_{\Omega}\mathbf{u}%
,\mathbf{v}\right)  \right\vert +\left\vert kb_{k}\left(  \left(  L_{\Omega
}\mathbf{u}\right)  ^{\nabla},\mathbf{v}^{\nabla}\right)  \right\vert
\leq\left(  k\left\Vert L_{\Omega}\mathbf{u}\right\Vert \right)  \left(
k\left\Vert \mathbf{v}\right\Vert \right)  +\left\vert kb_{k}%
^{\operatorname*{low}}\left(  \mathbf{u}^{\nabla},\mathbf{v}^{\nabla}\right)
\right\vert .
\]
We employ (\ref{blowest}) with $\rho={+}1$ to obtain%
\begin{align}
\left\vert b_{k}^{\operatorname*{low}}\left(  \mathbf{u}^{\nabla}%
,\mathbf{v}^{\nabla}\right)  \right\vert  &  \leq C_{b}k^{2}\left\Vert
\operatorname*{div}\nolimits_{\Gamma}L_{\Gamma}\mathbf{u}_{T}\right\Vert
_{H^{-3/2}\left(  \Gamma\right)  }\left\Vert \operatorname*{div}%
\nolimits_{\Gamma}L_{\Gamma}\mathbf{v}_{T}\right\Vert _{H^{-3/2}\left(
\Gamma\right)  }\nonumber\\
&  \leq Ck^{2}\left\Vert L_{\Gamma}\mathbf{u}_{T}\right\Vert _{\mathbf{H}%
^{-1/2}\left(  \Gamma\right)  }\left\Vert L_{\Gamma}\mathbf{v}_{T}\right\Vert
_{\mathbf{H}^{-1/2}\left(  \Gamma\right)  }%
\overset{(\ref{eq:LemLgammaPitu})}{\leq
}Ck^{2}\left\Vert \mathbf{u}\right\Vert _{\operatorname*{curl},\Omega
,1}\left\Vert \mathbf{v}\right\Vert _{\operatorname*{curl},\Omega,1}.
\label{estCkbklow2}%
\end{align}
Combining (\ref{estCkbk}) and (\ref{estCkbklow2}) leads to%
\begin{align*}
\left\vert kb_{k}\left(  \mathbf{u}^{\nabla},\mathbf{v}^{\nabla}\right)
\right\vert  &  \leq C\left\Vert \mathbf{u}\right\Vert _{\operatorname*{curl}%
,\Omega,k,\lambda}\left\Vert \mathbf{v}\right\Vert _{\operatorname*{curl}%
,\Omega,k,\lambda}+Ck^{3}\left\Vert \mathbf{u}\right\Vert
_{\operatorname*{curl},\Omega,1}\left\Vert \mathbf{v}\right\Vert
_{\operatorname*{curl},\Omega,1}\\
&  \leq Ck^{3}\left\Vert \mathbf{u}\right\Vert _{\operatorname*{curl}%
,\Omega,1}\left\Vert \mathbf{v}\right\Vert _{\operatorname*{curl},\Omega,1}.
\end{align*}
Taking into account the $L^{2}\left(  \Omega\right)  $ part in $\left(
\kern-.1em%
\left(
\kern-.1em%
\cdot,\cdot%
\kern-.1em%
\right)
\kern-.1em%
\right)  $ results in the estimate (\ref{lowfrequencyCSU}).%
\endproof

\begin{corollary}
\label{CorConstantsSphere}For $\Omega=B_{1}\left(  0\right)  $, the constants
in (\ref{DefCDtNk}), (\ref{Cbk}), (\ref{Cbkhightot}), (\ref{defCkLHOmega}),
and (\ref{DefCconthighk}) can be estimated by%
\begin{equation}
C_{\operatorname*{DtN},k}\leq Ck^{2},\quad C_{\operatorname*{cont},k}%
\leq\widetilde{C}_{b}k^{3},\quad C_{b,k}^{\nabla,\operatorname*{high}}%
\leq\widetilde{C}_{b},\quad C_{b,k}^{\operatorname*{curl},\operatorname*{high}%
}\leq C_{b}C_{\Gamma}^{2}\text{,\quad}C_{k}^{H,\Omega}\leq2,\quad
C_{b,k}^{\operatorname*{high}}\leq2+\widetilde{C}_{b}
\label{EstConstantsSphere}%
\end{equation}
with $k$-independent constants $C$, $C_{b}$ (cf.~Prop.~\ref{PropFrequbest}),
$\widetilde{C}_{b}$ (cf.~Prop.~\ref{PropStabCSU}), and $C_{\Gamma}$.
\end{corollary}

%

\proof
The estimate of $C_{\operatorname*{DtN},k}$ follows by combining
(\ref{blowest}) and (\ref{bkdaechleest}). Proposition~\ref{PropStabCSU}
implies the bound for $C_{\operatorname*{cont},k}$. Estimate (\ref{estCkbk})
implies the estimate of $C_{b,k}^{\nabla,\operatorname*{high}}$ as in
(\ref{defCbkhigh}). For $C_{b,k}^{\operatorname*{curl},\operatorname*{high}}$
we use (\ref{blowest}) to obtain%
\begin{align*}
k\left\vert b_{k}\left(  \mathbf{u}^{\operatorname*{curl}},\left(
\mathbf{v}^{\operatorname*{high}}\right)  ^{\operatorname*{curl}}\right)
\right\vert  &  =k\left\vert b_{k}^{\operatorname*{high}}\left(
\mathbf{u}^{\operatorname*{curl}},\mathbf{v}^{\operatorname*{curl}}\right)
\right\vert \leq C_{b}\left\Vert \operatorname*{curl}\nolimits_{\Gamma
}\mathbf{u}_{T}\right\Vert _{H^{-1/2}\left(  \Gamma\right)  }\left\Vert
\operatorname*{curl}\nolimits_{\Gamma}\mathbf{v}_{T}\right\Vert _{H^{-1/2}%
\left(  \Gamma\right)  }\\
&  \leq C_{b} C_{\Gamma}^{2} \|{\mathbf{u}}\|_{\operatorname{curl},\Omega,1}
\|{\mathbf{v}}\|_{\operatorname{curl},\Omega,1}%
\end{align*}
so that the estimate for $C_{b,k}^{\operatorname*{curl},\operatorname*{high}}$
is shown. Finally, $C_{k}^{H,\Omega}\leq2$ is proved in (\ref{lowhighuestb})
and the estimate of $C_{b}^{\operatorname*{high}}=C_{k}^{H,\Omega}%
+C_{b,k}^{\nabla,\operatorname*{high}}$ follows by combining the previous
estimates.%
\endproof

\section{Estimating the Terms in the Splitting (\ref{gammamainsplittot}b,c) of
$\left(
\kern-.1em%
\left(
\kern-.1em%
\mathbf{e}_{h},\mathbf{v}_{h}%
\kern-.1em%
\right)
\kern-.1em%
\right)  $\label{SecSplitting}}

\subsection{Estimating $\left(
\kern-.1em%
\left(
\kern-.1em%
\mathbf{e}_{h},\left(  \left(  \Pi_{h}^{\operatorname*{comp},\ast}%
-\Pi^{\operatorname*{comp},\ast}\right)  \mathbf{v}_{h}\right)
^{\operatorname*{high}}%
\kern-.1em%
\right)
\kern-.1em%
\right)  $ in (\ref{gammamainsplittot}b,c)\label{SecHOmega}}

In this section, we will prove the following Proposition~\ref{PropHOmegaSplit}%
. Recall the definition of $\tilde{\eta}_{4}^{\exp}$, $\eta_{6}%
^{\operatorname{alg}}$, $\tilde{\eta}_{7}^{\exp}$ in (\ref{DefEta5New}),
(\ref{PiEhc}), (\ref{Defeta7}), which involve the operator $\Pi_{h}^{E}$ as in
Assumption \ref{AdiscSp}.

\begin{proposition}
\label{PropHOmegaSplit}Let $\mathbf{e}_{h}=\mathbf{E}-\mathbf{E}_{h}$ denote
the Galerkin error and for $\mathbf{v}_{h}\in\mathbf{X}_{h}$ let $\Pi
_{h}^{\operatorname*{comp},\ast}$, $\Pi^{\operatorname*{comp},\ast}$ be
defined as in Definition \ref{DefHelmDecomp}. Let Assumption \ref{AdiscSp} be
satisfied. Then%
\begin{equation}
\left\vert \left(
\kern-.1em%
\left(
\kern-.1em%
\mathbf{e}_{h},\left(  \left(  \Pi^{\operatorname*{comp},\ast}-\Pi
_{h}^{\operatorname*{comp},\ast}\right)  \mathbf{v}_{h}\right)
^{\operatorname*{high}}%
\kern-.1em%
\right)
\kern-.1em%
\right)  \right\vert \leq C_{b,k}^{\operatorname*{high}}C_{r,k}\left\Vert
\mathbf{e}_{h}\right\Vert _{\operatorname*{curl},\Omega,k}\left\Vert
\mathbf{v}_{h}\right\Vert _{\operatorname*{curl},\Omega,k}.
\end{equation}
with%
\begin{equation}
C_{r,k}:=\left(  C_{b,k}^{\operatorname*{high}}+\frac{C_{\operatorname*{cont}%
,k}}{k^{2}}\tilde{\eta}_{4}^{\exp}\right)  \left(  \tilde{\eta}_{7}^{\exp
}+C_{\#,k}\eta_{6}^{\operatorname{alg}}\right)  \quad\text{and\quad}%
C_{\#,k}:=\left(  C_{k}^{H,\Omega}+C_{b,k}^{\nabla,\operatorname*{high}%
}\right)  C_{\Omega,k}. \label{defCrk}%
\end{equation}
The constant $C_{b,k}^{\operatorname*{high}}$ is as in (\ref{DefCconthighk}),
$C_{\operatorname*{cont},k}$ as in (\ref{Cbk}), and $C_{\Omega,k}$ as in
(\ref{DefCOmegak}).

For the case $\Omega=B_{1}\left(  0\right)  $ we have $C_{\operatorname*{cont}%
,k}\leq Ck^{3}$ while $C_{b,k}^{\operatorname{high}}$, $C_{r,k}$,
$\Omega_{\Omega,k}$ and $C_{\#,k}$ are bounded independently of $k$.
\end{proposition}

%

\proof
From (\ref{decoecheck1}) we conclude
\begin{equation}
\left.
\begin{array}
[c]{ll}
& \operatorname*{curl}\Pi_{h}^{\operatorname*{comp},\ast}\mathbf{v}%
_{h}=\operatorname*{curl}\Pi^{\operatorname*{comp},\ast}\mathbf{v}%
_{h}=\operatorname*{curl}\mathbf{v}_{h}\\
\text{and } & \operatorname*{curl}\Pi_{h}^{\operatorname*{curl},\ast}%
H_{\Omega}\mathbf{v}_{h}=\operatorname*{curl}\Pi^{\operatorname*{curl},\ast
}H_{\Omega}\mathbf{v}_{h}=\operatorname*{curl}H_{\Omega}\mathbf{v}_{h}%
\end{array}
\right\}  \quad\forall\mathbf{v}_{h}\in\mathbf{X}_{h}. \label{curlpropHcomp}%
\end{equation}
Let $\mathbf{r:=}\left(  \Pi^{\operatorname*{comp},\ast}-\Pi_{h}%
^{\operatorname*{comp},\ast}\right)  \mathbf{v}_{h}$ and let $\mathbf{q:=}%
\left(  I-\Pi_{h}^{E}\right)  \Pi^{\operatorname*{comp},\ast}\mathbf{v}_{h}$.
First we prove some curl-free properties. It holds%
\begin{align}
\operatorname*{curl}\left(  \Pi_{h}^{E}\Pi^{\operatorname*{comp},\ast}-\Pi
_{h}^{\operatorname*{comp},\ast}\right)  \mathbf{v}_{h}  &
\overset{\text{(\ref{decoecheck1}), Ass.~\ref{AdiscSp}}}{=}%
\operatorname*{curl}\left(  \Pi_{h}^{E}-I\right)  L_{\Omega}\mathbf{v}%
_{h}+\left(  \Pi_{h}^{F}\operatorname*{curl}\Pi^{\operatorname*{curl},\ast
}-\operatorname*{curl}\Pi_{h}^{\operatorname*{curl},\ast}\right)  H_{\Omega
}\mathbf{v}_{h}\label{curldiffzero}\\
&  =\operatorname*{curl}\left(  \Pi_{h}^{E}-I\right)  L_{\Omega}\mathbf{v}%
_{h}+\Pi_{h}^{F}\operatorname*{curl}H_{\Omega}\mathbf{v}_{h}%
-\operatorname*{curl}H_{\Omega}\mathbf{v}_{h}\nonumber\\
&  =\operatorname*{curl}\left(  \Pi_{h}^{E}-I\right)  L_{\Omega}\mathbf{v}%
_{h}+\operatorname*{curl}\left(  \Pi_{h}^{E}-I\right)  H_{\Omega}%
\mathbf{v}_{h}=\operatorname*{curl}\left(  \Pi_{h}^{E}-I\right)
\mathbf{v}_{h}\nonumber\\
&  \overset{\text{Ass.~\ref{AdiscSp}(a)}}{=}\operatorname*{curl}\left(
\mathbf{v}_{h}-\mathbf{v}_{h}\right)  =0,\nonumber
\end{align}
and also
\begin{subequations}
\label{curlqzero}
\begin{align}
&  \operatorname*{curl}\mathbf{r}\overset{\text{(\ref{curlpropHcomp})}%
}{=}0,\label{curlqzero-a}\\
&  \operatorname*{curl}\mathbf{q=}\operatorname*{curl}\left(  \Pi
^{\operatorname*{comp},\ast}-\Pi_{h}^{E}\Pi^{\operatorname*{comp},\ast
}\right)  \mathbf{v}_{h}\overset{\text{(\ref{curlpropHcomp})}}{=}%
\operatorname*{curl}\left(  \Pi_{h}^{\operatorname*{comp},\ast}-\Pi_{h}^{E}%
\Pi^{\operatorname*{comp},\ast}\right)  \mathbf{v}_{h}%
\overset{\text{(\ref{curldiffzero})}}{=}0. \label{curlqzero-b}%
\end{align}
\end{subequations}

We start our estimate with a continuity bound for the sesquilinear form
$\left(
\kern-.1em%
\left(
\kern-.1em%
\cdot,H_{\Omega}\cdot%
\kern-.1em%
\right)
\kern-.1em%
\right)  $ and employ (\ref{DefCconthighk}) to get%
\begin{equation}
\left\vert \left(
\kern-.1em%
\left(
\kern-.1em%
\mathbf{e}_{h},\mathbf{r}^{\operatorname*{high}}%
\kern-.1em%
\right)
\kern-.1em%
\right)  \right\vert \leq C_{b,k}^{\operatorname*{high}}\left\Vert
\mathbf{e}_{h}\right\Vert _{\operatorname*{curl},\Omega,k}\left\Vert
\mathbf{r}\right\Vert _{\operatorname*{curl},\Omega,k}%
\overset{\text{(\ref{curlpropHcomp})}}{=}C_{b,k}^{\operatorname*{high}%
}\left\Vert \mathbf{e}_{h}\right\Vert _{\operatorname*{curl},\Omega,k}\left(
k\left\Vert \mathbf{r}\right\Vert \right)  . \label{PropStokes1}%
\end{equation}
The coercivity of $\left(
\kern-.1em%
\left(
\kern-.1em%
\cdot,\cdot%
\kern-.1em%
\right)
\kern-.1em%
\right)  $ in the form (\ref{k+normdblescest}) leads to%
\begin{equation}
\left(  k\left\Vert \mathbf{r}\right\Vert \right)  ^{2}\leq\operatorname{Re}%
\left(
\kern-.1em%
\left(
\kern-.1em%
\mathbf{r},\mathbf{r}%
\kern-.1em%
\right)
\kern-.1em%
\right)  =\operatorname{Re}\left(
\kern-.1em%
\left(
\kern-.1em%
\mathbf{q},\mathbf{r}%
\kern-.1em%
\right)
\kern-.1em%
\right)  +\operatorname{Re}\left(
\kern-.1em%
\left(
\kern-.1em%
\left(  \Pi_{h}^{E}\Pi^{\operatorname*{comp},\ast}-\Pi_{h}%
^{\operatorname*{comp},\ast}\right)  \mathbf{v}_{h},\mathbf{r}%
\kern-.1em%
\right)
\kern-.1em%
\right)  . \label{kPiPihest}%
\end{equation}
We use the definition of $\Pi^{\nabla,\ast}$, $\Pi^{\operatorname*{curl},\ast
}$, $\Pi^{\operatorname*{comp},\ast}$ and its discrete versions as in
(\ref{vargammaoldadj}) and Definition \ref{DefHelmDecomp} to get%
\begin{equation}
\left(
\kern-.1em%
\left(
\kern-.1em%
\mathbf{w}_{h},\mathbf{r}%
\kern-.1em%
\right)
\kern-.1em%
\right)  =\left(
\kern-.1em%
\left(
\kern-.1em%
\mathbf{w}_{h},\left(  \Pi^{\operatorname*{curl},\ast}-\Pi_{h}%
^{\operatorname*{curl},\ast}\right)  H_{\Omega}\mathbf{v}_{h}%
\kern-.1em%
\right)
\kern-.1em%
\right)  =0\quad\forall\mathbf{w}_{h}\in\nabla S_{h}. \label{2ndscprodGalOrth}%
\end{equation}
From (\ref{curldiffzero}) and the exact sequence property (\ref{esp}) we
conclude that $\left(  \Pi_{h}^{E}\Pi^{\operatorname*{comp},\ast}-\Pi
_{h}^{\operatorname*{comp},\ast}\right)  \mathbf{v}_{h}=\nabla\psi_{h}$ for
some $\psi_{h}\in S_{h}$. The combination of this with (\ref{2ndscprodGalOrth}%
) for $\mathbf{w}_{h}=\nabla\psi_{h}$ implies that the last term in
(\ref{kPiPihest}) vanishes. Hence,%
\begin{equation}
\left(  k\left\Vert \mathbf{r}\right\Vert \right)  ^{2}\leq\operatorname{Re}%
\left(
\kern-.1em%
\left(
\kern-.1em%
\mathbf{q},\mathbf{r}%
\kern-.1em%
\right)
\kern-.1em%
\right)  =\operatorname{Re}\left(
\kern-.1em%
\left(
\kern-.1em%
H_{\Omega}\mathbf{q},\mathbf{r}%
\kern-.1em%
\right)
\kern-.1em%
\right)  +\operatorname{Re}\left(
\kern-.1em%
\left(
\kern-.1em%
L_{\Omega}\mathbf{q},\mathbf{r}%
\kern-.1em%
\right)
\kern-.1em%
\right)  . \label{hfp_low}%
\end{equation}
For the high-frequency part on the right-hand side we employ again
(\ref{DefCconthighk}) and obtain%
\begin{equation}
\operatorname{Re}\left(
\kern-.1em%
\left(
\kern-.1em%
H_{\Omega}\mathbf{q},\mathbf{r}%
\kern-.1em%
\right)
\kern-.1em%
\right)  \leq C_{b,k}^{\operatorname*{high}}\left\Vert \mathbf{q}\right\Vert
_{\operatorname*{curl},\Omega,k}\left\Vert \mathbf{r}\right\Vert
_{\operatorname*{curl},\Omega,k}\overset{\text{(\ref{curlqzero})}}{=}%
C_{b,k}^{\operatorname*{high}}\left(  k\left\Vert \mathbf{q}\right\Vert
\right)  \left(  k\left\Vert \mathbf{r}\right\Vert \right)  .
\label{PropStokes2}%
\end{equation}
The term $\left\Vert \mathbf{q}\right\Vert $ can be estimated by using the
definition of $\Pi^{\operatorname*{comp},\ast}$ as in Definition
\ref{DefHelmDecomp}%
\begin{align}
k\left\Vert \mathbf{q}\right\Vert  &  \leq k\left\Vert \left(  I-\Pi_{h}%
^{E}\right)  L_{\Omega}\mathbf{v}_{h}\right\Vert +k\left\Vert \left(
I-\Pi_{h}^{E}\right)  \Pi^{\operatorname*{curl},\ast}H_{\Omega}\mathbf{v}%
_{h}\right\Vert \nonumber\\
&  \leq\tilde{\eta}_{7}^{\exp}\left\Vert \mathbf{v}_{h}\right\Vert
_{\operatorname*{curl},\Omega,k}+\eta_{6}^{\operatorname{alg}}\left\Vert
\Pi^{\operatorname*{curl},\ast}H_{\Omega}\mathbf{v}_{h}\right\Vert
_{\mathbf{H}^{1}\left(  \Omega\right)  }\nonumber\\
&  \overset{\text{Lem. \ref{Lemembed}}}{\leq}\tilde{\eta}_{7}^{\exp}\left\Vert
\mathbf{v}_{h}\right\Vert _{\operatorname*{curl},\Omega,k}+C_{\Omega,k}%
\eta_{6}^{\operatorname{alg}}\left\Vert \Pi^{\operatorname*{curl},\ast
}H_{\Omega}\mathbf{v}_{h}\right\Vert _{\operatorname*{curl},\Omega
,1}\nonumber\\
&  \overset{\text{Lem. \ref{LemNablaCurl}}}{\leq}\left(  \tilde{\eta}%
_{7}^{\exp}+C_{\#,k}\eta_{6}^{\operatorname{alg}}\right)  \left\Vert
\mathbf{v}_{h}\right\Vert _{\operatorname*{curl},\Omega,k}.
\label{PropStokes3}%
\end{align}

To estimate the low frequency part in (\ref{hfp_low}) we observe that $%
\mbox{\boldmath$ \zeta$}%
:=\Pi^{\nabla}L_{\Omega}\mathbf{q}=\nabla\mathcal{N}_{4}^{\mathcal{A}%
}\mathbf{q}$ (cf. (\ref{graddoppelKlammer})) satisfies%
\[
\left(
\kern-.1em%
\left(
\kern-.1em%
\mbox{\boldmath$ \zeta$}%
,%
\mbox{\boldmath$ \xi$}%
\kern-.1em%
\right)
\kern-.1em%
\right)  =\left(
\kern-.1em%
\left(
\kern-.1em%
L_{\Omega}\mathbf{q},%
\mbox{\boldmath$ \xi$}%
\kern-.1em%
\right)
\kern-.1em%
\right)  \quad\forall%
\mbox{\boldmath$ \xi$}%
\in\nabla H^{1}\left(  \Omega\right)  .
\]
By choosing $%
\mbox{\boldmath$ \xi$}%
=\mathbf{r}$ we can use a Galerkin orthogonality in the form
(\ref{2ndscprodGalOrth}) to obtain for any $\mathbf{w}_{h}\in\nabla S_{h}$%
\[
\operatorname{Re}\left(
\kern-.1em%
\left(
\kern-.1em%
L_{\Omega}\mathbf{q},\mathbf{r}%
\kern-.1em%
\right)
\kern-.1em%
\right)  =\operatorname{Re}\left(
\kern-.1em%
\left(
\kern-.1em%
\mbox{\boldmath$ \zeta$}%
,\mathbf{r}%
\kern-.1em%
\right)
\kern-.1em%
\right)  =\operatorname{Re}\left(
\kern-.1em%
\left(
\kern-.1em%
\mbox{\boldmath$ \zeta$}%
-\mathbf{w}_{h},\mathbf{r}%
\kern-.1em%
\right)
\kern-.1em%
\right)  \leq C_{\operatorname*{cont},k}\left\Vert \mathbf{r}\right\Vert
_{\operatorname*{curl},\Omega,1}\left\Vert
\mbox{\boldmath$ \zeta$}%
-\mathbf{w}_{h}\right\Vert _{\operatorname*{curl},\Omega,1}.
\]
The last factor can be estimated by using (\ref{DefEta5New}),
(\ref{PropStokes3}), and the definition of $%
\mbox{\boldmath$ \zeta$}%
$:%

\begin{align}
\inf_{v_{h}\in S_{h}}\left\Vert \nabla\left(  \mathcal{N}_{4}^{\mathcal{A}%
}\mathbf{q}-v_{h}\right)  \right\Vert _{\operatorname*{curl},\Omega,1}  &
=\inf_{v_{h}\in S_{h}}\left\Vert \nabla\left(  \mathcal{N}_{4}^{\mathcal{A}%
}\mathbf{q}-v_{h}\right)  \right\Vert \leq\tilde{\eta}_{4}^{\exp}\left\Vert
\mathbf{q}\right\Vert _{\operatorname*{curl},\Omega,1}%
\overset{\text{(\ref{curlqzero})}}{=}\tilde{\eta}_{4}^{\exp}\left\Vert
\mathbf{q}\right\Vert \nonumber\\
&  \overset{\text{(\ref{PropStokes3})}}{\leq}\frac{\tilde{\eta}_{4}^{\exp}}%
{k}\left(  \tilde{\eta}_{7}^{\exp}+C_{\#,k}\eta_{6}^{\operatorname{alg}%
}\right)  \left\Vert \mathbf{v}_{h}\right\Vert _{\operatorname*{curl}%
,\Omega,k}.
\end{align}
Finally, we combine this estimate with (\ref{hfp_low}), (\ref{PropStokes2}),
(\ref{PropStokes3}) to bound the last factor in (\ref{PropStokes1})%
\begin{equation}
k\left\Vert \mathbf{r}\right\Vert \leq C_{r,k}\left\Vert \mathbf{v}%
_{h}\right\Vert _{\operatorname*{curl},\Omega,k}. \label{kPiCompPiComph}%
\end{equation}
We insert (\ref{kPiCompPiComph}) into (\ref{PropStokes1}) and arrive at the assertion.

The bounds for the constants are stated in Corollary \ref{CorConstantsSphere}.%
\endproof

\subsection{Estimate of $\left(
\kern-.1em%
\left(
\kern-.1em%
\mathbf{e}_{h},\Pi^{\operatorname*{curl},\ast}\mathbf{v}_{h}%
^{\operatorname*{high}}%
\kern-.1em%
\right)
\kern-.1em%
\right)  $ in (\ref{gammamainsplittot}b,c)\label{Seceovo}}

In this section, we investigate the second term of the right-hand side in
(\ref{gammamainsplit2}). Recall the definition of the adjoint solution
operators (\ref{smoothdualproblem}) and the corresponding adjoint
approximation properties (\ref{defetatildealg})--(\ref{Defeta7}).

\begin{proposition}
\label{PropehPicurlvhhigh}Let $\mathbf{e}_{h}=\mathbf{E}-\mathbf{E}_{h}$
denote the Galerkin error with splitting of ${\mathbf{v}}_{h} \in{\mathbf{X}%
}_{h}$ as in (\ref{decoecheck1}). Let Assumption~\ref{AdiscSp} be satisfied.
Then%
\begin{equation}
\left\vert \left(
\kern-.1em%
\left(
\kern-.1em%
\mathbf{e}_{h},\Pi^{\operatorname*{curl},\ast}\mathbf{v}_{h}%
^{\operatorname*{high}}%
\kern-.1em%
\right)
\kern-.1em%
\right)  \right\vert \leq C_{\#\#,k}\left(  C_{\#\#,k}+C_{b,k}%
^{\operatorname*{curl},\operatorname*{high}}+C_{\operatorname*{cont},k}%
\tilde{\eta}_{5}^{\exp}\right)  \tilde{\eta}_{2}^{\operatorname{alg}%
}\left\Vert \mathbf{e}_{h}\right\Vert _{\operatorname*{curl},\Omega
,k}\left\Vert \mathbf{v}_{h}\right\Vert _{\operatorname*{curl},\Omega,k}
\label{est2ndtermsplit}%
\end{equation}
with $C_{\#\#,k}:= C_{k}^{H,\Omega}+C_{b,k}^{\operatorname*{high}}$. For
$\Omega=B_{1}\left(  0\right)  $, it holds $C_{\operatorname*{cont},k}\leq
Ck^{3}$ while all other constants are bounded independently of $k$.
\end{proposition}

%

\proof
Let $\mathbf{s}:=\Pi^{\operatorname*{curl},\ast}\mathbf{v}_{h}%
^{\operatorname*{high}}\in\mathbf{V}_{0}^{\ast}$. We consider the adjoint
problem (cf. (\ref{adjoint3b})) with solution operator $\mathcal{N}_{2}\ $and
set $\mathbf{z}:=\mathcal{N}_{2}\mathbf{s}$. Galerkin orthogonality with
arbitrary $\mathbf{z}_{h}\in\mathbf{X}_{h}$ gives
\begin{equation}
\left(
\kern-.1em%
\left(
\kern-.1em%
\mathbf{e}_{h},\mathbf{s}%
\kern-.1em%
\right)
\kern-.1em%
\right)  =A_{k}\left(  \mathbf{e}_{h},\mathbf{z}\right)  =A_{k}\left(
\mathbf{e}_{h},\mathbf{z}-\mathbf{z}_{h}\right)  =A_{k}\left(  \mathbf{e}%
_{h},H_{\Omega}\left(  \mathbf{z}-\mathbf{z}_{h}\right)  \right)
+A_{k}\left(  \mathbf{e}_{h},L_{\Omega}\left(  \mathbf{z}-\mathbf{z}%
_{h}\right)  \right)  . \label{cont2}%
\end{equation}
For the first term we obtain%
\[
\left\vert A_{k}\left(  \mathbf{e}_{h},H_{\Omega}\left(  \mathbf{z}%
-\mathbf{z}_{h}\right)  \right)  \right\vert \leq\left\Vert
\operatorname*{curl}\mathbf{e}_{h}\right\Vert \left\Vert \operatorname*{curl}%
\left(  H_{\Omega}\left(  \mathbf{z-z}_{h}\right)  \right)  \right\Vert
+\left\vert \left(
\kern-.1em%
\left(
\kern-.1em%
\mathbf{e}_{h},H_{\Omega}\left(  \mathbf{z}-\mathbf{z}_{h}\right)
\kern-.1em%
\right)
\kern-.1em%
\right)  \right\vert {+}\left\vert {kb_{k}}\left(  {\mathbf{e}_{h}%
^{\operatorname*{curl}},}\left(  {H_{\Omega}}\left(  {\mathbf{z}%
-\mathbf{z}_{h}}\right)  \right)  ^{\operatorname*{curl}}\right)  \right\vert
{.}%
\]
The three terms on the right-hand side can be estimated by using the constants
in (\ref{Cbkhightot}), (\ref{defCkLHOmega}), (\ref{DefCconthighk}):
\begin{align*}
\left\Vert \operatorname*{curl}\left(  H_{\Omega}(\mathbf{z-z}_{h})\right)
\right\Vert  &  \leq\left\Vert H_{\Omega}\left(  \mathbf{z-z}_{h}\right)
\right\Vert _{\operatorname*{curl},\Omega,k}\leq C_{k}^{H,\Omega}\left\Vert
\mathbf{z-z}_{h}\right\Vert _{\operatorname*{curl},\Omega,k},\\
\left\vert {kb_{k}}\left(  {\mathbf{e}_{h}^{\operatorname*{curl}},}\left(
{H_{\Omega}}\left(  {\mathbf{z}-\mathbf{z}_{h}}\right)  \right)
^{\operatorname*{curl}}\right)  \right\vert  &  \leq C_{b,k}%
^{\operatorname*{curl},\operatorname*{high}}\left\Vert {\mathbf{e}_{h}%
}\right\Vert _{\operatorname*{curl},\Omega,k}\left\Vert \mathbf{z-z}%
_{h}\right\Vert _{\operatorname*{curl},\Omega,k},\\
\left\vert \left(
\kern-.1em%
\left(
\kern-.1em%
\mathbf{e}_{h},H_{\Omega}\left(  \mathbf{z}-\mathbf{z}_{h}\right)
\kern-.1em%
\right)
\kern-.1em%
\right)  \right\vert  &  \leq C_{b,k}^{\operatorname*{high}}\left\Vert
{\mathbf{e}_{h}}\right\Vert _{\operatorname*{curl},\Omega,k}\left\Vert
\mathbf{z-z}_{h}\right\Vert _{\operatorname*{curl},\Omega,k}.
\end{align*}
This leads to%
\[
\left\vert A_{k}\left(  \mathbf{e}_{h},H_{\Omega}\left(  \mathbf{z}%
-\mathbf{z}_{h}\right)  \right)  \right\vert \leq\left(  C_{\#\#,k}%
+C_{b,k}^{\operatorname*{curl},\operatorname*{high}}\right)  \left\Vert
{\mathbf{e}_{h}}\right\Vert _{\operatorname*{curl},\Omega,k}\left\Vert
\mathbf{z-z}_{h}\right\Vert _{\operatorname*{curl},\Omega,k}.
\]
For the second term in (\ref{cont2}) we obtain for arbitrary $\mathbf{\tilde
{z}}_{h}\in\mathbf{X}_{h}$
\begin{align}
\left\vert A_{k}\left(  \mathbf{e}_{h},L_{\Omega}\left(  \mathbf{z}%
-\mathbf{z}_{h}\right)  \right)  \right\vert  &  \leq\left\vert A_{k}\left(
\mathbf{e}_{h},L_{\Omega}\left(  \mathbf{z}-\mathbf{z}_{h}\right)
\mathbf{-\tilde{z}}_{h}\right)  \right\vert \nonumber\\
&  \overset{\text{(\ref{Cbk})}}{\leq}C_{\operatorname*{cont},k}\left\Vert
\mathbf{e}_{h}\right\Vert _{\operatorname*{curl},\Omega,1}\left\Vert
L_{\Omega}\left(  \mathbf{z}-\mathbf{z}_{h}\right)  -\mathbf{\tilde{z}}%
_{h}\right\Vert _{\operatorname*{curl},\Omega,1}. \label{3rdterm}%
\end{align}
This leads to the estimate%
\begin{equation}
\left\vert \left(
\kern-.1em%
\left(
\kern-.1em%
\mathbf{e}_{h},\mathbf{s}%
\kern-.1em%
\right)
\kern-.1em%
\right)  \right\vert \leq\left(  C_{\#\#,k}+C_{b,k}^{\operatorname*{curl}%
,\operatorname*{high}}\right)  \left\Vert \mathbf{e}_{h}\right\Vert
_{\operatorname*{curl},\Omega,k}\left\Vert \mathbf{z}-\mathbf{z}%
_{h}\right\Vert _{\operatorname*{curl},\Omega,k}+C_{\operatorname*{cont}%
,k}\left\Vert \mathbf{e}_{h}\right\Vert _{\operatorname*{curl},\Omega
,1}\left\Vert L_{\Omega}\left(  \mathbf{z}-\mathbf{z}_{h}\right)
-\mathbf{\tilde{z}}_{h}\right\Vert _{\operatorname*{curl},\Omega,1}.
\label{estDKe0hv0h}%
\end{equation}
With the definition of the adjoint approximation properties
(cf.~Sec.~\ref{SecAdjProbl}) we arrive at
\begin{align}
\inf_{{\mathbf{z}}_{h}\in{\mathbf{X}}_{h}}\left\Vert \mathbf{z}-\mathbf{z}%
_{h}\right\Vert _{\operatorname*{curl},\Omega,k}  &
\overset{\text{(\ref{defetatildealg})}}{\leq}\tilde{\eta}_{2}%
^{\operatorname{alg}}\left\Vert \Pi^{\operatorname*{curl},\ast}\mathbf{v}%
_{h}^{\operatorname*{high}}\right\Vert _{\operatorname*{curl},\Omega
,k},\label{zzheta1}\\
\inf_{\mathbf{z}_{h}}\inf_{\mathbf{\tilde{z}}_{h}}\left\Vert L_{\Omega}\left(
\mathbf{z}-\mathbf{z}_{h}\right)  -\mathbf{\tilde{z}}_{h}\right\Vert
_{\operatorname*{curl},\Omega,k}  &  \overset{\text{(\ref{Defeta6})}}{\leq
}\tilde{\eta}_{5}^{\exp}\inf_{\mathbf{z}_{h}}\left\Vert \mathbf{z}%
-\mathbf{z}_{h}\right\Vert _{\operatorname*{curl},\Omega,k}%
\overset{\text{(\ref{zzheta1})}}{\leq}\tilde{\eta}_{2}^{\operatorname{alg}%
}\tilde{\eta}_{5}^{\exp}\left\Vert \Pi^{\operatorname*{curl},\ast}%
\mathbf{v}_{h}^{\operatorname*{high}}\right\Vert _{\operatorname*{curl}%
,\Omega,k}.
\end{align}
The combination of these estimates with Lemma~\ref{LemNablaCurl} leads to
(\ref{est2ndtermsplit}).

The estimates of the constants for the case $\Omega=B_{1}\left(  0\right)  $
are stated in Corollary \ref{CorConstantsSphere}.%
\endproof

\subsection{Estimating $\left(
\kern-.1em%
\left(
\kern-.1em%
\mathbf{e}_{h},L_{\Omega}\left(  \Pi_{h}^{\operatorname*{comp},\ast}%
\mathbf{v}_{h}-\Pi^{\operatorname*{comp},\ast}\mathbf{v}_{h}\right)
\kern-.1em%
\right)
\kern-.1em%
\right)  $ and $\left(
\kern-.1em%
\left(
\kern-.1em%
\mathbf{e}_{h},L_{\Omega}\mathbf{v}_{h}%
\kern-.1em%
\right)
\kern-.1em%
\right)  $ in (\ref{gammamainsplittot}b,c)\label{SecLterms}}

Next, we investigate the first and last term of the right-hand side in
(\ref{gammamainsplit2}).

\begin{proposition}
\label{PropLowerOrderTerm}Let $\mathbf{e}_{h}=\mathbf{E}-\mathbf{E}_{h}$
denote the Galerkin error with splitting of ${\mathbf{v}}_{h} \in{\mathbf{X}%
}_{h}$ as in (\ref{decoecheck1}) and let Assumption~\ref{AdiscSp} be
satisfied. Then:
\begin{equation}
\left\vert \left(
\kern-.1em%
\left(
\kern-.1em%
\mathbf{e}_{h},L_{\Omega}\mathbf{r}%
\kern-.1em%
\right)
\kern-.1em%
\right)  \right\vert +\left\vert \left(
\kern-.1em%
\left(
\kern-.1em%
\mathbf{e}_{h},L_{\Omega}\mathbf{v}_{h}%
\kern-.1em%
\right)
\kern-.1em%
\right)  \right\vert \leq C_{\operatorname*{cont},k}\tilde{\eta}_{3}^{\exp
}\left(  1+C_{r,k}\right)  \left\Vert \mathbf{e}_{h}\right\Vert
_{\operatorname*{curl},\Omega,k}\left\Vert \mathbf{v}_{h}\right\Vert
_{\operatorname*{curl},\Omega,k} \label{PropLowerOrderTermest}%
\end{equation}
with $\mathbf{r}:=\Pi_{h}^{\operatorname*{comp},\ast}\mathbf{v}_{h}%
-\Pi^{\operatorname*{comp},\ast}\mathbf{v}_{h}$ and $C_{r,k}$ as in
(\ref{defCrk}).
\end{proposition}

%

\proof
Recall the definition of the solution operator $\mathcal{N}_{3}^{\mathcal{A}}$
from (\ref{smoothdualproblemd}) satisfying for given $\mathbf{s}\in\mathbf{X}$%
\[%
\begin{array}
[c]{cc}%
A_{k}\left(  \mathbf{w},\mathcal{N}_{3}^{\mathcal{A}}\mathbf{s}\right)
=\left(
\kern-.1em%
\left(
\kern-.1em%
\mathbf{w},L_{\Omega}\mathbf{s}%
\kern-.1em%
\right)
\kern-.1em%
\right)  & \forall\mathbf{w}\in\mathbf{X}.
\end{array}
\]

For the first term in (\ref{PropLowerOrderTermest}) we get in a similar
fashion as in (\ref{3rdterm})
\begin{align*}
\left\vert \left(
\kern-.1em%
\left(
\kern-.1em%
\mathbf{e}_{h},L_{\Omega}\mathbf{s}%
\kern-.1em%
\right)
\kern-.1em%
\right)  \right\vert  &  = \inf_{{\mathbf{z}}_{h} \in{\mathbf{X}}_{h}}
\left\vert A_{k}\left(  \mathbf{e}_{h},\mathcal{N}_{3}^{\mathcal{A}}%
\mathbf{s}-\mathbf{z}_{h}\right)  \right\vert \overset{\text{(\ref{Cbk}%
)}}{\leq} C_{\operatorname*{cont},k}\left\Vert \mathbf{e}_{h}\right\Vert
_{\operatorname*{curl},\Omega,1} \inf_{{\mathbf{z}}_{h} \in{\mathbf{X}}_{h}}
\left\Vert \mathcal{N}_{3}^{\mathcal{A}}\mathbf{s}-\mathbf{z}_{h}\right\Vert
_{\operatorname*{curl},\Omega,1}\\
&  \overset{\text{(\ref{defetatilde3exp})}}{\leq} C_{\operatorname*{cont}%
,k}\tilde{\eta}_{3}^{\exp}\left\Vert \mathbf{e}_{h}\right\Vert
_{\operatorname*{curl},\Omega,k}\left\Vert \mathbf{s}\right\Vert
_{\operatorname*{curl},\Omega,k}.
\end{align*}
This leads directly to the estimate of the second term in
(\ref{PropLowerOrderTermest}) by choosing $\mathbf{s}=\mathbf{v}_{h}$. For the
choice $\mathbf{s}=\mathbf{r}$, we combine (\ref{curlqzero-a}) with
(\ref{kPiCompPiComph}) to get$\displaystyle
\left\Vert \mathbf{r}\right\Vert _{\operatorname*{curl},\Omega,k}=k\left\Vert
\mathbf{r}\right\Vert \leq C_{r,k}\left\Vert \mathbf{v}_{h}\right\Vert
_{\operatorname*{curl},\Omega,k}. $%
\endproof


\section{Analysis of the Dual Problems}

\label{sec:dual-problems}

For the stability and convergence analysis, we have introduced various adjoint
approximation properties in Sec.~\ref{SecAdjProbl}. In this section, we
analyze the regularity of the adjoint solutions in Sec.~\ref{SecRegDualSol}
based on a solution formula which we will derive in Sec.~\ref{SecSolForm}. The
quantitative convergence rates require interpolation operators for $hp$ finite
element spaces that will be presented in Sections~\ref{sec:hp-approximation}.

\subsection{Solution Formulae\label{SecSolForm}}

In this section, we will develop a regularity theory to estimate the solutions
of the dual problems which have been introduced in Section \ref{SecAdjProbl}.
They belong to one of the following two types.

\textbf{Type 1:}%
\begin{align}
\text{Given }\mathbf{v}  &  \in\mathbf{H}\left(  \Omega,\operatorname*{div}%
\right)  ,\quad\mathbf{g},\mathbf{h}\in\mathbf{X\quad}\text{find }%
\mathbf{z}\in\mathbf{X}\text{ s.t.}\nonumber\\
A_{k}\left(  \mathbf{w},\mathbf{z}\right)   &  =k^{2}\left(  \mathbf{w}%
,\mathbf{v}\right)  +\operatorname*{i}kb_{k}\left(  \mathbf{w}^{\nabla
},\mathbf{g}^{\nabla}\right)  -\operatorname*{i}kb_{k}\left(  \mathbf{w}%
^{\operatorname*{curl}},\mathbf{h}^{\operatorname*{curl}}\right)  \quad
\forall\mathbf{w}\in\mathbf{X}. \label{dualproto}%
\end{align}
This is problem (\ref{adjoint3b}) with $\mathbf{v:}=\mathbf{g:=r}$ and
$\mathbf{h:}=\mathbf{0}$, problem (\ref{smoothdualproblemd}) with
$\mathbf{v}:=\mathbf{g}:=L_{\Omega}\mathbf{r}$ and $\mathbf{h}:=\mathbf{0}$,
and problem (\ref{adjproblm0}) with $\mathbf{v}=\mathbf{h}=\mathbf{g}%
:=L_{\Omega}\mathbf{w}${.}

\textbf{Type 2:}%
\begin{equation}
\text{Given }\mathbf{r}\in\mathbf{X}\text{\quad find }Z\in H^{1}\left(
\Omega\right)  /\mathbb{R}\text{ s.t.\quad}\left(
\kern-.1em%
\left(
\kern-.1em%
\nabla Z,\nabla\xi%
\kern-.1em%
\right)
\kern-.1em%
\right)  =\left(
\kern-.1em%
\left(
\kern-.1em%
L_{\Omega}\mathbf{r},\nabla\xi%
\kern-.1em%
\right)
\kern-.1em%
\right)  \qquad\forall\xi\in H^{1}\left(  \Omega\right)  . \label{type3}%
\end{equation}
This is problem (\ref{graddoppelKlammer}).

\subsubsection{Solution Formula for Problems of Type 1}

Integration by parts in the sesquilinear form $A_{k}\left(  \cdot
,\cdot\right)  $ gives%
\begin{align}
A_{k}\left(  \mathbf{w},\mathbf{z}\right)   &  =\left(  \operatorname*{curl}%
\mathbf{w},\operatorname*{curl}\mathbf{z}\right)  -k^{2}\left(  \mathbf{w}%
,\mathbf{z}\right)  -\operatorname*{i}k\left(  T_{k}\mathbf{w}_{T}%
,\mathbf{z}_{T}\right)  _{\Gamma}\nonumber\\
&  =\left(  \mathbf{w},\operatorname*{curl}\operatorname*{curl}\mathbf{z}%
-k^{2}\mathbf{z}\right)  -\left(  \gamma_{T}\mathbf{w},\Pi_{T}%
\operatorname*{curl}\mathbf{z}\right)  _{\Gamma}+\left(  \mathbf{w}%
_{T},\operatorname*{i}kT_{-k}\mathbf{z}_{T}\right)  _{\Gamma}\nonumber\\
&  =\left(  \mathbf{w},\operatorname*{curl}\operatorname*{curl}\mathbf{z}%
-k^{2}\mathbf{z}\right)  +\left(  \mathbf{w}_{T},\gamma_{T}%
\operatorname*{curl}\mathbf{z}\right)  _{\Gamma}+\left(  \mathbf{w}%
_{T},\operatorname*{i}kT_{-k}\mathbf{z}_{T}\right)  _{\Gamma}.
\label{Astrong4}%
\end{align}
In a similar way, we can express the right-hand side in (\ref{dualproto}) by%
\begin{align}
\text{r.h.s.}  &  =k^{2}\left(  \mathbf{w},\mathbf{v}\right)
+\operatorname*{i}k\left(  \left(  T_{k}\mathbf{w}^{\nabla},\mathbf{g}%
^{\nabla}\right)  _{\Gamma}-\left(  T_{k}\mathbf{w}^{\operatorname*{curl}%
},\mathbf{h}^{\operatorname*{curl}}\right)  _{\Gamma}\right) \nonumber\\
&  =k^{2}\left(  \mathbf{w},\mathbf{v}\right)  +\left(  \mathbf{w}^{\nabla
},\left(  \operatorname*{i}kT_{k}\right)  ^{\ast}\mathbf{g}^{\nabla}\right)
_{\Gamma}-\left(  \mathbf{w}^{\operatorname*{curl}},\left(  \operatorname*{i}%
kT_{k}\right)  ^{\ast}\mathbf{h}^{\operatorname*{curl}}\right)  _{\Gamma
}\nonumber\\
&  =k^{2}\left(  \mathbf{w},\mathbf{v}\right)  +\left(  \mathbf{w}%
_{T},-\operatorname*{i}kT_{-k}\left(  \mathbf{g}^{\nabla}-\mathbf{h}%
^{\operatorname*{curl}}\right)  \right)  _{\Gamma}. \label{Astrongrhs}%
\end{align}
The right-hand sides in (\ref{Astrong4}) and (\ref{Astrongrhs}) must be equal,
which leads to%
\begin{equation}%
\begin{array}
[c]{rll}%
\operatorname*{curl}\operatorname*{curl}\mathbf{z}-k^{2}\mathbf{z} &
=k^{2}\mathbf{v} & \text{in }\Omega,\\
\gamma_{T}\operatorname*{curl}\mathbf{z}+\operatorname*{i}kT_{-k}%
\mathbf{z}_{T} & =-\operatorname*{i}kT_{-k}\left(  \mathbf{g}^{\nabla
}-\mathbf{h}^{\operatorname*{curl}}\right)  & \text{on }\Gamma.
\end{array}
\label{fsp0}%
\end{equation}
In the next step, we eliminate the capacity operator $T_{-k}$ by considering a
full space problem with transmission condition. Note that for any given
$\mathbf{q}_{T}\in\mathbf{H}_{\operatorname*{curl}}^{-1/2}\left(
\Gamma\right)  $ the adjoint capacity operator $T_{-k}\mathbf{q}_{T}$ is
computed by first solving the exterior problem%
\begin{equation}%
\begin{array}
[c]{cl}%
-\operatorname*{i}k\mathbf{z}^{+}+\operatorname{curl}\mathbf{\tilde{H}}=0 &
\text{in }\mathbb{R}^{3}\backslash\overline{\Omega},\\
\operatorname*{i}k\mathbf{\tilde{H}}+\operatorname{curl}\mathbf{z}^{+}=0 &
\text{in }\mathbb{R}^{3}\backslash\overline{\Omega},\\
\gamma_{T}^{+}\mathbf{z}^{+}=\mathbf{q}_{T}\times\mathbf{n} & \text{on }%
\Gamma,\\
\left.
\begin{array}
[c]{c}%
\left\vert \mathbf{z}^{+}\left(  \mathbf{x}\right)  \right\vert \leq c/r\\
\left\vert \mathbf{\tilde{H}}\left(  \mathbf{x}\right)  \right\vert \leq c/r\\
\left\vert \mathbf{z}^{+}-\mathbf{\tilde{H}}\times\frac{\mathbf{x}}%
{r}\right\vert \leq c/r^{2}%
\end{array}
\right\}  & \text{as }r=\left\Vert \mathbf{x}\right\Vert \rightarrow\infty
\end{array}
\label{actionTmo}%
\end{equation}
so that $\gamma_{T}^{+}\mathbf{\tilde{H}}=T_{-k}\mathbf{q}_{T}$. In the
following we always choose $\mathbf{q}_{T}=\Pi_{T}\mathbf{z}$ in
(\ref{actionTmo}) with $\mathbf{z}$ being the solution of (\ref{dualproto}).

From the third equation in (\ref{actionTmo}) we obtain $\left[  \left(
\mathbf{z,z}^{+}\right)  \right]  _{0,\Gamma}=0$ and from the second equation
in (\ref{actionTmo})%
\begin{equation}
\gamma_{T}^{+}\operatorname{curl}\mathbf{z}^{+}=-\operatorname*{i}k\gamma
_{T}^{+}\mathbf{\tilde{H}}=-\operatorname*{i}kT_{-k}\mathbf{z}_{T}.
\label{fspb}%
\end{equation}
Hence,%
\[
\left[  \left(  \mathbf{z},\mathbf{z}^{+}\right)  \right]  _{1,\Gamma
}\overset{\text{(\ref{defjumps})}}{=}\gamma_{T}\operatorname{curl}%
\mathbf{z-}\gamma_{T}^{+}\operatorname{curl}\mathbf{z}^{+}%
\overset{\text{(\ref{fsp0}), (\ref{fspb})}}{=}-\operatorname*{i}kT_{-k}\left(
\mathbf{g}^{\nabla}-\mathbf{h}^{\operatorname*{curl}}\right)  .
\]
Let $\mathbf{v}_{\operatorname*{zero}}$ denote the extension of $\mathbf{v}$
to the full space by $0$ and define $\mathbf{Z}\in\mathbf{H}%
_{\operatorname*{loc}}\left(  \mathbb{R}^{3},\operatorname*{curl}\right)  $ by
$\left.  \mathbf{Z}\right\vert _{\Omega}=\mathbf{z}$ and $\left.
\mathbf{Z}\right\vert _{\Omega^{+}}=\mathbf{z}^{+}$. The combination with
(\ref{fsp0}) leads to (see \cite[(5.2.22)]{Nedelec01} for the radiation
condition)%
\begin{equation}%
\begin{array}
[c]{rll}%
\operatorname*{curl}\operatorname*{curl}\mathbf{Z}-k^{2}\mathbf{Z} &
=k^{2}\mathbf{v}_{\operatorname*{zero}} & \text{in }\mathbb{R}^{3}%
\backslash\Gamma,\\
\left[  \left(  \mathbf{z},\mathbf{z}^{+}\right)  \right]  _{0,\Gamma} & =0, &
\\
\left[  \left(  \mathbf{z},\mathbf{z}^{+}\right)  \right]  _{1,\Gamma} &
=-\operatorname*{i}kT_{-k}\left(  \mathbf{g}^{\nabla}-\mathbf{h}%
^{\operatorname*{curl}}\right)  , & \\
\left\vert \partial_{r}\mathbf{z}^{+}\left(  \mathbf{x}\right)
+\operatorname*{i}k\mathbf{z}^{+}\left(  \mathbf{x}\right)  \right\vert  &
\leq c/r^{2}, & \text{as }r=\left\Vert \mathbf{x}\right\Vert \rightarrow
\infty.
\end{array}
\label{fsprobl}%
\end{equation}
We first construct a particular solution for the corresponding full space
problem by ignoring the transmission conditions. Then we adjust this solution
to satisfy the transmission condition.

For this purpose we need the fundamental solution for the electric part of the
Maxwell problem in the full space:%
\[%
\begin{array}
[c]{rll}%
\operatorname*{curl}\operatorname*{curl}\mathbf{G}_{k}-k^{2}\mathbf{G}_{k} &
=\delta\mathbf{I} & \text{in }\mathbb{R}^{3},\\
\left\vert \partial_{r}\mathbf{G}_{k}\left(  \mathbf{x}\right)
-\operatorname*{i}k\mathbf{G}_{k}\left(  \mathbf{x}\right)  \right\vert  &
\leq c/r^{2} & \text{as }r=\left\Vert \mathbf{x}\right\Vert \rightarrow\infty.
\end{array}
\]
We eliminate in \cite[(5.2.1)]{Nedelec01} the magnetic field to get the
equations%
\[%
\begin{array}
[c]{rll}%
\operatorname*{curl}\operatorname*{curl}\mathbf{E}-k^{2}\mathbf{E} &
=\delta\mathbf{I} & \text{in }\mathbb{R}^{3},\\
\left\vert \partial_{r}\mathbf{E}\left(  \mathbf{x}\right)  -\operatorname*{i}%
k\mathbf{E}\left(  \mathbf{x}\right)  \right\vert  & \leq c/r^{2} & \text{as
}r=\left\Vert \mathbf{x}\right\Vert \rightarrow\infty.
\end{array}
\]
Hence, the fundamental solution is obtained by dividing the one in
\cite[(5.2.8)]{Nedelec01} by $\left(  \operatorname*{i}\omega\mu\right)  $ to obtain%

\begin{equation}
\mathbf{G}_{k}\left(  \mathbf{x}\right)  =g_{k}\left(  \left\Vert
\mathbf{x}\right\Vert \right)  \mathbf{I}+\frac{1}{k^{2}}\nabla\nabla
^{\intercal}g_{k}\left(  \left\Vert \mathbf{x}\right\Vert \right)
\quad\text{with\quad}g_{k}\left(  r\right)  :=\frac{\operatorname*{e}%
^{\operatorname*{i}kr}}{4\pi r}. \label{fundsol}%
\end{equation}
The second term in the sum is understood as a distribution, i.e., the
convolution with a function $\mathbf{f}\in C_{\operatorname{comp}}^{\infty
}\left(  \mathbb{R}^{3},\mathbb{C}^{3}\right)  $ is defined by%
\begin{equation}
\left(  \mathbf{G}_{k}\star\mathbf{f}\right)  \left(  x\right)  =\int%
_{\mathbb{R}^{3}}g_{k}\left(  \left\Vert x-y\right\Vert \right)
\mathbf{f}\left(  y\right)  dy+\frac{1}{k^{2}}\nabla\int_{\mathbb{R}^{3}}%
g_{k}\left(  \left\Vert \mathbf{x}-\mathbf{y}\right\Vert \right)
\operatorname*{div}\mathbf{f}\left(  \mathbf{y}\right)  d\mathbf{y}.
\label{fundconv}%
\end{equation}

{}From (\ref{fundconv}) we conclude that
\[
\mathbf{z}_{1}=k^{2}\int_{\Omega}g_{-k}\left(  \left\Vert \mathbf{\cdot
}-\mathbf{y}\right\Vert \right)  \mathbf{v}\left(  \mathbf{y}\right)
d\mathbf{y}+\nabla\int_{\mathbb{R}^{3}}g_{-k}\left(  \left\Vert \mathbf{x}%
-\mathbf{y}\right\Vert \right)  \left(  \operatorname*{div}\mathbf{v}%
_{\operatorname*{zero}}\right)  \left(  \mathbf{y}\right)  d\mathbf{y}%
\quad\text{in }\mathbb{R}^{3}%
\]
solves the differential equation (first line in (\ref{fsprobl})) in
$\mathbb{R}^{3}\backslash\Gamma$ and the radiation condition. The function
$\mathbf{v}_{\operatorname*{zero}}$ has a jump across $\Gamma$ and it is easy
to verify that the distributional divergence is given by
\[
\left(  \operatorname*{div}\nolimits_{\mathbb{R}^{3}}\mathbf{v}%
_{\operatorname*{zero}}\right)  \left(  \psi\right)  =\int_{\Omega}\left(
\operatorname*{div}\mathbf{v}\right)  \psi-\int_{\Gamma}\left\langle
\mathbf{v},\mathbf{n}\right\rangle \psi\qquad\forall\psi\in
C_{\operatorname*{comp}}^{\infty}\left(  \mathbb{R}^{3}\right)  .
\]
Hence,%
\[
\mathbf{z}_{1}=k^{2}\mathcal{N}_{-k}^{\operatorname*{Hh}}\left(
\mathbf{v}\right)  +\nabla\mathcal{N}_{-k}^{\operatorname*{Hh}}\left(
\operatorname*{div}\mathbf{v}\right)  -\nabla\mathcal{S}_{-k}%
^{\operatorname*{Hh}}\left(  \left\langle \mathbf{v},\mathbf{n}\right\rangle
\right)  =:\mathbf{z}_{1,1}+\mathbf{z}_{1,2}+\mathbf{z}_{1,3}%
\]
with the acoustic single layer potential%
\begin{equation}
\mathcal{S}_{k}^{\operatorname*{Hh}}\phi:=\int_{\Gamma}g_{k}\left(  \left\Vert
\mathbf{\cdot}-\mathbf{y}\right\Vert \right)  \phi\left(  \mathbf{y}\right)
d\Gamma_{\mathbf{y}} \label{eq:helmholtz-single-layer}%
\end{equation}
and the acoustic Newton potential
\begin{equation}
\mathcal{N}_{k}^{\operatorname*{Hh}}w:=\int_{\Omega}g_{k}\left(  \left\Vert
\mathbf{\cdot}-\mathbf{y}\right\Vert \right)  w\left(  \mathbf{y}\right)
d\Gamma_{\mathbf{y}}. \label{eq:helmholtz-newton-potential}%
\end{equation}
We assumed $\mathbf{v}\in\mathbf{H}\left(  \Omega,\operatorname*{div}\right)
$. Well-known mapping properties of $\mathcal{S}_{k}^{\operatorname*{Hh}}$ and
$\mathcal{N}_{k}^{\operatorname*{Hh}}$ (cf. \cite{SauterSchwab2010}) imply
that
\[
\mathbf{z}_{1,1}\in\mathbf{H}_{\operatorname*{loc}}^{2}\left(  \mathbb{R}%
^{3}\right)  \text{ so that }\left[  \mathbf{z}_{1,1}\right]  _{0,\Gamma
}=0\quad\text{and\quad}\left[  \mathbf{z}_{1,1}\right]  _{1,\Gamma}=0.
\]
By the same reasoning we know that $\mathbf{z}_{1,2}\in\mathbf{H}%
_{\operatorname*{loc}}^{1}\left(  \mathbb{R}^{3}\right)  $ and also
$\operatorname*{curl}\mathbf{z}_{1,2}=0$. Hence $\left[  \mathbf{z}%
_{1,2}\right]  _{0,\Gamma}=\left[  \mathbf{z}_{1,2}\right]  _{1,\Gamma}=0$.
Since $\left\langle \mathbf{v},\mathbf{n}\right\rangle \in H^{-1/2}\left(
\Gamma\right)  $ we know that $\mathcal{S}_{-k}^{\operatorname*{Hh}}\left(
\left\langle \mathbf{v},\mathbf{n}\right\rangle \right)  \in\mathbf{H}%
_{\operatorname*{loc}}^{1}\left(  \mathbb{R}^{3}\right)  $ and
$\operatorname*{curl}\nabla\mathcal{S}_{-k}^{\operatorname*{Hh}}\left(
\left\langle \mathbf{v},\mathbf{n}\right\rangle \right)  =0$ so that $\left[
\mathbf{z}_{1,3}\right]  _{1,\Gamma}=0$. Since $\gamma_{\tau}\nabla$ is a
tangential differential operator its jump vanishes on functions in
$\mathbf{H}_{\operatorname*{loc}}^{1}\left(  \mathbb{R}^{3}\right)  $. This
implies that%
\begin{equation}
\left[  \mathbf{z}_{1}\right]  _{0,\Gamma}=0\quad\text{and\quad}\left[
\mathbf{z}_{1}\right]  _{1,\Gamma}=0. \label{gammastau0}%
\end{equation}
To obtain the full solution we introduce the single layer operator for the
Maxwell problem (cf. \cite[(3.11)]{BHP01}) by%
\begin{equation}
\mathcal{S}_{k}^{\operatorname*{Mw}}\left(
\mbox{\boldmath$ \phi$}%
\right)  =\mathcal{S}_{k}^{\operatorname*{Hh}}\left(
\mbox{\boldmath$ \phi$}%
\right)  +\frac{1}{k^{2}}\nabla\mathcal{S}_{k}^{\operatorname*{Hh}}\left(
\operatorname{div}_{\Gamma}%
\mbox{\boldmath$ \phi$}%
\right)  . \label{eq:maxwell-single-layer}%
\end{equation}

From \cite[(5.5.29)]{Nedelec01} we get that%
\[
\left[  \mathcal{S}_{k}^{\operatorname*{Mw}}%
\mbox{\boldmath$ \phi$}%
\right]  _{1,\Gamma}=-%
\mbox{\boldmath$ \phi$}%
.
\]
The combination of this, the third equation in (\ref{fsprobl}), and
(\ref{gammastau0}) show that
\[
\mathbf{z}_{2}:=\mathcal{S}_{-k}^{\operatorname*{Mw}}\left(  \operatorname*{i}%
kT_{-k}\left(  \mathbf{g}^{\nabla}-\mathbf{h}^{\operatorname*{curl}}\right)
\right)
\]
satisfies $\operatorname{curl}\operatorname{curl}\mathbf{z}_{2}-k^{2}%
\mathbf{z}_{2}=0$ in $\mathbb{R}^{3}\backslash\Gamma$, the transmission
condition (2nd and 3rd equation in (\ref{fsprobl})), and the Silver-M\"{u}ller
radiation conditions for the dual problem. Next we give a formula for the full
solution of (\ref{fsprobl})%
\begin{align}
\mathbf{Z}  &  =k^{2}\int_{\Omega}g_{-k}\left(  \left\Vert \mathbf{\cdot
}-\mathbf{y}\right\Vert \right)  \mathbf{v}\left(  \mathbf{y}\right)
d\mathbf{y}+\nabla\int_{\Omega}g_{-k}\left(  \left\Vert \cdot-\mathbf{y}%
\right\Vert \right)  \left(  \operatorname{div}\mathbf{v}\right)  \left(
\mathbf{y}\right)  d\mathbf{y}\nonumber\\
&  -\nabla\int_{\Gamma}g_{-k}\left(  \left\Vert \cdot-\mathbf{y}\right\Vert
\right)  \left\langle \mathbf{v},\mathbf{n}\right\rangle \left(
\mathbf{y}\right)  d\mathbf{y}\label{solformdualprob}\\
&  +\left(  \operatorname*{i}k\right)  \int_{\Gamma}g_{-k}\left(  \left\Vert
\mathbf{\cdot}-\mathbf{y}\right\Vert \right)  T_{-k}\left(  \mathbf{g}%
^{\nabla}-\mathbf{h}^{\operatorname*{curl}}\right)  \left(  \mathbf{y}\right)
d\Gamma_{\mathbf{y}}-\frac{1}{\operatorname*{i}k}\nabla\int_{\Gamma}%
g_{-k}\left(  \left\Vert \mathbf{\cdot}-\mathbf{y}\right\Vert \right)
\operatorname{div}_{\Gamma}T_{-k}\mathbf{g}^{\nabla}\left(  \mathbf{y}\right)
d\Gamma_{\mathbf{y}},\nonumber
\end{align}
where we used $\operatorname{div}_{\Gamma}T_{-k}\mathbf{h}%
^{\operatorname*{curl}}=0$ (cf. (\ref{divsucurl=0})).

\begin{theorem}
\quad

\begin{enumerate}
\item For $\mathbf{v}\in\mathbf{V}_{0}^{\ast}$, $\mathbf{g}=\mathbf{v}$, and
$\mathbf{h}=\mathbf{0}$, the solution of (\ref{dualproto}) is given by%
\begin{equation}
\mathbf{z}=k^{2}\int_{\Omega}g_{-k}\left(  \left\Vert \mathbf{\cdot
}-\mathbf{y}\right\Vert \right)  \mathbf{v}\left(  \mathbf{y}\right)
d\mathbf{y}+\operatorname*{i}k\int_{\Gamma}g_{-k}\left(  \left\Vert
\mathbf{\cdot}-\mathbf{y}\right\Vert \right)  T_{-k}\mathbf{v}^{\nabla}\left(
\mathbf{y}\right)  d\Gamma_{\mathbf{y}}. \label{zvv}%
\end{equation}

\item For $\mathbf{v}=\mathbf{0}$, formula (\ref{solformdualprob}) simplifies
to a combined layer potential%
\begin{equation}
\mathbf{z}=\operatorname*{i}k\int_{\Gamma}g_{-k}\left(  \left\Vert
\mathbf{\cdot}-\mathbf{y}\right\Vert \right)  T_{-k}\left(  \mathbf{g}%
^{\nabla}-\mathbf{h}^{\operatorname*{curl}}\right)  \left(  \mathbf{y}\right)
d\Gamma_{\mathbf{y}}-\frac{1}{\operatorname*{i}k}\nabla\int_{\Gamma}%
g_{-k}\left(  \left\Vert \mathbf{\cdot}-\mathbf{y}\right\Vert \right)
\operatorname{div}_{\Gamma}T_{-k}\mathbf{g}^{\nabla}\left(  \mathbf{y}\right)
d\Gamma_{\mathbf{y}}. \label{v=0}%
\end{equation}

\end{enumerate}
\end{theorem}

%

\proof
For the choices as in (\ref{zvv}), the properties (\ref{ImpBedb}) allow us to
simplify (\ref{solformdualprob}) and to obtain (\ref{zvv}). Formula
(\ref{v=0}) follows simply by setting $\mathbf{v}=\mathbf{0}$ in
(\ref{solformdualprob}).%
\endproof

\subsubsection{Solution Formula for Type 2 Problems in the Unit Ball
$B_{1}(0)$ \label{SolFormType3}}

The problem of Type 2 (cf.~(\ref{type3})) is a Poisson-type problem.
Integration by parts leads to its strong formulation. We recall
$\operatorname{div}L_{\Omega}{\mathbf{r}}=0$ by (\ref{eq:defLOmega-strong-b})
so that
\begin{equation}%
\begin{array}
[c]{cl}%
-\Delta Z=0 & \text{in }\Omega,\\
\frac{\partial Z}{\partial\mathbf{n}}-\frac{\operatorname*{i}}{k}%
\operatorname*{div}_{\Gamma}T_{k}\nabla_{\Gamma}Z=\left\langle L_{\Omega
}\mathbf{r},\mathbf{n}\right\rangle -\frac{\operatorname*{i}}{k}%
\operatorname*{div}_{\Gamma}T_{k}^{\operatorname*{low}}\mathbf{r}_{T} &
\text{on }\Gamma.
\end{array}
\label{doppelklammerstrong}%
\end{equation}
To analyze problem, we introduce the Dirichlet-to-Neumann operator $T_{\Delta
}:H^{1/2}(\Gamma)\rightarrow H^{-1/2}(\Gamma)$ that maps $g\in H^{1/2}%
(\Gamma)$ to $\partial_{n}u$, where $u$ is the (weak) solution of%
\[
\Delta u=0\text{ in }\Omega,\qquad\qquad u=g\quad\text{on }\Gamma.
\]
This allows us to formulate (\ref{doppelklammerstrong}) as follows (with
$L_{\Gamma}$ as in Def.~\ref{DefFreqSplit})%
\begin{equation}%
\begin{array}
[c]{cl}%
-\Delta Z=0 & \text{in }\Omega,\\
T_{\Delta}Z-\frac{\operatorname*{i}}{k}\operatorname*{div}_{\Gamma}T_{k}%
\nabla_{\Gamma}Z=\left\langle L_{\Omega}\mathbf{r},\mathbf{n}\right\rangle
-\frac{\operatorname*{i}}{k}\operatorname*{div}_{\Gamma}T_{k}%
^{\operatorname*{low}}\mathbf{r}_{T} & \text{on }\Gamma.
\end{array}
\label{doppelklammerstrong2}%
\end{equation}
We employ expansions of $\left\langle L_{\Omega}\mathbf{r},\mathbf{n}%
\right\rangle $ and $\mathbf{r}_{T}$ in the forms%
\begin{equation}
\left\langle L_{\Omega}\mathbf{r},\mathbf{n}\right\rangle =\sum_{\ell
=0}^{\infty}\sum_{m\in\iota_{\ell}}\kappa_{\ell}^{m}Y_{\ell}^{m}%
\quad\text{and\quad}\mathbf{r}_{T}=\sum_{\ell=1}^{\infty}\sum_{m\in\iota
_{\ell}}\left(  r_{\ell}^{m}\mathbf{T}_{\ell}^{m.}+R_{\ell}^{m}\nabla_{\Gamma
}Y_{\ell}^{m}\right)  \label{eq:kappa-R}%
\end{equation}
so that the right-hand side in the second equation of
(\ref{doppelklammerstrong2}) is%
\begin{align}
\left\langle L_{\Omega}\mathbf{r},\mathbf{n}\right\rangle -\frac
{\operatorname*{i}}{k}\operatorname*{div}\nolimits_{\Gamma}T_{k}%
^{\operatorname*{low}}\mathbf{r}_{T}  &  =\sum_{\ell=0}^{\infty}\sum
_{m\in\iota_{\ell}}\left(  \kappa_{\ell}^{m}Y_{\ell}^{m}-\frac
{\operatorname*{i}}{k}\operatorname*{div}\nolimits_{\Gamma}T_{k}%
^{\operatorname*{low}}\left(  r_{\ell}^{m}\mathbf{T}_{\ell}^{m}+R_{\ell}%
^{m}\nabla_{\Gamma}Y_{\ell}^{m}\right)  \right) \nonumber\\
&  \overset{\text{\cite[(2.4.173), (5.3.93)]{Nedelec01}}}{=}\sum_{\ell>\lambda
k}^{\infty}\sum_{m\in\iota_{\ell}}\kappa_{\ell}^{m}Y_{\ell}^{m}+\sum_{\ell
\leq\lambda k}\sum_{m\in\iota_{\ell}}\left(  \kappa_{\ell}^{m}-\frac
{\ell\left(  \ell+1\right)  }{z_{\ell}\left(  k\right)  +1}R_{\ell}%
^{m}\right)  Y_{\ell}^{m}. \label{doppelklammerstrong2rhs}%
\end{align}
Note that $Y_{0}^{0}=\frac{1}{\sqrt{4\pi}}$ is constant and hence $\kappa
_{0}^{0}=\left(  \left\langle L_{\Omega}\mathbf{r},\mathbf{n}\right\rangle
,Y_{0}^{0}\right)  _{\Gamma}=\left(  \operatorname*{div}L_{\Omega}%
\mathbf{r},\frac{1}{\sqrt{4\pi}}\right)  _{\Omega}=0$. Hence the summation
index for the second sum in (\ref{doppelklammerstrong2rhs}) can be restricted
to $1\leq\ell\leq\lambda k$. The representation (\ref{doppelklammerstrong2rhs}%
) motivates the ansatz for the trace of $Z$%
\[
\left.  Z\right\vert _{\Gamma}=\sum_{\ell=0}^{\infty}\sum_{m\in\iota_{\ell}%
}Z_{\ell}^{m}Y_{\ell}^{m}.
\]
The left-hand side in the second equation of (\ref{doppelklammerstrong2})
becomes
\begin{align}
T_{\Delta}Z-\frac{\operatorname*{i}}{k}\operatorname*{div}\nolimits_{\Gamma
}T_{k}\nabla_{\Gamma}Z  &  =\sum_{\ell=0}^{\infty}\sum_{m\in\iota_{\ell}%
}Z_{\ell}^{m}\left(  T_{\Delta}Y_{\ell}^{m}-\frac{\operatorname*{i}}%
{k}\operatorname*{div}\nolimits_{\Gamma}T_{k}\nabla_{\Gamma}Y_{\ell}%
^{m}\right) \nonumber\\
&  \overset{\text{\cite[(2.5.22), (5.3.93)]{Nedelec01}}}{=}\sum_{\ell
=1}^{\infty}\ell\left(  1{-}\frac{\ell+1}{z_{\ell}\left(  k\right)
+1}\right)  \sum_{m\in\iota_{\ell}}Z_{\ell}^{m}Y_{\ell}^{m}.
\label{doppelklammerstrong2lhs}%
\end{align}
The right-hand sides in (\ref{doppelklammerstrong2rhs}) and
(\ref{doppelklammerstrong2lhs}) must be equal. Thus%
\begin{equation}
Z_{\ell}^{m}=\left\{
\begin{array}
[c]{ll}%
\varphi_{\ell}^{m}:=\dfrac{1}{\ell}\left(  \dfrac{z_{\ell}\left(  k\right)
+1}{z_{\ell}\left(  k\right)  {-}\ell}\right)  \kappa_{\ell}^{m}-\dfrac
{\ell+1}{z_{\ell}\left(  k\right)  {-}\ell}R_{\ell}^{m} & \ell\leq\lambda k,\\
\Phi_{\ell}^{m}:=\dfrac{1}{\ell}\left(  \dfrac{z_{\ell}\left(  k\right)
+1}{z_{\ell}\left(  k\right)  {-}\ell}\right)  \kappa_{\ell}^{m} &
\ell>\lambda k.
\end{array}
\right.  \label{defphilmPhilm}%
\end{equation}
Hence, the solution $Z$ of (\ref{doppelklammerstrong2}) is the solution of the
following Laplace equation with non-homogeneous Dirichlet boundary conditions%
\begin{equation}%
\begin{array}
[c]{crll}
& -\Delta Z & =0 & \text{in }\Omega,\\
& Z & =g_{D} & \text{on }\Gamma,\\
\text{with} & g_{D} & :=%
{\displaystyle\sum\limits_{\ell\leq\lambda k}}
{\displaystyle\sum\limits_{m\in\iota_{\ell}}}
\varphi_{\ell}^{m}Y_{\ell}^{m}+%
{\displaystyle\sum\limits_{\ell>\lambda k}^{\infty}}
{\displaystyle\sum\limits_{m\in\iota_{\ell}}}
\Phi_{\ell}^{m}Y_{\ell}^{m}. &
\end{array}
\label{recastDiri}%
\end{equation}


\subsection{Regularity of the Dual Problems\label{SecRegDualSol} in $\Omega = B_1(0)$}



\subsubsection{The High-Frequency Case\label{SecRegHighFreq}}


We consider the regularity of the solution in (\ref{adjoint3b}) for a
right-hand side $\mathbf{r}\leftarrow\mathbf{v}_{0}\in\mathbf{V}_{0}%
^{\ast}$. Recall the definition of $\nabla^{p}$ in (\ref{defLaplhochn}).

\begin{proposition}
\label{PropN2Arough}
Let $\Omega = B_1(0)$. 
Let $\mathbf{v}_{0}\in\mathbf{V}_{0}^{\ast}$ and
$\mathbf{z}=\mathcal{N}_{2}\mathbf{v}_{0}$ with $\mathcal{N}_{2}$ given by
(\ref{adjoint3b}). There exists a $k$-dependent splitting $\mathcal{N}%
_{2}{\mathbf v}_0=\mathcal{N}_{2}^{\operatorname*{rough}}{\mathbf v}_0+\mathcal{N}_{2}^{\mathcal{A}}{\mathbf v}_0$
such that
\begin{equation}%
\begin{array}
[c]{rll}%
\left\Vert \mathcal{N}_{2}^{\operatorname*{rough}}\mathbf{v}_{0}\right\Vert
_{\mathbf{H}^{2}\left(  \Omega\right)  } & \leq C_{\operatorname*{rough}%
}k\left\Vert \mathbf{v}_{0}\right\Vert _{\operatorname*{curl},\Omega,k}, & \\
\left\Vert \nabla^{p}\mathcal{N}_{2}^{\mathcal{A}}\mathbf{v}_{0}\right\Vert  &
\leq C_{\mathcal{A},2}k^{3}\gamma_{\mathcal{A},2}^p\left(  \max\left\{
p+1,k\right\}  \right)  ^{p}\left\Vert \mathbf{v}_{0}\right\Vert
_{\operatorname*{curl},\Omega,k} & \forall p\in\mathbb{N}_{0},
\end{array}
\label{PropFullsNewton1}%
\end{equation}
where $C_{\operatorname{rough}}$, $C_{\mathcal{A},2}$, $\gamma_{\mathcal{A}%
,2}>0$ are constants independent of $p$, $k$ and $\mathbf{v}_0$.
\end{proposition}

%

\proof
The solution of the dual problem (\ref{adjoint3b}) is given (cf.~(\ref{zvv}),
(\ref{orthobk})) by
\begin{align*}
\mathbf{z}  &  =\left(  -\operatorname*{i}k\mathbf{z}_{1}+\mathbf{z}%
_{2}\right)  \operatorname*{i}k\\
\text{with\quad}\mathbf{z}_{1}  &  :=\int_{\Omega}g_{-k}\left(  \left\Vert
\mathbf{\cdot}-\mathbf{y}\right\Vert \right)  \mathbf{v}_{0}\left(
\mathbf{y}\right)  d\mathbf{y\quad}\text{and\quad}\mathbf{z}_{2}:=\int%
_{\Gamma}g_{-k}\left(  \left\Vert \mathbf{\cdot}-\mathbf{y}\right\Vert
\right)  T_{-k}\mathbf{v}_{0}^{\nabla}\left(  \mathbf{y}\right)
d\Gamma_{\mathbf{y}}.
\end{align*}

{}From the decomposition lemma in \cite[Lemma~{3.5}]{MelenkSauterMathComp} we
get a $k$-dependent additive splitting $\mathbf{z}_{1}:=\mathbf{z}%
_{1}^{\operatorname*{rough}}+\mathbf{z}_{1}^{\mathcal{A}}$ such that%
\begin{equation}%
\begin{array}
[c]{rll}%
\left\Vert \nabla^{m}\mathbf{z}_{1}^{\operatorname*{rough}}\right\Vert  & \leq
Ck^{m-2}\left\Vert \mathbf{v}_{0}\right\Vert  & \forall m\in\left\{
0,1,2\right\}  ,\\
\left\Vert \nabla^{p}\mathbf{z}_{1}^{\mathcal{A}}\right\Vert  & \leq
Ck^{p-1}\left\Vert \mathbf{v}_{0}\right\Vert  & \forall p\in\mathbb{N}_{0}%
\end{array}
\label{z1roughA0}%
\end{equation}
for a constant $C$ independent of $k$ and $\mathbf{v}_{0}$. For the function
$\mathbf{z}_{2}$ we employ the splitting%
\[
\mathbf{v}_{0}^{\nabla,\operatorname*{low}}:=L_{\Gamma}\left(  \mathbf{v}%
_{0}^{\nabla}\right)  \quad\text{and\quad}\mathbf{v}_{0}^{\nabla
,\operatorname*{high}}:=H_{\Gamma}\left(  \mathbf{v}_{0}^{\nabla}\right)
\]
and define $\mathbf{z}_{2}^{\operatorname*{low}}:=\mathcal{S}_{-k}%
^{\operatorname*{Hh}}\left(  T_{-k}\mathbf{v}_{0}^{\nabla,\operatorname*{low}%
}\right)  $ and $\mathbf{z}_{2}^{\operatorname*{high}}:=\mathbf{z}%
_{2}-\mathbf{z}_{2}^{\operatorname*{low}}$. From \cite[Lem.~{3.4}, Thm.~{5.3}%
]{Melenk_map_helm_int} we conclude that there exists a splitting
$\mathbf{z}_{2}^{\operatorname*{high}}=\mathbf{z}_{2}^{\operatorname*{rough}%
}+\mathbf{z}_{2}^{\mathcal{A}}$ such that, for $\mathbf{w}:=T_{-k}%
\mathbf{v}_{0}^{\nabla,\operatorname*{high}}$,%
\begin{equation}%
\begin{array}
[c]{rll}%
\left\Vert \nabla^{m}\mathbf{z}_{2}^{\operatorname*{rough}}\right\Vert  & \leq
C\left\Vert \mathbf{w}\right\Vert _{\mathbf{H}_{T}^{1/2}\left(  \Gamma\right)
} & \forall m\in\{0,1,2\},\\
\left\Vert \nabla^{p}\mathbf{z}_{2}^{\mathcal{A}}\right\Vert  & \leq\tilde
{C}\tilde{\gamma}^{p}\max\left\{  p+1,k\right\}  ^{p+1}\left\Vert
\mathbf{w}\right\Vert _{\mathbf{H}_{T}^{-3/2}\left(  \Gamma\right)  } &
\forall p\in\mathbb{N}_{0}.
\end{array}
\label{z2roughAfirst}%
\end{equation}
Here the constants $C$, $\tilde{C}$, $\tilde{\gamma}$ are independent of $p$,
$k$ and $\mathbf{w}$. This motivates the definition of the operator
$\mathcal{N}_{2}^{\operatorname*{rough}}:\mathbf{V}_{0}^{\ast}\rightarrow
\mathbf{H}^{2}\left(  \Omega\right)  $ by%
\begin{equation}
\mathcal{N}_{2}^{\operatorname*{rough}}\mathbf{v}_{0}:=\mathbf{z}%
_{1}^{\operatorname*{rough}}+\mathbf{z}_{2}^{\operatorname*{rough}}.
\label{defN1rough}%
\end{equation}

To estimate the norms of $\mathbf{w}$ in (\ref{z2roughAfirst}) we employ the
third estimate in Lemma~\ref{Lemzlest} for $s\leq3/2$ (we also use that
(\ref{zlconjcompl}) gives $z_{\ell}\left(  -k\right)  =\overline{z_{\ell
}\left(  k\right)  }$ and $\left\vert \frac{-k}{z_{\ell}\left(  -k\right)
+1}\right\vert =\left\vert \frac{k}{z_{\ell}\left(  k\right)  +1}\right\vert
$): From the definition of $\mathbf{w}$ and (\ref{uTaddsplit}) --
(\ref{DefTk}) we conclude that $\mathbf{w}$ has the representation%
\[
\mathbf{v}_0^\nabla =\sum_{\ell=1}^{\infty}\sum_{m\in\iota_{\ell}}V_{\ell}^{m}%
\nabla_{\Gamma}Y_{\ell}^{m}%
\]
for some coefficients $V_{\ell}^{m}$. Hence%
\begin{align}
\left\Vert \mathbf{w}\right\Vert _{\mathbf{H}_{T}^{s}\left(  \Gamma\right)
}^{2}  &  \leq\sum_{\ell>\lambda k}\sum_{m\in\iota_{\ell}}\left(  \ell\left(
\ell+1\right)  \right)  ^{s+1}\left\vert \frac{k}{z_{\ell}\left(  k\right)
+1}\right\vert ^{2}\left\vert V_{\ell}^{m}\right\vert ^{2}\leq Ck^{2}%
\sum_{\ell>\lambda k}\sum_{m\in\iota_{\ell}}\left(  \ell+1\right)
^{2s-3}\left(  \ell\left(  \ell+1\right)  \right)  ^{3/2}\left\vert V_{\ell
}^{m}\right\vert ^{2}\nonumber\\
&  \leq Ck^{2s-1}\sum_{\ell>\lambda k}\sum_{m\in\iota_{\ell}}\left(
\ell\left(  \ell+1\right)  \right)  ^{3/2}\left\vert V_{\ell}^{m}\right\vert
^{2}\overset{\text{(\ref{divscalednorm})}}{\leq}Ck^{2s-1}\left\Vert
\operatorname*{div}\nolimits_{\Gamma}\mathbf{v}_{0}^{\nabla}\right\Vert
_{H^{-1/2}\left(  \Gamma\right)  }^{2}\leq Ck^{2s-1}\left\Vert \mathbf{v}%
_{0}^{\nabla}\right\Vert _{\mathbf{H}_{T}^{1/2}\left(  \Gamma\right)  }%
^{2}\nonumber\\
&  \leq Ck^{2s-1}\left\Vert \mathbf{v}_{0}\right\Vert _{\mathbf{H}%
^{1}\left(  \Omega\right)  }^{2}\overset{\text{Lemma~\ref{Lemembedspec}}}
{\leq}Ck^{2s-1}\left\Vert \mathbf{v}_{0}\right\Vert
_{\operatorname*{curl},\Omega,1}^{2}. \label{Tmk}%
\end{align}
We set $s=1/2$ in (\ref{Tmk}) to derive%
\begin{equation}
\left\Vert \mathbf{w}\right\Vert _{\mathbf{H}_{T}^{1/2}\left(  \Gamma\right)
}\leq C\left\Vert \mathbf{v}_{0}\right\Vert _{\operatorname*{curl},\Omega,k}.
\label{Est1/2normw}%
\end{equation}

The combination of the first lines in (\ref{z1roughA0}) and
(\ref{z2roughAfirst}) with (\ref{Est1/2normw}) leads to the first estimate in
(\ref{PropFullsNewton1}).

To estimate $\left\Vert \mathbf{w}\right\Vert _{\mathbf{H}_{T}^{-3/2}\left(
\Gamma\right)  }$ we employ (\ref{Tmk}) for $s=-3/2$ and obtain%
\begin{equation}
\left\Vert \mathbf{w}\right\Vert _{\mathbf{H}_{T}^{-3/2}\left(  \Gamma\right)
}\leq Ck^{-2}\left\Vert \mathbf{v}_{0}\right\Vert _{\operatorname*{curl}%
,\Omega,1}. \label{w-3/2est}%
\end{equation}
Taking into account the second estimate in (\ref{z2roughAfirst}) results in%
\[
\left\Vert \nabla^{p}\mathbf{z}_{2}^{\mathcal{A}}\right\Vert \leq C\gamma
^{p}\max\left\{  p+1,k\right\}  ^{p-1}\left\Vert \mathbf{v}_{0}\right\Vert
_{\operatorname*{curl},\Omega,1}.
\]

The term $\mathbf{z}_{2}^{\operatorname*{low}}$ is defined as the acoustic
single layer potential applied to the function $T_{-k}\mathbf{v}_{0}%
^{\nabla,\operatorname*{low}}$. The analysis of such a term will be carried
out in Section \ref{SecRegLowFreq} and it follows from (\ref{nablacde}) (where
the function $c$ corresponds to $\mathbf{z}_{2}^{\operatorname*{low}}$) that%
\begin{equation}
\mathbf{z}_{2}^{\operatorname*{low}}\in\mathcal{A}\left(  Ck^{2}\left\Vert
\mathbf{v}_{0}\right\Vert _{\operatorname*{curl},\Omega,1},\gamma
,\Omega\right)  , \label{z2lowest}%
\end{equation}
where $C$ and $\gamma$ are positive constants independent of $k$ and
$\mathbf{v}_{0}$. The combination of the second estimates in (\ref{z1roughA0}%
), (\ref{z2roughAfirst}) with (\ref{w-3/2est}) and (\ref{z2lowest}) leads to
the second estimate in (\ref{PropFullsNewton1}).%
\endproof


\subsubsection{The Low-Frequency Cases\label{SecRegLowFreq}}


First, we study the regularity of the solution operator $\mathcal{N}%
_{3}^{\mathcal{A}}$ as in (\ref{smoothdualproblemd}) which is of Type 1 with
$\mathbf{r}=\mathbf{g}=L_{\Omega}\mathbf{v}$ and $\mathbf{h}=0$. Since
$L_{\Omega}\mathbf{v}$ is, in general, not in $\mathbf{V}_{0}^{\ast}$ we have
to employ the solution formula (\ref{solformdualprob}), where the second
summand can be dropped due to $\operatorname{div}L_{\Omega}\mathbf{v}=0$
(cf.~(\ref{eq:defLOmega-strong-b})). We set%
\begin{equation}%
\begin{array}
[c]{ll}%
a:=\mathcal{N}_{-k}^{\operatorname*{Hh}}\left(  L_{\Omega}\mathbf{v}\right)
, & b:=\mathcal{S}_{-k}^{\operatorname*{Hh}}\left(  \left\langle L_{\Omega
}\mathbf{v},\mathbf{n}\right\rangle \right)  ,\\
c:=\mathcal{S}_{-k}^{\operatorname*{Hh}}\left(  T_{-k}\left(  L_{\Omega
}\mathbf{v}\right)  ^{\nabla}\right)  , & d:=\mathcal{S}_{-k}%
^{\operatorname*{Hh}}\left(  \operatorname{div}_{\Gamma}T_{-k}\left(
L_{\Omega}\mathbf{v}\right)  ^{\nabla}\right)  
\end{array}
\label{abcde}%
\end{equation}
so that
\begin{equation}
\mathbf{z}:=\mathcal{N}_{3}^{\mathcal{A}}\mathbf{v}=k^{2}a-\nabla
b+\operatorname*{i}kc+\frac{\operatorname*{i}}{k}\nabla d. \label{zabcde}%
\end{equation}

\begin{proposition}
\label{PropN3} 
Let $\Omega = B_1(0)$. 
There exist positive
constants $C_{\mathcal{A},3}$ and $\gamma_{\mathcal{A},3}$ independent of $k$
such that for any $\mathbf{v} \in{\mathbf{X}}$
\[
\mathcal{N}_{3}^{\mathcal{A}}\mathbf{v}\in\mathcal{A}\left(  C_{\mathcal{A}%
,3}k^{3}\left\Vert \mathbf{v}\right\Vert _{\operatorname*{curl},\Omega
,k},\gamma_{\mathcal{A},3},\Omega\right)  .
\]

\end{proposition}

%

\proof
We determine the analyticity classes for the functions in the splitting
(\ref{zabcde}), distinguishing between the terms related to the acoustic
Newton potential $\mathcal{N}_{-k}^{\operatorname*{Hh}}$ and the acoustic
single layer operator.

\textbf{a) Newton potential}. We start by writing a function $q=\mathcal{N}%
_{-k}^{\operatorname*{Hh}}\left(  g\right)  $ as a solution of a transmission
problem: Let
\begin{align*}
-\Delta q-k^{2}q  &  =g_{\operatorname*{zero}}\text{ in }\mathbb{R}%
^{3}\backslash\Gamma\text{\quad with\quad}g_{\operatorname*{zero}}:=\left\{
\begin{array}
[c]{ll}%
g & \text{in }\Omega,\\
0 & \text{in }\mathbb{R}^{3}\backslash\overline{\Omega}.
\end{array}
\right. \\
\left[  q\right]  _{\Gamma}  &  =\left[  \frac{\partial q}{\partial\mathbf{n}%
}\right]  _{\Gamma}=0,\\
\left\vert \frac{\partial q}{\partial r}+\operatorname*{i}kq\right\vert  &
=o\left(  \left\Vert \mathbf{x}\right\Vert ^{-1}\right)  \quad\text{as
}\left\Vert \mathbf{x}\right\Vert \rightarrow\infty\text{.}%
\end{align*}
Next, we will determine the class of analyticity for the function $q$ by using
the results in \cite{Melenk_map_helm_int}. For this, we have to investigate
the analyticity class of $g=L_{\Omega}\mathbf{v}$. From Theorem~\ref{TheoLau}
we conclude with $C_{1}$, $\gamma_{1}$ independent of $k$ and ${\mathbf{v}}$
\begin{equation}
L_{\Omega}\mathbf{v}\in\mathcal{A}\left(  C_{\mathbf{v},1},\gamma_{1}%
,\Omega\right)  \quad\text{with\quad}C_{\mathbf{v},1}:=C_{1}k^{3/2}\left\Vert
\mathbf{v}\right\Vert _{\operatorname*{curl},\Omega,k}.
\label{eq:analyticity-LOmegav}%
\end{equation}
This allows us to use \cite[Thm.~{B.4}]{Melenk_map_helm_int} to deduce the
analyticity class for $\mathcal{N}_{-k}^{\operatorname*{Hh}}\left(  L_{\Omega
}\mathbf{v}\right)  $. We introduce the weighted $H^{1}$-norm by $\left\Vert
v\right\Vert _{\mathcal{H},\Omega}:=\left(  \left\Vert \nabla v\right\Vert
^{2}+k^{2}\left\Vert v\right\Vert ^{2}\right)  ^{1/2}$ and obtain%
\[
\left.  \mathcal{N}_{-k}^{\operatorname*{Hh}}\left(  L_{\Omega}\mathbf{v}%
\right)  \right\vert _{\Omega}\in\mathcal{A}\left(  C_{\mathbf{v},2}%
,\gamma_{3},\Omega\right)  ,
\]
with%
\[
C_{\mathbf{v},2}:=C_{3}\left(  k^{-2}C_{\mathbf{v},1}+k^{-1}\left\Vert
\mathcal{N}_{-k}^{\operatorname*{Hh}}\left(  L_{\Omega}\mathbf{v}\right)
\right\Vert _{\mathcal{H},B_{R}(0)}\right)  ;
\]
here, $B_{R}(0)$ is an (arbitrarily chosen) ball containing $\overline{\Omega
}$. {}From \cite[Lemma 3.5]{MelenkSauterMathComp}, we get $\left\Vert
\mathcal{N}_{-k}^{\operatorname*{Hh}}\left(  L_{\Omega}\mathbf{v}\right)
\right\Vert _{\mathcal{H},B_{R}(0)}\leq C\left\Vert L_{\Omega}\mathbf{v}%
\right\Vert $ so that%
\[
C_{\mathbf{v},2}\leq CC_{3}\left(  k^{-2}C_{\mathbf{v},1}+k^{-1}\left\Vert
L_{\Omega}\mathbf{v}\right\Vert \right)  \leq CC_{3}\left(  C_{1}%
k^{-1/2}\left\Vert \mathbf{v}\right\Vert _{\operatorname*{curl},\Omega
,k}+k^{-2}\left\Vert \mathbf{v}\right\Vert _{\operatorname*{curl},\Omega
,k}\right)  \leq C_{4}k^{-1/2}\left\Vert \mathbf{v}\right\Vert
_{\operatorname*{curl},\Omega,k}.
\]
Hence
\begin{equation}
k^{2}a\in\mathcal{A}\left(  C_{4}k^{3/2}\left\Vert \mathbf{v}\right\Vert
_{\operatorname*{curl},\Omega,k},\gamma_{4},\Omega\right)  . \label{k2a}%
\end{equation}

\textbf{b) Single layer potential:} We write a function $q=\mathcal{S}%
_{-k}^{\operatorname*{Hh}}\left(  g\right)  $ as a solution of a transmission
problem: Let $\gamma_{0}$ denote the standard, one-sided trace operator for
$\Gamma$ from the interior and $\gamma_{0}^{+}$ the one from the exterior. The
one-sided normal trace (from the interior) is denoted by $\gamma_{1}%
:=\partial/\partial\mathbf{n}$ and by $\gamma_{1}^{+}$ from the exterior. The
respective jumps are $\left[  u\right]  _{\Gamma}=\gamma_{0}^{+}u-\gamma_{0}u$
and $\left[  u\right]  _{\mathbf{n},\Gamma}=\gamma_{1}^{+}u-\gamma_{1}u$. The
well-known jump relations for the single layer potential yield for the
potential $q=\mathcal{S}_{-k}^{\operatorname*{Hh}}\left(  g\right)  $
\begin{align*}
-\Delta q-k^{2}q  &  =0\text{ in }\mathbb{R}^{3}\backslash\Gamma,\\
\left[  q\right]  _{\Gamma}  &  =0\quad\text{and\quad}\left[  q\right]
_{\mathbf{n},\Gamma}=-g,\\
\left\vert \frac{\partial q}{\partial r}+\operatorname*{i}kq\right\vert  &
=o\left(  \left\Vert \mathbf{x}\right\Vert ^{-1}\right)  \quad\text{as
}\left\Vert \mathbf{x}\right\Vert \rightarrow\infty\text{.}%
\end{align*}
The essential part of the regularity estimates are those near the
boundary/interface $\Gamma$, where the analyticity of the jump $g$ and the
geometry $\Gamma$ come into play. We follow the standard procedure of locally
flattening $\Gamma$ so that \cite[Thm.~{5.5.4}]{MelenkHabil} becomes
applicable.
In view of (\ref{abcde}) we have to analyze the transmission problem for 3
different choices of $g$:
\begin{equation}
g\in\left\{  g_{1},g_{2},\mathbf{g}_{3}\right\}  \quad\text{with\quad}%
g_{1}:=\left\langle L_{\Omega}\mathbf{v},\mathbf{n}\right\rangle \text{,\quad
}g_{2}:=\operatorname{div}_{\Gamma}T_{-k}\left(  L_{\Omega}\mathbf{v}\right)
^{\nabla}\text{,\quad}\mathbf{g}_{3}:=T_{-k}\left(  L_{\Omega}\mathbf{v}%
\right)  ^{\nabla}. \label{g1g2g3}%
\end{equation}
\emph{1.~step (analyticity classes of $g$):}
In the following ${\mathcal{U}}_{\Gamma}$ is a sufficiently small neighborhood
of $\Gamma$ whose size depends solely on $\Gamma$.
Lemma~\ref{lemma:ntimesLomega-analytic} directly implies the existence of an
extension $g_{1}^{\ast}$ of $g_{1}$ into ${\mathcal{U}}_{\Gamma}$
\begin{equation}
g_{1}^{\ast}\in\mathcal{A}\left(  \tilde{C}k^{3/2}\left\Vert \mathbf{v}%
\right\Vert _{\operatorname*{curl},\Omega,k},\tilde{\gamma},\mathcal{U}%
_{\Gamma}\right)  . \label{g1starA}%
\end{equation}

To define extensions of $g_{2}$, $\mathbf{g}_{3}$ we repeat the arguments of
Lemma~\ref{LemLgammaPitu}. From the expansion
\[
\Pi_{T}\mathbf{v}:=\sum_{\ell=1}^{\infty}\sum_{m\in\iota_{\ell}}\left(
v_{\ell}^{m}\overrightarrow{\operatorname*{curl}\nolimits_{\Gamma}}Y_{\ell
}^{m}+V_{\ell}^{m}\nabla_{\Gamma}Y_{\ell}^{m}\right)
\]
we get
\begin{align*}
\mathbf{g}_{3}  &  =T_{-k}\left(  L_{\Omega}\mathbf{v}\right)  ^{\nabla
}\overset{\text{(\ref{DefTk})}}{=}\sum_{1\leq\ell\leq\lambda k}\overline
{\left(  \frac{\operatorname*{i}k}{z_{\ell}\left(  k\right)  +1}\right)  }%
\sum_{m\in\iota_{\ell}}V_{\ell}^{m}\nabla_{\Gamma}Y_{\ell}^{m},\\
g_{2}  &  =\operatorname{div}_{\Gamma}T_{-k}\left(  L_{\Omega}\mathbf{v}%
\right)  ^{\nabla}\overset{\text{(\ref{divsucurl=0})}}{=}\sum_{1\leq\ell
\leq\lambda k}\overline{\left(  \frac{\operatorname*{i}k\ell\left(
\ell+1\right)  }{z_{\ell}\left(  k\right)  +1}\right)  }\sum_{m\in\iota_{\ell
}}V_{\ell}^{m}Y_{\ell}^{m}.
\end{align*}
Recall the analytic extension $\tilde Y^{m}_{\ell}$ of $Y^{m}_{\ell}$ with the
property (\ref{estgradYtilde}) from the proof of Lemma~\ref{LemLgammaPitu}. We
define the analytic extensions of $g_{2}$, $\mathbf{g}_{3}$ by
\[
g_{2}^{\ast}:=\sum_{1\leq\ell\leq\lambda k}\overline{\left(  \frac
{\operatorname*{i}k\ell\left(  \ell+1\right)  }{z_{\ell}\left(  k\right)
+1}\right)  }\sum_{m\in\iota_{\ell}}V_{\ell}^{m}\tilde{Y}_{\ell}^{m}%
\quad\text{and\quad}\mathbf{g}_{3}^{\ast}:=\sum_{1\leq\ell\leq\lambda
k}\overline{\left(  \frac{\operatorname*{i}k}{z_{\ell}\left(  k\right)
+1}\right)  }\sum_{m\in\iota_{\ell}}V_{\ell}^{m}\nabla\tilde{Y}_{\ell}^{m}.
\]
We obtain by using Cauchy-Schwarz inequalities
\begin{align*}
\left\Vert \nabla^{n}g_{2}^{\ast}\right\Vert _{L^{2}\left(  \mathcal{U}%
_{\Gamma}\right)  }  &  \leq\sum_{1\leq\ell\leq\lambda k}\left\vert
\frac{\operatorname*{i}k\ell\left(  \ell+1\right)  }{z_{\ell}\left(  k\right)
+1}\right\vert \sum_{m\in\iota_{\ell}}\left\vert V_{\ell}^{m}\right\vert
\left\Vert \nabla^{n}\tilde{Y}_{\ell}^{m}\right\Vert _{L^{2}\left(
\mathcal{U}_{\Gamma}\right)  }\overset{\text{(\ref{estk/z+1}),
(\ref{m1/2curldiva})}}{\leq}Ck^{7/2}\gamma^{n}\max\left\{  k,n\right\}
^{n}\left\Vert \mathbf{v}_{T}\right\Vert _{-1/2,\operatorname*{curl}_{\Gamma}%
},\\
\left\Vert \nabla^{n}\mathbf{g}_{3}^{\ast}\right\Vert _{L^{2}\left(
\mathcal{U}_{\Gamma}\right)  }  &  \leq\sum_{1\leq\ell\leq\lambda k}\left\vert
\frac{\operatorname*{i}k}{z_{\ell}\left(  k\right)  +1}\right\vert \sum
_{m\in\iota_{\ell}}\left\vert V_{\ell}^{m}\right\vert \left\Vert \nabla
^{n+1}\tilde{Y}_{\ell}^{m}\right\Vert _{L^{2}\left(  \mathcal{U}_{\Gamma
}\right)  }\\
&  \leq\tilde{C}\tilde{\gamma}^{n+1}k\max\left\{  k,n+1\right\}  ^{n}%
\sum_{1\leq\ell\leq\lambda k}\sqrt{\ell\left(  \ell+1\right)  }\sum_{m\in
\iota_{\ell}}\left\vert V_{\ell}^{m}\right\vert \\
&  \leq\hat{C}\hat{\gamma}^{n}k^{5/2}\max\left\{  k,n+1\right\}
^{n}\left\Vert \mathbf{v}_{T}\right\Vert _{-1/2,\operatorname*{curl}_{\Gamma}%
}.
\end{align*}
We combine this with Theorem~\ref{traceTHM1} and have proved that the
extensions $g_{1}^{\ast}$, $g_{2}^{\ast}$, $\mathbf{g}_{3}^{\ast}$ belong to
the analyticity classes%
\[
g_{1}^{\ast}\in\mathcal{A}\left(  C_{1}k^{3/2}\left\Vert \mathbf{v}\right\Vert
_{\operatorname*{curl},\Omega,k},\gamma_{1},\mathcal{U}_{\Gamma}\right)
,\quad g_{2}^{\ast}\in\mathcal{A}\left(  C_{2}k^{7/2}\left\Vert \mathbf{v}%
\right\Vert _{\operatorname*{curl},\Omega,1},\gamma_{2},\mathcal{U}_{\Gamma
}\right)  ,\quad\mathbf{g}_{3}^{\ast}\in\mathcal{A}\left(  C_{3}%
k^{5/2}\left\Vert \mathbf{v}\right\Vert _{\operatorname*{curl},\Omega
,1},\gamma_{3},\mathcal{U}_{\Gamma}\right)  ,
\]
where $C_{j}$, $\gamma_{j}$ are independent of $k$ and $\mathbf{v}$.

\emph{2.~step (a priori bounds for potential $q$):} Note that, for an
(arbitrary) fixed ball $B_{R}(0)$ with $\overline{\Omega}\subset B_{R}%
(0)$,\cite[Lemma~{3.4}, Thm.~{5.3}]{Melenk_map_helm_int} imply
\begin{equation}
\left\Vert \mathcal{S}_{-k}^{\operatorname*{Hh}}\left(  \mathbf{g}_{3}\right)
\right\Vert _{\mathcal{H},B_{R}(0)}\leq C\sum_{\ell=0}^{1}k^{2-\ell}\left\Vert
\mathbf{g}_{3}\right\Vert _{\mathbf{H}_{T}^{-3/2+\ell}\left(  \Gamma\right)
}\quad\text{and\quad}\left\Vert \mathcal{S}_{-k}^{\operatorname*{Hh}}\left(
g_{i}\right)  \right\Vert _{\mathcal{H},B_{R}(0)}\leq C\sum_{\ell=0}%
^{1}k^{2-\ell}\left\Vert g_{i}\right\Vert _{H^{-3/2+\ell}\left(
\Gamma\right)  },\quad i=1,2. \label{S-kHHg}%
\end{equation}
The $\left\Vert \cdot\right\Vert _{H^{-3/2+\ell}\left(  \Gamma\right)  }$ norm
of $g_{i}$ can be estimated for $\ell=0,1$ as follows:

\textbf{Estimating }$g_{1}:$%
\begin{align*}
\left\Vert \left\langle L_{\Omega}\mathbf{v},\mathbf{n}\right\rangle
\right\Vert _{H^{-3/2+\ell}\left(  \Gamma\right)  }  &  \leq C\left\Vert
\left\langle L_{\Omega}\mathbf{v},\mathbf{n}\right\rangle \right\Vert
_{H^{-1/2}\left(  \Gamma\right)  }\leq C\left\Vert L_{\Omega}\mathbf{v}%
\right\Vert _{\mathbf{H}\left(  \Omega,\operatorname*{div}\right)  }\\
&  \overset{\text{(\ref{eq:defLOmega-strong-b})}}{=}C\left\Vert L_{\Omega
}\mathbf{v}\right\Vert \leq\frac{C}{k}\left\Vert L_{\Omega}\mathbf{v}%
\right\Vert _{\operatorname*{curl},\Omega,k}\leq Ck^{-1}\left\Vert
\mathbf{v}\right\Vert _{\operatorname*{curl},\Omega,k}.
\end{align*}

\textbf{Estimating }$g_{2}:$%
\begin{align*}
\left\Vert g_{2}\right\Vert _{H^{-3/2+n}\left(  \Gamma\right)  }^{2}  &
\overset{\text{(\ref{DefHsGammaNorm})}}{\leq}\sum_{1\leq\ell\leq\lambda
k}\left(  \ell\left(  \ell+1\right)  \right)  ^{-3/2+n}\left\vert
\frac{\operatorname*{i}k\ell\left(  \ell+1\right)  }{z_{\ell}\left(  k\right)
+1}\right\vert ^{2}\sum_{m\in\iota_{\ell}}\left\vert V_{\ell}^{m}\right\vert
^{2}\\
&  \overset{\text{(\ref{estk/z+1})}}{\leq}Ck^{2}\sum_{1\leq\ell\leq\lambda
k}\left(  \ell\left(  \ell+1\right)  \right)  ^{1/2+n}\sum_{m\in\iota
_{\ell}}\left\vert V_{\ell}^{m}\right\vert ^{2}\\
&  \overset{\text{(\ref{m1/2curldiva})}}{\leq}Ck^{2+2n}\left\Vert
\mathbf{v}_{T}\right\Vert _{-1/2,\operatorname*{curl}_{\Gamma}}^{2}%
\overset{\text{Thm. \ref{traceTHM1}}}{\leq}Ck^{2+2n}\left\Vert
\mathbf{v}\right\Vert _{\operatorname*{curl},\Omega,1}^{2}.
\end{align*}

\textbf{Estimating }$g_{3}:$
\begin{align*}
\left\Vert \mathbf{g}_{3}\right\Vert _{\mathbf{H}_{T}^{-3/2+n}\left(
\Gamma\right)  }^{2}  &  \overset{\text{(\ref{DefHsGammaTNorm})}}{\leq}%
\sum_{1\leq\ell\leq\lambda k}\left(  \ell\left(  \ell+1\right)  \right)
^{-1/2+n}\left\vert \frac{\operatorname*{i}k}{z_{\ell}\left(  k\right)
+1}\right\vert ^{2}\sum_{m\in\iota_{\ell}}\left\vert V_{\ell}^{m}\right\vert
^{2}\\
&  \leq Ck^{2+2n}\sum_{1\leq\ell\leq\lambda k}\left(  \ell\left(
\ell+1\right)  \right)  ^{-1/2}\sum_{m\in\iota_{\ell}}\left\vert V_{\ell}%
^{m}\right\vert ^{2}\leq Ck^{2+2n}\left\Vert \mathbf{v}\right\Vert
_{\operatorname*{curl},\Omega,1}^{2}.
\end{align*}
The combination with (\ref{S-kHHg}) leads to%
\[
\left\Vert \mathcal{S}_{-k}^{\operatorname*{Hh}}\left(  g_{1}\right)
\right\Vert _{\mathcal{H},B_{R}(0)}\leq Ck\left\Vert \mathbf{v}\right\Vert
_{\operatorname*{curl},\Omega,k},\quad\left\Vert \mathcal{S}_{-k}%
^{\operatorname*{Hh}}\left(  g_{2}\right)  \right\Vert _{\mathcal{H},B_{R}%
(0)}\leq Ck^{3}\left\Vert \mathbf{v}\right\Vert _{\operatorname*{curl}%
,\Omega,1},\quad\left\Vert \mathcal{S}_{-k}^{\operatorname*{Hh}}\left(
\mathbf{g}_{3}\right)  \right\Vert _{\mathcal{H},B_{R}(0)}\leq Ck^{3}%
\left\Vert \mathbf{v}\right\Vert _{\operatorname*{curl},\Omega,1}.
\]
\emph{3.~step (analyticity of potential $q$):} The above steps and
\cite[Thm.~5.5.4]{MelenkHabil} give
\begin{equation}
\nabla b\in\mathcal{A}\left(  Ck^{3/2}\left\Vert \mathbf{v}\right\Vert
_{\operatorname*{curl},\Omega,k},\gamma,\Omega\right)  ,\quad kc\in
\mathcal{A}\left(  Ck^{3}\left\Vert \mathbf{v}\right\Vert
_{\operatorname*{curl},\Omega,1},\gamma,\Omega\right)  ,\quad\frac{1}{k}\nabla
d\in\mathcal{A}\left(  Ck^{5/2}\left\Vert \mathbf{v}\right\Vert
_{\operatorname*{curl},\Omega,1},\gamma,\Omega\right)  . \label{nablacde}%
\end{equation}
{}From the decomposition (\ref{zabcde}) and the assertions (\ref{k2a}),
(\ref{nablacde}) we conclude that%
\[
\mathcal{N}_{3}^{\mathcal{A}}\mathbf{v}\in\mathcal{A}\left(  C_{\mathcal{A}%
,3}k^{3}\left\Vert \mathbf{v}\right\Vert _{\operatorname*{curl},\Omega
,k},\gamma_{\mathcal{A},3},\Omega\right)
\]
for constants $C_{\mathcal{A},3}$, $\gamma_{\mathcal{A},3}$ independent of $k$
and $\mathbf{v}$.%
\endproof

Next, we analyze the regularity of the solution operator $\mathcal{N}%
_{4}^{\mathcal{A}}$ of (\ref{graddoppelKlammer}).

\begin{proposition}
\label{PropN4A} Let $\Omega= B_{1}(0)$. There exist positive constants
$C_{\mathcal{A},4}$ and $\gamma_{\mathcal{A},4}$ depending only on $\Gamma$
and the cut-off parameter $\lambda$ such that for any $\mathbf{r}\in
\mathbf{X}$
\[
\nabla\mathcal{N}_{4}^{\mathcal{A}}\mathbf{r}\in\mathcal{A}\left(
C_{\mathcal{A},4}k^{5/2}\left\Vert \mathbf{r}\right\Vert
_{\operatorname*{curl},\Omega,1},\gamma_{\mathcal{A},4},\Omega\right)  .
\]

\end{proposition}

%

\proof
We first analyze $g_{D}$ of (\ref{recastDiri}) (in Steps~1--3) and
subsequently the solution $Z$ of (\ref{recastDiri}) in Step~4. As in the proof
of Proposition~\ref{PropN3}, we let $\mathcal{U}_{\Gamma}$ be a sufficiently
small neighborhood of $\Gamma$.

\emph{1.~step (analyticity class of $g_{D}^{\ast}$):} With the analytic
extensions $\tilde{Y}_{\ell}^{m}$ of the eigenfunctions $Y_{\ell}^{m}$
(cf.~(\ref{estgradYtilde})) we extend $g_{D}$ to ${\mathcal{U}}_{\Gamma}$ by
\[
g_{D}^{\ast}:=\sum_{\ell=1}^{\infty}\sum_{m\in\iota_{\ell}}Z_{\ell}^{m}%
\tilde{Y}_{\ell}^{m}\overset{(\ref{defphilmPhilm})}{=}\sum_{\ell=1}^{\infty
}\sum_{m\in\iota_{\ell}}\frac{1}{\ell}\left(  \frac{z_{\ell}(k)+1}{z_{\ell
}(k)-\ell}\right)  \kappa_{\ell}^{m}\tilde{Y}_{\ell}^{m}-\sum_{\ell\leq\lambda
k}\sum_{m\in\iota_{\ell}}\frac{\ell+1}{z_{\ell}(k)-\ell}R_{\ell}^{m}\tilde
{Y}_{\ell}^{m},
\]
where $\kappa_{\ell}^{m}$ and $R_{\ell}^{m}$ are given by (\ref{eq:kappa-R}).
We note that the coefficients $\kappa_{\ell}^{m}$ are controlled by
Lemma~\ref{lemma:ntimesLomega-analytic}. For the coefficients $R_{\ell}^{m}$
we estimate
\[
\sum_{\ell\leq\lambda k}\sum_{m\in\iota_{\ell}}\ell|R_{\ell}^{m}|\lesssim
k^{3/2}\left(  \sum_{\ell\leq\lambda k}\sum_{m\in\iota_{\ell}}\ell|R_{\ell
}^{m}|^{2}\right)  ^{1/2}\overset{(\ref{m1/2curldiva})}{\lesssim}k^{3/2}%
\Vert{\mathbf{r}}_{T}\Vert_{-1/2,\operatorname{curl}_{\Gamma}}{\lesssim
}k^{3/2}\Vert{\mathbf{r}}\Vert_{\operatorname{curl},\Omega,1}.
\]
\emph{2.~step (symbol estimates):} We have%
\[
\left\vert z_{\ell}\left(  k\right)  +1\right\vert ^{2}=\left(
1+\operatorname{Re}\left(  z_{\ell}\left(  k\right)  \right)  \right)
^{2}+\left(  \operatorname{Im}z_{\ell}\left(  k\right)  \right)  ^{2}.
\]
From (\cite[(2.6.23)]{Nedelec01}) we know that 
$1 \leq - \operatorname{Re} z_\ell(k) \leq \ell+1$ so that 
$\left(  1+\operatorname{Re}%
\left(  z_{\ell}\left(  k\right)  \right)  \right)  \leq0$ and hence
$\left\vert 1+\operatorname{Re}\left(  z_{\ell}\left(  k\right)  \right)
\right\vert \leq\left\vert -\ell+\operatorname{Re}\left(  z_{\ell}\left(
k\right)  \right)  \right\vert $. Thus%
\[
\left\vert \dfrac{1}{\ell}\left(  \dfrac{z_{\ell}\left(  k\right)  +1}%
{z_{\ell}\left(  k\right)  {-}\ell}\right)  \right\vert \leq\frac{1}{\ell}%
\leq\frac{2}{\ell+1}%
\]
and also we have%
\[
\left\vert \dfrac{\ell}{z_{\ell}\left(  k\right)  {-}\ell}\right\vert
\leq\frac{\ell}{\left\vert z_{\ell}\left(  k\right)  +1\right\vert }\leq
C\left\{
\begin{array}
[c]{ll}%
\ell & \ell\leq\lambda k,\\
1 & \ell\geq\lambda k.
\end{array}
\right.
\]

\emph{3.~step (analyticity classes of $g_{D}^{\ast}$):} We claim: there are
$C$, $\gamma^{\prime}>0$ independent of $k$ and ${\mathbf{r}}$ such that
\begin{equation}
g_{D}^{\ast}\in{\mathcal{A}}(Ck^{3/2}\Vert{\mathbf{r}}\Vert
_{\operatorname{curl},\Omega,1},\gamma^{\prime},{\mathcal{U}}_{\Gamma})
\label{eq:analyticity-gD}%
\end{equation}
Using (\ref{estgradYtilde}) and the symbol estimates of Step~2 we estimate
with the abbreviation $\lambda_{\ell}=\ell(\ell+1)$
and the constant $\gamma_{\mathcal{A},\Gamma}^\prime$ of 
Lemma~\ref{lemma:ntimesLomega-analytic} 
\begin{align}
\left\Vert \nabla^{n}g_{D}^{\ast}\right\Vert _{L^{2}\left(  \mathcal{U}%
_{\Gamma}\right)  }  &  \lesssim\gamma^{n}\biggl\{\sum_{\ell\leq
k\gamma_{\mathcal{A},\Gamma}^{\prime}}\sum_{m\in\iota_{\ell}}|\kappa_{\ell
}^{m}|\max\left\{  \sqrt{\lambda_{\ell}},n\right\}  ^{n}+\sum_{\ell
>k\gamma_{\mathcal{A},\Gamma}^{\prime}}\sum_{m\in\iota_{\ell}}\frac{1}{\ell
+1}|\kappa_{\ell}^{m}|\max\left\{  \sqrt{\lambda_{\ell}},n\right\}
^{n}\label{nabladiriext}\\
&  \quad\mbox{}+\sum_{\ell\leq\lambda k}\sum_{m \in \iota_{\ell}} \ell|R_{\ell}^{m}|\max\left\{
\sqrt{\lambda_{\ell}},n\right\}  ^{n}\biggr\}=:\gamma^{n}\biggl\{\cdots
\biggr\}.\nonumber
\end{align}
We estimate the expression $\{\cdot\}$ in curly braces further with
Lemma~\ref{lemma:ntimesLomega-analytic} and suitable $\tilde \gamma$: 
\begin{align*}
\biggl\{\cdots\biggr\}  &  \lesssim k^{3/2}\Vert{\mathbf{r}}\Vert
_{\operatorname{curl},\Omega,1}\tilde \gamma^n \max\{k,n\}^n 
+\sum_{\ell>k\gamma_{\mathcal{A},\Gamma
}^{\prime}}\sum_{m\in\iota_{\ell}}\frac{1}{\ell+1}|\kappa_{\ell}^{m}|\left[
\lambda_{\ell}^{n/2}+n^{n}\right] \\
&  \lesssim k^{3/2}\Vert{\mathbf{r}}\Vert_{\operatorname{curl},\Omega
,1}\tilde \gamma^n \max\{k,n\}^n 
+k^{-1}\sum_{\ell>k\gamma_{\mathcal{A},\Gamma}^{\prime}}\sum_{m\in
\iota_{\ell}}|\kappa_{\ell}^{m}|\left[  \lambda_{\ell}^{n/2}+n^{n}\right] \\
&  \lesssim k^{3/2}\Vert{\mathbf{r}}\Vert_{\operatorname{curl},\Omega
,1}\tilde \gamma^n \max\{k,n\}^n +k\tilde{\gamma}^{n}n^{n}\Vert{\mathbf{r}}\Vert_{\operatorname{curl}%
,\Omega,1};
\end{align*}
in the last step, we employed
(\ref{eq:lemma:ntimesLomega-analytic-12}) once with $\alpha=n$ and once with
$\alpha=0$. This shows (\ref{eq:analyticity-gD}). 
We also note
\begin{equation}
\Vert g_{D}\Vert_{H^{1/2}(\Gamma)}\lesssim\Vert g_{D}^{\ast}\Vert
_{H^{1}({\mathcal{U}})}\overset{(\ref{eq:analyticity-gD})}{\lesssim}%
k^{5/2}\Vert{\mathbf{r}}\Vert_{\operatorname{curl},\Omega,1}.
\label{eq:H1/2-gD}%
\end{equation}
\emph{3.~step (interior regularity):} Given $\mathbf{r}\in\mathbf{X}$, the
function $Z=\mathcal{N}_{4}^{\mathcal{A}}\mathbf{r}$ solves
(\ref{recastDiri}). First, interior regularity as derived in
\cite[Prop.~{5.5.1}]{MelenkHabil} gives%
\begin{equation}
\Vert\nabla^{n+1}Z\Vert_{L^{2}(\Omega\setminus\mathcal{U}_{\Gamma})}\leq
C\gamma^{n}\left(  n+1\right)  ^{n}\|Z\|_{H^1(\Omega)}  \leq
C\gamma^{n}\left(  n+1\right)  ^{n}\left\Vert g_{D}\right\Vert _{H^{1/2}%
\left(  \Gamma\right)  }\qquad\forall n\in\mathbb{N}_{0}.
\label{interiorestimateZ}%
\end{equation}
This is the desired bound away from $\Gamma$ in view of (\ref{eq:H1/2-gD}).

\emph{4.~step:} For the behavior of $Z$ near $\Gamma$, we write $Z=Z_{0}%
-g_{D}^{\ast}$. Near $\Gamma$, the function $Z_{0}$ satisfies
\begin{equation}
-\Delta Z_{0}=-\Delta g_{D}^{\ast}\ \text{in }\mathcal{U}_{\Gamma}%
\quad\text{and\ }Z_{0}|_{\Gamma}=0. \label{eqforZ0}%
\end{equation}
{}From (\ref{eq:analyticity-gD}) we get $\Delta g_{D}^{\ast}\in{\mathcal{A}}(C
k^{7/2},\gamma,{\mathcal{U}}_{\Gamma})$ for suitably adjusted constants $C$,
$\gamma> 0$. Also we have%
\begin{align}
\left\Vert \nabla Z_{0}\right\Vert _{L^{2}\left(  \mathcal{U}_{\Gamma}\right)
}  &  \leq\left\Vert \nabla Z\right\Vert +\left\Vert \nabla g_{D}^{\ast
}\right\Vert _{L^{2}\left(  \mathcal{U}_{\Gamma}\right)  }\lesssim\left\Vert
g_{D}\right\Vert _{H^{1/2}\left(  \Gamma\right)  }+Ck^{5/2}\left\Vert
\mathbf{r}\right\Vert _{\operatorname*{curl},\Omega,1}
\overset{(\ref{eq:H1/2-gD})}{\lesssim} k^{5/2} \|{\mathbf{r}}%
\|_{\operatorname{curl},\Omega,1}. \label{gdstarinline}%
\end{align}
One concludes with the aid of Theorem~\ref{thm:poisson} (and suitable
localization as well as flattening of the boundary) that $Z_{0}$ in
(\ref{eqforZ0}) satisfies%
\[
\nabla Z_{0}\in\mathcal{A}\left(  Ck^{5/2}\left\Vert \mathbf{r}\right\Vert
_{\operatorname*{curl},\Omega,1},\gamma,\mathcal{U}_{\Gamma}\right)  ,
\]
again with adjusted constants $C$, $\gamma$. This in turn implies $\left.
\nabla Z\right\vert _{\mathcal{U}_{\Gamma}}\in\mathcal{A}\left(
Ck^{5/2}\left\Vert \mathbf{r}\right\Vert _{\operatorname*{curl},\Omega
,1},\gamma,\mathcal{U}_{\Gamma}\right)  $.
\endproof

\begin{proposition}
\label{PropRegN1A} Let $\Omega= B_{1}(0)$. There exist positive constants
$C_{\mathcal{A},1}$ and $\gamma_{\mathcal{A},1}$ depending only on $\Gamma$
and the cut-off parameter $\lambda$ such that for any $\mathbf{v}\in
\mathbf{X}$
\[
\mathcal{N}_{1}^{\mathcal{A}}\mathbf{v}\in\mathcal{A}\left(  C_{\mathcal{A}%
,1}k^{3}\left\Vert \mathbf{v}\right\Vert _{\operatorname*{curl},\Omega
,k},\gamma_{\mathcal{A},1},\Omega\right)  .
\]

\end{proposition}

%

\proof
For given $\mathbf{v}\in\mathbf{X}$, the solution $\mathbf{z}:=\mathcal{N}%
_{1}^{\mathcal{A}}\mathbf{v}$ can be split into%
\[
\mathbf{z}=\mathcal{N}_{3}^{\mathcal{A}}\mathbf{v}+\mathbf{\tilde{z}}%
\]
with the solution $\mathbf{\tilde{z}}\in\mathbf{X}$ of%
\[
A_{k}\left(  \mathbf{w},\mathbf{\tilde{z}}\right)  =-\operatorname*{i}%
kb_{k}\left(  \mathbf{w}^{\operatorname*{curl}},\left(  L_{\Omega}%
\mathbf{v}\right)  ^{\operatorname*{curl}}\right)  \qquad\forall\mathbf{w}%
\in\mathbf{X}.
\]
{}From (\ref{solformdualprob}) we get the following representation of the
solution%
\begin{equation}
\mathbf{\tilde{z}}=-\operatorname*{i}k\int_{\Gamma}g_{-k}\left(  \left\Vert
\mathbf{\cdot}-\mathbf{y}\right\Vert \right)  T_{-k}\left(  \left(  L_{\Omega
}\mathbf{v}\right)  ^{\operatorname*{curl}}\right)  \left(  \mathbf{y}\right)
d\Gamma_{\mathbf{y}}.\nonumber
\end{equation}
Fourier expansion of $\left(  L_{\Gamma}\mathbf{v}_{T}\right)
^{\operatorname*{curl}}$ leads to (cf. (\ref{DefTk}))
\[%
\mbox{\boldmath$ \mu$}%
:=T_{-k}\left(  L_{\Omega}\mathbf{v}\right)  ^{\operatorname*{curl}}%
=\sum_{1\leq\ell\leq\lambda k}\overline{\left(  \frac{z_{\ell}\left(
k\right)  +1}{\operatorname*{i}k}\right)  }\sum_{m\in\iota_{\ell}}v_{\ell}%
^{m}\mathbf{T}_{\ell}^{m}.
\]
An extension of $%
\mbox{\boldmath$ \mu$}%
^{\ast}$ is given by%
\[%
\mbox{\boldmath$ \mu$}%
^{\ast}:=\sum_{1\leq\ell\leq\lambda k}\overline{\left(  \frac{z_{\ell}\left(
k\right)  +1}{\operatorname*{i}k}\right)  }\sum_{m\in\iota_{\ell}}v_{\ell}%
^{m}\mathbf{\tilde{T}}_{\ell}^{m},
\]
where $\mathbf{\tilde{T}}_{\ell}^{m}:=\nabla\tilde{Y}_{\ell}^{m}%
\times\mathbf{n}^{\ast}$ with $\mathbf{n}^{\ast}\left(  \mathbf{x}\right)
:=\mathbf{x}/\left\Vert \mathbf{x}\right\Vert $ and $\tilde{Y}_{\ell}^{m}$ as
in (\ref{estgradYtilde}). Now we proceed as in the proof of Proposition
\ref{PropN3}. First, we derive the estimates%
\begin{align*}
\left\Vert
\mbox{\boldmath$ \mu$}%
\right\Vert _{\mathbf{H}_{T}^{-3/2}\left(  \Gamma\right)  }^{2}  &
=\sum_{1\leq\ell\leq\lambda k}\left(  \ell\left(  \ell+1\right)  \right)
^{-1/2}\left\vert \frac{z_{\ell}\left(  k\right)  +1}{\operatorname*{i}%
k}\right\vert ^{2}\sum_{m\in\iota_{\ell}}\left\vert v_{\ell}^{m}\right\vert
^{2}\\
&  \overset{\text{Lem. \ref{Lemzlest}}}{\leq}\sum_{1\leq\ell\leq\lambda
k}\left(  \ell\left(  \ell+1\right)  \right)  ^{-1/2}\left(  1+\frac{\ell}%
{k}\right)  ^{2}\sum_{m\in\iota_{\ell}}\left\vert v_{\ell}^{m}\right\vert ^{2}
\leq C\left\Vert \mathbf{v}_{T}\right\Vert _{\mathbf{H}_{\operatorname*{curl}%
}^{-1/2}\left(  \Gamma\right)  }^{2}\leq C\left\Vert \mathbf{v}\right\Vert
_{\operatorname*{curl},\Omega,1}^{2}%
\end{align*}
and%
\begin{align*}
\left\Vert \nabla^{n}%
\mbox{\boldmath$ \mu$}%
^{\ast}\right\Vert _{L^{2}\left(  \mathcal{U}_{\Gamma}\right)  }  &  \leq
\sum_{1\leq\ell\leq\lambda k}\left\vert \left(  \frac{z_{\ell}\left(
k\right)  +1}{\operatorname*{i}k}\right)  \right\vert \sum_{m\in\iota_{\ell}%
}\left\vert v_{\ell}^{m}\right\vert \left\Vert \nabla^{n}\mathbf{\tilde{T}%
}_{\ell}^{m}\right\Vert _{L^{2}\left(  \mathcal{U}_{\Gamma}\right)  }\\
&  \overset{\text{(\ref{eq:analyticity-extension})}}{\leq}C\max\left\{
k,n+1\right\}  ^{n+1}\gamma^{n}\sum_{1\leq\ell\leq\lambda k}\left(
1+\frac{\ell}{k}\right)  \sum_{m\in\iota_{\ell}}\left\vert v_{\ell}%
^{m}\right\vert \leq\tilde{C}k^{2}\max\left\{  k,n+1\right\}  ^{n}%
\tilde{\gamma}^{n}\left\Vert \mathbf{v}\right\Vert _{\operatorname*{curl}%
,\Omega,1}.
\end{align*}
The application of $\mathcal{S}_{-k}^{\operatorname*{Hh}}$ to $%
\mbox{\boldmath$ \mu$}%
$ can then be estimated by%
\[
\left\Vert k\mathcal{S}_{-k}^{\operatorname*{Hh}}\left(  T_{-k}\left(
L_{\Omega}\mathbf{v}\right)  ^{\operatorname*{curl}}\right)  \right\Vert
_{\mathcal{H},\Omega}\overset{\text{(\ref{S-kHHg})}}{\leq}Ck^{3}\left\Vert
\mbox{\boldmath$ \mu$}%
\right\Vert _{\mathbf{H}_{T}^{-3/2}\left(  \Gamma\right)  }\leq Ck^{3}%
\left\Vert \mathbf{v}\right\Vert _{\operatorname*{curl},\Omega,1}%
\]
and $k\mathcal{S}_{-k}^{\operatorname*{Hh}}\left(  T_{-k}\left(  L_{\Omega
}\mathbf{v}\right)  ^{\operatorname*{curl}}\right)  \in\mathcal{A}\left(
Ck^{2}\left\Vert \mathbf{v}\right\Vert _{\operatorname*{curl},\Omega,k}%
,\gamma,\Omega\right)  $. The combination with Proposition \ref{PropN3} leads
to the assertion.%
\endproof


\section{Approximation Operators for $S_{p+1}({\mathcal{T}}_{h})$ and
$\boldsymbol{\mathcal{N}}_{p}^{\operatorname*{I}}({\mathcal{T}}_{h}%
)$\label{SecApproxOps}}


The relevant $hp$ finite element spaces have been introduced in
Section~\ref{SecCurlConfFEM}. A key property of these spaces is that both
lines in the diagram in Fig.~\ref{fig:commuting-diagram-allgemein} are exact
sequences, \cite{nedelec80,hiptmair-acta,Monkbook}. In particular, therefore,
(\ref{esp}) is satisfied for the pair $(S_{h},{\mathbf{X}}_{h})=(S_{p+1}%
({\mathcal{T}}_{h}),\boldsymbol{\mathcal{N}}_{p}^{\operatorname*{I}%
}({\mathcal{T}}_{h}))$. The operators $\Pi_{h}^{E}$ and $\Pi_{h}^{F}$ of
Assumption~\ref{AdiscSp} are constructed to satisfy the stronger
\textquotedblleft commuting diagram property\textquotedblright\ that make the
diagram in Fig.~\ref{fig:commuting-diagram} commute. In that case, the
operator $\Pi_{h}^{E}$ is defined on the space $\prod_{K\in{\mathcal{T}}_{h}%
}{\mathbf{H}}^{1}(K,\operatorname*{curl})\cap{\mathbf{X}}\supset\{{\mathbf{u}%
}\in{\mathbf{H}}^{1}(\Omega)\,|\,\operatorname{curl}{\mathbf{u}}%
\in\operatorname{curl}{\mathbf{X}}_{h}\}$. \begin{figure}[ptb]%
\[
\begin{CD} \mathbb{R} @> >> H^1(\Omega) @> \nabla >> H(\Omega, \operatorname*{curl}) @> \operatorname*{curl} >> H(\Omega,\operatorname*{div}) @> \operatorname*{div} >> L^2(\Omega)\\ @AA {\iota}A @AA {\iota} A @AA {\iota} A @AA {\iota} A @AA {\iota} A \\ \mathbb{R} @> >> S_{p+1}({\mathcal{T}}_{h}) @> \nabla >> \boldsymbol{\mathcal N}^I_p({\mathcal{T}}_{h}) @> \operatorname*{curl} >> \mathbf{RT}_p({\mathcal{T}}_{h}) @> \operatorname*{div} >> Z_p({\mathcal T}_h) \end{CD}
\]
\caption{Continuous and discrete exact sequences.}%
\label{fig:commuting-diagram-allgemein}%
\end{figure}


\subsection{Optimal Simultaneous $hp$-Approximation in $L^{2}$ and
${\mathbf{H}}(\operatorname{curl})$}

We restrict our attention to approximation operators that are constructed element-by-element.

\begin{definition}
[element-by-element construction]\label{def:element-by-element} An operator
$\widehat{\Pi}^{\operatorname*{grad}}:H^{2}(\widehat{K})\rightarrow
{{\mathcal{P}}}_{p+1}$ is said to admit an \emph{element-by-element
construction} if the operator $\Pi^{\operatorname*{grad}}$ defined 
elementwise on $H^{1}(\Omega
)\cap\prod_{K\in{\mathcal{T}}_{h}}H^{2}(K)$ by
$(\Pi^{\operatorname*{grad}}u)|_{K}:=(\widehat{\Pi}^{\operatorname*{grad}%
}(u\circ F_{K}))\circ F_{K}^{-1}$ maps into the conforming subspace
$S^{p+1}({\mathcal{T}}_{h})\subset H^{1}(\Omega)$.

An operator $\widehat{\Pi}^{\operatorname*{curl}}:{\mathbf{H}}^{1}%
(\widehat{K},\operatorname*{curl})\rightarrow\boldsymbol{\mathcal{N}}%
_{p}^{\operatorname*{I}}(\widehat{K})$ is said to admit an
\emph{element-by-element construction} if the operator $\Pi
^{\operatorname*{curl}}$ defined elementwise on ${\mathbf{H}}(\Omega,\operatorname*{curl})\cap
\prod_{K\in{\mathcal{T}}_{h}}{\mathbf{H}}^{1}(K,\operatorname*{curl})$ 
by $(\Pi^{\operatorname*{curl}}u)|_{K}:=(F_{K}^{\prime}%
)^{-T}(\widehat{\Pi}^{\operatorname*{curl}}((F_{K}^{\prime})^{\intercal
}\mathbf{u}\circ F_{K}))\circ F_{K}^{-1}$ maps into the conforming subspace
$\boldsymbol{\mathcal{N}}_{p}^{\operatorname*{I}}({\mathcal{T}}_{h}%
)\subset{\mathbf{H}}(\Omega,\operatorname*{curl})$.

An operator $\widehat{\Pi}^{\operatorname*{div}}:{\mathbf{H}}^{1}%
(\widehat{K},\operatorname*{div})\rightarrow\mathbf{RT}_{p}(\widehat{K})$ is
said to admit an \emph{element-by-element construction} if the operator
$\Pi^{\operatorname*{div}}$ defined elementwise on ${\mathbf{H}}(\Omega,\operatorname*{div})\cap
\prod_{K\in{\mathcal{T}}_{h}}{\mathbf{H}}^{1}(K,\operatorname*{div})$ 
by
\[
(\Pi^{\operatorname*{div}}u)|_{K}:=(\operatorname*{det}(F_{K}^{\prime}%
))^{-1}F_{K}^{\prime}(\widehat{\Pi}^{\operatorname*{div}}(\operatorname*{det}%
F_{K}^{\prime})(F_{K}^{\prime})^{-1}\mathbf{u}\circ F_{K}))\circ F_{K}^{-1}%
\]
maps into the conforming subspace $\mathbf{RT}_{p}({\mathcal{T}}_{h}%
)\subset{\mathbf{H}}(\Omega,\operatorname*{div})$. Finally, any operator
$\widehat{\Pi}^{L^{2}}:{L}^{2}(\widehat{K})\rightarrow{\mathcal{P}}%
_{p}(\widehat{K})$ leads to a globally defined $L^{2}(\Omega)$-conforming
operator by the following \emph{element-by-element} construction: $(\Pi
^{L^{2}}u)|_{K}:=(\widehat{\Pi}^{L^{2}}(\mathbf{u}\circ F_{K}))\circ
F_{K}^{-1}$.
\end{definition}

As it is typical, we will construct such operators on the reference
tetrahedron $\widehat{K}$ in such a way that the value of the operator
restricted to a lower-dimensional entity (i.e., a vertex, an edge, or a face)
is completely determined by the value of the function on that entity. For
scalar functions the operator $\Pi_{p}$ of \cite[Def.~{5.3}, Thm.~{B.4}%
]{MelenkSauterMathComp} is an example that we will build on; it can be viewed
as a variant of the projection-based interpolation technique of
\cite{demkowicz08} that also underlies the construction of the operator
$\Pi_{h}^{E}$. Important features of the construction of $\Pi_{p}$ are:
$(\Pi_{p}u)(V)=u(V)$ for all vertices $V$; it has the property that $(\Pi
_{p}u)|_{e}$ is the projection of $u|_{e}$ onto a space of polynomials of
degree $p$ on each edge $e$ under the constraint that $\Pi_{p}u$ has already
been fixed in the vertices; it has the property that $(\Pi_{p}u)|_{f}$ is the
(constrained) projection of $u|_{f}$ onto a space of polynomials of degree $p$
on each face $f$ under the constraint that $\Pi_{p}u$ has already been fixed
on edges. We note that the fact that $\Pi_{p}$ is a (constrained) projection
on polynomial spaces for the edges and faces makes the definition independent
of the parametrization of the edges and faces of the reference tetrahedron.

We need approximation operators suitable for the approximation in the norm
$\|\cdot\|_{\operatorname*{curl},\Omega,k}$. Such an operator can be defined
in an element-by-element fashion on the reference tetrahedron:

\begin{lemma}
\label{lemma:Hcurl-interpolation} Let $s>3/2$. There exist operators
$\widehat{\Pi}_{p}^{\operatorname*{curl},s}:{\mathbf{H}}^{s}(\widehat{K}%
)\rightarrow\boldsymbol{\mathcal{N}}_{p}^{\operatorname*{I}}(\widehat{K})$
with the following properties:

\begin{enumerate}
[(i)]

\item $\widehat{\Pi}_{p}^{\operatorname*{curl},s}$ admits an
element-by-element construction as in Definition~\ref{def:element-by-element}.

\item For $p\geq s-1$ we have
\begin{equation}
(p+1)\Vert\mathbf{u}-\widehat{\Pi}_{p}^{\operatorname*{curl},s}\mathbf{u}%
\Vert_{\mathbf{L}^{2}(\widehat{K})}+\Vert\mathbf{u}-\widehat{\Pi}%
_{p}^{\operatorname*{curl},s}\mathbf{u}\Vert_{{\mathbf{H}}^{1}(\widehat{K}%
)}\leq Cp^{-(s-1)}|\mathbf{u}|_{{\mathbf{H}}^{s}(\widehat{K})}.
\label{eq:lemma:Hcurl-interpolation-10}%
\end{equation}

\item Let $\mathbf{u}$ satisfy, for some $C_{\mathbf{u}}$, $\overline{\gamma}%
$, $h>0$, and $\kappa\geq1$
\begin{equation}
\Vert\nabla^{n}\mathbf{u}\Vert_{\mathbf{L}^{2}(\widehat{K})}\leq
C_{\mathbf{u}}(\overline{\gamma}h)^{n}\max\{n,\kappa\}^{n}\qquad\forall
n\in\mathbb{N},\quad n\geq2.
\end{equation}
Assume furthermore
\begin{equation}
h+\kappa h/p\leq\widetilde{C}.
\end{equation}
Then there exist constants $C$, $\sigma>0$ depending solely on $\widetilde{C}$
and $\overline{\gamma}$ such that
\begin{equation}
\Vert\mathbf{u}-\widehat{\Pi}_{p}^{\operatorname*{curl},s}\mathbf{u}%
\Vert_{W^{2,\infty}(\widehat{K})}\leq CC_{\mathbf{u}}\left[  \left(  \frac
{h}{\sigma+h}\right)  ^{p+1}+\left(  \frac{\kappa h}{\sigma p}\right)
^{p+1}\right]  . \label{eq:lemma:Hcurl-interpolation-20}%
\end{equation}

\end{enumerate}
\end{lemma}

%

\proof
Let $\Pi_{p}:H^{s}(\widehat{K})\rightarrow{\mathcal{P}}_{p}$, $s>3/2$, be the
scalar polynomial approximation operator\footnote{In \cite[Def.~{5.3},
Thm.~{B.4}]{MelenkSauterMathComp} the element-by-element construction of the
polynomial approximation on the reference element only fixes $\Pi_{p}$ on
$\partial\widehat{K}$. The operator $\Pi_{p}$ is fully determined by adding a
final minimization step to fix the interior degrees of freedom on the
reference element.} of \cite[Def.~{5.3}, Thm.~{B.4}]{MelenkSauterMathComp}. A
key property of $\Pi_{p}$ is that, as described above, one has that the
restriction of $\Pi_{p}u$ to a vertex, edge, or face is completely determined
by $u$ restricted to that entity. We write, e.g., for a face $f$: $\Pi
_{p}\left(  \left.  u\right\vert _{f}\right)  :=\left.  \left(  \Pi
_{p}u\right)  \right\vert _{f}$. We define the operator $\widehat{\Pi}%
_{p}^{\operatorname*{curl},s}:\mathbf{H}^{s}(\widehat{K})\rightarrow\left(
{\mathcal{P}}_{p}\right)  ^{3}\subset\boldsymbol{\mathcal{N}}_{p}%
^{\operatorname*{I}}(\widehat{K})$ by componentwise application to
$\mathbf{u}=\left(  u_{i}\right)  _{i=1}^{3}$, i.e.,
\[
\Pi^{\operatorname*{curl},s}{\mathbf{u}}:=\left(  \Pi_{p}u_{i}\right)
_{i=1}^{3}.
\]
\emph{1.~step:} We show that $\widehat{\Pi}_{p}^{\operatorname*{curl},s}$
admits an element-by-element construction. We show this by asserting that the
tangential component $\Pi_{T}\widehat{\Pi}_{p}^{\operatorname*{curl}%
,s}\mathbf{u}$ depends solely on the tangential component $\Pi_{T}\mathbf{u}$.
Fix a face $f$ of $\widehat{K}$ with normal $\mathbf{n}$. Note that
$\mathbf{n}$ is constant on $f$. The tangential component of
$\Pi^{\operatorname*{curl},s}{\mathbf{u}}$ on $f$ is
\[
\left.  \left(  \Pi_{T}\left(  \Pi^{\operatorname*{curl},s}{\mathbf{u}%
}\right)  \right)  \right\vert _{f}=\left.  \left(  \left(  \Pi_{p}%
u_{i}\right)  _{i=1}^{3}\right)  \right\vert _{f}-\left.  \left(  \sum
_{j=1}^{3}n_{j}\Pi_{p}u_{j}\right)  \mathbf{n}\right\vert _{f}.
\]
Using that $(\Pi_{p}{\mathbf{u}})|_{f}$ is completely determined by the values
of ${\mathbf{u}}$ on $f$ and using that the normal vector ${\mathbf{n}}$ is
constant on $f$, we infer with the understanding that $\Pi_{p}$ acts
componentwise on a vector-valued object
\[
\left.  \left(  \Pi_{T}\left(  \Pi^{\operatorname*{curl},s}{\mathbf{u}%
}\right)  \right)  \right\vert _{f}=\Pi_{p}\left(  {\mathbf{u}}|_{f}\right)
-\Pi_{p}\left(  \left(  {\mathbf{n}}\cdot{\mathbf{u}}|_{f}\right)
{\mathbf{n}}\right)  =\Pi_{p}\left(  {\mathbf{u}}|_{f}-\left(  \left(
{\mathbf{n}}\cdot{\mathbf{u}}|_{f}\right)  {\mathbf{n}}\right)  \right)
=\Pi_{p}(\Pi_{T}{\mathbf{u}})|_{f},
\]
which is the desired claim.

\emph{2.~step:} Estimate (\ref{eq:lemma:Hcurl-interpolation-10}) then follows
from \cite[Thm.~{B.4}]{MelenkSauterMathComp}.

\emph{3.~step:} From \cite[Lemma~{C.2}]{MelenkSauterMathComp}, we conclude
that (\ref{eq:lemma:Hcurl-interpolation-20}) holds.%
\endproof


\subsection{Projection Operators with Commuting Diagram Property}


The operator $\Pi_{p}^{\operatorname*{curl},s}$, which is obtained by an
elementwise use of $\widehat{\Pi}_{p}^{\operatorname*{curl},s}$ of
Lemma~\ref{lemma:Hcurl-interpolation}
(cf.~Definition~\ref{def:element-by-element} for the transformation rule) has
($p$-optimal) approximation properties in $\Vert\cdot\Vert
_{\operatorname*{curl},\Omega,k}$ as it has simultaneously $p$-optimal
approximation properties in $L^{2}$ and $H^{1}$. However, it is not a
projection and does not have the commuting diagram property. We therefore
present a second operator, $\widehat{\Pi}^{\operatorname{curl},c}$, in
Theorem~\ref{thm:projection-based-interpolation} with this property. The
construction is given in \cite{melenk_rojik_2018} and similar to that in
\cite{demkowicz-buffa05,demkowicz08}. We point out that the difference between
Theorem~\ref{thm:projection-based-interpolation} from \cite{melenk_rojik_2018}
and the works \cite{demkowicz-buffa05,demkowicz08} is that, by assuming
$H^{2}(\widehat{K})$- and ${\mathbf{H}}^{1}(\widehat{K},\operatorname{curl}%
)$-regularity, Theorem~\ref{thm:projection-based-interpolation} features the
optimal $p$-dependence, thus avoiding the factors of $\log p$ present in
\cite{demkowicz-buffa05,demkowicz08}. \begin{figure}[ptb]%
\[
\begin{CD} \mathbb{R} @> >> H^2(\widehat K) @> \nabla >> {\mathbf H}^1(\widehat K, \operatorname*{curl}) @> \operatorname*{curl} >> {\mathbf H}^1(\widehat K, \operatorname*{div}) @> \operatorname*{div} >> H^1(\widehat K)\\ @VV V @VV \widehat\Pi^{\operatorname{grad},c}_{p+1} V @VV \widehat \Pi^{\operatorname{curl},c}_p V @VV \widehat\Pi^{\operatorname{div},c}_{p} V @VV \widehat \Pi^{L^2}_p V \\ \mathbb{R} @> >> {\mathcal P}_{p+1}(\widehat K) @> \nabla >> \boldsymbol{\mathcal N}^I_p(\widehat K) @> \operatorname*{curl} >> \mathbf{RT}_p(\widehat K) @> \operatorname*{div} >> {\mathcal P}_p(\widehat K) \end{CD}
\]
\caption{Commuting diagram on reference element $\protect\widehat{K}$.}%
\label{fig:commuting-diagram}%
\end{figure}\begin{figure}[ptb]%
\[
\begin{CD} \mathbb{R} @> >> H^2(\Omega) @> \nabla >> {\mathbf H}^1(\Omega, \operatorname*{curl}) @> \operatorname*{curl} >> {\mathbf H}^1(\Omega, \operatorname*{div}) @> \operatorname*{div} >> H^1(\Omega)\\ @VV V @VV \Pi^{\operatorname{grad},c}_{p+1} V @VV \Pi^{\operatorname{curl},c}_p V @VV \Pi^{\operatorname{div},c}_{p} V @VV \Pi^{L^2}_p V \\ \mathbb{R} @> >> S_{p+1}({\mathcal{T}}_{h}) @> \nabla >> \boldsymbol{\mathcal N}^I_p({\mathcal{T}}_{h}) @> \operatorname*{curl} >> \mathbf{RT}_p({\mathcal{T}}_{h}) @> \operatorname*{div} >> Z_p({\mathcal{T}}_{h}) \end{CD}
\]
\caption{ Commuting diagram on mesh ${\mathcal{T}}_{h}$}%
\label{fig:commuting-diagram-global}%
\end{figure}

\begin{theorem}
[\!\!\cite{melenk_rojik_2018}]\label{thm:projection-based-interpolation} There
are linear \emph{projection} operators $\widehat{\Pi}_{p+1}%
^{\operatorname*{grad},c}$, $\widehat{\Pi}_{p}^{\operatorname*{curl},c}$,
$\widehat{\Pi}_{p}^{\operatorname*{div},c}$, $\widehat{\Pi}_{p}^{L^{2}}$ such
that the following holds:

\begin{enumerate}
[(i)]

\item \label{item:thm:projection-based-interpolation-i} The diagram in
Fig.~\ref{fig:commuting-diagram} commutes.

\item \label{item:thm:projection-based-interpolation-ii} The operators
$\widehat{\Pi}_{p+1}^{\operatorname*{grad},c}$, $\widehat{\Pi}_{p}%
^{\operatorname*{curl},c}$, $\widehat{\Pi}_{p}^{\operatorname*{div},c}$,
$\widehat{\Pi}_{p}^{L^{2}}$ admit element-by-element constructions as in
Definition~\ref{def:element-by-element}. The global operators $\Pi
_{p+1}^{\operatorname*{grad},c}$, $\Pi_{p}^{\operatorname*{curl},c}$, $\Pi
_{p}^{\operatorname*{div},c}$, $\Pi_{p}^{L^{2}}$ obtained from the operators
$\widehat{\Pi}_{p+1}^{\operatorname*{grad},c}$, $\widehat{\Pi}_{p}%
^{\operatorname*{curl},c}$, $\widehat{\Pi}_{p}^{\operatorname*{div},c}$,
$\widehat{\Pi}_{p}^{L^{2}}$ by an element-by-element construction are also
linear projection operators and the diagram in
Fig.~\ref{fig:commuting-diagram-global} commutes.

\item \label{item:thm:projection-based-interpolation-iii} For all $\varphi\in
H^{2}(\widehat{K})$ there holds
\[
\Vert\varphi-\widehat{\Pi}_{p+1}^{\operatorname*{grad},c}\varphi\Vert
_{H^{s}(\widehat{K})}\leq C_{s}p^{-1-(1-s)}\inf_{v\in{\mathcal{P}}%
_{p+1}(\widehat{K})}\Vert\varphi-v\Vert_{H^{2}(\widehat{K})},\qquad
s\in\lbrack0,1].
\]

\item \label{item:thm:projection-based-interpolation-iv} For all ${\mathbf{u}%
}\in{\mathbf{H}}^{1}(\widehat{K},\operatorname{curl})$ there holds
\[
\Vert{\mathbf{u}}-\widehat{\Pi}_{p}^{\operatorname*{curl},c}{\mathbf{u}}%
\Vert_{{\mathbf{H}}(\widehat{K},\operatorname{curl})}\leq Cp^{-1}%
\inf_{{\mathbf{v}}\in\boldsymbol{\mathcal{N}}_{p}^{\operatorname*{I}%
}(\widehat{K})}\Vert{\mathbf{u}}-{\mathbf{v}}\Vert_{{\mathbf{H}}%
^{1}(\widehat{K},\operatorname{curl})}.
\]

\item \label{item:thm:projection-based-interpolation-v} For all $k\geq1$ and
all ${\mathbf{u}}\in{\mathbf{H}}^{k}(\widehat{K})$ with $\operatorname*{curl}%
{\mathbf{u}}\in\boldsymbol{\mathcal{P}}_{p}$ there holds
\begin{equation}
\Vert{\mathbf{u}}-\widehat{\Pi}_{p}^{\operatorname*{curl},c}{\mathbf{u}}%
\Vert_{\mathbf{L}^{2}(\widehat{K})}\leq C_{k}p^{-k}\Vert{\mathbf{u}}%
\Vert_{\mathbf{H}^{k}(\widehat{K})}.
\label{eq:lemma:projection-based-interpolation-approximation-10}%
\end{equation}
If $p\geq k-1$, then the full norm $\Vert{\mathbf{u}}\Vert_{{\mathbf{H}}%
^{k}(\widehat{K})}$ can be replaced with the seminorm $|{\mathbf{u}%
}|_{{\mathbf{H}}^{k}(\widehat{K})}$.
\end{enumerate}
\end{theorem}

\subsection{$hp$-FEM Approximation\label{sec:hp-approximation}}

Our $hp$-FEM convergence result will be formulated for the specific class of
meshes which have been introduced in Section \ref{SecCurlConfFEM}. For such
meshes, we can formulate approximation results for both, the operators
$\Pi_{p}^{\operatorname*{curl},s}$ and $\Pi_{p}^{\operatorname*{curl},c}$. In
both cases, we will need to relate functions defined on $K$ to their pull-back
to the reference tetrahedron $\widehat{K}$. The appropriate transformations
are described in Definition~\ref{def:element-maps}: For scalar functions
$\varphi$ defined on $K$ and vector-valued functions ${\mathbf{u}}$ defined on
$K$, we let
\begin{equation}
\widehat{\varphi}=\varphi\circ F_{K},\qquad\widehat{\mathbf{u}}=(F_{K}%
^{\prime})^{\top}({\mathbf{u}}\circ F_{K}).
\label{eq:transformation-convention}%
\end{equation}

\begin{lemma}
\label{lemma:scaling} Let the regular mesh ${\mathcal{T}}_{h}$ satisfy
Assumption~\ref{def:element-maps}.

\begin{enumerate}
[(i)]

\item With implied constants depending only on $C_{\operatorname{affine}}$,
$C_{\operatorname{metric}}$, $\gamma$ there holds for all $K \in{\mathcal{T}%
}_{h}$
\begin{align}
&  |\widehat{\varphi}|_{H^{j}(\widehat{K})}\sim h_{K}^{j-3/2}|\varphi
|_{H^{j}(K)},\quad j\in\{0,1\},\qquad|\widehat{\varphi}|_{H^{2}(\widehat{K}%
)}\lesssim h_{K}^{2-3/2}\Vert\varphi\Vert_{H^{2}(K)}%
,\label{eq:lemma:scaling-10}\\
&  \left\Vert \widehat{\mathbf{u}}\right\Vert _{\mathbf{L}^{2}(\widehat{K}%
)}\sim h_{K}^{1-3/2}\left\Vert {\mathbf{u}}\right\Vert _{\mathbf{L}^{2}%
(K)},\quad\left\Vert \operatorname{curl}\widehat{\mathbf{u}}\right\Vert
_{\mathbf{L}^{2}(\widehat{K})}\sim h_{K}^{2-3/2}\left\Vert \operatorname{curl}%
{\mathbf{u}}\right\Vert _{\mathbf{L}^{2}(K)},\quad|\widehat{\mathbf{u}%
}|_{{\mathbf{H}}^{2}(\widehat{K})}\lesssim h_{K}^{3-3/2}\Vert{\mathbf{u}}%
\Vert_{{\mathbf{H}}^{2}(K)}. \label{eq:lemma:scaling-20}%
\end{align}

\item Let $\overline{\gamma}>0$. Then there exist $\gamma^{\prime}$, $C>0$
depending only on $\overline{\gamma}$ and the constants of
Assumption~\ref{def:element-maps} such that%
\begin{align}
\Vert\nabla^{n}\varphi\Vert_{\mathbf{L}^{2}(K)}  &  \leq C_{\varphi}%
\overline{\gamma}^{n}\max\{n,k\}^{n}\quad\forall n\in{\mathbb{N}}_{0}%
\quad\Longrightarrow\quad\Vert\nabla^{n}\widehat{\varphi}\Vert_{\mathbf{L}%
^{2}(\widehat{K})}\leq CC_{\varphi}h_{K}^{-3/2}\left(  h_{k}\gamma^{\prime
}\right)  ^{n}\max\{n,k\}^{n}\quad\forall n\in{\mathbb{N}}_{0}%
,\label{eq:lemma:scaling-30}\\
\Vert\nabla^{n}{\mathbf{u}}\Vert_{\mathbf{L}^{2}(K)}  &  \leq C_{\mathbf{u}%
}\overline{\gamma}^{n}\max\{n,k\}^{n}\quad\forall n\in{\mathbb{N}}_{0}%
\quad\Longrightarrow\quad\Vert\nabla^{n}\widehat{\mathbf{u}}\Vert
_{\mathbf{L}^{2}(\widehat{K})}\leq Ch_{K}^{1-3/2}C_{\mathbf{u}}\left(
h_{K}\gamma^{\prime}\right)  ^{n}\max\{n,k\}^{n}\quad\forall n\in{\mathbb{N}%
}_{0}. \label{eq:lemma:scaling-40}%
\end{align}

\end{enumerate}
\end{lemma}

\begin{proof}
We will not show (\ref{eq:lemma:scaling-10}). For (\ref{eq:lemma:scaling-20}),
the first and third estimate in (\ref{eq:lemma:scaling-20}) follow by
inspection, the second equivalence follows from (cf., e.g., \cite[Cor.~{3.58}%
]{Monkbook})
\[
F_{K}^{\prime}\operatorname{curl}\widehat{\mathbf{u}}=(\operatorname{det}%
F_{K}^{\prime})(\operatorname{curl}{\mathbf{u}})\circ F_{K}.
\]
The implications (\ref{eq:lemma:scaling-30}), (\ref{eq:lemma:scaling-40}) are
obtained by similar arguments. We will therefore focus on
(\ref{eq:lemma:scaling-40}). Recalling that the element map $F_{K}$ has the
form $F_{K}=R_{K}\circ A_{K}$, we introduce the function
$\widetilde{\mathbf{u}}:=(R_{K}^{\prime})^{\top}{\mathbf{u}}\circ R_{K}$,
which is defined on $\widetilde{K}:=A_{K}(\widehat{K})$. Using
\cite[Lemma~{4.3.1}]{MelenkHabil} (and noting as in the proof
\cite[Lemma~{C.1}]{MelenkSauterMathComp} that the original 2d arguments
extends to 3d), we get the existence of $C$, $\widetilde{\gamma}$, which
depend solely on the constants of Assumption~\ref{def:element-maps} and on
$\overline{\gamma}$, such that
\[
\Vert\nabla^{n}\widetilde{\mathbf{u}}\Vert_{\mathbf{L}^{2}(\widetilde{K})}\leq
CC_{\mathbf{u}}(\widetilde{\gamma})^{n}\max\{n,k\}^{n}\quad\forall
n\in{\mathbb{N}}_{0}.
\]
Next, we observe $\widehat{\mathbf{u}}=(A_{K}^{\prime})^{\top}%
\widetilde{\mathbf{u}}\circ A_{K}$. Using that $A_{K}$ is affine, it is easy
to deduce
\[
\Vert\nabla^{n}\widehat{\mathbf{u}}\Vert_{\mathbf{L}^{2}(\widehat{K})}\leq
CC_{\mathbf{u}} h_K^{1-3/2} \max\{n,k\}^{n}\left(  h_{K}\gamma^{\prime}\right)  ^{n}%
\quad\forall n\in{\mathbb{N}}_{0},
\]
which is the desired estimate.
\end{proof}

\begin{lemma}
\label{lemma:Picurls-approximation} Let ${\mathcal{T}}_{h}$ be a regular mesh
satisfying Assumption~\ref{def:element-maps} and assume $p\geq1$.

\begin{enumerate}
[(i)]

\item \label{item:lemma:Picurls-approximation-i}
\[
\Vert{\mathbf{u}}-\Pi_{p}^{\operatorname*{curl},s}{\mathbf{u}}\Vert
_{\mathbf{L}^{2}(K)}+h_{K}p^{-1}\Vert\operatorname{curl}({\mathbf{u}}-\Pi
_{p}^{\operatorname*{curl},s}{\mathbf{u}})\Vert_{\mathbf{L}^{2}(K)}\lesssim
h_{K}^{2}p^{-2}\Vert{\mathbf{u}}\Vert_{{\mathbf{H}}^{2}(K)}.
\]

\item \label{item:lemma:Picurls-approximation-ii} Let $\widetilde{C}>0$ be
given. If ${\mathbf{u}}$ satisfies (\ref{eq:lemma:scaling-40}), then there
exist $C$, $\sigma>0$ depending only on $\widetilde{C}$ and $\overline{\gamma
}$ and the constants of Assumption~\ref{def:element-maps} such that under the
side constraint
\begin{equation}
h_{K}+\frac{kh_{K}}{p}\leq\widetilde{C}%
\end{equation}
the following approximation result holds:
\begin{equation}
\Vert{\mathbf{u}}-\Pi_{p}^{\operatorname*{curl},s}{\mathbf{u}}\Vert
_{\mathbf{L}^{2}(K)}+h_{K}p^{-1}\Vert\operatorname{curl}({\mathbf{u}}-\Pi
_{p}^{\operatorname*{curl},s}{\mathbf{u}})\Vert_{\mathbf{L}^{2}(K)}%
\lesssim\left(  \left(  \frac{h_{K}}{h_{K}+\sigma}\right)  ^{p+1}+\left(
\frac{kh_{K}}{\sigma p}\right)  ^{p+1}\right)  . \label{approxpropPpcurlsL2}%
\end{equation}

\end{enumerate}
\end{lemma}

\begin{proof}
\emph{Proof of (\ref{item:lemma:Picurls-approximation-i}):} From
Lemma~\ref{lemma:Hcurl-interpolation} with $s=2$ we have on the reference
tetrahedron
\[
p\Vert\widehat{\mathbf{u}}-\widehat{\Pi}_{p}^{\operatorname*{curl}%
,s}\widehat{\mathbf{u}}\Vert_{\mathbf{L}^{2}(\widehat{K})}+\Vert
\widehat{\mathbf{u}}-\widehat{\Pi}_{p}^{\operatorname*{curl},s}%
\widehat{\mathbf{u}}\Vert_{{\mathbf{H}}^{1}(\widehat{K})}\lesssim
p^{-1}|\widehat{\mathbf{u}}|_{{\mathbf{H}}^{2}(\widehat{K})}.
\]
Hence, using (\ref{eq:lemma:scaling-20}) we infer
\[
ph_{K}^{1-3/2}\Vert\mathbf{u}-\Pi_{p}^{\operatorname*{curl},s}\mathbf{u}%
\Vert_{\mathbf{L}^{2}(K)}+h_{K}^{2-3/2}\Vert\operatorname{curl}({\mathbf{u}%
}-\Pi_{p}^{\operatorname*{curl},s}{\mathbf{u}})\Vert_{\mathbf{L}^{2}%
(K)}\lesssim h_{K}^{3-3/2}p^{-1}\Vert{\mathbf{u}}\Vert_{{\mathbf{H}}^{2}(K)}.
\]
\emph{Proof of (\ref{item:lemma:Picurls-approximation-ii}):} We proceed as
above. The transformation rules of Lemma~\ref{lemma:scaling} and
Lemma~\ref{lemma:Hcurl-interpolation} give
\begin{equation}
\left\Vert \widehat{\mathbf{u}}-\widehat{\Pi}_{p}^{\operatorname*{curl}%
,s}\widehat{\mathbf{u}}\right\Vert _{W^{2,\infty}(\widehat{K})}\leq
CC_{\mathbf{u}}h_{K}^{1-3/2}\left(  \left(  \frac{h_{K}}{h_{K}+\sigma}\right)
^{p+1}+\left(  \frac{kh_{K}}{\sigma p}\right)  ^{p+1}\right)  .
\label{W2infrefel}%
\end{equation}
Since the norm $\left\Vert \cdot\right\Vert _{W^{2,\infty}\left(
\widehat{K}\right)  }$ is stronger than $\left\Vert \cdot\right\Vert
_{\mathbf{L}^{2}\left(  \widehat{K}\right)  }$ and $\left\Vert
\operatorname*{curl}\cdot\right\Vert _{\mathbf{H}^{1}\left(  \widehat{K}%
\right)  }$ the result follows by transforming back to $K$ using
Lemma~\ref{lemma:scaling}.
\end{proof}

For the operator $\Pi_{p}^{\operatorname*{curl},c}$ we have the following
approximation results:

\begin{lemma}
\label{lemma:Picurlcom-approximation} Let ${\mathcal{T}}_{h}$ be a regular
mesh satisfying Assumption~\ref{def:element-maps}. Then for $p\geq1$:

\begin{enumerate}
[(i)]

\item \label{item:lemma:Picurlcom-approximation-i}
\begin{equation}
h_{K}^{-1}\Vert{\mathbf{u}}-\Pi_{p}^{\operatorname*{curl},c}{\mathbf{u}}%
\Vert_{\mathbf{L}^{2}(K)}+\Vert\operatorname{curl}({\mathbf{u}}-\Pi
_{p}^{\operatorname*{curl},c}{\mathbf{u}})\Vert_{\mathbf{L}^{2}(K)}\leq
Ch_{K}(p+1)^{-1}\Vert{\mathbf{u}}\Vert_{{\mathbf{H}}^{2}(K)}.
\end{equation}

\item \label{item:lemma:Picurlcom-approximation-ii} Assume the hypotheses of
Lemma~\ref{lemma:Picurls-approximation},
(\ref{item:lemma:Picurls-approximation-ii}). Then%
\[
h_{K}^{-1}\Vert{\mathbf{u}}-\Pi_{p}^{\operatorname*{curl},c}{\mathbf{u}}%
\Vert_{\mathbf{L}^{2}(K)}+\Vert\operatorname{curl}({\mathbf{u}}-\Pi
_{p}^{\operatorname*{curl},c}{\mathbf{u}})\Vert_{\mathbf{L}^{2}(K)}%
\lesssim\left(  \left(  \frac{h_{K}}{h_{K}+\sigma}\right)  ^{p}+\frac{k}%
{p}\left(  \frac{kh_{K}}{\sigma p}\right)  ^{p}\right)  .
\]

\item \label{item:lemma:Picurlcom-approximation-iii} For ${\mathbf{u}}%
\in{\mathbf{H}}^{1}(K,\operatorname{curl})$ with $\operatorname{curl}%
\widehat{\mathbf{u}}\in({\mathcal{P}}_{p}(\widehat{K}))^{3}$ there holds
\[
\Vert{\mathbf{u}}-\Pi_{p}^{\operatorname*{curl},c}{\mathbf{u}}\Vert
_{\mathbf{L}^{2}(K)}\leq Ch_{K}p^{-1}\Vert{\mathbf{u}}\Vert_{{\mathbf{H}}%
^{1}(K)}.
\]

\end{enumerate}
\end{lemma}

\begin{proof}
\emph{Proof of (\ref{item:lemma:Picurlcom-approximation-i}):} Using
Lemma~\ref{lemma:scaling}, we get from
Theorem~\ref{thm:projection-based-interpolation} and the assumption $p\geq1$
\begin{align*}
&  h_{K}^{-1}\Vert{\mathbf{u}}-\Pi_{p}^{\operatorname*{curl},c}{\mathbf{u}%
}\Vert_{\mathbf{L}^{2}(K)}+\Vert\operatorname{curl}({\mathbf{u}}-\Pi
_{p}^{\operatorname*{curl},c}{\mathbf{u}})\Vert_{\mathbf{L}^{2}(K)}\sim
h_{K}^{-2+3/2}\Vert\widehat{\mathbf{u}}-\widehat{\Pi}_{p}%
^{\operatorname*{curl},c}\widehat{\mathbf{u}}\Vert_{{\mathbf{H}}%
(\widehat{K},\operatorname{curl})}\\
&  \lesssim p^{-1}h_{K}^{-2+3/2}\inf_{{\mathbf{v}}\in{\mathcal{P}}_{p}^{3}%
}\Vert\widehat{\mathbf{u}}-{\mathbf{v}}\Vert_{{\mathbf{H}}^{1}(\widehat{K}%
,\operatorname{curl})}\lesssim p^{-1}h_{K}^{-2+3/2}|\widehat{\mathbf{u}%
}|_{{\mathbf{H}}^{2}(\widehat{K})}\lesssim p^{-1}h_{K}\Vert{\mathbf{u}}%
\Vert_{{\mathbf{H}}^{2}(K)}.
\end{align*}
\emph{Proof of (\ref{item:lemma:Picurlcom-approximation-ii}):} We start as
above. The novel aspect is that $\inf_{{\mathbf{v}}\in\boldsymbol{\mathcal{N}%
}_{p}^{\operatorname*{I}}(\widehat{K})}\Vert\widehat{\mathbf{u}}-{\mathbf{v}%
}\Vert_{{\mathbf{H}}^{1}(\widehat{K},\operatorname{curl})}$ can be estimated
as in the proof of Lemma~\ref{lemma:Picurls-approximation}%
\begin{align*}
h_{K}^{-1}\left\Vert {\mathbf{u}}-\Pi_{p}^{\operatorname*{curl},c}{\mathbf{u}%
}\right\Vert _{\mathbf{L}^{2}(K)}+\left\Vert \operatorname{curl}\left(
{\mathbf{u}}-\Pi_{p}^{\operatorname*{curl},c}{\mathbf{u}}\right)  \right\Vert
_{\mathbf{L}^{2}(K)}  &  \lesssim p^{-1}h_{K}^{-2+3/2}\inf_{{\mathbf{v}}%
\in{\mathcal{P}}_{p}^{3}}\Vert\widehat{\mathbf{u}}-{\mathbf{v}}\Vert
_{{\mathbf{H}}^{1}(\widehat{K},\operatorname{curl})}\\
&  \overset{\text{(\ref{W2infrefel})}}{\lesssim}CC_{\mathbf{u}}\left(  \left(
\frac{h_{K}}{h_{K}+\sigma}\right)  ^{p}+\frac{k}{p}\left(  \frac{kh_{K}%
}{\sigma p}\right)  ^{p}\right)  .
\end{align*}

\emph{Proof of (\ref{item:lemma:Picurlcom-approximation-iii}):} With
Lemma~\ref{lemma:scaling} and Theorem~\ref{thm:projection-based-interpolation}%
, (\ref{item:thm:projection-based-interpolation-v}) we estimate
\[
\Vert{\mathbf{u}}-\Pi_{p}^{\operatorname*{curl},c}{\mathbf{u}}\Vert
_{\mathbf{L}^{2}(K)}\sim h^{-1+3/2}\Vert\widehat{\mathbf{u}}-\widehat{\Pi}%
_{p}^{\operatorname*{curl},c}\widehat{\mathbf{u}}\Vert_{\mathbf{L}%
^{2}(\widehat{K})}\lesssim h^{-1+3/2}p^{-1}|\widehat{\mathbf{u}}%
|_{{\mathbf{H}}^{1}(\widehat{K})}\lesssim h^{-1+3/2}p^{-1}h_{K}^{2-3/2}%
\Vert{\mathbf{u}}\Vert_{{\mathbf{H}}^{1}(K)},
\]
which completes the proof.
\end{proof}

\begin{lemma}
\label{lemma:resolution-condition-detail} Let $h_{0}$, $\sigma$, $c_{2}>0$,
$\alpha\ge0$, $k \ge1$. Then for every $\varepsilon>0$ there is $c_{1} > 0$
(depending only on $h_{0}$, $\sigma$, $c_{2}$, $\alpha$, $\varepsilon$) such
that for any $h \in(0,h_{0}]$ and $p \ge \max\{1,c_{2} \log k\}$ with $kh/p \leq c_{1}$
there holds
\[
k^{\alpha}\left(  \frac{h}{h+\sigma}\right)  ^{p} \leq\varepsilon.
\]

\end{lemma}

\begin{proof}
Without loss of generality, we restrict to $\varepsilon \in (0,1]$.
The case $k \in[1,\operatorname*{e}]$ is easily seen. For $k \ge \operatorname*{e}$,
we note that $h \mapsto h/(h+\sigma)$ is monotone increasing and we consider the two cases, ``$h$
small'' and ``$h$ large''. For the first case we
select $\overline{h} \in(0,h_{0}]$ such that
\begin{equation}
\label{eq:lemma:resolution-condition-details-10}
\alpha + c_{2}  \ln \bigl({\overline{h}}/({\overline{h}+\sigma})\bigr)
\leq \ln\varepsilon \leq 0. 
\end{equation}
We emphasize that $\overline{h}$ depends
only on $\varepsilon$, $\alpha$, and $c_2$ but is independent of $k$.
By the monotonicity of $h \mapsto h/(h+\sigma)$
and the constraint $p \ge c_2 \ln k$
we have for $0 < h \leq\overline{h}$ and in view of
$\ln k \ge \ln \operatorname*{e} = 1$ the estimate
\begin{align*}
\ln \left( k^\alpha \left(\frac{h}{h+\sigma}\right)^p \right)
\leq
\ln \left( k^\alpha \left(\frac{\overline{h}}{\overline{h}+\sigma}\right)^{c_2 \log k} \right)
 = \underbrace{\ln k}_{\ge 1}
\underbrace{ \left[ \alpha + c_2 \ln \frac{\overline{h}}{\overline{h}+\sigma} \right]
           }_{\leq \ln \varepsilon \leq 0}
\leq \ln \varepsilon.
\end{align*}
For the second case, i.e., $h \in(\overline{h},h_{0}]$, we
fix $c_1$ such that
\begin{equation}
\label{eq:lemma:resolution-condition-details-20}
\alpha + \frac{\overline{h}}{c_1}  \ln \frac{h_0}{h_0+\sigma} \leq \ln\varepsilon \leq 0. 
\end{equation}
We note that $c_1$ depends only on $\varepsilon$, $\alpha$, $h_0$, and $c_2$. (Recall that
$\overline{h}$ depends only on $\varepsilon$, $\alpha$, $c_2$).
For $h \in (\overline{h} , h_0]$ we get from $p \ge k h/c_{1}\ge k \overline{h}/c_{1}$ and the monotonicity
of $h \mapsto h/(h+\sigma)$
\begin{align*}
\ln \left(
k^{\alpha}\left(  \frac{h}{h+\sigma}\right)  ^{p}
     \right)
& \leq \ln \left( k^{\alpha}\left(
\frac{h_{0}}{h_{0}+\sigma}\right)  ^{k \overline{h}/c_{1}}
         \right)
 = \alpha \ln k + \frac{k \overline{h}}{c_1} \ln \frac{h_0}{h_0 + \sigma}  \\
& \!\!\!\!\stackrel{\ln k \leq k}{\leq} \underbrace{ k}_{\ge \operatorname{e} \ge 1}
\underbrace{ \left[ \alpha + \frac{\overline{h}}{c_1} \ln \frac{h_0}{h_0 + \sigma}\right]
           }_{\leq \ln \varepsilon \leq 0}
\leq \ln \varepsilon.
\end{align*}
\end{proof}


\appendix


\section{Proof of Lemma~\ref{Lemzlest}\label{SecA}}

In this appendix we prove Lemma~\ref{Lemzlest}. The first two estimates in
(\ref{estk/z+1}) are proved in the following lemma.

\begin{lemma}
For any $\lambda>1$ there holds%
\[
\frac{k}{\left\vert z_{n}\left(  k\right)  +1\right\vert }\leq\left\{
\begin{array}
[c]{ll}%
{2\sqrt{2}} k & n\in{\mathbb{N}}_{0},\\
{2\sqrt{2}} \left(  \frac{2}{\lambda}+1\right)  \dfrac{k}{\left(  n+1\right)
} & n>\lambda k^{2}.
\end{array}
\right.
\]

\end{lemma}

%

\proof
We follow the reasoning in \cite[Thm.~{2.6.1}]{Nedelec01}. The coefficient
$z_{n}\left(  k\right)  $ can be expressed by%
\[
z_{n}\left(  k\right)  =-\frac{\left(  m_{n}^{2}\right)  ^{\prime}}{m_{n}^{2}%
}+k\frac{\operatorname*{i}}{m_{n}^{2}},
\]
where%
\[
\mu=\left(  2n+1\right)  ^{2}\quad\text{and\quad}\left\{
\begin{array}
[c]{ll}%
m_{n}^{2}=%
{\displaystyle\sum\limits_{m=0}^{n}}
\dfrac{\delta_{m}\left(  \mu\right)  }{k^{2m}}, & \left(  m_{n}^{2}\right)
^{\prime}=%
{\displaystyle\sum_{m=0}^{n}}
\left(  m+1\right)  \dfrac{\delta_{m}\left(  \mu\right)  }{k^{2m}},\\
\delta_{m}\left(  \mu\right)  =\dfrac{\left(  2m\right)  !}{\left(  m!\right)
^{2}16^{m}}\gamma_{m}\left(  \mu\right)  , & \gamma_{m}\left(  \mu\right)  :=%
{\displaystyle\prod\limits_{s=1}^{m}}
\left(  \mu-\left(  2s-1\right)  ^{2}\right)  .
\end{array}
\right.
\]
Define%
\[
a_{m,n}:=\delta_{m}\left(  \left(  2n+1\right)  ^{2}\right)  =\frac{\left(
2m\right)  !\left(  n+m\right)  !}{\left(  m!\right)  ^{2}4^{m}\left(
n-m\right)  !}.
\]
With the function%
\begin{equation}
\rho_{n}\left(  k\right)  :=\frac{%
{\displaystyle\sum\limits_{m=0}^{n}}
\dfrac{a_{m,n}}{k^{2m}}}{%
{\displaystyle\sum\limits_{m=0}^{n}}
\left(  m+1\right)  \dfrac{a_{m,n}}{k^{2m}}} \label{defrhon}%
\end{equation}
we estimate%
\begin{align}
\frac{k}{\left\vert z_{n}\left(  k\right)  +1\right\vert }  &  =\frac
{km_{n}^{2}}{\left\vert m_{n}^{2}-\left(  m_{n}^{2}\right)  ^{\prime
}+k\operatorname*{i}\right\vert }\overset{\mu=\left(  2n+1\right)  ^{2}}{\leq
}\sqrt{2}k\frac{%
{\displaystyle\sum\limits_{m=0}^{n}}
\dfrac{a_{m,n}}{k^{2m}}}{k+%
{\displaystyle\sum\limits_{m=1}^{n}}
m\dfrac{a_{m,n}}{k^{2m}}}\label{middleomegaoverz}\\
&  \leq{2\sqrt{2}}k\rho_{n}\left(  k\right)  \overset{\text{ansatz}}{\leq
}{2\sqrt{2}}k\left(  \frac{k^{2}+\beta}{k^{2}+C_{n}\beta}\right)  .
\label{eq:ansatz-for-rho_n}%
\end{align}
The ansatz (\ref{eq:ansatz-for-rho_n}) is equivalent to%
\[
\left(  k^{2}+C_{n}\beta\right)
{\displaystyle\sum\limits_{m=0}^{n}}
\dfrac{a_{m,n}}{k^{2m}}\leq\left(  k^{2}+\beta\right)
{\displaystyle\sum\limits_{m=0}^{n}}
\left(  m+1\right)  \dfrac{a_{m,n}}{k^{2m}},
\]
which, by multiplying out and rearranging terms, is equivalent to%
\begin{align*}
&  k^{2}a_{0,n}+%
{\displaystyle\sum\limits_{m=0}^{n-1}}
\left(  a_{m+1,n}+C_{n}\beta a_{m,n}\right)  \dfrac{1}{k^{2m}}+C_{n}%
\beta\dfrac{a_{n,n}}{k^{2n}}\\
&  \leq k^{2}a_{0,n}+%
{\displaystyle\sum\limits_{m=0}^{n-1}}
\left(  \left(  m+2\right)  a_{m+1,n}+\beta\left(  m+1\right)  a_{m,n}\right)
\dfrac{1}{k^{2m}}+\beta\left(  n+1\right)  \dfrac{a_{n,n}}{k^{2n}}.
\end{align*}
Hence, we have to stipulate%
\begin{align*}
\left(  a_{m+1,n}+C_{n}\beta a_{m,n}\right)   &  \leq\left(  \left(
m+2\right)  a_{m+1,n}+\beta\left(  m+1\right)  a_{m,n}\right)  ,\qquad
m=0,\ldots,n-1,\\
C_{n}  &  \leq n+1.
\end{align*}
We select $C_{n}:=(n+1)$ and insert this in the left-hand side of the first
condition to obtain%
\[
0\leq\left(  m+1\right)  a_{m+1,n}+\beta\left(  m+1-\left(  n+1\right)
\right)  a_{m,n}.
\]
Inserting the definitions of $a_{m,n}$ leads to%
\[
0\leq\left(  m+1\right)  \frac{\left(  2m+2\right)  !\left(  n+m+1\right)
!}{\left(  \left(  m+1\right)  !\right)  ^{2}4^{m+1}\left(  n-\left(
m+1\right)  \right)  !}+\beta\left(  m+1-\left(  n+1\right)  \right)
\frac{\left(  2m\right)  !\left(  n+m\right)  !}{\left(  m!\right)  ^{2}%
4^{m}\left(  n-m\right)  !}.
\]
This is equivalent to
\[
\beta\left(  n-m\right)  \frac{\left(  2m\right)  !\left(  n+m\right)
!}{\left(  m!\right)  ^{2}4^{m}\left(  n-m\right)  !}\leq\left(  m+1\right)
\frac{\left(  2m+2\right)  !\left(  n+m+1\right)  !}{\left(  \left(
m+1\right)  !\right)  ^{2}4^{m+1}\left(  n-\left(  m+1\right)  \right)  !}%
\]
and in turn leads to the condition
\[
\beta\leq\left(  m+1\right)  \frac{\left(  2m+1\right)  \left(  2m+2\right)
\left(  n+m+1\right)  }{\left(  m+1\right)  ^{2}4}=\left(  m+\frac{1}%
{2}\right)  \left(  n+m+1\right)  ,\qquad m=0,\ldots,n-1.
\]
We select $\beta=\frac{n+1}{2}$, which finally leads to
\[
\frac{k}{\left\vert z_{n}\left(  k\right)  +1\right\vert }\leq{2\sqrt{2}%
}k\left(  \frac{2k^{2}+n+1}{2k^{2}+\left(  n+1\right)  ^{2}}\right)
\leq\left\{
\begin{array}
[c]{ll}%
{2\sqrt{2}}k & \forall n\in{\mathbb{N}}_{0}\\
{2\sqrt{2}}\left(  \frac{2}{\lambda}+1\right)  \dfrac{k}{\left(  n+1\right)
}, & n+1>\lambda k^{2}.
\end{array}
\right.
\]%
\endproof

The proof of the third estimate in (\ref{estk/z+1}) is more technical and is
the assertion of the next lemma.

\begin{lemma}
For every $\lambda_{0} > 1$ there is $C_{0} > 0$ depending only on
$\lambda_{0}$ such that
\[
\frac{n+1}{\left\vert z_{n}\left(  k\right)  +1\right\vert }\leq C_{0}%
\qquad\forall n\geq\lambda_{0} k.
\]

\end{lemma}

%

\proof
Recall the definition of the function $\rho_{n}$ in (\ref{defrhon}). We will
prove%
\[
\left(  n+1\right)  \rho_{n}\left(  k\right)  \leq\tilde{C}_{0}\qquad\forall
n\geq\lambda_{0}k
\]
{}from which the statement follows in view of (\ref{middleomegaoverz}).

\textbf{Step 1: } We claim that $\rho_{n}$ is monotone increasing with respect
to $k$. To see this, we compute
\[
\rho_{n}^{\prime}\left(  k\right)  =\frac{-\left(
{\displaystyle\sum\limits_{m=0}^{n}}
\left(  m+1\right)  \dfrac{a_{m,n}}{k^{2m}}\right)  \left(
{\displaystyle\sum\limits_{m=0}^{n}}
2m\dfrac{a_{m,n}}{k^{2m+1}}\right)  +%
{\displaystyle\sum\limits_{m=0}^{n}}
\left(  \dfrac{a_{m,n}}{k^{2m}}\right)
{\displaystyle\sum\limits_{m=0}^{n}}
\left(  2m\right)  \left(  m+1\right)  \dfrac{a_{m,n}}{k^{2m+1}}}{\left(
{\displaystyle\sum\limits_{m=0}^{n}}
\left(  m+1\right)  \dfrac{a_{m,n}}{k^{2m}}\right)  ^{2}}.
\]
Thus, it is sufficient to prove that the numerator (denoted by $d_{n}\left(
k\right)  $) is positive. We write{%
\[
d_{n}\left(  k\right)  =2{\displaystyle\sum\limits_{m=0}^{n}}%
{\displaystyle\sum\limits_{\ell=0}^{n}\ell}\left(  \ell-m\right)
\dfrac{a_{\ell,n}a_{m,n}}{k^{2m+2\ell+1}}.
\]
We now exploit the fact that the coefficients $a_{\ell,n}$ are non-negative.
The double sum on the right-hand side can be interpreted as a quadratic form.
Note that we have, for vectors $\mathbf{x}$ and matrices ${\mathbf{B}}$,
\[
2\mathbf{x}^{\intercal}{\mathbf{B}}\mathbf{x}=\mathbf{x}^{\intercal
}({\mathbf{B}}^{\intercal}+{\mathbf{B}})\mathbf{x}\geq0
\]
if the vector $\mathbf{x}$ has non-negative entries and the symmetric part
$1/2({\mathbf{B}}^{\intercal}+{\mathbf{B}})$ of the matrix ${\mathbf{B}}$ has
non-negative entries. For ${\mathbf{B}}_{\ell,m}:=\ell\left(  \ell-m\right)  $
we compute
\[
{\mathbf{B}}_{\ell,m}+{\mathbf{B}}_{m,\ell}=(\ell-m)^{2}\geq0.
\]
} {\textbf{Step 2: }{}The monotonicity of $\rho_{n}$ shown in Step~1 implies
for $n\geq\lambda_{0}k$%
\begin{equation}
\rho_{n}\left(  k\right)  \leq\rho_{n}\left(  n/\lambda_{0}\right)  =\frac{%
{\displaystyle\sum\limits_{m=0}^{n}}
a_{m,n}\left(  \frac{\lambda_{0}}{n}\right)  ^{2m}}{%
{\displaystyle\sum\limits_{m=0}^{n}}
\left(  m+1\right)  a_{m,n}\left(  \frac{\lambda_{0}}{n}\right)  ^{2m}}%
=:\rho_{n}^{\operatorname{I}}. \label{defrhonI}%
\end{equation}
We next show that the dominant contribution to the sums in (\ref{defrhonI})
arises from few coefficients with index $m$ close to $n\sqrt{1-\lambda
_{0}^{-2}}$. To that end, we analyze the coefficients $a_{m,n}$ with
Stirling's formula in the form
\[
\sqrt{2\pi}\exp\left(  \frac{1}{12}\right)  \sqrt{n+1}\left(  \frac
{n}{\operatorname*{e}}\right)  ^{n}\geq n!\geq\sqrt{n+1}\left(  \frac
{n}{\operatorname*{e}}\right)  ^{n}.
\]
Upon setting $C_{1}:=2\pi\exp(1/6)$ and $C_{2}:=(2\pi)^{-3/2}\exp(-1/4)$, we
get%
\begin{align}
a_{m,n}  &  =\frac{\left(  2m\right)  !\left(  n+m\right)  !}{\left(
m!\right)  ^{2}4^{m}\left(  n-m\right)  !}\leq C_{1} \underbrace{ \frac
{\sqrt{n+m+1}}{\sqrt{n-m+1}} }_{\leq\sqrt{2m+1} \text{ for $m \leq n$}}
\frac{\sqrt{2m+1}}{m+1}\frac{\left(  n+m\right)  ^{n+m}}{\left(  n-m\right)
^{n-m}\operatorname*{e}^{2m}}\label{eq:a_mn-20}\\
&  \leq2C_{1}\frac{\left(  n+m\right)  ^{n+m}}{\left(  n-m\right)
^{n-m}\operatorname*{e}^{2m}},\\
a_{m,n}  &  \geq C_{2}\frac{\sqrt{2m+1}\sqrt{n+m+1}}{(m+1)\sqrt{n-m+1}}%
\frac{\left(  n+m\right)  ^{n+m}}{\left(  n-m\right)  ^{n-m}\operatorname*{e}%
^{2m}}. \label{amnbelow}%
\end{align}
The dominant contribution of $a_{m,n}(\lambda_{0}/n)^{2m}$ is
\[
b_{m,n}:=\frac{\left(  n+m\right)  ^{n+m}}{\left(  n-m\right)  ^{n-m}%
\operatorname*{e}^{2m}}\left(  \frac{\lambda_{0}}{n}\right)  ^{2m}.
\]
The maximum of $m\mapsto b_{m,n}$ in the interval $\left[  0,n\right]
\subset\mathbb{R}$ is attained at $\tilde{m}=n\mu_{0}$ with $\mu_{0}%
=\sqrt{1-\lambda_{0}^{-2}}$ and value%
\[
\tilde{b}_{n}=c_{\mu_{0}}^{n}\quad\text{with }c_{\mu_{0}}=\frac{\left(
1+\mu_{0}\right)  ^{1+\mu_{0}}}{\left(  1-\mu_{0}\right)  ^{1-\mu_{0}}}\left(
\frac{\lambda_{0}}{\operatorname*{e}}\right)  ^{2\mu_{0}}.
\]
We also introduce the factor
\[
f_{m,n}:=\frac{\sqrt{n+m+1}}{\sqrt{n-m+1}}\frac{\sqrt{2m+1}}{m+1},
\]
so as to be able to describe $a_{m,n}\left(  \frac{\lambda_{0}}{n}\right)
^{2m}\sim f_{m,n}b_{m,n}$ uniformly in $m$, $n$. }

\textbf{Case 1:} We consider the range
\[
0\leq n\leq\max\{\frac{2}{1-\mu_{0}},\frac{2}{\mu_{0}},\frac{\lambda_{0}^{2}%
}{\mu_{0}c_{0}},c_{5}\},
\]
where the parameter $c_{0}$ is given by Lemma~\ref{LemTaylor} (with
$\lambda=\lambda_{0}$ there) and $c_{5}$ is defined in (\ref{eq:c_5}); both
constants depend solely on $\lambda_{0}$. This is a finite set so
\[
\sup_{0\leq n\leq\max\{\frac{2}{1-\mu_{0}},\frac{2}{\mu_{0}},\frac{\lambda
_{0}^{2}}{\mu_{0}c_{0}},c_{5}\}}(n+1)\rho_{n}^{\operatorname*{I}}=:\tilde
{C}_{1}<\infty
\]
depends solely on $\lambda_{0}$.

\textbf{Case 2:} We assume
\begin{equation}
n\geq\max\{\frac{2}{1-\mu_{0}},\frac{2}{\mu_{0}},\frac{\lambda_{0}^{2}}%
{\mu_{0}c_{0}},c_{5}\}. \label{eq:condition-on-n}%
\end{equation}
We split the summations $\sum_{m=0}^{n}$ in the representation of $\rho
_{n}^{\operatorname{I}}$ (cf. (\ref{defrhonI})) as $S_{n}^{\operatorname{I}%
}+S_{n}^{\operatorname{II}}$ with
\[
S_{n}^{\operatorname{I}}:=\sum_{n\tilde{\delta}_{0}\leq m\leq n}\frac{a_{m,n}%
}{\tilde{b}_{n}}\left(  \frac{\lambda_{0}}{n}\right)  ^{2m},\qquad
S_{n}^{\operatorname{II}}:=\sum_{{0}\leq m<n\tilde{\delta}_{0}}\frac{a_{m,n}%
}{\tilde{b}_{n}}\left(  \frac{\lambda_{0}}{n}\right)  ^{2m},
\]
where
\begin{equation}
\tilde{\delta}_{0}:=\mu_{0}^{3}.
\end{equation}
In view of%
\[
\min\left\{  m+1:m\geq n\tilde{\delta}_{0}\right\}  \geq1+n\tilde{\delta}_{0}%
\]
we have
\[
\rho_{n}^{\operatorname{I}}\leq\frac{S_{n}^{\operatorname{I}}+S_{n}%
^{\operatorname{II}}}{\left(  1+n\tilde{\delta}_{0}\right)  S_{n}%
^{\operatorname{I}}}=:\rho_{n}^{\operatorname{II}}.
\]
In order to estimate the terms $S_{n}^{\operatorname*{I}}$, $S_{n}%
^{\operatorname*{II}}$, we have to investigate the behavior of $a_{m,n}\left(
\frac{\lambda_{0}}{n}\right)  ^{2m}/\tilde{b}_{n}$ depending on the distance
of $m$ from $\tilde{m}$. We write $m=n\mu_{0}\left(  1+\varepsilon\right)  $
for some $\varepsilon\in\mathbb{R}$ that satisfies $0<\mu_{0}\left(
1+\varepsilon\right)  <1$. This gives
\begin{align}
C_{2}f_{m,n}\left(  \gamma_{\lambda_{0}}\left(  \varepsilon\right)  \right)
^{n}  &  \leq\frac{a_{m,n}\left(  \frac{\lambda_{0}}{n}\right)  ^{2m}}%
{\tilde{b}_{n}}\leq2C_{1}\left(  \gamma_{\lambda_{0}}\left(  \varepsilon
\right)  \right)  ^{n}\label{defgammalambda}\\
\text{with\quad}\gamma_{\lambda_{0}}\left(  \varepsilon\right)   &
:=\frac{\left(  1+\mu_{0}\left(  1+\varepsilon\right)  \right)  ^{1+\mu
_{0}\left(  1+\varepsilon\right)  }}{\left(  1-\mu_{0}\left(  1+\varepsilon
\right)  \right)  ^{1-\mu_{0}\left(  1+\varepsilon\right)  }}\frac{\left(
1-\mu_{0}\right)  ^{1-\mu_{0}}}{\left(  1+\mu_{0}\right)  ^{1+\mu_{0}}}\left(
\frac{\lambda_{0}}{\operatorname*{e}}\right)  ^{2\mu_{0}\varepsilon}.\nonumber
\end{align}
\textbf{Estimate of $S_{n}^{\operatorname{I}}$:} The dominant contribution in
the numerator of $\rho_{n}^{\operatorname*{II}}$ will be seen to be
$S_{n}^{\operatorname*{I}}$, for which we therefore need a lower bound. Our
strategy is to estimate this sum by a single summand, namely, the summand
corresponding to an integer $m$ close to $\tilde{m}=n\mu_{0}$. For
$m\in\left\{  \left\lfloor n\mu_{0}\right\rfloor ,\left\lceil n\mu
_{0}\right\rceil \right\}  $ we have%
\[
m-n\mu_{0}=n\mu_{0}\varepsilon_{m}\text{ with }\varepsilon_{m}\in\left\{
-\frac{n\mu_{0}-\left\lfloor n\mu_{0}\right\rfloor }{n\mu_{0}},\frac
{\left\lceil n\mu_{0}\right\rceil -n\mu_{0}}{n\mu_{0}}\right\}  .
\]
For these two values of $m$ (in fact, we will only need the one with $m\leq
\mu_{0}n$), we have $m=n\mu_{0}(1+\varepsilon_{m})$ with $|\varepsilon
_{m}|\leq(n\mu_{0})^{-1}$ and (cf. (\ref{eq:condition-on-n}))
\begin{align}
\frac{\mu_{0}}{2}n  &  \leq n\mu_{0}-1\leq m\leq n\mu_{0}+1\leq\frac{1+\mu
_{0}}{2}n,\label{eq:estimate_for_m-10}\\
\lambda_{0}^{2}\left\vert \varepsilon_{m}\right\vert  &  \leq\frac{\lambda
_{0}^{2}}{n\mu_{0}}\overset{\text{(\ref{eq:c_5})}}{\leq}c_{0}.
\label{eq:estimate_for_m-11}%
\end{align}
The estimates (\ref{eq:estimate_for_m-10}), (\ref{eq:estimate_for_m-11}) make
Lemma~\ref{LemTaylor} applicable, which gives%
\begin{equation}
1\geq\gamma_{\lambda_{0}}\left(  \varepsilon_{m}\right)  \geq1-c_{2}%
\lambda_{0}^{2}\varepsilon_{m}^{2}\geq1-c_{2}c_{0}|\varepsilon_{m}|\geq
1-\frac{c_{6}}{n}\quad\text{with }c_{6}=\frac{c_{2}c_{0}}{\mu_{0}}.
\label{eq:c_6}%
\end{equation}
The estimate (\ref{eq:estimate_for_m-10}) leads to two-sided bounds for
$f_{m,n}$:
\begin{align*}
f_{m,n}  &  \leq\frac{2n+1}{n\mu_{0}/2\sqrt{n(1-\mu_{0})/2}}\overset{n\geq
1}{\leq}\frac{6\sqrt{2}}{\mu_{0}\sqrt{1-\mu_{0}}}n^{-1/2}=:c_{7}n^{-1/2},\\
f_{m,n}  &  \geq\frac{\sqrt{n+\mu_{0}n}\sqrt{\mu_{0}n}}{\sqrt{n-\mu_{0}%
n}(1+\mu_{0})n}=:c_{8}n^{-1/2}.
\end{align*}
Define $c_{5}>0$ such that, with $c_{0}$ given by Lemma~\ref{LemTaylor},
\begin{equation}
n\geq c_{5}\quad\Longrightarrow\left(  (1-c_{6}/n)^{n}\geq\frac{1}%
{2}\operatorname{e}^{-c_{6}}\quad\mbox{ and }\quad\frac{\lambda_{0}^{2}}%
{n\mu_{0}}\leq c_{0}\right)  \label{eq:c_5}%
\end{equation}
This leads to%
\begin{equation}
S_{n}^{\operatorname{I}}\geq C_{2}f_{m,n}\left(  \gamma_{\lambda_{0}}\left(
\varepsilon_{m}\right)  \right)  ^{n}\geq C_{2}c_{8}n^{-1/2}\left(
1-\frac{c_{6}}{n}\right)  ^{n}\geq\frac{1}{2}C_{2}c_{8}\operatorname{e}%
\nolimits^{-c_{6}}n^{-1/2}. \label{eq:SI}%
\end{equation}
\textbf{Estimate of $S_{n}^{\operatorname{II}}$:} Let $c_{0}\in\left(
0,1\right)  $ be the constant in Lemma~\ref{LemTaylor} (note that we may
assume, without loss of generally, $c_{0}<1$). Upon writing $m\in
\{0,\ldots,\lfloor n\tilde{\delta}_{0}\rfloor\}$ in the form $m=\mu
_{0}n(1+\varepsilon_{m})$, we find in view of $\tilde{\delta}_{0}=\mu_{0}^{3}$
that $|\varepsilon_{m}|\geq\lambda_{0}^{-2}$. Hence, the monotonicity
properties of the function $\gamma_{\lambda_{0}}$ of Lemma~\ref{LemTaylor}
imply
\begin{equation}
\gamma_{\lambda_{0}}(\varepsilon_{m})\leq1-\frac{c_{2}}{c_{0}}\lambda_{0}%
^{-2}.
\end{equation}
We therefore get
\begin{equation}
S_{n}^{\operatorname{II}}=\sum_{0\leq m\leq\lfloor\tilde{\delta}_{0}n\rfloor
}\frac{a_{m,n}}{\tilde{b}_{n}}\left(  \frac{\lambda_{0}}{n}\right)  ^{2m}%
\leq2C_{1}\sum_{0\leq m\leq\lfloor\tilde{\delta}_{0}n\rfloor}\left(
1-\frac{c_{2}}{c_{0}}\lambda_{0}^{-2}\right)  ^{n}\leq2C_{1}(n+1)\left(
1-\frac{c_{2}}{c_{0}}\lambda_{0}^{-2}\right)  ^{n}. \label{eq:SII}%
\end{equation}
The combination of (\ref{eq:SII}) and (\ref{eq:SI}) shows $S_{n}%
^{\operatorname{I}}+S_{n}^{\operatorname{II}}\leq CS_{n}^{\operatorname{I}}$
for some constant $C>0$ that depends solely on $\lambda_{0}$. This concludes
the proof.
\endproof

\begin{lemma}
\label{LemTaylor} For $\lambda> 1$ and $\mu:= \sqrt{1 - \lambda^{-2}}$
introduce the function
\[
(-1,\mu^{-1}-1) \ni\varepsilon\mapsto\gamma_{\lambda}\left(  \varepsilon
\right)  :=\frac{\left(  1+\mu\left(  1+\varepsilon\right)  \right)
^{1+\mu\left(  1+\varepsilon\right)  }}{\left(  1-\mu\left(  1+\varepsilon
\right)  \right)  ^{1-\mu\left(  1+\varepsilon\right)  }}\frac{\left(
1-\mu\right)  ^{1-\mu}}{\left(  1+\mu\right)  ^{1+\mu}}\left(  \frac{\lambda
}{\operatorname*{e}}\right)  ^{2\mu\varepsilon}.
\]
Let $\lambda_{0} > 1$. Then there are constants $c_{0}$, $c_{1}$, $c_{2} > 0$
depending solely on $\lambda_{0}$ such that the following holds for every
$\lambda\ge\lambda_{0}$: For every $\varepsilon$ satisfying
\begin{equation}
\label{eq:lemma:Taylor-5}|\varepsilon| \lambda^{2} \leq c_{0}%
\end{equation}
the function $\gamma_{\lambda}$ satisfies
\begin{equation}
\label{eq:lemma:Taylor-6}1 - c_{1} \lambda^{2} \varepsilon^{2} \leq
\gamma_{\lambda}(\varepsilon) \leq1 - c_{2} \lambda^{2} \varepsilon^{2}.
\end{equation}
Furthermore, the function $\gamma_{\lambda}$ is monotone increasing on
$(-1,0)$ and monotone decreasing on $(0,\mu^{-1}-1)$. In particular,
therefore,
\begin{equation}
\label{eq:lemma:Taylor-7}0 < \gamma_{\lambda}(\varepsilon) \leq1 - \frac
{c_{2}}{c_{0}} \lambda^{-2} \qquad\forall\varepsilon\in(-1,\mu^{-1}%
-1)\setminus(-c_{0} \lambda^{-2}, c_{0} \lambda^{-2}).
\end{equation}

\end{lemma}

%

\proof
Define the function
\begin{equation}
g_{\lambda}\left(  \varepsilon\right)  :=\ln\left(  \left(  1-\mu^{2}\left(
1+\varepsilon\right)  ^{2}\right)  \lambda^{2}\right)  \label{eq:g}%
\end{equation}
and observe
\begin{align}
&  g_{\lambda}^{\prime}\left(  \varepsilon\right)  =-2\mu^{2}\frac
{1+\varepsilon}{1-\mu^{2}\left(  1+\varepsilon\right)  ^{2}},\qquad
g_{\lambda}^{\prime\prime}\left(  \varepsilon\right)  =-2\mu^{2}\frac
{1+\mu^{2}\left(  \varepsilon+1\right)  ^{2}}{\left(  1-\mu^{2}\left(
1+\varepsilon\right)  ^{2}\right)  ^{2}},\label{derofg}\\
&  \gamma_{\lambda}^{\prime}=\mu\gamma_{\lambda}g_{\lambda},\qquad
\gamma_{\lambda}^{\prime\prime}=\mu\gamma_{\lambda}\left(  \mu g_{\lambda}%
^{2}+g_{\lambda}^{\prime}\right)  ,\qquad\gamma_{\lambda}^{\prime\prime\prime
}=\mu\gamma_{\lambda}\left(  \mu^{2}g_{\lambda}^{3}+3\mu g_{\lambda}%
g_{\lambda}^{\prime}+g_{\lambda}^{\prime\prime}\right)  .
\end{align}
\textbf{Step 1:} (monotonicity properties of $\gamma_{\lambda}$) The function
$\gamma_{\lambda}$ is defined in the interval $(-1,\mu^{-1}-1)$.

\emph{Claim:} $\gamma_{\lambda}$ is strictly increasing on $(-1,0)$, strictly
decreasing on $(0,\mu^{-1}-1)$, and thus has a proper maximum at
$\varepsilon=0$. To see these monotonicity properties, we note that
$\gamma_{\lambda}\geq0$ and that $g_{\lambda}(\varepsilon)>0$ for
$\varepsilon<0$ and $g_{\lambda}(\varepsilon)<0$ for $\varepsilon>0$. We
calculate%
\begin{equation}
\gamma_{\lambda}(0)=1,\qquad\gamma_{\lambda}^{\prime}(0)=0,\qquad
\gamma_{\lambda}^{\prime\prime}(0)=-2(\lambda^{2}-1)^{3/2}\lambda^{-1}.
\label{eq:lemma:Taylor-100}%
\end{equation}

\textbf{Step 2:} Use $\mu=\sqrt{1-\lambda^{-2}}$ to write
\begin{equation}
\left(  1-\mu^{2}(1+\varepsilon)^{2}\right)  \lambda^{2}=1-(\lambda
^{2}-1)(2\varepsilon+\varepsilon^{2}). \label{eq:lemma:Taylor-150}%
\end{equation}
Fix $q\in(0,1)$ and consider $\varepsilon$ satisfying
\begin{equation}
0<\mu(1+\varepsilon)<1\quad\mbox{ and }\quad(\lambda^{2}-1)|2\varepsilon
+\varepsilon^{2}|\leq q<1. \label{eq:lemma:Taylor-200}%
\end{equation}
{}From (\ref{eq:lemma:Taylor-150}) and (\ref{eq:lemma:Taylor-200}) we infer
\[
(1-q)\lambda^{-2}\leq1-\mu^{2}(1+\varepsilon)^{2}\leq(1+q)\lambda^{-2}.
\]
This, together with $0<\mu(1+\varepsilon)<1$ and $\mu\in(0,1)$ implies
\begin{equation}
|g_{\lambda}(\varepsilon)|\leq\max\{|\ln(1-q)|,\ln(1+q)\},\qquad|g_{\lambda
}^{\prime}(\varepsilon)|\leq\frac{2\lambda^{2}}{1-q},\qquad|g_{\lambda
}^{\prime\prime}(\varepsilon)|\leq\frac{4\lambda^{4}}{(1-q)^{2}}.
\label{eq:lemma:Taylor-500}%
\end{equation}
Taylor's theorem now implies for every $\varepsilon$ satisfying
(\ref{eq:lemma:Taylor-200}) the existence of an $\varepsilon^{\prime}$ in the
interval $\left(  0,\varepsilon\right)  $ with endpoints $0$ and $\varepsilon$
such that
\begin{equation}
\gamma_{\lambda}(\varepsilon)=\gamma_{\lambda}(0)+\gamma_{\lambda}^{\prime
}(0)\varepsilon+\frac{1}{2}\gamma_{\lambda}^{\prime\prime}(0)\varepsilon
^{2}+\frac{1}{3!}\gamma_{\lambda}^{\prime\prime\prime}(\varepsilon^{\prime
})\varepsilon^{3}=1-2(\lambda^{2}-1)^{3/2}\lambda^{-1}\varepsilon^{2}+\frac
{1}{3!}\gamma_{\lambda}^{\prime\prime\prime}(\varepsilon^{\prime}%
)\varepsilon^{3}.
\end{equation}
The remainder term $\gamma_{\lambda}^{\prime\prime\prime}(\varepsilon^{\prime
})$ is estimated using (\ref{eq:lemma:Taylor-500}) as follows (note that
$\gamma_{\lambda}\geq0$ and has maximum $1$) as%
\[
|\gamma_{\lambda}^{\prime\prime\prime}(\varepsilon^{\prime})|\leq\max
\{\ln(1+q),|\ln(1-q)|\}^{3}+6\lambda^{2}\max\{\ln(1+q),|\ln(1-q)|\}(1-q)^{-1}%
+4\lambda^{4}(1-q)^{-2}\leq C_{1}\lambda^{4}%
\]
for a constant $C_{1}$ that depends solely on $\lambda_{0}>1$ and the chosen
$q$. Finally, there are constants $C_{2}$, $C_{3}>0$ depending solely on
$\lambda_{0}>1$ such that
\begin{equation}
C_{2}\lambda^{2}\leq2(\lambda^{2}-1)^{3/2}\lambda^{-1}\leq C_{3}\lambda^{2}.
\end{equation}
We conclude for $\varepsilon$ satisfying (\ref{eq:lemma:Taylor-200})
\[
1-C_{2}\lambda^{2}\varepsilon^{2}-\frac{C_{1}}{3!}\lambda^{4}\varepsilon
^{3}\leq\gamma_{\lambda}(\varepsilon)\leq1-C_{3}\lambda^{2}\varepsilon
^{2}+\frac{C_{1}}{3!}\lambda^{4}\varepsilon^{3}.
\]
The two-sided bound (\ref{eq:lemma:Taylor-6}) now follows if we assume
(\ref{eq:lemma:Taylor-5}) for $c_{0}$ sufficiently small so that the terms
$\lambda^{4}\varepsilon^{3}$ are small compared to the terms involving
$\lambda^{2}\varepsilon^{2}$. We note that the condition
(\ref{eq:lemma:Taylor-5}) for sufficiently small $c_{0}$ also implies
(\ref{eq:lemma:Taylor-200}). Finally, the estimate (\ref{eq:lemma:Taylor-7})
is a consequence of (\ref{eq:lemma:Taylor-6}) and the monotonicity properties
of $\gamma_{\lambda}$.
\endproof


\section{Equivalence of $\left\Vert \cdot\right\Vert _{\mathbf{H}^{1}\left(
\Omega\right)  }$ and $\left\Vert \cdot\right\Vert _{\operatorname*{curl}%
,\Omega,1}$ in $\mathbf{V}_{0}$ and $\mathbf{V}_{0}^{\ast}$}

The spaces $\mathbf{V}_{0}$ and $\mathbf{V}_{0}^{\ast}$ as in (\ref{defVo})
involve the capacity operator (cf.~Lemma~\ref{LemEq2scprodstrong}). For the
case that $\Gamma$ is the surface of the ball, they are subspaces of
$\mathbf{H}^{1}\left(  \Omega\right)  $ as shown in the following lemma. In
contrast to Lemma~\ref{Lemembed} we obtain $k$-explicit bounds for the norm estimates.

\begin{lemma}
\label{Lemembedspec}Let $\Omega= B_{1}(0)$ and let $\mathbf{V}_{0}$,
$\mathbf{V}_{0}^{\ast}$ be defined as in (\ref{defVo}). Then, $\mathbf{V}%
_{0}\cup\mathbf{V}_{0}^{\ast}\subset\mathbf{H}^{1}\left(  \Omega\right)  $
and
\begin{equation}
\left\Vert \mathbf{u}\right\Vert _{\mathbf{H}^{1}\left(  \Omega\right)  }%
\leq\left\Vert \mathbf{u}\right\Vert _{\operatorname*{curl},\Omega,1}%
\qquad\forall\mathbf{u\in V}_{0}\cup\mathbf{V}_{0}^{\ast}, \label{normestV0}%
\end{equation}
i.e., the constant $C_{\Omega,k}$ in Lemma \ref{Lemembed} equals $1$ for
$\Omega=B_{1}\left(  0\right)  $.
\end{lemma}

%

\proof
The inclusion $\mathbf{V}_{0}\cup\mathbf{V}_{0}^{\ast}\subset\mathbf{H}%
^{1}\left(  \Omega\right)  $ follows from Lemma \ref{Lemembed} and it remains
to prove the norm estimates. Let $\mathbf{u\in V}_{0}$. Then, from
\cite[(2.5.151), (2.5.152), Lemma 5.4.2]{Nedelec01} we have%
\begin{align*}
&  \left(  \nabla\mathbf{u},\nabla\mathbf{v}\right)  -\left(
\operatorname*{curl}\mathbf{u},\operatorname*{curl}\mathbf{v}\right)  -\left(
\operatorname*{div}\mathbf{u},\operatorname*{div}\mathbf{v}\right) \\
&  \qquad=-\left(  \operatorname*{div}\nolimits_{\Gamma}\mathbf{u}%
_{T},\left\langle \mathbf{v},\mathbf{n}\right\rangle \right)  _{\Gamma
}-\left(  \left\langle \mathbf{u},\mathbf{n}\right\rangle ,\operatorname*{div}%
\nolimits_{\Gamma}\mathbf{v}_{T}\right)  _{\Gamma}-2\left(  \left\langle
\mathbf{u},\mathbf{n}\right\rangle ,\left\langle \mathbf{v},\mathbf{n}%
\right\rangle \right)  _{\Gamma}-\left(  \mathbf{u}_{T},\mathbf{v}_{T}\right)
_{\Gamma}.
\end{align*}
We choose $\mathbf{v}=\mathbf{u}$ and employ (\ref{ImpBeda}) to obtain after
rearranging terms%
\begin{align}
\left\Vert \nabla\mathbf{u}\right\Vert ^{2}  &  =\left\Vert
\operatorname*{curl}\mathbf{u}\right\Vert ^{2}-2\operatorname{Re}\left(
\operatorname*{div}\nolimits_{\Gamma}\mathbf{u}_{T},\left\langle
\mathbf{u},\mathbf{n}\right\rangle \right)  _{\Gamma}-2\left\Vert \left\langle
\mathbf{u},\mathbf{n}\right\rangle \right\Vert _{\Gamma}^{2}-\left\Vert
\mathbf{u}_{T}\right\Vert _{\Gamma}^{2}\nonumber\\
&  \overset{\text{(\ref{ImpBeda})}}{=}\left\Vert \operatorname*{curl}%
\mathbf{u}\right\Vert ^{2}+\frac{2}{k}\operatorname{Im}\left(
\operatorname{div}_{\Gamma}T_{k}\mathbf{u}_{T},\operatorname*{div}%
\nolimits_{\Gamma}\mathbf{u}_{T}\right)  _{\Gamma}-2\left\Vert \left\langle
\mathbf{u},\mathbf{n}\right\rangle \right\Vert _{\Gamma}^{2}-\left\Vert
\mathbf{u}_{T}\right\Vert _{\Gamma}^{2}\nonumber\\
&  \leq\left\Vert \operatorname*{curl}\mathbf{u}\right\Vert ^{2}+\frac{2}%
{k}\operatorname{Im}\left(  \operatorname{div}_{\Gamma}T_{k}\mathbf{u}%
_{T},\operatorname*{div}\nolimits_{\Gamma}\mathbf{u}_{T}\right)  _{\Gamma}.
\label{Voest}%
\end{align}
{}From \cite[(5.3.91), (5.3.93)]{Nedelec01} we conclude that%
\[
\left(  \left(  \operatorname{div}_{\Gamma}T_{k}\mathbf{u}_{T}\right)
,\operatorname*{div}\nolimits_{\Gamma}{\mathbf{u}}_{T}\right)
_{\Gamma}=\sum_{\ell=1}^{\infty}\sum_{m\in\iota_{\ell}}\operatorname*{i}%
\ell^{2}\left(  \ell+1\right)  ^{2}\frac{k}{z_{\ell}\left(  k\right)
+1}\left\vert U_{\ell}^{m}\right\vert ^{2}.
\]
Since
\[
\operatorname*{Im}\left(  \frac{\operatorname*{i}}{z_{\ell}\left(  k\right)
+1}\right)  =\frac{\operatorname*{Im}\left(  \operatorname*{i}\left(
\overline{z_{\ell}\left(  k\right)  }+1\right)  \right)  }{\left\vert z_{\ell
}\left(  k\right)  +1\right\vert ^{2}}=\frac{\operatorname{Re}\left(  z_{\ell
}\left(  k\right)  +1\right)  }{\left\vert z_{\ell}\left(  k\right)
+1\right\vert ^{2}}\overset{\text{\cite[(2.6.23)]{Nedelec01}}}{\leq}0,
\]
the second summand in (\ref{Voest}) is non-positive so
that$\displaystyle \left\Vert \nabla\mathbf{u}\right\Vert \leq\left\Vert
\operatorname*{curl}\mathbf{u}\right\Vert .
$ This implies the first estimate in (\ref{normestV0}) while the statement
about ${\mathbf{u}} \in{\mathbf{V}}^{\ast}_{0}$ is simply a repetition of
these arguments.%
\endproof

\section{Vector Spherical Harmonics}

For $\mathbf{x}\in\mathbb{R}^{3}$, $r=\left\Vert \mathbf{x}\right\Vert $, and
$\mathbf{\hat{x}}:=\mathbf{x}/r$ we introduce the vectorial spherical
harmonics (VSH) as in \cite[Thm.~{2.46}]{Kirsch_Hettlich_2015} (with a
different scaling)%
\[
\mathbf{Y}_{\ell}^{m}\left(  \mathbf{\hat{x}}\right)  :=\mathbf{\hat{x}%
}Y_{\ell}^{m}\left(  \mathbf{\hat{x}}\right)  ,\quad\mathbf{U}_{\ell}%
^{m}\left(  \mathbf{\hat{x}}\right)  :=\nabla_{\Gamma}Y_{\ell}^{m}\left(
\mathbf{\hat{x}}\right)  ,\quad\mathbf{V}_{\ell}^{m}\left(  \mathbf{\hat{x}%
}\right)  :=\nabla_{\Gamma}Y_{\ell}^{m}\left(  \mathbf{\hat{x}}\right)
\times\mathbf{\hat{x}}.
\]
{}From \cite[Thm.~{5.36}]{Kirsch_Hettlich_2015} we conclude that any
$\mathbf{u}\in\mathbf{X}$ has an expansion of the form%
\begin{equation}
\mathbf{u}\left(  r\mathbf{\hat{x}}\right)  =\sum_{\ell=0}^{\infty}\sum
_{m\in\iota_{\ell}}\left(  u_{\ell}^{m}\left(  r\right)  \mathbf{Y}_{\ell}%
^{m}\left(  \mathbf{\hat{x}}\right)  +v_{\ell}^{m}\left(  r\right)
\mathbf{U}_{\ell}^{m}\left(  \mathbf{\hat{x}}\right)  +w_{\ell}^{m}\left(
r\right)  \mathbf{V}_{\ell}^{m}\left(  \mathbf{\hat{x}}\right)  \right)  .
\label{urepvsh}%
\end{equation}
We use the relations (cf. \cite[p.271]{Kirsch_Hettlich_2015})\footnote{There
is a sign error in the second last relation on \cite[p.271]%
{Kirsch_Hettlich_2015}.}%
\begin{align*}
\operatorname*{curl}\left(  u_{\ell}^{m}\left(  r\right)  \mathbf{Y}_{\ell
}^{m}\left(  \mathbf{\hat{x}}\right)  \right)   &  =\frac{u_{\ell}^{m}\left(
r\right)  }{r}\mathbf{V}_{\ell}^{m}\left(  \mathbf{\hat{x}}\right)
,\quad\operatorname*{curl}\left(  v_{\ell}^{m}\left(  r\right)  \mathbf{U}%
_{\ell}^{m}\left(  \mathbf{\hat{x}}\right)  \right)  =-\frac{1}{r}\left(
rv_{\ell}^{m}\left(  r\right)  \right)  ^{\prime}\mathbf{V}_{\ell}^{m}\left(
\mathbf{\hat{x}}\right)  ,\\
\operatorname*{curl}\left(  w_{\ell}^{m}\left(  r\right)  \mathbf{V}_{\ell
}^{m}\left(  \mathbf{\hat{x}}\right)  \right)   &  =\frac{1}{r}\left(
rw_{\ell}^{m}\left(  r\right)  \right)  ^{\prime}\mathbf{U}_{\ell}^{m}\left(
\mathbf{\hat{x}}\right)  +w_{\ell}^{m}\left(  r\right)  \frac{\ell\left(
\ell+1\right)  }{r}\mathbf{Y}_{\ell}^{m}\left(  \mathbf{\hat{x}}\right)  ,
\end{align*}
so that $\operatorname*{curl}\mathbf{u}$ is given by%
\[
\operatorname*{curl}\mathbf{u}\left(  r\mathbf{\hat{x}}\right)  =\sum_{\ell
=0}^{\infty}\sum_{m\in\iota_{\ell}}\frac{1}{r}\left(  \left(  u_{\ell}%
^{m}\left(  r\right)  -\left(  rv_{\ell}^{m}\left(  r\right)  \right)
^{\prime}\right)  \mathbf{V}_{\ell}^{m}\left(  \mathbf{\hat{x}}\right)
+\left(  rw_{\ell}^{m}\left(  r\right)  \right)  ^{\prime}\mathbf{U}_{\ell
}^{m}\left(  \mathbf{\hat{x}}\right)  +w_{\ell}^{m}\left(  r\right)
\ell\left(  \ell+1\right)  \mathbf{Y}_{\ell}^{m}\left(  \mathbf{\hat{x}%
}\right)  \right)  .
\]
Using the orthogonality relations of the vectorial spherical harmonics we get%
\begin{align}
\left\Vert \mathbf{u}\right\Vert ^{2}  &  =\sum_{\ell=0}^{\infty}\sum
_{m\in\iota_{\ell}}\int_{\mathbb{R}}r^{2}\left(  \left\vert u_{\ell}%
^{m}\left(  r\right)  \right\vert ^{2}+\ell\left(  \ell+1\right)  \left(
\left\vert v_{\ell}^{m}\left(  r\right)  \right\vert ^{2}+\left\vert w_{\ell
}^{m}\left(  r\right)  \right\vert ^{2}\right)  \right)  dr,\label{vshL2Curla}%
\\
\left\Vert \operatorname*{curl}\mathbf{u}\right\Vert ^{2}  &  =\sum_{m\in
\iota_{\ell}}\int_{\mathbb{R}}\ell\left(  \ell+1\right)  \left(  \sum_{\ell
=0}^{\infty}\left\vert u_{\ell}^{m}\left(  r\right)  -\left(  rv_{\ell}%
^{m}\left(  r\right)  \right)  ^{\prime}\right\vert ^{2}+\left\vert \left(
rw_{\ell}^{m}\left(  r\right)  \right)  ^{\prime}\right\vert ^{2}+\ell\left(
\ell+1\right)  \left\vert w_{\ell}^{m}\left(  r\right)  \right\vert
^{2}\right)  dr. \label{vshL2Curlb}%
\end{align}
For $a>0$, we introduce operators $L_{a}^{\operatorname*{VSH}}%
:\mathbf{X}\rightarrow\mathbf{X}$ and $H_{a}^{\operatorname*{VSH}}%
:\mathbf{X}\rightarrow\mathbf{X}$ for functions $\mathbf{u}$ as in
(\ref{urepvsh}) by%
\begin{equation}
L_{a}^{\operatorname*{VSH}}\mathbf{u}=\sum_{\ell:0\leq\ell\leq a}\sum
_{m\in\iota_{\ell}}\left(  u_{\ell}^{m}\left(  r\right)  \mathbf{Y}_{\ell}%
^{m}\left(  \mathbf{\hat{x}}\right)  +v_{\ell}^{m}\left(  r\right)
\mathbf{U}_{\ell}^{m}\left(  \mathbf{\hat{x}}\right)  +w_{\ell}^{m}\left(
r\right)  \mathbf{V}_{\ell}^{m}\left(  \mathbf{\hat{x}}\right)  \right)
,\quad H_{a}^{\operatorname*{VSH}}\mathbf{u}=\mathbf{u}-L_{a}%
^{\operatorname*{VSH}}\mathbf{u}. \label{defLaHaVSH}%
\end{equation}
{}From (\ref{vshL2Curla}), (\ref{vshL2Curlb}) we conclude the stability of the
splitting%
\begin{equation}%
\begin{array}
[c]{cc}%
\left\Vert L_{a}^{\operatorname*{VSH}}\mathbf{u}\right\Vert \leq\left\Vert
\mathbf{u}\right\Vert , & \left\Vert \operatorname*{curl}L_{a}%
^{\operatorname*{VSH}}\mathbf{u}\right\Vert \leq\left\Vert
\operatorname*{curl}\mathbf{u}\right\Vert ,\\
\left\Vert H_{a}^{\operatorname*{VSH}}\mathbf{u}\right\Vert \leq\left\Vert
\mathbf{u}\right\Vert , & \left\Vert \operatorname*{curl}H_{a}%
^{\operatorname*{VSH}}\mathbf{u}\right\Vert \leq\left\Vert
\operatorname*{curl}\mathbf{u}\right\Vert .
\end{array}
\label{lowhighuestVSH}%
\end{equation}
In addition the splitting is orthogonal:%
\[
\left(  L_{a}^{\operatorname*{VSH}}\mathbf{u},H_{a}^{\operatorname*{VSH}%
}\mathbf{u}\right)  =\left(  \operatorname*{curl}L_{a}^{\operatorname*{VSH}%
}\mathbf{u,}\operatorname*{curl}H_{a}^{\operatorname*{VSH}}\mathbf{u}\right)
=0.
\]
Note that on the unit sphere, it holds%
\begin{align*}
\Pi_{T}\mathbf{Y}_{\ell}^{m}\left(  \mathbf{\hat{x}}\right)   &
=\mathbf{\hat{x}}\times\left(  Y_{\ell,m}\left(  \mathbf{\hat{x}}\right)
\mathbf{\hat{x}}\times\mathbf{\hat{x}}\right)  =0,\\
\Pi_{T}\mathbf{U}_{\ell}^{m}\left(  \mathbf{\hat{x}}\right)   &
=\mathbf{\hat{x}}\times\left(  \nabla_{\Gamma}Y_{\ell}^{m}\left(
\mathbf{\hat{x}}\right)  \mathbf{\times\hat{x}}\right)  =\nabla_{\Gamma
}Y_{\ell}^{m}\left(  \mathbf{\hat{x}}\right)  ,\\
\Pi_{T}\mathbf{V}_{\ell}^{m}\left(  \mathbf{\hat{x}}\right)   &
=\mathbf{\hat{x}}\times\left(  \left(  \nabla_{\Gamma}Y_{\ell}^{m}\left(
\mathbf{\hat{x}}\right)  \times\mathbf{\hat{x}}\right)  \mathbf{\times\hat{x}%
}\right)  =-\mathbf{\hat{x}}\times\nabla_{\Gamma}Y_{\ell}^{m}\left(
\mathbf{\hat{x}}\right)  =\mathbf{T}_{\ell}^{m},
\end{align*}
where $\mathbf{T}_{\ell}^{m}$ is as in \cite[(2.4.173)]{Nedelec01}. Hence, the
application of the trace map $\Pi_{T}$ yields%
\begin{equation}
\mathbf{u}_{T}=\Pi_{T}\mathbf{u}=\sum_{\ell=0}^{\infty}\sum_{m\in\iota_{\ell}%
}\left(  v_{\ell}^{m}\nabla_{\Gamma}Y_{\ell}^{m}+w_{\ell}^{m}\mathbf{T}_{\ell
}^{m}\right)  , \label{defutexpmm}%
\end{equation}
where $v_{\ell}^{m}=v_{\ell}^{m}\left(  1\right)  $, $w_{\ell}^{m}:=w_{\ell
}^{m}\left(  1\right)  $. A key observation for the case of the unit sphere is
that for any $\mathbf{u}\in\mathbf{X}$, the function $L_{\lambda
k}^{\operatorname*{VSH}}\mathbf{u}$ satisfies $\Pi_{T}L_{\lambda
k}^{\operatorname*{VSH}}\mathbf{u}=L_{\Gamma}\Pi_{T}\mathbf{u}$, where
$L_{\Gamma}$ was introduced in Definition~\ref{DefFreqSplit}.

\begin{lemma}
\label{lemma:minimization-by-vsh} Let $\Omega= B_{1}(0)$ and $L_{\Omega}$ be
as in Definition~\ref{DefFreqSplit}. Then: $\Pi_{T}L_{\lambda k}%
^{\operatorname*{VSH}}=L_{\Gamma}\Pi_{T}$ and
\[
\left\Vert L_{\Omega}\mathbf{u}\right\Vert _{\operatorname*{curl},\Omega
,k}\leq\Vert\mathbf{u}\Vert_{\operatorname*{curl},\Omega,k}\qquad
\forall\mathbf{u}\in\mathbf{X}.
\]
Furthermore, the stability constants in (\ref{defCkLHOmega}) satisfy
$C_{k}^{L,\Omega}\leq1$ and $C_{k}^{H,\Omega}\leq2$.
\end{lemma}

%

\proof
Since $L_{\Omega}$ is the minimum norm extension
(cf.~Definition~\ref{DefFreqSplit}) the bound (\ref{lowhighuestVSH}) lead to%
\[
\left\Vert L_{\Omega}\mathbf{u}\right\Vert _{\operatorname*{curl},\Omega
,k}^{2}\leq\Vert L_{\lambda k}^{\operatorname*{VSH}}\mathbf{u}\Vert
_{\operatorname*{curl},\Omega,k}^{2}=k^{2}\Vert L_{\lambda k}%
^{\operatorname*{VSH}}\mathbf{u}\Vert^{2}+\Vert\operatorname*{curl}L_{\lambda
k}^{\operatorname*{VSH}}\mathbf{u}\Vert^{2}\leq k^{2}\Vert\mathbf{u}\Vert
^{2}+\Vert\operatorname*{curl}\mathbf{u}\Vert^{2}\leq\Vert\mathbf{u}%
\Vert_{\operatorname*{curl},\Omega,k}^{2}.
\]%
\endproof

\section{Analytic regularity of Maxwell and Maxwell-like Problems}


\subsection{Local Smoothness}

Consider for a bounded Lipschitz domain $\omega\subset{\mathbb{R}}^{3}$
\begin{subequations}
\label{eq:maxwell-variable-coeff}
\end{subequations}%
\begin{align}
\operatorname*{curl}\left(  \mathbf{A}\left(  x\right)  \operatorname*{curl}%
\mathbf{u}\right)   &  =\mathbf{f\quad}\text{in }\omega,\tag{%
\ref{eq:maxwell-variable-coeff}%
%
a}\label{eq:maxwell-variable-coeff-a}\\
\operatorname*{div}\left(  \mathbf{B}\left(  x\right)  \mathbf{u}\right)   &
=g\quad\text{in }\omega,\tag{%
\ref{eq:maxwell-variable-coeff}%
%
b}\label{eq:maxwell-variable-coeff-b}\\
\Pi_{T}\mathbf{u}  &  =\mathbf{0}\quad\text{on }\partial\omega. \tag{%
\ref{eq:maxwell-variable-coeff}%
%
c}\label{eq:maxwell-variable-coeff-c}%
\end{align}

We have smoothness of ${\mathbf{u}}$ under regularity assumptions on the
right-hand sides:

\begin{lemma}
\label{lemma:maxwell-H2-regularity} Let $\partial\omega$ be a smooth bounded
Lipschitz domain that is star-shaped with respect to a ball. Let ${\mathbf{A}%
}$, ${\mathbf{B}}\in C^{\infty}(\overline{\omega})$ be pointwise symmetric
positive definite. Then:

\begin{enumerate}
[(i)]

\item \label{item:lemma:maxwell-H2-regularity-i} If ${\mathbf{u}}
\in{\mathbf{H}}_{0}(\omega,\operatorname{curl})$ and $\operatorname{div}
({\mathbf{B}} {\mathbf{u}}) \in L^{2}(\omega)$, then ${\mathbf{u}}
\in{\mathbf{H}}^{1}(\omega)$ with
\[
\|{\mathbf{u}}\|_{{\mathbf{H}}^{1}(\omega)} \leq C \left[
\|\operatorname{div} ({\mathbf{B}} {\mathbf{u}})\|_{L^{2}(\omega)} +
\|\operatorname{curl} {\mathbf{u}}\|_{{\mathbf{L}}^{2}(\omega)}\right]  .
\]

\item \label{item:lemma:maxwell-H2-regularity-ii} If ${\mathbf{u}}
\in{\mathbf{H}}_{0}(\omega,\operatorname{curl})$ satisfies
(\ref{eq:maxwell-variable-coeff}) for some ${\mathbf{f}} \in{\mathbf{H}}%
^{s}(\omega)$, $g \in H^{s+1}(\omega)$, $s \in{\mathbb{N}}_{0}$, then
${\mathbf{u}} \in{\mathbf{H}}^{s+2}(\omega)$ and
\[
\|{\mathbf{u}}\|_{{\mathbf{H}}^{s+2}(\omega)} \leq C_{s}\left[  \|{\mathbf{f}%
}\|_{{\mathbf{H}}^{s}(\omega)} + \|g\|_{H^{s+1}(\omega)} \right]  .
\]

\end{enumerate}
\end{lemma}

\begin{proof}
We use the right inverse $R^{\operatorname{curl}}$ of the $\operatorname{curl}%
$-operator and use its mapping properties due to \cite{costabel-mcintosh10} as
formulated in \cite[Lemma~{6.4}]{melenk_rojik_2018}; specifically, we employ
$R^{\operatorname{curl}}:{\mathbf{H}}^{s}(\omega) \rightarrow{\mathbf{H}%
}^{s+1}(\omega)$ for any $s \in{\mathbb{N}}_{0}$. We will also repeatedly use
decompositions formulated in \cite[Lemma~{6.5}]{melenk_rojik_2018}, i.e., for
$s \in{\mathbb{N}}_{0}$ and ${\mathbf{v}} \in{\mathbf{H}}^{s}(\omega
,\operatorname{curl})$ there is $\varphi\in H^{s+1}(\omega)$ such that
\begin{equation}
\label{eq:lemma:maxwell-H2-regularity-1}{\mathbf{v}} = \nabla\varphi+
R^{\operatorname{curl}}(\operatorname{curl} {\mathbf{v}}).
\end{equation}

\emph{Proof of (\ref{item:lemma:maxwell-H2-regularity-i}):} Using
(\ref{eq:lemma:maxwell-H2-regularity-1}), we write
\begin{equation}
{\mathbf{u}}=\nabla\varphi+R^{\operatorname{curl}}(\operatorname{curl}%
{\mathbf{u}}). \label{eq:lemma:maxwell-H2-regularity-5}%
\end{equation}
The mapping property $R^{\operatorname{curl}}:{\mathbf{L}}^{2}(\omega
)\rightarrow{\mathbf{H}}^{1}(\omega)$ implies $R^{\operatorname{curl}%
}(\operatorname{curl}{\mathbf{u}})\in{\mathbf{H}}^{1}(\omega)$. Using $\Pi
_{T}{\mathbf{u}}=0$, we infer $\nabla_{\partial\omega}\varphi=-\Pi
_{T}R^{\operatorname{curl}}(\operatorname{curl}{\mathbf{u}})\in{\mathbf{H}%
}_{T}^{1/2}(\partial\omega)$ so that, by the smoothness of $\partial\omega$,
we have $g_{D}:=\varphi|_{\partial\omega}\in H^{3/2}(\partial\omega)$.
Multiplying (\ref{eq:lemma:maxwell-H2-regularity-5}) by ${\mathbf{B}}$ and
applying the divergence reveals that $\varphi$ solves
\begin{equation}
g=\operatorname{div}\left(  \mathbf{B}{\mathbf{u}}\right)  =\operatorname{div}%
\left(  {\mathbf{B}}\nabla\varphi\right)  +\operatorname{div}\left(
\mathbf{B}R^{\operatorname{curl}}(\operatorname{curl}{\mathbf{u}})\right)
\quad\mbox{ in $\omega$},\qquad\varphi=g_{D}\quad\mbox{ on $\partial\omega$}.
\label{eq:lemma:maxwell-H2-regularity-7}%
\end{equation}
This is a standard Poisson type problem for $\varphi$, and the smoothness of
$\partial\omega$ and ${\mathbf{B}}$ then imply $\varphi\in H^{2}(\omega)$
with
\begin{equation}
\Vert\varphi\Vert_{H^{2}(\omega)}\lesssim\Vert g-\operatorname{div}\left(
\mathbf{B}R^{\operatorname{curl}}(\operatorname{curl}{\mathbf{u}})\right)
\Vert_{L^{2}(\omega)}+\Vert g_{D}\Vert_{H^{3/2}(\partial\omega)}\lesssim\Vert
g\Vert_{L^{2}(\omega)}+\Vert\operatorname{curl}{\mathbf{u}}\Vert_{L^{2}%
(\omega)}. \label{eq:lemma:maxwell-H2-regularity-9}%
\end{equation}
\emph{Proof of (\ref{item:lemma:maxwell-H2-regularity-ii}):} We set
${\mathbf{w}}:=\operatorname{curl}{\mathbf{u}}$ and note
\begin{equation}
\operatorname{div}{\mathbf{w}}=0,\qquad{\mathbf{n}}\cdot{\mathbf{w}%
}={\mathbf{n}}\cdot\operatorname{curl}{\mathbf{u}}=\operatorname{curl}%
_{\partial\omega}\Pi_{T}{\mathbf{u}}=0.
\label{eq:lemma:maxwell-H2-regularity-10}%
\end{equation}
\emph{1.~step:} {} From (\ref{eq:lemma:maxwell-H2-regularity-1}) we see that
we can we write, for some $\varphi\in H^{1}(\omega)$,
\begin{equation}
{\mathbf{A}}{\mathbf{w}}=\nabla\varphi+R^{\operatorname{curl}}%
(\operatorname{curl}({\mathbf{A}}{\mathbf{w}}%
))\overset{\text{(\ref{eq:maxwell-variable-coeff-a})}}{=}\nabla\varphi
+R^{\operatorname{curl}}({\mathbf{f}}).
\label{eq:lemma:maxwell-H2-regularity-15}%
\end{equation}
Hence, ${\mathbf{w}}={\mathbf{A}}^{-1}\left(  R^{\operatorname{curl}%
}({\mathbf{f})}+\nabla\varphi\right)  $ and we get from
(\ref{eq:lemma:maxwell-H2-regularity-10}) that $\varphi$ satisfies
\begin{equation}
-\operatorname{div}\left(  {\mathbf{A}}^{-1}\nabla\varphi\right)
=\operatorname{div}\left(  {\mathbf{A}}^{-1}R^{\operatorname{curl}}%
(\mathbf{f})\right)  \quad\mbox{ in $\omega$},\qquad{\mathbf{n}}%
\cdot{\mathbf{A}}^{-1}\nabla\varphi=-{\mathbf{n}}\cdot{\mathbf{A}}%
^{-1}R^{\operatorname{curl}}({\mathbf{f}})\quad\mbox{ on $\partial\omega$.}
\label{eq:lemma:maxwell-H2-regularity-20}%
\end{equation}
The mapping properties of $R^{\operatorname{curl}}:{\mathbf{H}}^{s}%
(\omega)\rightarrow{\mathbf{H}}^{s+1}(\omega)$ give $R^{\operatorname{curl}%
}({\mathbf{f}})\in{\mathbf{H}}^{s+1}(\omega)$ so that the scalar shift theorem
for Poisson type problems gives in fact $\varphi\in H^{s+2}(\omega)$ with
$\Vert\varphi\Vert_{H^{s+2}(\omega)}\leq C\Vert{\mathbf{f}}\Vert_{{\mathbf{H}%
}^{s}(\omega)}.$
Inserting this regularity information in
(\ref{eq:lemma:maxwell-H2-regularity-15}) provides ${\mathbf{w}}\in
{\mathbf{H}}^{s+1}(\omega)$ with
\begin{equation}
\Vert{\mathbf{w}}\Vert_{{\mathbf{H}}^{s+1}(\omega)}\leq C\Vert{\mathbf{f}%
}\Vert_{{\mathbf{H}}^{s}(\omega)}. \label{eq:lemma:maxwell-H2-regularity-25}%
\end{equation}
\emph{2.~step:} From (\ref{item:lemma:maxwell-H2-regularity-i}) we have
${\mathbf{u}}\in{\mathbf{H}}^{1}(\omega)$ and from the first~step we get
$\operatorname{curl}{\mathbf{u}}\in{\mathbf{H}}^{s+1}(\omega)$. In particular,
${\mathbf{u}}\in{\mathbf{H}}^{1}(\omega,\operatorname{curl})$. Hence,
(\ref{eq:lemma:maxwell-H2-regularity-1}) allows us to write, for some
$\varphi\in H^{2}(\omega)$
\begin{equation}
{\mathbf{u}}=\nabla\varphi+R^{\operatorname{curl}}%
(\underbrace{\operatorname{curl}{\mathbf{u}}}_{\in{\mathbf{H}}^{s+1}(\omega
)}). \label{eq:lemma:maxwell-H2-regularity-35}%
\end{equation}
\emph{3.~step:} An equation for $\varphi$ is obtained in two steps: using
$\Pi_{T}{\mathbf{u}}=0$, we see again that
\[
\nabla_{\partial\omega}\varphi=-\Pi_{T}R^{\operatorname{curl}}%
(\operatorname{curl}{\mathbf{u}})\in{\mathbf{H}}^{s+3/2}(\partial\omega),
\]
where we used the trace estimate and the mapping properties of
$R^{\operatorname{curl}}$. We conclude $g_{D}:=\varphi|_{\partial\omega}\in
H^{s+5/2}(\partial\omega)$. Multiplying
(\ref{eq:lemma:maxwell-H2-regularity-35}) with ${\mathbf{B}}$ and applying the
divergence operator reveals a Poisson type problem for $\varphi$:
\begin{equation}
g=\operatorname{div}{\mathbf{B}}{\mathbf{u}}=\operatorname{div}\left(
{\mathbf{B}}\nabla\varphi\right)  +\operatorname{div}\left(  {\mathbf{B}%
}R^{\operatorname{curl}}(\operatorname{curl}{\mathbf{u}})\right)
\quad\mbox{ in $\omega$},\qquad\varphi=g_{D}\quad\mbox{ on $\partial\omega$.}
\label{eq:lemma:maxwell-H2-regularity-45}%
\end{equation}
By standard elliptic regularity in view of the smoothness of $\partial\omega$
and ${\mathbf{B}}$, we get $\varphi\in H^{s+3}(\omega)$ with
\begin{align}
\Vert\varphi\Vert_{H^{s+3}(\omega)}  &  \lesssim\Vert g-\operatorname{div}%
\left(  {\mathbf{B}}R^{\operatorname{curl}}(\operatorname{curl}{\mathbf{u}%
})\right)  \Vert_{H^{s+1}(\omega)}+\Vert g_{D}\Vert_{H^{s+5/2}(\partial
\omega)}\lesssim\Vert g\Vert_{H^{s+1}(\omega)}+\Vert\operatorname{curl}%
{\mathbf{u}}\Vert_{{\mathbf{H}}^{s+1}(\omega)}%
\nonumber\label{eq:lemma:maxwell-H2-regularity-65}\\
&  \lesssim\Vert g\Vert_{H^{s+1}(\omega)}+\Vert{\mathbf{f}}\Vert_{{\mathbf{H}%
}^{s}(\omega)}.
\end{align}
\emph{4.~step:} Inserting the information
(\ref{eq:lemma:maxwell-H2-regularity-65}) in
(\ref{eq:lemma:maxwell-H2-regularity-35}) implies ${\mathbf{u}}\in{\mathbf{H}%
}^{s+2}(\omega)$ together with $\Vert{\mathbf{u}}\Vert_{{\mathbf{H}}%
^{s+2}(\omega)}\lesssim\Vert g\Vert_{H^{s+1}(\omega)}+\Vert{\mathbf{f}}%
\Vert_{{\mathbf{H}}^{s}(\omega)}.$

\end{proof}


\subsection{Local Analytic Regularity}

We show analytic regularity of solutions of elliptic systems of the form
(\ref{eq:local-system}) on half-balls $B_{r}^{+}:= \{{\mathbf{x}}=
(x_{1},x_{2},x_{3})\in{\mathbb{R}}^{3}\,|\, |{\mathbf{x}}| < r, x_{3} > 0\}$.
We denote $\Gamma_{R}:= \{{\mathbf{x}} \in B_{R}(0)\,|\, {\mathbf{x}}_{3} =
0\}$.

On $B_{R}^{+}$ with $R\leq1$ we consider smooth functions ${\mathbf{u}}$ that
satisfy the following equations for some $\varepsilon>0$:
%
\begin{subequations}
\label{eq:local-system}%
\begin{align}
&  -\varepsilon^{2}\sum_{\alpha,\beta,i=1}^{3}\partial_{\alpha}\left(
A_{\alpha\beta}^{ij}\partial_{\beta}{\mathbf{u}}_{j}\right)  +\varepsilon
\sum_{\beta,j=1}^{3}B_{\beta}^{ij}\partial_{\beta}{\mathbf{u}}_{j}+\sum
_{j=1}^{3}C^{ij}{\mathbf{u}}_{j}={\mathbf{f}}_{i},\qquad i=1,2,3,\\
&  {\mathbf{u}}_{1}={\mathbf{u}}_{2}=0\quad\mbox {on $\Gamma_R$},\\
&  \partial_{3}{\mathbf{u}}_{3}=\varepsilon^{-1}(G+b{\mathbf{u}}_{3}%
)+\sum_{j=1}^{3}d_{j}\partial_{j}{\mathbf{u}}_{3}+\sum_{j=1}^{3}e_{j}%
\partial_{3}{\mathbf{u}}_{j}\quad\mbox{ on $\Gamma_R$}.
\end{align}
We assume that the coefficients are analytic, i.e., (cf.~Def.~\ref{DefClAnFct}%
)
%
\end{subequations}
\begin{subequations}
\label{eq:analyticity-data}%
\begin{align}
(A_{\alpha\beta}^{ij})_{i,j,\alpha,\beta}  &  \in{\mathcal{A}}^{\infty}%
(C_{A},\gamma_{A},B_{R}^{+}), & (B_{\beta}^{ij})_{i,j,\beta}  &
\in{\mathcal{A}}^{\infty}(C_{B},\gamma_{B},B_{R}^{+}), & (C^{ij})_{i,j}  &
\in{\mathcal{A}}^{\infty}(C_{C},\gamma_{C},B_{R}^{+}),\\
b  &  \in{\mathcal{A}}^{\infty}(C_{b},\gamma_{b},B_{R}^{+}), & (d_{j})_{j}  &
\in{\mathcal{A}}^{\infty}(C_{d},\gamma_{d},B_{R}^{+}), & (e_{j})_{j}  &
\in{\mathcal{A}}^{\infty}(C_{e},\gamma_{e},B_{R}^{+})
\end{align}
%
here, we have written, e.g., $(d_{j})_{j}$ to emphasize that the objects are
tensor-valued and the multiindex notation is understood as in
(\ref{defLaplhochn}).
Concerning the tensor $A_{\alpha\beta}^{ij}$ and the coefficients $d_{j}$,
$e_{j}$ we will furthermore make the following structural assumption:
\end{subequations}
\begin{equation}
\label{eq:structural-assumption}A_{\alpha\beta}^{ij}(0)=\delta_{ij}%
\delta_{\alpha\beta},\qquad d_{j}(0)=0,\qquad e_{j}(0)=0.
\end{equation}
This structural assumption implies that the leading order differential
operator in (\ref{eq:local-system}) reduces to a block Laplace operator at the
origin and that the boundary conditions for the third component ${\mathbf{u}%
}_{3}$ reduce to Neumann boundary conditions. In other words: the system
decouples at the origin. Hence, for sufficiently small $R$, we can reduce the
regularity analysis of the system to that of scalar problems, and this is the
avenue taken in the remainder of this appendix.

\begin{remark}
The structural assumption on $A_{\alpha\beta}^{ij}$ implies the
\textquotedblleft very strong ellipticity\textquotedblright/Legendre condition
for the leading order differential operator (near the origin). No sign
conditions are imposed on the coefficients $B_{\alpha}^{ij}$, $C^{ij}$,
$b_{j}$, $d_{j}$, which could even by complex. The condition $\varepsilon>0$
can always be enforced by a scaling so that \textsl{mutatis mutandis} the
ensuing theory is also valid for complex $\varepsilon$. \hbox{}\hfill
\rule{0.8ex}{0.8ex}
\end{remark}

It is convenient to introduce ${\mathcal{E}} \in(0,1]$ by
\begin{equation}
\label{eq:peclet}{\mathcal{E}}^{-1} := \frac{C_{B}}{\varepsilon} + \frac
{\sqrt{C_{C}}}{\varepsilon} + \frac{C_{b}}{\varepsilon} + 1,
\end{equation}
which implies the estimates
\begin{equation}
\label{eq:peclet-estimates}\frac{C_{C}}{\varepsilon^{2}} \leq{\mathcal{E}%
}^{-2}, \qquad\frac{C_{B}}{\varepsilon} \leq{\mathcal{E}}^{-1}, \qquad
\frac{C_{b}}{\varepsilon} \leq{\mathcal{E}}^{-1}, \qquad\frac{\mathcal{E}%
}{\varepsilon} \leq\frac{1}{C_{B}+ \sqrt{C_{c}} + C_{b}}.
\end{equation}
We will make the following assumptions on the right-hand sides
%
\begin{subequations}
\label{eq:analyticity-rhs}%
\begin{align}
\|\nabla^{p} {\mathbf{f}}\|_{L^{2}(B_{R})}  &  \leq C_{f} \gamma_{f}^{p}
\max\{p/R,{\mathcal{E}}^{-1}\}^{p} \qquad\forall p \in{\mathbb{N}}_{0},\\
\|\nabla^{p} G\|_{L^{2}(B_{R})}  &  \leq C_{G} \gamma_{G}^{p} \max
\{p/R,{\mathcal{E}}^{-1}\}^{p} \qquad\forall p \in{\mathbb{N}}_{0}.
\end{align}
Given the special role of the variable $x_{3}$, we will interchangeably use
the notation ${\mathbf{x}}=(x,y)$ with $x=({\mathbf{x}}_{1},{\mathbf{x}}_{2})$
and $y={\mathbf{x}}_{3}$. Analytic regularity of the solution of
(\ref{eq:local-system}) will be characterized in
Theorem~\ref{thm:local-regularity} by the following seminorms:
\end{subequations}
\begin{equation}
N_{R,p,q}^{\prime}(v)=\frac{1}{[p+q]!}\sup_{R/2\leq r<R}(R-r)^{p+q+2}%
\Vert\partial_{y}^{q+2}\nabla_{x}^{p}v\Vert_{L^{2}(B_{r}^{+})},\quad
p\geq0,q\geq-2. \label{eq:Nprime}%
\end{equation}
Our procedure to control $N_{R,p,q}^{\prime}({\mathbf{u}})$ is the standard
one by first controlling tangential derivatives and then using the
differential equation to control normal derivatives. We follow
\cite[Sec.~{5.5}]{MelenkHabil}. In the proofs, we implicitly assume that the
solution ${\mathbf{u}}\in C^{\infty}(B_{R}^{+})$. This could be proved by
carefully arguing with the difference quotient method or, alternatively, by
asserting the smoothness of the solution by a separate argument (this is how
we proceed in the present application of Theorem~\ref{thm:local-regularity}).


\subsubsection{Control of Tangential Derivatives}


We introduce the following auxiliary notation suitable for controlling
tangential derivatives (cf.~\cite[Sec.~{5.5}]{MelenkHabil}) 
\begin{subequations}
\label{eq:Mprime}
\begin{align}
[p]  &  := \max\{1,p\},\\
M^{\prime}_{R,p}(v)  &  = \frac{1}{p!} \sup_{R/2 \leq r < R} (R-r)^{p+2}
\|\nabla_{x}^{p} v\|_{L^{2}(B^{+}_{r})},\\
N^{\prime}_{R,p}(v)  &  =
\begin{cases}
\displaystyle \frac{1}{p!} \sup_{R/2 \leq r < R} (R-r)^{p+2} \|\nabla^{2}
\nabla_{x}^{p} v\|_{L^{2}(B^{+}_{r})} & \mbox{ if } p \ge0\\
\displaystyle \sup_{R/2 \leq r < R} (R-r)^{p+2} \|\nabla^{2+p} v\|_{L^{2}%
(B^{+}_{r})} & \mbox{ if } p =-2, -1,
\end{cases}
\\
H_{R,p}(v)  &  := \frac{1}{[p-1]!} \sup_{R/2 \leq r < R} (R-r)^{p+1}
\!\!\left[  \|\nabla_{x}^{p} v\|_{L^{2}(B^{+}_{r})} + \frac{R-r}{[p]}
\|\nabla^{p}_{x} \nabla v\|_{L^{2}(B^{+}_{r})}\right]  .
\end{align}
\end{subequations}

\begin{lemma}
\label{lemma:Nprimep} There exists a universal constant $C_{I}> 0$ such that
for $f$, $G$ sufficiently smooth, there holds:

\begin{enumerate}
[(i)]

\item \label{item:lemma:Nprimep-i} Let $u$ solve $-\Delta u =f$ on $B_{R}^{+}$
and $u|_{\Gamma_{R}} = 0$.
\begin{equation}
N^{\prime}_{R,p}(u) \leq C_{I}\left[  M^{\prime}_{R,p}(f) + N^{\prime}%
_{R,p-1}(u) + N^{\prime}_{R,p-2}(u)\right]  \qquad\forall p \ge0.
\end{equation}
For $p = 0$, we have the sharper estimate $N^{\prime}_{R,0}(u) \leq C_{I}
\left[  M^{\prime}_{R,0}(f) + N^{\prime}_{R,-1}(u)\right]  $.

\item \label{item:lemma:Nprimep-ii} Let $u$ solve $-\Delta u =f$ on $B_{R}%
^{+}$ and $\partial_{y} u|_{\Gamma_{R}} = G$. Then
\begin{equation}
N^{\prime}_{R,p}(u) \leq C_{I}\left[  M^{\prime}_{R,p}(f) + H_{R,p}(G) +
N^{\prime}_{R,p-1}(u) + N^{\prime}_{R,p-2}(u)\right]  \qquad\forall p \ge0.
\end{equation}
For $p = 0$, we have the sharper estimate $N^{\prime}_{R,0}(u) \leq
C_{I}\left[  M^{\prime}_{R,0}(f) + H_{R,0}(G) + N^{\prime}_{R,-1}(u) \right]
$.
\end{enumerate}
\end{lemma}

\begin{proof}
For the proof of (\ref{item:lemma:Nprimep-i}), see \cite[Lemma~{5.5.15}%
]{MelenkHabil} or \cite[Lemma~{5.7.3'}]{morrey66}. Statement
(\ref{item:lemma:Nprimep-ii}) is essentially taken from \cite[Lemma~{5.5.23}%
]{MelenkHabil}. The special cases $p = 0$ follow from the general case and the
first Poincar\'e inequality in the case (\ref{item:lemma:Nprimep-i}) and the
second Poincar\'e inequality in the case (\ref{item:lemma:Nprimep-ii}).
\end{proof}

\begin{lemma}
\label{lemma:tangential-derivatives} Let ${\mathbf{u}}$ satisfy
(\ref{eq:local-system}) with coefficients and data satisfying
(\ref{eq:analyticity-data}), (\ref{eq:structural-assumption}), and
(\ref{eq:analyticity-rhs}). Let $C_{I}$ be given by Lemma~\ref{lemma:Nprimep}.
Let $R\leq1$ be such that
\begin{equation}
\label{eq:Rsmall}3 C_{I} \left(  C_{A} \gamma_{A} + C_{d} \gamma_{d} + C_{e}
\gamma_{e} \right)  R \leq\frac{1}{2}.
\end{equation}
Then there is $K > 1$ depending only on the constants appearing in
(\ref{eq:analyticity-data}) and on $\gamma_{f}$, $\gamma_{G}$ such that
\begin{align}
\label{eq:lemma:tangential-derivatives}N^{\prime}_{R,p}({\mathbf{u}})  &  \leq
C_{\mathbf{u}} K^{p+2} \frac{\max\{R/{\mathcal{E}},p+3\}^{p+2}}{p!}, \qquad p
\ge-1,\\
C_{\mathbf{u}}  &  = \min\{1,R/{\mathcal{E}}\} (1 + {\mathcal{E}} C_{A}
\gamma_{A}) {\mathcal{E}}\|\nabla{\mathbf{u}}\|_{L^{2}(B^{+}_{R})} +
\min\{1,R/{\mathcal{E}}\}^{2} ({\mathcal{E}}/\varepsilon)^{2} \left[  C_{f} +
C_{C} \|{\mathbf{u}}\|_{L^{2}(B_{R}^{+})} \right] \nonumber\\
&  \quad\mbox{} + C_{G} (1+\gamma_{G}) \min\{1,R/{\mathcal{E}}\}
({\mathcal{E}}/\varepsilon)\nonumber\\
&  \quad\mbox{} + C_{b} (1+\gamma_{b} R) \min\{1,R/{\mathcal{E}}%
\}({\mathcal{E}}/\varepsilon) \|{\mathbf{u}}\|_{L^{2}(B^{+}_{R})} + C_{b}
\min\{1,R/{\mathcal{E}}\}^{2} ({\mathcal{E}}/\varepsilon) {\mathcal{E}}
\|\nabla{\mathbf{u}}\|_{L^{2}(B^{+}_{R})}\nonumber\\
&  \quad\mbox{} + (C_{d} \gamma_{d} + C_{e} \gamma_{e} ) \min
\{1,R/{\mathcal{E}}\}^{2} {\mathcal{E}}^{2} \|\nabla{\mathbf{u}}%
\|_{L^{2}(B^{+}_{R})}.\nonumber
\end{align}

\end{lemma}

\begin{proof}
We start with the observation
\begin{equation}
\min\{1,R/{\mathcal{E}}\}\max\{1,R/{\mathcal{E}}\}=R/{\mathcal{E}}.
\label{eq:minmax}%
\end{equation}
The proof will be by induction on $p$ and we will employ
Lemma~\ref{lemma:Nprimep}. To that end, recall $A_{\alpha\beta}^{ij}%
(0)=\delta_{\alpha\beta}\delta_{ij}$ from (\ref{eq:structural-assumption}). We
write (\ref{eq:local-system}) as
%
\begin{subequations}
\label{eq:local-system-rewritten}%
\begin{align}
-\Delta{\mathbf{u}}_{i}  &  =\varepsilon^{-2}{\mathbf{f}}_{i}-\varepsilon
^{-2}\sum_{j=1}^{3}C^{ij}{\mathbf{u}}_{j}-\varepsilon^{-1}\sum_{\beta,j=1}%
^{3}\widetilde{B}_{\beta}^{ij}\partial_{\beta}{\mathbf{u}}_{j}+\sum
_{\alpha,\beta,j=1}^{3}\Bigl(\underbrace{A_{\alpha\beta}^{ij}-A_{\alpha\beta
}^{ij}(0)}_{=:\widetilde{A}_{\alpha\beta}^{ij}}\Bigr)\partial_{\alpha}%
\partial_{\beta}{\mathbf{u}}_{j},\label{eq:local-system-rewritten-a}\\
{\mathbf{u}_{1}}  &  ={\mathbf{u}}_{2}=0\quad
\mbox{ on $\Gamma$},\label{eq:local-system-rewritten-b}\\
\partial_{3}{\mathbf{u}}_{3}  &  =\varepsilon^{-1}(G+b{\mathbf{u}}_{3}%
)+\sum_{j=1}^{3}d_{j}\partial_{j}{\mathbf{u}}_{3}+\sum_{j=1}^{3}e_{j}%
\partial_{3}{\mathbf{u}}_{j}\quad\mbox{ on $\Gamma$},
\label{eq:local-system-rewritten-c}%
\end{align}
where the coefficient
\end{subequations}
\[
\widetilde{B}_{\beta}^{ij}:=B_{\beta}^{ij}+\varepsilon\sum_{\alpha=1}%
^{3}\partial_{\alpha}\left(  A_{\alpha\beta}^{ij}-A_{\alpha\beta}%
^{ij}(0)\right)  =B_{\beta}^{ij}+\varepsilon\sum_{\alpha=1}^{3}\partial
_{\alpha}A_{\alpha\beta}^{ij}%
\]
is again an analytic function with $(\widetilde{B}_{\beta}^{ij})_{i,j,\beta
}\in{\mathcal{A}}^{\infty}(C_{\widetilde{B}},\gamma_{\widetilde{B}},B_{R}%
^{+})$ with
$C_{\widetilde{B}}:=C_{B}+C_{A}\gamma_{A}\varepsilon$ and $\gamma
_{\widetilde{B}}:=\gamma_{B}+2\gamma_{A}$. (Note: $C_{B}\gamma_{B}%
^{p}+\varepsilon C_{A}\gamma_{A}^{p+1}(p+1)\leq C_{B}\gamma_{B}^{p}%
+\varepsilon C_{A}\gamma_{A}(2\gamma_{A})^{p}$). The system
(\ref{eq:local-system-rewritten}) is of the form analyzed in
Lemma~\ref{lemma:Nprimep}. We therefore get
\begin{align}
N_{R,p}^{\prime}({\mathbf{u}})  &  \leq C_{I}\Bigl[\varepsilon^{-2}%
M_{R,p}^{\prime}({\mathbf{f}})+\sum_{i=1}^{3}M_{R,p}^{\prime}\Bigl(\varepsilon
^{-2}\sum_{j}C^{ij}{\mathbf{u}}_{j}+\varepsilon^{-1}\sum_{\beta,j}%
\widetilde{B}_{\beta}^{ij}\partial_{\beta}{\mathbf{u}}_{j}-\sum_{\alpha
,\beta,j}\widetilde{A}_{\alpha\beta}^{ij}\partial_{\alpha}\partial_{\beta
}{\mathbf{u}}_{j}\Bigr)\label{eq:lemma:tangential-derivatives-10}\\
&  \quad\mbox{}+\varepsilon^{-1}H_{R,p}(G)+\varepsilon^{-1}H_{R,p}%
\bigl(b{\mathbf{u}}_{3}\bigr)+H_{R,p}\bigl(\sum_{j=1}^{3}d_{j}\partial
_{j}{\mathbf{u}}_{3}\bigr)+H_{R,p}\bigl(\sum_{j=1}^{3}e_{j}\partial
_{3}{\mathbf{u}}_{j}\bigr)+N_{R,p-1}^{\prime}({\mathbf{u}})+N_{R,p-2}^{\prime
}({\mathbf{u}})\Bigr].\nonumber
\end{align}
\emph{1.~step:} For $p=-1$, the assertion
(\ref{eq:lemma:tangential-derivatives}) follows directly from $K\geq1$, the
definition of $C_{\mathbf{u}}$, and (\ref{eq:minmax}) since
\begin{equation}
\label{eq:N-1}N^{\prime}_{R,-1}({\mathbf{u}}) \leq R \|\nabla{\mathbf{u}%
}\|_{L^{2}(B^{+}_{R})} \overset{(\ref{eq:minmax})}{=} \min
\{1,R/\mathcal{E}\} \mathcal{E} \|\nabla{\mathbf{u}}\|_{L^{2}(B^{+}_{R})}
\max\{1,R/\mathcal{E}\}.
\end{equation}

\emph{2.~step:} For $p=0$, we employ the sharpened versions of
Lemma~\ref{lemma:Nprimep} which leads to
(\ref{eq:lemma:tangential-derivatives-10}) for $p=0$ where the last term,
$N_{p-2}({\mathbf{u}})$, is dropped. In view of
(\ref{eq:structural-assumption}), we have
\begin{equation}
\label{eq:estimate-dj-ej}\|(d_{j})_j\|_{L^{\infty}(B_{R}^{+})} \leq C_{d}
\gamma_{d} R, \quad\|(e_{j})_j\|_{L^{\infty}(B_{R}^{+})} 
\leq C_{e} \gamma_{e} R,
\quad\|(\widetilde{A}^{ij}_{\alpha\beta})_{i,j,\alpha,\beta}\|_{L^{\infty}(B_{R}^{+})} \leq C_{A}
\gamma_{A} R.
\end{equation}
We estimate with the sharpened version of Lemma~\ref{lemma:Nprimep}:
\begin{align*}
N_{R,0}^{\prime}({\mathbf{u}})  &  \leq C_{I}\Bigl[\varepsilon^{-2}%
M_{R,0}^{\prime}({\mathbf{f}})+\sum_{i=1}^{3}M_{R,0}^{\prime}\Bigl(\varepsilon
^{-2}\sum_{j}C^{ij}{\mathbf{u}}_{j}+\varepsilon^{-1}\sum_{\beta,j}%
\widetilde{B}_{\beta}^{ij}\partial_{\beta}{\mathbf{u}}_{j}+\sum_{\alpha
,\beta,j}\widetilde{A}_{\alpha\beta}^{ij}\partial_{\alpha}\partial_{\beta
}{\mathbf{u}}_{j}\Bigr)\\
&  \quad\mbox{}+\varepsilon^{-1}H_{R,0}(G)+\varepsilon^{-1}H_{R,0}%
\bigl(b{\mathbf{u}}_{3}\bigr)+H_{R,0}\bigl(\sum_{j=1}^{3}d_{j}\partial
_{j}{\mathbf{u}}_{3}\bigr)+H_{R,0}\bigl(\sum_{j=1}^{3}e_{j}\partial
_{3}{\mathbf{u}}_{j}\bigr)+N_{R,-1}^{\prime}({\mathbf{u}})\Bigr]\\
&  \overset{(\ref{eq:estimate-dj-ej})}{\leq}3C_{I}\Bigl[(R/2)^{2}%
\varepsilon^{-2}C_{f}+(R/2)^{2}\varepsilon^{-2}C_{C}\Vert{\mathbf{u}}%
\Vert_{L^{2}(B_{R}^{+})}+C_{\widetilde{B}}(R/2)^{2}\varepsilon^{-1}\Vert
\nabla{\mathbf{u}}\Vert_{L^{2}(B_{R}^{+})}+C_{A}\gamma_{A}RN_{R,0}^{\prime
}({\mathbf{u}})\\
&  \quad\mbox{}+C_{G}R/2\varepsilon^{-1}+C_{G}\gamma_{G}(R/2)^{2}%
\varepsilon^{-1}\max\{1/R,{\mathcal{E}}^{-1}\}\\
&  \quad\mbox{}+{C_{b}}(1+\gamma_{b}R)R/2\varepsilon^{-1}\Vert{\mathbf{u}%
}\Vert_{L^{2}(B_{R}^{+})}+C_{b}(R/2)^{2}\varepsilon^{-1}\Vert\nabla
{\mathbf{u}}\Vert_{L^{2}(B_{R}^{+})}\\
&  \quad\mbox{}+3C_{d}\gamma_{d}(R/2)^{2}\Vert\nabla{\mathbf{u}}\Vert
_{L^{2}(B_{R}^{+})}+C_{d}\gamma_{d}RN_{R,0}^{\prime}({\mathbf{u}})+3
C_{e}\gamma_{e}(R/2)^{2}\Vert\nabla{\mathbf{u}}\Vert_{L^{2}(B_{R}^{+})}%
+C_{e}\gamma_{e}RN_{R,0}^{\prime}({\mathbf{u}})+N_{R,-1}^{\prime}({\mathbf{u}%
})\Bigr]\\
&  \leq3C_{I}\biggl[\frac{1}{4}(R/{\mathcal{E}})^{2}({\mathcal{E}}%
/\varepsilon)^{2}\left\{  C_{f}+C_{C}\Vert{\mathbf{u}}\Vert_{L^{2}(B_{R}^{+}%
)}\right\}  +\frac{1}{4}(R/{\mathcal{E}})^{2}C_{\widetilde{B}}({\mathcal{E}%
}/\varepsilon){\mathcal{E}}\Vert\nabla{\mathbf{u}}\Vert_{L^{2}(B_{R}^{+}%
)}+C_{A}\gamma_{A}RN_{R,0}^{\prime}({\mathbf{u}})\\
&  \quad\mbox{}+\frac{C_{G}}{2}R/{\mathcal{E}}({\mathcal{E}}/\varepsilon
)+\frac{C_{G}}{4}\gamma_{G}R/{\mathcal{E}}({\mathcal{E}}/\varepsilon
)\max\{1,R/{\mathcal{E}}\}\\
&  \quad\mbox{}+\frac{C_{b}}{2}(1+\gamma_{b}R)R/{\mathcal{E}}({\mathcal{E}%
}/\varepsilon)\Vert{\mathbf{u}}\Vert_{L^{2}(B_{R}^{+})}+\frac{C_{b}}%
{4}(R/{\mathcal{E}})^{2}({\mathcal{E}}/\varepsilon)\mathcal{E}\Vert
\nabla{\mathbf{u}}\Vert_{L^{2}(B_{R}^{+})}\\
&  \quad\mbox{}+\frac{3}{4} C_{d}\gamma_{d}(R/{\mathcal{E}})^{2}{\mathcal{E}%
}^{2}\Vert\nabla{\mathbf{u}}\Vert_{L^{2}(B_{R}^{+})}+C_{d}\gamma_{d}%
RN_{R,0}^{\prime}({\mathbf{u}})+\frac{3}{4} C_{e}\gamma_{e}(R/{\mathcal{E}%
})^{2}{\mathcal{E}}^{2}\Vert\nabla{\mathbf{u}}\Vert_{L^{2}(B_{R}^{+})}\\
&  \quad\mbox{}+C_{e}\gamma_{e}RN_{R,0}^{\prime}({\mathbf{u}})+N_{R,-1}%
^{\prime}({\mathbf{u}})\biggr].
\end{align*}
The condition (\ref{eq:Rsmall}) allows us to absorb the three terms $C_{A}
\gamma_{A} R N^{\prime}_{R,0}({\mathbf{u}})$, $C_{d} \gamma_{d} R N^{\prime
}_{R,0}({\mathbf{u}})$, and $C_{e} \gamma_{e} R N^{\prime}_{R,0}({\mathbf{u}%
})$ of the right-hand side in the left-hand side at the expense of a factor
$2$. We next use (\ref{eq:minmax}), the trivial estimates $1 \leq
\max\{1,R/\mathcal{E}\}$, $\min\{1,R/\mathcal{E}\}^{2} \leq\min
\{1,R/\mathcal{E}\}$, the observation $(\mathcal{E}/\varepsilon)
C_{\widetilde{B}} = (\mathcal{E}/\varepsilon) C_{B} + C_{A} \gamma
_{A}\mathcal{E} \leq1 + C_{A} \gamma_{A} \mathcal{E}$ (by
(\ref{eq:peclet-estimates})), and (\ref{eq:N-1}) to see that we have arrived
at $N^{\prime}_{R,0}({\mathbf{u}}) \leq C C_{\mathbf{u}} \max\{1,R/\mathcal{E}%
\}^{2}$. Thus, by selecting $K$ sufficiently large, we have shown the case
$p=0$.

\emph{3.~step:} For $p\geq1$, we proceed by induction, assuming that
(\ref{eq:lemma:tangential-derivatives}) is valid up to $p-1$, which implies,
\begin{equation}
N_{R,p-q}^{\prime}({\mathbf{u}})\leq C_{\mathbf{u}}K^{p+2-q}\frac
{\max\{p+3,R/\mathcal{E}\}^{p+2}}{p!},\qquad q=1,\ldots,p+1.
\label{eq:induction-hyp}%
\end{equation}
We need to estimate the terms in (\ref{eq:lemma:tangential-derivatives-10}).
To bound the terms $M_{R,p}^{\prime}(\sum_{j,\beta}\widetilde{B}_{\beta}%
^{ij}\partial_{\beta}{\mathbf{u}})$ in terms of $N_{R,p-q-1}^{\prime
}({\mathbf{u}})$, it is useful to note the simple facts (cf.~also
\cite[(5.7.19)]{morrey66})
\begin{equation}
|\nabla\nabla_{x}^{p}{\mathbf{u}}|^{2}\leq|\nabla^{2}\nabla_{x}^{p-1}%
{\mathbf{u}}|^{2},\quad p\geq1,\qquad|\nabla\nabla_{x}^{p}{\mathbf{u}}%
|^{2}=|\nabla{\mathbf{u}}|^{2},\quad p=0. \label{eq:morrey-5.7.19}%
\end{equation}
To estimate these terms, we compute (cf.~\cite[Lemma~{5.5.13}]{MelenkHabil}
for similar calculations) with (\ref{eq:estimate-dj-ej}) for the third
estimate:
\begin{align*}
\varepsilon^{-2}M_{R,p}^{\prime}\Bigl(\bigl(\sum_{j}C^{ij}{\mathbf{u}}_{j}\bigr)_i\Bigr)  &
\leq\frac{C_{c}}{4}\sum_{q=0}^{p}\left(  \gamma_{c}\frac{R}{2}\right)
^{q}\left(  \frac{R}{\varepsilon}\right)  ^{2}\frac{[p-q-2]!}{(p-q)!}%
N_{R,p-q-2}^{\prime}({\mathbf{u}}),\\
\varepsilon^{-1}M_{R,p}^{\prime}\Bigl(\bigl(\sum_{j,\beta}\widetilde{B}_{\beta}%
^{ij}\partial_{\beta}{\mathbf{u}}_{j}\bigr)_i\Bigr)  &  \leq\frac{C_{\widetilde{B}}%
}{2}\sum_{q=0}^{p}\left(  \gamma_{\widetilde{B}}\frac{R}{2}\right)  ^{q}%
\frac{R}{\varepsilon}\frac{[p-q-1]!}{(p-q)!}N_{R,p-q-1}^{\prime}({\mathbf{u}%
}),\\
M_{R,p}^{\prime}\Bigl(\bigl(\sum_{j,\alpha,\beta}\widetilde{A}_{\alpha\beta}%
^{ij}\partial_{\alpha}\partial_{\beta}{\mathbf{u}}_{j}\bigr)_i\Bigr)  &  \leq
C_{A}\gamma_{A}RN_{R,p}^{\prime}({\mathbf{u}})+C_{A}\sum_{q=1}^{p}\left(
\gamma_{A}\frac{R}{2}\right)  ^{q}N_{R,p-q}^{\prime}({\mathbf{u}}),\\
\varepsilon^{-1}\frac{1}{[p-1]!}\sup_{R/2\leq r<R}(R-r)^{p+1}\Vert\nabla
_{x}^{p}(b{\mathbf{u}}_{3})\Vert_{L^{2}(B_{r}^{+})}  &  \leq\frac{C_{b}}%
{2}\frac{R[p]}{\varepsilon}\sum_{q=0}^{p}\left(  \frac{\gamma_{b}R}{2}\right)
^{q}\frac{[p-q-2]!}{(p-q)!}N_{R,p-q-2}^{\prime}({\mathbf{u}}_{3}),\\
\varepsilon^{-1}\frac{1}{[p]!}\sup_{R/2\leq r<R}(R-r)^{p+2}\Vert\nabla_{x}%
^{p}\nabla(b{\mathbf{u}}_{3})\Vert_{L^{2}(B_{r}^{+})}  &  \leq\frac{C_{b}}%
{2}\frac{R(p+1)}{\varepsilon}\sum_{q=0}^{p+1}\left(  \frac{\gamma_{b}R}%
{2}\right)  ^{q}\frac{[p-q-1]!}{(p-q+1)!}N_{R,p-q-1}^{\prime}({\mathbf{u}}%
_{3}),\\
\frac{1}{[p-1]!}\sup_{R/2\leq r<R}(R-r)^{p+1}\Vert\nabla_{x}^{p}(\sum
_{j=1}^{3}d_{j}\partial_{j}{\mathbf{u}}_{3})\Vert_{L^{2}(B_{r}^{+})}  &  \leq
C_{d}\gamma_{d}RN_{R,p-1}^{\prime}({\mathbf{u}}_{3})+C_{d}\sum_{q=1}%
^{p}\left(  \frac{\gamma_{d}R}{2}\right)  ^{q}\frac{[p]}{[p-q]}N_{R,p-q-1}%
^{\prime}({\mathbf{u}}_{3}),\\
\frac{1}{[p]!}\sup_{R/2\leq r<R}(R-r)^{p+2}\Vert\nabla_{x}^{p}\nabla
(\sum_{j=1}^{3}d_{j}\partial_{j}{\mathbf{u}}_{3})\Vert_{L^{2}(B_{r}^{+})}  &
\leq C_{d}R\gamma_{d}N_{R,p}^{\prime}({\mathbf{u}}_{3})+C_{d}\sum_{q=1}%
^{p+1}\left(  \frac{\gamma_{d}R}{2}\right)  ^{q}\frac{p+1}{[p-q+1]}%
N_{R,p-q}^{\prime}({\mathbf{u}}_{3}),\\
\frac{1}{[p-1]!}\sup_{R/2\leq r<R}(R-r)^{p+1}\Vert\nabla_{x}^{p}(\sum
_{j=1}^{3}e_{j}\partial_{3}{\mathbf{u}}_{j})\Vert_{L^{2}(B_{r}^{+})}  &  \leq
C_{e}\gamma_{e}RN_{R,p-1}^{\prime}({\mathbf{u}})+C_{e}\sum_{q=1}^{p}\left(
\frac{\gamma_{e}R}{2}\right)  ^{q}\frac{[p]}{[p-q]}N_{R,p-q-1}^{\prime
}({\mathbf{u}}),\\
\frac{1}{[p]!}\sup_{R/2\leq r<R}(R-r)^{p+2}\Vert\nabla_{x}^{p}\nabla
(\sum_{j=1}^{3}e_{j}\partial_{3}{\mathbf{u}}_{j})\Vert_{L^{2}(B_{r}^{+})}  &
\leq C_{e}R\gamma_{e}N_{R,p}^{\prime}({\mathbf{u}})+C_{e}\sum_{q=1}%
^{p+1}\left(  \frac{\gamma_{e}R}{2}\right)  ^{q}\frac{p+1}{[p-q+1]}%
N_{R,p-q}^{\prime}({\mathbf{u}}).
\end{align*}
We choose
\[
K>\max\{1,\gamma_{f}/2,\gamma_{G}/2,\gamma_{A}/2,\gamma_{\widetilde{B}%
}/2,\gamma_{C}/2,\gamma_{b}/2,\gamma_{d}/2,\gamma_{e}/2\}
\]
such that the expression in brackets $[\cdots]$ in
(\ref{eq:lemma:tangential-derivatives-100}) is smaller than $1/(6C_{I})$ (and,
of course, such that the case $p=0$ is proved); note that $\mathcal{E}%
/\varepsilon$ is controlled in view of (\ref{eq:peclet-estimates}). The
calculation in \cite[p.~{206}, bottom]{MelenkHabil} for $M_{R,p}^{\prime
}({\mathbf{f}})$ and similar calculations for $H_{R,p}(G)$ give
\begin{align}
\varepsilon^{-2}M_{R,p}^{\prime}({\mathbf{f}})  &  \leq C_{f}\min
\{1,R/{\mathcal{E}}\}^{2}({\mathcal{E}}/\varepsilon)^{2}K^{p+2}\frac
{\max\{p+3,R/\mathcal{E}\}^{p+2}}{p!}\frac{1}{4}K^{-2}\left(  \frac{\gamma
_{f}}{2K}\right)  ^{p},\\
\varepsilon^{-1}H_{R,p}(G)  &  \leq C_{G}\min\{1,R/{\mathcal{E}}%
\}({\mathcal{E}}/\varepsilon)K^{p+2}\frac{\max\{p+3,R/\mathcal{E}\}^{p+2}}%
{p!}\left[  \frac{1}{2}K^{-2}\left(  \frac{\gamma_{G}}{2K}\right)  ^{p}%
+\frac{1}{2K}\left(  \frac{\gamma_{G}}{2K}\right)  ^{p+1}\right]  .
\end{align}
We use the induction assumption (\ref{eq:induction-hyp})
and (\ref{eq:minmax}) (to deal with the cases where $N_{R,-2}^{\prime
}({\mathbf{u}})$ is involved) to estimate
\begin{align*}
R^{2}\varepsilon^{-2}\frac{[p-q-2]!}{(p-q)!}N_{R,p-q-2}^{\prime}({\mathbf{u}%
})  &  \leq(\mathcal{E}/\varepsilon)^{2}C_{\mathbf{u}}K^{p+2-q-2}\frac
{\max\{p+3,R/\mathcal{E}\}^{p+2}}{p!},\qquad q=0,\ldots,p,\\
\varepsilon^{-1}\frac{[p-q-1]!}{(p-q)!}N_{R,p-q-1}^{\prime}({\mathbf{u}})  &
\leq(\mathcal{E}/\varepsilon)C_{\mathbf{u}}K^{p+2-q-1}\frac{\max
\{p+3,R/\mathcal{E}\}^{p+2}}{p!},\qquad q=0,\ldots,p,\\
\lbrack p]R/\varepsilon\frac{\lbrack p-q-2]!}{(p-q)!}N_{R,p-q-2}^{\prime
}({\mathbf{u}})  &  \leq(\mathcal{E}/\varepsilon)C_{\mathbf{u}}K^{p+2-q-2}%
\frac{\max\{p+3,R/\mathcal{E}\}^{p+2}}{p!},\qquad q=0,\ldots,p,\\
(p+1)R/\varepsilon\frac{\lbrack p-q-1]!}{(p-q+1)!}N_{R,p-q-1}^{\prime
}({\mathbf{u}})  &  \leq(\mathcal{E}/\varepsilon)C_{\mathbf{u}}K^{p+2-q-1}%
\frac{\max\{p+3,R/\mathcal{E}\}^{p+2}}{p!},\qquad q=0,\ldots,p+1,\\
\frac{\lbrack p]}{[p-q]}N_{R,p-q-1}^{\prime}({\mathbf{u}})  &  \leq
C_{\mathbf{u}}K^{p+2-q-1}\frac{\max\{p+3,R/\mathcal{E}\}^{p+2}}{p!},\qquad
q=1,\ldots,p,\\
\frac{\lbrack p+1]}{[p-q+1]}N_{R,p-q}^{\prime}({\mathbf{u}})  &  \leq
C_{\mathbf{u}}K^{p+2-q}\frac{\max\{p+3,R/\mathcal{E}\}^{p+2}}{p!},\qquad
q=1,\ldots,p+1,\\
&
\end{align*}
We note that (\ref{eq:peclet-estimates}) gives
\begin{equation}
C_{C}(\mathcal{E}/\varepsilon)^{2}\leq1,\quad C_{\widetilde{B}}(\mathcal{E}%
/\varepsilon)\leq1+C_{A}\gamma_{A}\mathcal{E},\quad C_{b}(\mathcal{E}%
/\varepsilon)\leq1. \label{eq:peclet-further}%
\end{equation}
Inserting all of the above in (\ref{eq:lemma:tangential-derivatives-10})
yields\footnote{the factor $3$ in $3C_{I}$ is due to the summation over $i$
and likely suboptimal} together with the geometric series
\begin{align}
N_{R,p}^{\prime}({\mathbf{u}})  &  \leq3C_{I}\left(  C_{A}\gamma_{A}%
R+C_{d}\gamma_{d}R+C_{e}\gamma_{e}R\right)  N_{R,p}^{\prime}({\mathbf{u}%
})+C_{\mathbf{u}}K^{p+2}\frac{\max\{p+3,R/\mathcal{E}\}^{p+2}}{p!}\nonumber\\
&  \ \times3C_{I}\Bigl[\frac{1}{4K}\left(  \frac{\gamma_{f}}{2K}\right)
^{p}+\frac{1}{2}2K^{-2}\left(  \frac{\gamma_{G}}{2K}\right)  ^{p}+\frac
{1}{4K^{2}}\frac{1}{1-(\gamma_{c}R/(2K))}+K^{-1}\frac{1+C_{A}\gamma
_{A}\mathcal{E}}{1-R\gamma_{\widetilde{B}}/(2K)}\nonumber\\
&  \qquad+\frac{C_{A}\gamma_{A}R}{2K}\frac{1}{1-\gamma_{A}R/(2K)}+\frac
{1}{2K^{2}}\frac{1}{1-\gamma_{b}R/(2K)}+\frac{1}{2K}\frac{1}{1-\gamma
_{b}R/(2K)}\nonumber\\
&  \qquad+(1+K^{-1})\frac{C_{d}\gamma_{d}R}{2K}\frac{1}{1-\gamma_{d}%
R/(2K)}+(1+K^{-1})\frac{C_{e}\gamma_{e}R}{2K}\frac{1}{1-\gamma_{e}%
R/(2K)}+K^{-1}+K^{-2}\Bigr] \label{eq:lemma:tangential-derivatives-100}%
\end{align}
By the choice of $K$, the expression in brackets, $[\cdots]$, is smaller than
$1/(6C_{I})$ and by (\ref{eq:Rsmall}) the expression $3C_{I}(C_{A}\gamma
_{A}R+C_{d}\gamma_{d}R+C_{e}\gamma_{e}R)\leq1/2$. Hence, the induction step is completed.
\end{proof}


\subsubsection{Control of Normal Derivatives}

Recall $N^{\prime}_{R,p,q}$ from (\ref{eq:Nprime}) and define
\begin{equation}
\label{eq:Mprimepq}M^{\prime}_{R,p,q}({\mathbf{v}}) := \frac{1}{[p+q]!}
\sup_{R/2 \leq r < R} (R - r)^{p+q+2} \|\partial_{y}^{q} \nabla_{x}^{p}
{\mathbf{v}}\|_{L^{2}(B_{r}^{+})}.
\end{equation}

\begin{theorem}
\label{thm:local-regularity} Let ${\mathbf{u}}$ satisfy (\ref{eq:local-system}%
) with coefficients and data satisfying (\ref{eq:analyticity-data}),
(\ref{eq:structural-assumption}), and (\ref{eq:analyticity-rhs}). Let $R\leq1$
be such that, with the universal constant $C_{I}$ given by
Lemma~\ref{lemma:Nprimep},
\[
(3C_{I}+6)\left(  C_{A}\gamma_{A}+C_{d}\gamma_{d}+C_{e}\gamma_{e}\right)
R\leq\frac{1}{2}.
\]
Then there are $K_{1}$, $K_{2}>1$ depending only on the constants appearing in
(\ref{eq:analyticity-data}) and on $\gamma_{f}$, $\gamma_{G}$ such that with
$C_{\mathbf{u}}$ given by (\ref{eq:lemma:tangential-derivatives})
\begin{equation}
N_{R,p,q}^{\prime}({\mathbf{u}})\leq C_{\mathbf{u}}K_{1}^{p+2}K_{2}^{q+2}%
\frac{\max\{R/\mathcal{E},p+q+3\}^{p+q+2}}{[p+q]!} \qquad\forall p\geq
0,q\geq-2\mbox{ with $(p,q) \ne (0,-2)$.} \label{eq:local-regularity-1}%
\end{equation}

\end{theorem}

\begin{proof}
Let $K$ be given by Lemma~\ref{lemma:tangential-derivatives}. We select $K_{1}
\ge K$, $K_{2}$ such that
\[
K_{2}>\max\{1,\gamma_{f}/2,\gamma_{A}/2,\gamma_{\widetilde{B}}/2,\gamma
_{C}/2\}
\]
and such that the expression in brackets, $[\cdots]$, in
(\ref{eq:thm:local-regularity-bracket}) is smaller than $1/12$. This is indeed
possible by first selecting $K_{1} \ge K$ sufficiently large (e.g., such that
$C_{A} \gamma_{A} /(2K_{1})/(1-(\gamma_{A}/(2K_{1})) < 1/12$) and then
selecting $K_{2}$ sufficiently large. The proof is by induction on $q$. For
$q\in\{-2,-1,0\}$ and all $p\in{\mathbb{N}}_{0}$ (with the exception of the
excluded case $(q,p)=(-2,0)$) the result follows directly from
Lemma~\ref{lemma:tangential-derivatives}. Let us assume that
(\ref{eq:local-regularity-1}) holds (for all $p\in{\mathbb{N}}_{0}$) up to
$q-1$ for some $q\geq1$.

Starting point is the observation that for a smooth solution
$\widetilde{\mathbf{u}}$ of
\begin{equation}
\label{eq:thm:local-regularity-5}-\partial_{y}^{2} \widetilde{\mathbf{u}} =
\Delta_{x} \widetilde{\mathbf{u}} + \widetilde{\mathbf{f}} \quad
\mbox{on $B_R^+$}
\end{equation}
we have by the definition of the seminorms $N^{\prime}_{R,r,s}$, $M^{\prime
}_{R,r,s}$ the estimate
\begin{equation}
\label{eq:thm:local-regularity-10}N^{\prime}_{R,p,q}(\widetilde{\mathbf{u}})
\leq2\left[  N^{\prime}_{R,p+2,q-2}(\widetilde{\mathbf{u}}) + M^{\prime}%
_{p,q}(\widetilde{\mathbf{f}}) \right]  , \qquad p \ge0, q \ge0.
\end{equation}
The system (\ref{eq:local-system-rewritten}) is of the form
(\ref{eq:thm:local-regularity-5}) with
\begin{align}
\label{eq:thm:local-regularity-20}\widetilde{\mathbf{f}}_{i}  &  = \sum
_{j=1}^{3} \widetilde{A}^{ij}_{33} \partial_{y}^{2} {\mathbf{u}}_{j} +
\varepsilon^{-2} {\mathbf{f}}_{i} - \varepsilon^{-2} \sum_{j} C^{ij}
{\mathbf{u}}_{j} - \varepsilon^{-1} \sum_{j,\beta} \widetilde{B}^{ij}_{\beta
}\partial_{\beta}{\mathbf{u}}_{j} + \sum_{\substack{j,\alpha,\beta
\\(\alpha,\beta) \ne(3,3)}} \widetilde{A}^{ij}_{\alpha\beta} \partial_{\alpha
}\partial_{\beta}{\mathbf{u}}_{j}.
\end{align}
We estimate
\begin{align*}
&  \varepsilon^{-2} M^{\prime}_{R,p,q}({\mathbf{f}}) \leq\frac{C_{f}}{4}
\min\{1,R/\mathcal{E}\}^{2} ({\mathcal{E}}/\varepsilon)^{2} \left(
\frac{\gamma_{f}}{2}\right)  ^{p+q} \frac{\max\{p+q+3,R/{\mathcal{E}%
}\}^{p+q+2}}{[p+q]!},\\
&  \varepsilon^{-2} M^{\prime}_{R,p,q}((\sum_{j} C^{ij} {\mathbf{u}}_{j})_i) \leq
C_{C} \varepsilon^{-2} \sum_{r=0}^{p}\sum_{s=0}^{q} \binom{p}{r}\binom{q}{s}
\gamma_{C}^{r+s} \frac{s!r!}{(p+q)!} \sup_{R/2 \leq r < R} (R - r)^{p+q+2}
\|\partial_{y}^{q-s} \nabla_{x}^{p-r} {\mathbf{u}}\|_{L^{2}(B^{+}_{r})}\\
&  \qquad\qquad\qquad\leq\frac{C_{C}}{4} (R/\varepsilon)^{2} \sum_{r=0}%
^{p}\sum_{s=0}^{q} \binom{p}{r}\binom{q}{s} \left(  \frac{ \gamma_{C} R}%
{2}\right)  ^{r+s} \frac{s!r! [p+q-r-s-2]!}{(p+q)!} N^{\prime}_{R,p-r,q-s-2}%
({\mathbf{u}}),\\
&  \varepsilon^{-1} M^{\prime}_{R,p,q}((\sum_{j,\beta} \widetilde{B}%
^{ij}_{\beta}\partial_{\beta}{\mathbf{u}}_{j})_i) \leq\\
&  \quad\frac{C_{\widetilde{B}} }{2} (R/\varepsilon) \sum_{r=0}^{p}\sum
_{s=0}^{q} \binom{p}{r}\binom{q}{s} \left(  \frac{ \gamma_{\widetilde{B}}
R}{2}\right)  ^{r+s} \frac{s!r! [p+q-r-s-1]!}{(p+q)!} \left[  N^{\prime
}_{R,p-r,q-s-1}({\mathbf{u}}) + N^{\prime}_{R,p+1-r,q-s-2}({\mathbf{u}%
})\right]  ,\\
&  M^{\prime}_{R,p,q}((\sum_{\substack{j,\alpha,\beta\\(\alpha,\beta) \ne
(3,3)}} \widetilde{A}^{ij}_{\alpha\beta} \partial_{\alpha}\partial_{\beta
}{\mathbf{u}}_{j})_i) \leq\\
&  \quad C_{A} \sum_{r=0}^{p}\sum_{s=0}^{q} \binom{p}{r}\binom{q}{s} \left(
\frac{ \gamma_{A} R}{2}\right)  ^{r+s} \frac{s!r! [p+q-r-s]!}{(p+q)!} \left[
N^{\prime}_{R,p+2-r,q-s-2}({\mathbf{u}}) + N^{\prime}_{R,p+1-r,q-s-1}%
({\mathbf{u}})\right]  ,\\
&  M^{\prime}_{R,p,q}((\sum_{j} \widetilde{A}^{ij}_{33} \partial^{2}_{3}
{\mathbf{u}}_{j})_{i}) \leq C_{A} \gamma_{A} R N^{\prime}_{R,p,q}({\mathbf{u}%
}) +\\
&  \quad C_{A} \sum_{(r,s) \ne(0,0)} \binom{p}{r}\binom{q}{s} \left(  \frac{
\gamma_{A} R}{2}\right)  ^{r+s} \frac{s!r! [p+q-r-s]!}{(p+q)!} N^{\prime
}_{R,p-r,q-s}({\mathbf{u}}).
\end{align*}
We introduce the abbreviation
\begin{equation}
m_{p,q}:= \max\{p+q+3,R/{\mathcal{E}}\}^{p+q+2}.
\end{equation}
so that, for $q^{\prime}\leq q-1$ the induction hypothesis reads ${[}%
p^{\prime}+q^{\prime}{]}! N^{\prime}_{R,p^{\prime},q^{\prime}}({\mathbf{u}})
\leq C_{\mathbf{u}} K_{1}^{p^{\prime}+2} K_{2}^{q^{\prime}+2} m_{p^{\prime
},q^{\prime}}$. We have
\begin{align*}
\varepsilon^{-2} M^{\prime}_{R,p,q}({\mathbf{f}})  &  \leq\left[
\min\{1,R/\mathcal{E}\}^{2} \frac{C_{f} ({\mathcal{E}}/\varepsilon)^{2}}{4
K_{1}^{2} K_{2}^{2}} \left(  \frac{\gamma_{f}}{2 K_{1}} \right)  ^{p} \left(
\frac{\gamma_{f}}{K_{2}} \right)  ^{q}\right]  K_{1}^{p+2} K_{2}^{q+2}
\frac{m_{p,q}}{(p+q)!}.
\end{align*}
We recall the elementary estimates
\begin{equation}
\binom{p}{r} r! \leq p^{r} , \qquad0 \leq r \leq p.
\end{equation}
Recalling (\ref{eq:peclet-estimates}) (cf.\ also (\ref{eq:peclet-further})) we
get from the induction hypothesis
\begin{align*}
&  \varepsilon^{-2} M^{\prime}_{R,p,q}\Bigl(\sum_{j} C^{ij} {\mathbf{u}}%
_{j}\Bigr) \leq C_{\mathbf{u}} \frac{C_{C} (\mathcal{E}/\varepsilon)^{2}}{4}
\sum_{r,s} \binom{p}{r}\binom{q}{s} \left(  \frac{\gamma_{C} R}{2} \right)
^{r+s} \frac{r!s! (R/\mathcal{E})^{2}}{[p+q]!} K_{1}^{p-r+2} K_{2}^{q-s}
m_{p-r,q-s-2}\\
&  \quad\leq C_{\mathbf{u}} \frac{1}{4} \sum_{r,s} \left(  \frac{\gamma_{C}
R}{2} \right)  ^{r+s} K_{1}^{p-r+2} K_{2}^{q-s} \frac{m_{p,q}}{[p+q]!},\\
&  \varepsilon^{-1} M^{\prime}_{R,p,q}\Bigl(\sum_{j,\beta} \partial_{\beta
}\widetilde{B}^{ij}_{\beta}{\mathbf{u}}_{j}\Bigr)\\
&  \quad\leq C_{\mathbf{u}} {C_{\widetilde{B}}} (\mathcal{E}/\varepsilon)
\sum_{r,s} \binom{p}{r}\binom{q}{s} \left(  \frac{\gamma_{\widetilde{B}} R}{2}
\right)  ^{r+s} \frac{r!s! R/\mathcal{E}}{[p+q]!} \left[  K_{1}^{p-r+2}
K_{2}^{q-s+1} m_{p-r,q-s-1} + K_{1}^{p-r+3} K_{2}^{q-s} m_{p+1-r,q-s-2}
\right] \\
&  \quad\leq C_{\mathbf{u}} (1+C_{A} \gamma_{A} \mathcal{E}) \sum_{r,s}
\left(  \frac{\gamma_{\widetilde{B}} R}{2} \right)  ^{r+s} K_{1}^{p-r+2}
K_{2}^{q+2-s} \frac{m_{p,q}}{[p+q]!} \left[  K_{2}^{-1} + K_{1} K_{2}%
^{-2}\right]  ,\\
&  M^{\prime}_{R,p,q}\Bigl(\sum_{\substack{j,\alpha,\beta\\(\alpha,\beta)
\ne(3,3)}} \widetilde{A}^{ij}_{\alpha\beta} \partial_{\alpha}\partial_{\beta
}{\mathbf{u}}_{j}\Bigr)\\
&  \quad\leq C_{\mathbf{u}} {C_{A}} \sum_{r,s} \binom{p}{r}\binom{q}{s}
\left(  \frac{\gamma_{A} R}{2} \right)  ^{r+s} \frac{r!s!}{[p+q]!} \left[
K_{1}^{p-r+4} K_{2}^{q-s} m_{p+2-r,q-s-2} + K_{1}^{p-r+3} K_{2}^{q-s+1}
m_{p+1-r,q-s-1} \right] \\
&  \leq C_{\mathbf{u}} {C_{A}} \sum_{r,s} \left(  \frac{\gamma_{A} R}{2}
\right)  ^{r+s} K_{1}^{p-r+2} K_{2}^{q+2-s} \frac{m_{p,q}}{[p+q]!} \left[
K_{1}^{2} K_{2}^{-2} + K_{1} K_{2}^{-1}\right]  ,\\
&  M^{\prime}_{R,p,q}\Bigl(\sum_{j} \widetilde{A}^{ij}_{33} \partial^{2}_{3}
{\mathbf{u}}_{j}\Bigr) \leq C_{A} \gamma_{A} R N^{\prime}_{R,p,q}({\mathbf{u}%
}) + C_{\mathbf{u}} {C_{A}} \sum_{(r,s) \ne(0,0)} \left(  \frac{\gamma_{A}
R}{2} \right)  ^{r+s} K_{1}^{p-r+2} K_{2}^{q+2-s} \frac{m_{p,q}}{[p+q]!}.
\end{align*}
We note for $\delta_{1}$, $\delta_{2} \in(0,1)$ that
\begin{align*}
\sum_{(r,s) \ne(0,0)} \delta_{1}^{r} \delta_{2}^{s} = \frac{1}{(1-\delta
_{1})(1-\delta_{2})} -1 = \frac{\delta_{1}+\delta_{2} - \delta_{1}\delta_{2}%
}{(1-\delta_{1})(1-\delta_{2})} \leq\frac{\delta_{1}+\delta_{2}}{(1-\delta
_{1})(1-\delta_{2})}.
\end{align*}
Inserting these estimates in (\ref{eq:thm:local-regularity-10}) gives (recall
that $j$ from $1$ to $3$ which accounts for a generous factor $3$):
\begin{align}
\label{eq:thm:local-regularity-bracket}N^{\prime}_{R,p,q}({\mathbf{u}})  &
\leq6 C_{A} \gamma_{A} R N^{\prime}_{R,p,q}({\mathbf{u}}) + C_{\mathbf{u}}
K_{1}^{p+2} K_{2}^{q+2} \frac{m_{p,q}}{(p+q)!} 6 \Bigl[\\
&  \frac{1}{4 K_{1}^{2} K_{2}^{2}} \left(  \frac{\gamma_{f}}{2 K_{1}}\right)
^{p} \left(  \frac{\gamma_{f}}{ K_{2}}\right)  ^{q} + \frac{1 }{4 K_{2}^{2}}
\frac{1}{1 - \gamma_{C} R/(2 K_{1})} \frac{1}{1 - \gamma_{C} R/(2 K_{2}%
)}\nonumber\\
&  \mbox{} + (1+C_{A} \gamma_{A}\mathcal{E}) (K_{2}^{-1} + K_{1} K_{2}^{-2})
\frac{1}{1 - \gamma_{\widetilde{B}} R/(2 K_{1})} \frac{1}{1 - \gamma
_{\widetilde{B}} R/(2 K_{2})}\nonumber\\
&  \mbox{} + {C_{A}} (K_{1}^{2} K_{2}^{-2} + K_{1} K_{2}^{-1} + \frac
{\gamma_{A} R}{2 K_{1}} + \frac{\gamma_{A} R}{2 K_{2}} ) \frac{1}{1 -
\gamma_{A} R/(2 K_{1})} \frac{1}{1 - \gamma_{A} R/(2 K_{2})} + K_{1}^{2}
K_{2}^{-2} \Bigr]\nonumber
\end{align}
By the choice of $K_{2}$, the expression in brackets, $[\cdots]$, is smaller
than $1/12$ and by assumption on $R$, the term $C_{A} \gamma_{A} R \leq1/2$.
Hence, the induction step is completed.
\end{proof}


\section{Analytic regularity for Poisson Problems}

We consider, on the half-ball $B_{R}^{+}$, solutions $u$ of
\begin{equation}
-\operatorname{div}\left(  A(x)\nabla u\right)  =f\quad
\mbox{ in $B^+_R$},\qquad u|_{\Gamma_{R}}=0. \label{eq:poisson-dirichlet}%
\end{equation}
Here, the matrix $A$ is pointwise symmetric positive definite and satisfies
\begin{equation}
A\in{\mathcal{A}}^{\infty}(C_{A},\gamma_{A},B_{R}^{+}),\qquad A\geq
\lambda_{min}>0. \label{eq:cond-on-A}%
\end{equation}
The data $f$ is assumed to satisfy, for some $\varepsilon\in(0,1]$
\begin{equation}
\Vert\nabla^{p}f\Vert_{L^{2}(B_{R}^{+})}\leq C_{f}\gamma_{f}^{p}%
\max\{p/R,\varepsilon^{-1}\}^{p}\qquad\forall p\in{\mathbb{N}}_{0}.
\label{eq:cond-f}%
\end{equation}
Note that this problem has been considered in \cite[Lemma~{5.5.15}%
]{MelenkHabil} where a recursion for the tangential derivatives, i.e., for the
seminorm $N_{R,p}^{\prime}(u)$, is derived. We use this result here to derive
the following estimate.

\begin{lemma}
\label{lemma:poisson-tangential} Assume (\ref{eq:cond-on-A}) and
(\ref{eq:cond-f}) and $R \leq1$. Then there exists $K >0$ depending solely on
$\lambda_{min}$, $C_{A}$, $\gamma_{A}$, $\gamma_{f}$ such that a solution $u$
of (\ref{eq:poisson-dirichlet}) satisfies
\begin{align}
\label{eq:lemma:poisson-tangential-10}N^{\prime}_{R,p}(u) \leq K^{p+2} \left[
C_{f} R^{2} \frac{\max\{p+1,R/\varepsilon\}^{p}}{p!} + (p+1) R \|\nabla
u\|_{L^{2}(B^{+}_{R})}\right]  \qquad\forall p \ge0.
\end{align}
Additionally, $N^{\prime}_{R,-1}(u) \leq R/2 \|\nabla u\|_{L^{2}(B^{+}_{R})}$.
\end{lemma}

\begin{proof}
The estimate $N_{R,-1}^{\prime}(u)\leq R/2\Vert\nabla u\Vert_{L^{2}(B_{R}%
^{+})}$ is a direct consequence of the definition. The case $p=0$ follows
directly from \cite[Lemma~{5.5.15}]{MelenkHabil}. The case $p=1$ follows from
an inspection of the arguments below. For $p\geq2$, the proof is by induction
on $p$, assuming that (\ref{eq:lemma:poisson-tangential-10}) holds for all
$p^{\prime}\leq p-1$ for some $p\geq2$. From \cite[Lemma~{5.5.15}%
]{MelenkHabil} we get
\begin{align}
N_{R,p}^{\prime}(u)  &  \leq C_{B}^{\prime}\bigl[C_{f}\left(  \frac{R}%
{2}\right)  ^{2}\left(  \frac{\gamma_{f}}{2}\right)  ^{p}\frac{\max
\{p,R/\varepsilon\}^{p}}{p!}+C_{A}(p+1)\left(  \frac{\gamma_{A}R}{2}\right)
^{p+1}N_{R,-1}^{\prime}(u)\label{eq:lemma:poisson-tangential-20}\\
&  +C_{A}\sum_{q=1}^{p}\frac{(p+1)!}{(p+1-q)!}\left(  \frac{\gamma_{A}R}%
{2}\right)  ^{q}\frac{(p-q)!}{p!}N_{R,p-q}^{\prime}(u)+N_{R,p-1}^{\prime
}(u)+N_{R,p-2}^{\prime}(u)\bigr].\nonumber
\end{align}
The induction hypothesis gives for $q=1,\ldots,p$
\[
\frac{(p+1)!}{(p+1-q)!}\frac{(p-q)!}{p!}N_{R,p-q}^{\prime}(u)\leq
K^{p+2-q}\left[  C_{f}R^{2}\frac{\max\{p+1,R/\varepsilon\}^{p}}{p!}%
+(p+1)R\Vert\nabla u\Vert_{L^{2}(B_{R}^{+})}\right]  =:K^{p+2-q}B_{p},
\]
where $B_{p}$ abbreviates the expression in brackets, $\bigl[\cdots\bigr]$.
Inserting the above and the induction hypothesis in
(\ref{eq:lemma:poisson-tangential-20}) gives, assuming $\gamma_{A}R/(2K)<1$,
\[
N_{R,p}^{\prime}(u)\leq K^{p+2}C_{B}^{\prime}B_{p}\Bigl[\frac{1}{4K^{2}%
}\left(  \frac{\gamma_{f}}{2K}\right)  ^{p}+C_{A}K^{-1}\left(  \frac
{\gamma_{A}R}{2K}\right)  ^{p+1}+C_{A}\frac{\gamma_{A}R}{2K}\frac{1}%
{1-\gamma_{A}R/(2K)}+K^{-1}+K^{-2}\Bigr].
\]
Selecting $K$ sufficiently large shows that the factor $C_{B}^{\prime}%
[\cdots]$ can be made smaller than $1$, which concludes the induction argument.
\end{proof}

The following theorem generalizes \cite[Prop. 5.5.2]{MelenkHabil} from
homogeneous Dirichlet boundary condition on the whole boundary of $B_{R}^{+}$
to the condition $\left.  u\right\vert _{\Gamma_{R}}=0$.

\begin{theorem}
\label{thm:poisson} Assume (\ref{eq:cond-on-A}) and (\ref{eq:cond-f}) and
$R\leq1$. Then there exist $K_{1}$, $K_{2}\geq1$ depending solely on
$\lambda_{min}$, $C_{A}$, $\gamma_{A}$, $\gamma_{f}$ such that a solution $u$
of (\ref{eq:poisson-dirichlet}) satisfies, for all $p\geq0$, $q\geq-2$ with
$(p,q)\neq(0,-2)$
\begin{equation}
N_{R,p,q}^{\prime}(u)\leq K_{1}^{p+2}K_{2}^{q+2}\left[  C_{f}R^{2}\frac
{\max\{p+q+3,R/\varepsilon\}^{p+q}}{(p+q)!}+(p+q+3)R\Vert\nabla u\Vert
_{L^{2}(B_{R}^{+})}\right]  . \label{eq:thm:poisson-tangential-10}%
\end{equation}

\end{theorem}

\begin{proof}
We control the normal derivatives as in the proof of
Theorem~\ref{thm:local-regularity}. Inspection of the arguments leading to
\cite[(5.5.30)]{MelenkHabil} shows that we have
\begin{equation}
-\partial_{y}^{2}u=\widetilde{f}+\widetilde{A}\nabla u+B:\nabla^{2}u,
\label{eq:5.5.30}%
\end{equation}
where, for $C^{\prime}$, $\gamma>0$ depending solely on $\lambda_{min}$,
$C_{A}$, $\gamma_{A}$, $\gamma_{f}$,
\begin{align}
&  \Vert\nabla^{p}\widetilde{f}\Vert_{L^{2}(B_{R}^{+})}\leq C^{\prime}%
C_{f}\gamma^{p}\max\{p/R,\varepsilon^{-1}\}^{p}\qquad\forall p\in{\mathbb{N}%
}_{0},\\
&  \widetilde{A},B\in{\mathcal{A}}^{\infty}(C^{\prime},\gamma,B_{R}^{+}),\quad
B_{33}\equiv0.
\end{align}
We abbreviate
\begin{equation}
M_{p,q}:=\left[  C_{f}R^{2}\frac{\max\{p+q+3,R/\varepsilon\}^{p+q}}%
{(p+q)!}+(p+q+3)R\Vert\nabla u\Vert_{L^{2}(B_{R}^{+})}\right]  =:M_{p,q}%
^{(1)}+M_{p,q}^{(2)}. \label{eq:Mpq}%
\end{equation}
The proof is by induction on $q$, the cases $q\in\{-2,-1,0\}$ being shown in
Lemma~\ref{lemma:poisson-tangential} if we select $K_{1}=K$ with $K$ given by
Lemma~\ref{lemma:poisson-tangential}. Assume that
(\ref{eq:thm:poisson-tangential-10}) holds for all $q^{\prime}\leq q-1$ for
some $q\geq0$ and all $p$. {}From (\ref{eq:5.5.30}) we get
\begin{align}
N_{R,p,q}^{\prime}(u)  &  \leq M_{R,p,q}^{\prime}(\widetilde{f})+M_{R,p,q}%
^{\prime}(\widetilde{A}\nabla u)+M_{R,p,q}^{\prime}(B\colon\nabla^{2}u),\\
M_{R,p,q}^{\prime}(\widetilde{f})  &  \leq\frac{C^{\prime}}{4}\left(
\frac{\gamma}{2}\right)  ^{p+q}M_{p,q},
\end{align}
where the estimate for $M_{R,p,q}^{\prime}(\widetilde{f})$ follows from a
direct calculation. The terms $M_{R,p,q}^{\prime}(\widetilde{A}\nabla u)$ and
$M_{R,p,q}^{\prime}(B\colon\nabla^{2}u)$ are treated as in the proof of
Theorem~\ref{thm:local-regularity}. This is done in the following steps.

\emph{1.~step:} We note that for $p$, $\alpha\in{\mathbb{N}}_{0}$ the
function
\begin{equation}
\label{eq:binomial-monotonicity}r \mapsto\frac{(p-r+\alpha)!}{(p-r)!} =
(p-r+1)\cdots(p-r+\alpha) \qquad\mbox{ is monotone decreasing.}
\end{equation}

\emph{2.~step:} The induction hypothesis and the monotonicity assertion
(\ref{eq:binomial-monotonicity}) imply
\begin{align}
\label{eq:foo-1} &  \binom{p}{r}\binom{q}{s} \frac{s! r! (p+q-r-s)!}{(p+q)!}
\left[  N^{\prime}_{R,p-r+2,q-s-2}(u) + N^{\prime}_{R,p+1-r,q-s-1}(u)\right]
\\
&  \qquad\leq K_{1}^{p-r+2} K_{2}^{q-s+2} M_{p,q} \left[  K_{1}^{2} K_{2}^{-2}
+ K_{1} K_{2}^{-1}\right]  .\nonumber
\end{align}
To see this, one write $M_{p,q} = M_{p,q}^{(1)} + M_{p,q}^{(2)}$ as in
(\ref{eq:Mpq}). The terms involving $M_{p,q}^{(1)}$ are treated exactly as in
the treatment of $M^{\prime}_{R,p,q}(\sum_{j,\alpha,\beta\colon(\alpha,\beta)
\ne(3,3)} \widetilde{A}^{ij}_{\alpha\beta} \partial_{\alpha}\partial_{\beta
}{\mathbf{u}})$ in the proof of Theorem~\ref{thm:local-regularity}. The terms
involving $M_{p,q}^{(2)}$ are handled by noting that a two-fold application of
the monotonicity assertion (\ref{eq:binomial-monotonicity}) implies
\begin{equation}
\label{eq:foo-2}\binom{p}{r}\binom{q}{s} r!s! \frac{(p-r+q-s)!}{(p+q)!} \leq1.
\end{equation}
\emph{3.~step:} Analogously, as in the treating of the term $M^{\prime
}_{R,p,q}(\sum_{j,\alpha,\beta: (\alpha,\beta) \ne(3,3)} \widetilde{A}%
^{ij}_{\alpha\beta} \partial_{\alpha}\partial_{\beta}{\mathbf{u}}_{j})$ in the
proof of Theorem~\ref{thm:local-regularity} we get for $M^{\prime}%
_{R,p,q}(B:\nabla^{2} u)$ by observing that the assumption $B_{33} \equiv0$
leads to a sum of terms of the form (\ref{eq:foo-1})
\begin{align}
M^{\prime}_{R,p,q}(B : \nabla^{2} u)  &  \leq K_{1}^{p+2} K_{2}^{q+2} M_{p,q}
\left[  K_{1}^{2} K_{2}^{-2} + K_{1} K_{2}^{-1} \right]  C^{\prime}\sum_{r,s}
\left(  \frac{\gamma R}{2 K_{1}}\right)  ^{r} \left(  \frac{\gamma R}{2 K_{2}%
}\right)  ^{s}.
\end{align}
\emph{4.~step:} The induction hypothesis gives
\begin{align}
\label{eq:foo-10} &  \binom{p}{r}\binom{q}{s} \frac{s! r! [p+q-r-s-1]!}%
{(p+q)!} \left[  N^{\prime}_{R,p-r,q-s-1}(u) + N^{\prime}_{R,p+1-r,q-s-2}%
(u)\right] \nonumber\\
&  \qquad\leq3 K_{1}^{p-r+2} K_{2}^{q-s+2} \frac{M_{p,q}}{p+q+3} \left[
K_{2}^{-1} + K_{1} K_{2}^{-2}\right]  ;
\end{align}
in this estimate, for the contribution $M_{p,q}^{(1)}$ of $M_{p,q}$, the
calculation is as in the treatment of $M^{\prime}_{R,p,q}((\widetilde{B}%
^{ij}_{\beta} \partial_{\beta}\mathbf{u}_{j})_i)$; for the contribution
$M_{p,q}^{(2)}$ of $M_{p,q}$ one has to consider
\begin{align*}
\frac{p!}{(p-r)!} \frac{q!}{(q-s)!} \frac{[p-r+q-s-1]!}{(p+q)!} (p-r+q-s+3)
\leq3 \frac{p!}{(p-r)!} \frac{q!}{(q-s)!} \frac{(p-r+q-s)!}{(p+q)!}
\overset{(\ref{eq:foo-2})}{\leq} 3,
\end{align*}
where the term $(p-r+q-s+3)$ arises from $M_{p-r,q-s}^{(2)}$. Since
$M_{p,q}^{(2)} = (p+q+3) R \|\nabla u\|_{L^{2}(B^{+}_{R})}$, the result
(\ref{eq:foo-10}) follows.

\emph{5.~step:} As in the treatment of the terms $M_{R,p,q}^{\prime
}(\widetilde{B}_{\beta}^{ij}\partial_{\beta}{\mathbf{u}}_{j})$ in the proof of
Theorem~\ref{thm:local-regularity}, we get from (\ref{eq:foo-10}) the bound
\[
M_{R,p,q}^{\prime}(\widetilde{A}\nabla u)\leq3K_{1}^{p+2}K_{2}^{q+2}%
M_{p,q}\left[  K_{2}^{-1}+K_{1}K_{2}^{-2}\right]  \frac{C^{\prime}}{p+q+3}%
\sum_{r,s}\left(  \frac{\gamma R}{2K_{1}}\right)  ^{r}\left(  \frac{\gamma
R}{2K_{2}}\right)  ^{s}.
\]
\emph{6.~step:} Selecting $K_{2}$ sufficiently large depending solely on
$C^{\prime}$ and $\gamma$ completes the induction argument.
\end{proof}



\subsection*{Acknowledgements}

The authors cordially thank Joachim Sch\"{o}berl (TU Wien) for interesting
discussions on the topic.

The second author also wishes to thank the \textsl{Institut f\"{u}r Analysis
und Scientific Computing}, TU Wien, for its support during his sabbatical,
where parts of this work were carried out. He is also very thankful to family
Huber-Pieslinger for their warm hospitality during his stay in Vienna.
\bibliographystyle{abbrv}
\bibliography{maxwell}

\end{document}